\documentclass[leqno,11pt]{amsart}
 
\usepackage{amsmath}
\usepackage{amsthm}  
\usepackage{verbatim}
\usepackage{enumerate} 
\usepackage{mathtools}  
\usepackage{amssymb}
\usepackage{mathrsfs}
\usepackage[top=1.2in, bottom=1.2in, left=0.8in, right=0.8in]{geometry}  
\usepackage[colorlinks]{hyperref}
\hypersetup{
	colorlinks,
	citecolor=blue,
	linkcolor=red
}   
\numberwithin{equation}{section}
\usepackage{xypic}
\usepackage{bm}
\usepackage{tikz}

\newtheorem{theorem}{\textbf{Theorem}}[section]
\newtheorem{theorem*}{\textbf{Theorem}}
\newtheorem{definition}[theorem]{\textbf{Definition}}
\newtheorem{proposition}[theorem]{\textbf{Proposition}}
\newtheorem{lemma}[theorem]{\textbf{Lemma}}

\newtheorem{step}[theorem*]{\textbf{Step}}
\newtheorem{corollary}[theorem]{\textbf{Corollary}}
\newtheorem{remark}[theorem]{\textbf{Remark}}
\newtheorem{assumption}[theorem]{\textbf{Assumption}}
\newtheorem{example}[theorem]{\textbf{Example}}

\newtheorem{definition/proposition}[theorem]{\textbf{Definition/Proposition}}
\newtheorem{innercustomgeneric}{\customgenericname}
\providecommand{\customgenericname}{}
\newcommand{\newcustomtheorem}[2]{%
	\newenvironment{#1}[1]
	{%
		\renewcommand\customgenericname{#2}%
		\renewcommand\theinnercustomgeneric{##1}%
		\innercustomgeneric
	}
	{\endinnercustomgeneric}
}
\newcustomtheorem{customthm}{Theorem}
\newcustomtheorem{customprop}{Proposition}
\newcustomtheorem{customcoro}{Corollary}

\def\A{{\mathbb A}}
\def\N{{\mathbb N}}
\def\R{\mathbb{R}}
\def\Z{{\mathbb Z}}

\def\C{{\mathbb C}}
\def\E{{\mathbb E}}
\def\F{{\mathbb F}}

\def\W{{\mathbb W}}

\def\H{{\mathbb H}}
\def\T{{\mathbb T}}
\def\K{{\mathbb K}}
\def\M{{\mathbb M}}

\def\i{\iota}

\def\cA{{\mathcal A}}

\def\cC{{\mathcal C}}

\def\cE{{\mathcal E}}
\def\cF{{\mathcal F}}

\def\cH{{\mathcal H}}

\def\cL{{\mathcal L}}
\def\cM{{\mathcal M}}
\def\cN{{\mathcal N}}
\def\cO{{\mathcal O}}
\def\cP{{\mathcal P}}
\def\cQ{{\mathcal Q}}

\def\cT{{\mathcal T}}
\def\cU{{\mathcal U}}
\def\cV{{\mathcal V}}
\def\cW{{\mathcal W}}
\def\cX{{\mathcal X}}
\def\cY{{\mathcal Y}}
\def\cZ{{\mathcal Z}}

\def\bC{{\bm C}}

\def\bE{{\bm E}}
\def\bF{{\bm F}}

\def\bO{{\bm O}}

\def\bQ{{\bm Q}}

\def\bT{{\bm T}}
\def\bU{{\bm U}}
\def\bV{{\bm V}}
\def\bW{{\bm W}}
\def\bX{{\bm X}}
\def\bY{{\bm Y}}
\def\bZ{{\bm Z}}

\def\hX{{\widehat\cX}}

\def\hZ{{\widehat Z}}

\def\os{{\overline{s}}}

\def\rD{{\rm D}}

\def\rd{{\rm d}}

\def\sc{{\rm sc}}

\def\la{\langle\,}
\def\ra{\,\rangle}
\def\st{\: \big| \:}

\DeclareMathOperator{\Ima}{im}
\DeclareMathOperator{\ind}{ind}
\DeclareMathOperator{\Hom}{Hom}
\DeclareMathOperator{\Id}{id}
\DeclareMathOperator{\ID}{Id}

\DeclareMathOperator{\stab}{stab}
\DeclareMathOperator{\supp}{supp}
\DeclareMathOperator{\diff}{Diff}
\DeclareMathOperator{\coker}{coker}
\DeclareMathOperator{\mor}{Mor}
\DeclareMathOperator{\obj}{Obj}
\DeclareMathOperator{\Map}{Map}
\DeclareMathOperator{\Fix}{Fix}
\DeclareMathOperator{\rank}{rank}

\DeclareMathOperator{\codim}{codim}
\DeclareMathOperator{\End}{End}
\DeclareMathOperator{\eff}{eff}

\title{Quotient Theorems in Polyfold Theory and $S^1$-Equivariant Transversality}
\author{Zhengyi Zhou}

\begin{document}
	\maketitle
\begin{abstract}
We introduce group actions on polyfolds and polyfold bundles. We prove quotient theorems for polyfolds, when the group action has finite isotropy. We prove that the sc-Fredholm property is preserved under quotient if the base polyfold is infinite dimensional. The quotient construction is the main technical tool in the construction of equivariant fundamental class in \cite{equi}. We also analyze the equivariant transversality near the fixed locus in the polyfold setting. In the case of $S^1$-action with fixed locus, we give a sufficient condition for the existence of equivariant transverse perturbations. We outline the application to Hamiltonian-Floer cohomology and a proof of the weak Arnold conjecture for general symplectic manifolds, assuming the existence of Hamiltonian-Floer cohomology polyfolds.
\end{abstract}
\tableofcontents
\section{Introduction}
Polyfold theory \cite{HWZ1,HWZ3,hofer2010sc,HWZ2,hofer2017polyfold} developed by Hofer, Wysocki, and Zehnder provides an analytic framework to solve the regularization problem of moduli spaces in symplectic geometry. The moduli spaces in symplectic geometry, e.g. Gromov-Witten moduli spaces \cite{hofer2017applications} and Floer type moduli spaces \cite{hofer2017application,li,wehrheim2012fredholm}, can be realized as zero sets of sc-Fredholm sections of polyfold bundles. The abstract perturbation machinery in polyfold theory assures that the sc-Fredholm sections can be perturbed into transverse positions, such that the perturbed zero sets have smooth structure and remain compact. Such process is referred to as polyfold regularization and perturbed smooth compact moduli spaces are referred to as regularized moduli spaces. The integration theory developed in \cite{hofer2010integration} then provides meaningful ``counting" by integrating differential forms on the regularized moduli spaces.  Those numbers form atomic components of algebraic invariants like Gromov-Witten invariants and symplectic field theory (SFT).

In this paper, we study group actions on polyfold bundles and quotient constructions in polyfold theory. There are mainly two types of group actions appearing in moduli problems in symplectic geometry. One is actions induced by actions on the ambient space. As a typical example,  if a compact Lie group $G$ acts on the symplectic manifold preserving the symplectic form, then the $G$-action lifts to an action on the Gromov-Witten moduli spaces for any $G$-invariant almost complex structure. Such consideration leads to equivariant Gromov-Witten invariants developed in \cite{givental1996equivariant}. Many equivariant Floer theories \cite{bourgeois2016s,hendricks2016flexible,hendricks2016simplicial,seidel2014equivariant} also fit into this category. The other type is actions induced by reparametrization on the domain Riemann surfaces. A typical example is the $S^1$-action induced by rotating the $S^1$-coordinate of the cylinders in Hamiltonian-Floer theory when both the Hamiltonian and the almost complex structure are time independent. This $S^1$-action plays an important role in Floer's proof \cite{floer1989symplectic} of the weak Arnold conjecture. In this paper, we introduce group actions on polyfolds such that the group actions on the moduli spaces in examples above is a restriction of group actions on the ambient polyfolds. The goal of this paper is showing that the quotients of polyfold bundles exists uniquely and equivariant sc-Fredholm sections induce sc-Fredholm sections on quotients. 
 
\subsection{Main result}
Roughly speaking, a polyfold is a topological space $Z$ which contains a complicated moduli space $\cM:=s^{-1}(0)$ as zero set of a section $s:Z\to W$.  Polyfolds have level structures, i.e. we have a sequence of continuous inclusions of subsets equipped with different topology $\ldots Z_{i+1}\subset Z_{i}\subset \ldots \subset Z_0 := Z$ and $Z_\infty := \cap_{i\in \N} Z_i$ is dense in $Z_i$ for all $i\ge 0$. Then we can shift the levels up by $k$ to get another polyfold $Z^k$ by $(Z^k)_i:=Z_{k+i}$. All the information needed to regularize $\cM$ is encoded in $Z_\infty$, which contains $\cM$ and any of its regularizing perturbations. The level structure is required only to express in what sense the section $s:Z_\infty\to W_\infty$ is sc-smooth and sc-Fredholm, and forgetting the first finitely many levels is immaterial. Thus the regularization theory for $\cM\subset Z^k$ is independent of the shift $k\ge 0$. Our main result is the following quotient theorem after a level shift.

\begin{theorem}\label{thm:main}
	Let $Z$ be an infinite-dimensional tame polyfold, $p:W\to Z$ be a regular tame strong polyfold bundle (Definition \ref{def:regbund}) and $s:Z\to W$ a proper sc-Fredholm section.  Assume a compact Lie group $G$ acts on $p$ sc-smoothly (Definition \ref{def:actbundle}) such that the induced $G$-action on $Z$ only has finite isotropy and $s$ is $G$-equivariant. Then the following holds.
	\begin{enumerate}
		\item There is a $G$-invariant open set $\hZ\subset Z^2$ containing $Z_\infty$, such that induced projection $\overline{p}: p^{-1}(\hZ)/G\to\hZ/G$ is a strong polyfold bundle and the quotient map $\pi_G:p^{-1}(\hZ) \to p^{-1}(\hZ)/G $ is $\sc^\infty$ strong bundle map.
		\item $s$ induces a proper sc-Fredholm section $\os:\hZ^1/G\to p^{-1}(\hZ^1)/G$ by $\pi_G^*\os=s|_{\hZ^1}$.
		\item If $s$ is oriented and the $G$-action preserves the orientation (Definition \ref{def:preori}), then $\os$ is orientable.
	\end{enumerate} 
\end{theorem}

In contrast to group actions on Kuranishi structures considered in \cite{fukaya2017lie,pardon2016algebraic}, one of the major differences is that we do not assume equivariant local models for group actions. Our definition of group action is merely a functor from $G$ to the category of polyfold bundles, plus local representability (which is local in $G$ and does not assume any equivariant property) to characterize the sc-smoothness. In particular, any sc-smooth polyfold bundle map $G\times W \to W$ satisfying the obvious group property is a sc-smooth group action. 

Sc-Fredholm section was introduced in \cite{HWZ2} to replace the naive definition of Fredholm section such that implicit function theorem holds. The novel extra conditions in the sc-Fredholm notion is necessary as the naive definition does not imply  implicit function theorem \cite{counter}. Therefore preserving sc-Fredholm property is the most important property in any abstract constructions of polyfolds. Unlike the fiber product construction in \cite{ben2018fiber}, the quotient problem has some flexibility in its nature. It turns out the infinite dimensional assumption  in Theorem \ref{thm:main} provides enough flexibility such that sc-Fredholm property is preserved under quotient.

\subsection{Applications}
Theorem \ref{thm:main} provides an infrastructure to study equivariant theories in symplectic geometry using polyfolds. In the following, we discuss briefly some applications of the quotient theorem. 
\subsubsection{Equivariant fundamental class}
One consequence of Theorem \ref{thm:main} is that the Borel construction works in the polyfold category. Let $p:W\to Z$ be a regular strong polyfold bundle with a $G$-action and $s:Z\to W$ an equivariant oriented proper sc-Fredholm section. Assume $Z$ is infinite dimensional and has no boundary and the $G$-action preserves the orientation of $s$. We can construct an equivariant fundamental class using the Borel construction. That is we have a well-defined $H^*_G(pt)$ module map:
\begin{equation}\label{eqn:efc}
s_*:H_G^*(Z) \to H^{*-\ind s}_G(pt).
\end{equation}
The details of the construction and the precise statements will appear in \cite{equi}. As an application of the equivariant fundamental class, let $p_{A,g,m}:W_{A,g,m}\to Z_{A,g,m}$ be the Gromov-Witten polyfold bundle defined in \cite{hofer2017applications} for a symplectic manifold $(M,\omega)$ with a compatible almost complex structure $J$, where $A \in H_2(M)$ is a homology class and $g, m\in N$ are the genus and number of marked points. Assume $G$ acts on $M$ preserving $J$, then the $G$-action lifts to $p_{g,A,m}$ such that the sc-Fredholm section $s_{A,g,m}:Z_{A,g,m}\to W_{A,g,m}$ induced by the Cauchy-Riemann operator is $G$-equivariant. Then equivariant Gromov-Witten invariants can be defined as $$I_{A,g,m}:H_G^*(M)^{\otimes m} \to H^*_G(pt), \quad \alpha_1\otimes \ldots \otimes \alpha_m \mapsto {s_{A,g,m}}_*(ev_1^*\alpha_1\wedge \ldots \wedge ev_m^*\alpha_m),$$
where $ev_i$ is the evaluation map at $i$-th marked point. One property of \eqref{eqn:efc} is that when equivariant transversality holds, $s_*$ is the integration over the moduli space $s^{-1}(0)$. As a consequence, the equivariant Gromov-Witten invariants defined using $s_*$ coincides with the invariants in \cite{givental1996equivariant}. The other more important property of $s_*$ is the localization property for torus actions. Under some technical assumptions, which is satisfied in the Gromov-Witten polyfolds, we show in \cite{equi} that we have a formula in the following form:
$$s_*(\alpha) = s^{T^n}_*(\alpha \wedge e^{-1}_{T^n}(N)),$$
where $s^{T^n}:Z^{T^n} \to W^{T^n}$ is the restriction of $s$ to the fixed locus and $e^{-1}_{T^n}(N)$ is a class in the localized equivariant cohomology of $Z^{T^n}$, which can be understood as the inverse of the equivariant Euler class of the ``virtual normal bundle" $N$.

The Borel construction also yields construction for equivariant Floer theories. In \cite{MB}, we study Morse-Bott cohomology and equivariant cohomology in the framework of flow categories. In particular, we outline the generalization of equivariant theory to polyfolds with boundary and corners which is expected to give definitions of equivariant Floer cohomology in various settings.

\subsubsection{$S^1$-equivariant transversality and the weak Arnold conjecture}
In this paper, we explore under what conditions equivariant regularization can be achieved by a transverse, yet still equivariant perturbation of the section. In the case of Banach manifolds, if the group action only has finite isotropy, Cieliebak, Riera and Salamon \cite{cieliebak2003equivariant} constructed an Euler class by equivariant multivalued perturbations. Since multivalued perturbation has become a part of the perturbation package in polyfold theory \cite[Chapter 13]{hofer2017polyfold}, using Theorem \ref{thm:main} and the polyfold perturbation machinery, we know that if the group action only has finite isotropy then we can find equivariant multivalued transverse perturbations. 

Beyond the finite isotropy case, it is known that equivariant transversality is often obstructed even in finite dimensional case, e.g. see \cite{costenoble1992}. In this more general case, we analyze the equivariant transversality near the fixed locus in Section \ref{s5}. Theorem \ref{thm:equitran} gives a sufficient condition which guarantees the existence of equivariant transverse $\sc^+$-perturbations near the fixed locus. In the special case of $S^1$-action, since $S^1$-actions have finite isotropy away from the fixed locus,  a global equivariant transverse perturbation exists if equivariant transverse perturbations exist near the fixed locus.

\begin{corollary}\label{s1}
	Suppose $p:W\to Z$ is a regular tame strong polyfold bundle with a proper sc-Fredholm section $s$ and $Z$ is infinite dimensional. Assume $S^1$ acts on $p$ such that $s$ is equivariant and the tubular neighborhood assumption (Definition \ref{def:tube}) holds. Provided that for all weights $\lambda \in \N^+$ and $x\in s^{-1}(0)\cap Z^{S^1}$, we have $\ind \rD^\lambda s_x+2>\ind \rD^{S^1}s_x$. Then there exists a $S^1$-invariant neighborhood $\hZ\subset Z^3$ containing $Z_\infty$ and an equivariant $\sc^+$-multisection perturbation $\kappa$ on $\hZ$, such that 
	\begin{enumerate}
		\item $s+\kappa$ is transverse to zero and proper;
		\item there exist a $S^1$-invariant neighborhood $U\subset Z^3$ of $(Z^{S^1})^3$, such that $\kappa|_U$ is single valued.
	\end{enumerate}
\end{corollary}

In Section \ref{s6}, we apply Corollary \ref{s1} to Hamiltonian-Floer cohomology polyfolds with $C^2$ small time-independent Hamiltonians and time-independent almost complex structures. In this case, the $S^1$-action is given by reparametrization in the $S^1$-coordinate of the domain cylinder. In particular, we can generalize Floer's proof of the weak Arnold conjecture \cite{floer1989symplectic} to any closed symplectic manifold, once the construction of Hamiltonian-Floer polyfolds is completed.

\subsection{Organization and notation}
This paper is organized as follows. In section \ref{s2}, we prove a special form of Theorem \ref{thm:main} for free quotients of M-polyfolds. Sections \ref{s3} introduces group actions on polyfolds and proves Theorem \ref{thm:main} up to the assertion on orientation. Orientation is discussed in Section \ref{s4}. In Section \ref{s5}, we analyze the equivariant transversality on the fixed locus. Section \ref{s6} outlines the proof of the weak Arnold conjecture using Corollary \ref{s1}.
We will assume the basics of sc-Banach spaces and sc-calculus from \cite[Chapter 1]{hofer2017polyfold}. We use the following notation throughout this paper.
\begin{itemize}
	\item Blackboard bold letters $\E, \H, \F$ stand for sc-Banach spaces, \cite[Definition 1.1]{hofer2017polyfold}. $\E_i$ is the $i$-th level Banach space of $\E$ and $\E_\infty: = \cap_{i=0}^\infty \E_i$. Points in $\E_\infty$ are also referred to as smooth points. $||\cdot||_m$ denotes the norm on $\E_m$.
	\item For $k\ge 0$, $\E^k$ is the $k$ levels shifted sc-Banach space, i.e. $(\E^k)_i:=\E_{i+k}$.
	\item An open set $D$ in $\E$ is an open set $D\subset \E_0$. We define $D_k := D\cap \E_k$, which is also open set in $\E_k$ for any $k\in \N$. 
	\item Assume $f:\E\to \F$ is a $\sc^\infty$ map \cite[Definition 1.9]{hofer2017polyfold}. We denote by ${\rm D}f_a(b)\in \F_i$ the image of the linearization ${\rm D}f$ at the point $a\in \E_{i+1}$ with the input vector $b\in \E_i$.
\end{itemize}

\subsection*{Acknowledgements}
The results presented here are part of my Ph.D. thesis; I would like to express my deep gratitude to my thesis advisor Katrin Wehrheim for guidance, encouragement and enlightening discussions. I would like to thank Benjamin Filippenko for explaining his work on fiber products of polyfolds. Part of the writing was completed  during  my  stay  at  the  Institute  for  Advanced  Study  supported  by  the National Science Foundation under Grant No.  DMS-1638352.  It is a great pleasure to acknowledge the Institute for its warm hospitality. This paper is dedicated to the memory of Chenxue.

\section{Free Quotients of M-polyfolds}\label{s2}
M-polyfolds were introduced by Hofer, Wysocki and Zehnder \cite[Definition 2.8]{hofer2017polyfold} as a generalization of manifolds. In this section we prove that the quotient of a tame M-polyfold by a free group action is still a tame M-polyfold and an equivariant sc-Fredholm section descends to a sc-Fredholm section on the quotient. 

\begin{theorem}{\label{thm:free}}
Let $\cX$ be an infinite-dimensional tame M-polyfold  (Definition \ref{def:infty}),  $p :\cY \to \cX$ a tame strong M-polyfold bundle (Definition \ref{strbund}) and $s:\cX \to \cY$ a proper sc-Fredholm section (Definition \ref{def:fredchart}). Assume a compact Lie group $G$ acts on the tame strong M-polyfold bundle $p:\cY \to \cX$ sc-smoothly (Definition \ref{def:act}) such that the induced action $\rho_{\cX}$ on the base $\cX$ is free and the section $s$ is $G$-equivariant. Then the following holds.
\begin{enumerate}
	\item\label{part:freeone} There exists a $G$-invariant open set $\hX\subset \cX^2$ containing $\cX_\infty$, such that the induced projection $\overline{p}:p^{-1}(\hX)/G \to \hX/G$ is a tame strong M-polyfold bundle and the quotient map $\pi_G: p^{-1}(\hX) \to p^{-1}(\hX)/G$ is a sc-smooth strong bundle map.
	\item\label{part:freetwo} $s$ induces a proper sc-Fredholm section  $\os:\hX^1/G\to p^{-1}(\hX^1)/G$ by $\pi_G^*\os = s|_{\hX^1}$.
\end{enumerate}	
\end{theorem}
Part \eqref{part:freeone} of Theorem \ref{thm:free} is proven in Section \ref{subsec:free} and part \eqref{part:freetwo} is proven in Section \ref{subsec:freeFred}. The uniqueness of the M-polyfold structure is addressed in Proposition \ref{prop:unique}.
\begin{remark} A few remarks on the level shift are in order.
	\begin{itemize}
		\item Tame M-polyfolds are M-polyfolds with well-behaved boundary and corner structures, which play important roles in Floer-type theories like Hamiltonian-Floer cohomology and SFT.  The discussions on the boundary and corners of tame M-polyfolds resp. tame polyfolds can be found in \cite[\S 2.3]{hofer2017polyfold}.
		\item 	Like polyfolds, M-polyfolds have level structures $\ldots \subset \cX_i \subset \ldots \cX_1\subset \cX_0 = \cX$, such that $\cX_\infty:=\bigcap_{i\in \N}\cX_i$ is dense in every level $\cX_i$.  We can shift levels up by $k$ to get $\cX^k$, i.e. $(\cX^k)_i := \cX_{k+i}$. All the information needed to regularize $s^{-1}(0)$ is encoded in $\cX_\infty$. In particular, $s^{-1}(0)$ and any of its regularizing perturbations are contained in $\cX_\infty$.  
		\item 	The two levels shifted in part \eqref{part:freeone} of Theorem \ref{thm:free} come from the construction of the slices to the group action, see Lemma \ref{lemma:slice}. The extra level shift in part \eqref{part:freetwo} is from Lemma \ref{lemma:fred}.
		\item 	The level shift in Theorem \ref{thm:free} seems to be necessary. For example, we have the sc-Banach space $\E$ of continuous functions on $S^1$, i.e. $\E_i:= C^i(S^1)$. Then we have a sc-smooth $S^1$-action on $\E$ by $\theta \bullet f:= f(\theta + \cdot)$, for $f\in \E$ and $\theta \in S^1$. Let $\cX := \{f\in \E| f(\theta + \cdot) \ne f, \forall \theta \in S^1\}$. Then $\cX$ is an open subset of $\E$, hence a M-polyfold. The $S^1$-action restricted to $\cX$ is free. Then we can give the quotient $\cX^1/S^1$ a M-polyfold structure, since we need $C^1$-differentiability to write down the local slice conditions, see \cite[\S 2.2, \S 4.3]{fabert2016polyfolds}  for details. It is not clear whether one can give $\cX/S^1$ a M-polyfold structure such that the quotient map is sc-smooth.
		\item 	 It is not clear to us whether one can construct the quotient bundle and section by shifting only one level.  Since in applications, the invariants are derived from the zero set $s^{-1}(0)$, which is contained in $\cX_{\infty}$. Therefore, shifting one level and three levels do not make an essential difference. 
	\end{itemize}
\end{remark}	
\begin{remark}
	the infinite dimensional assumption  provides enough flexibility such that sc-Fredholm property is preserved under quotient, see Proposition \ref{prop:goodslice}. This condition is satisfied in all known applications \cite{hofer2017application,hofer2017applications,li,wehrheim2012fredholm}.
\end{remark}

\subsection{Free quotients of tame M-polyfolds and tame strong M-polyfold bundles}\label{subsec:free}
If a compact Lie group $G$ acts on a finite dimensional manifold $M$ freely, then the quotient $M/G$ is a smooth manifold, see e.g. \cite{lee2003smooth}. This section proves the analogue for tame M-polyfolds and tame M-polyfold bundles. The proof also provides a local prototype for our main theorem on polyfolds.

\subsubsection{M-polyfolds and M-polyfold bundles}
This subsection reviews some definitions from \cite{hofer2017polyfold} that will be crucial for our construction. Let $\R_+ := [0,\infty)$. We begin with the local models that generalize open subsets of $\R_+^m$ for manifolds with boundary and corner. A partial quadrant is defined to be $\R_+^m \times \E$.
For every $(r,e) \in \R_+^m\times \E$, the \textbf{degeneracy index} $d: \R_+^m \times \E\to \N$ is defined to be 
$$d(r,e) := \#\left\{i \in \{1,\ldots, m\}| \text{the $i$-th coordinate of }r \text{ is zero.}\right\}.$$
For every $x\in \R_+^m\times \E$, we can define the minimal linear subspace $(\R_+^m\times \E)_x \subset \R^m\times \E$ \cite[Definition 2.16]{hofer2017polyfold} as follows: If $x=(r_1,\ldots, r_m, e)\in \R_+^m \times \E$, then
\begin{equation}\label{eqn:reduced}(\R_+^m\times \E)_x :=  \{(v_1,\ldots, v_m, f)| v_i = 0 \text{ if } r_i = 0 \} \subset \R^m \times \E.\end{equation}
This can be understood as the tangent space of the intersection of all the faces containing $x$, i.e. the tangent space of the corner of degeneracy index $d(x)$. 

\begin{remark}
	The original definition of partial quadrant \cite[Definition 1.6]{hofer2017polyfold} is a closed convex subset $C\subset \E'$, such that there exists another sc-Banach space $\E$ and a linear sc-isomorphism $\Psi: \E' \to \R^m \times \E$ satisfying $\Psi(C) = \R_+^m \times \E$. Instead of including $\Psi$ in the discussion, we will work with the standard model  $\R_+^m \times \E$ to simplify notation.
\end{remark}

\begin{definition}[{\cite[Definition 2.17]{hofer2017polyfold}}]\label{splicing}
	Let $U$ be an open subset of $\R_+^m\times \E$. A $\sc^\infty$ map $r:U\to U$ is called a \textbf{tame sc-retraction} if the following conditions hold:
	\begin{itemize}
		\item $r\circ r = r$;
		\item $d(r(x))=d(x)$ for all $x\in U$;
		\item at every x in $r(U)_\infty:=r(U)\cap(\R_+^m\times \E)_\infty$, there exists a sc-subspace $\A \subset (\R_+^m\times \E)_x$, such that $\R^m\times \E=\rD r_x(\R^m\times \E)\oplus \A$.
	\end{itemize}
A pair $(O, \R_+^m\times \E)$ is called a \textbf{tame sc-retract} if there exists a tame sc-retraction $r$ on an open subset $U \subset \R_+^m\times \E$, such that $r(U) = O$.
\end{definition}
\begin{remark}
	A \textbf{sc-retraction} \cite[Definition 2.1]{hofer2017polyfold} is a sc-smooth map $r:U\to U$ such that $r\circ r = r$. The extra conditions for tameness guarantee well- behaved boundary and corner, see \cite[\S 2.3]{hofer2017polyfold}. 
\end{remark}
\begin{remark}
	In \cite{hofer2017polyfold}, a sc-retract is a tuple $(O,C,\E')$ with $C\subset \E'$ a partial quadrant. Since we fix the form of partial quadrants throughout this paper, we have simplified the notation for sc-retracts to $(O, C=\R_+^m\times \E)$.
\end{remark}

The notion of smoothness for maps between open subsets of $\R^m_+$ is generalized by polyfold theory in two ways: First, sc-smoothness for maps between open subsets of sc-Banach spaces is defined in \cite[Definition 1.9]{hofer2017polyfold}. Second, for maps between sc-retracts,  sc-smoothness is defined as follows.

\begin{definition}[{\cite[Definition 2.4]{hofer2017polyfold}}]\label{def:smooth}
	Let $(O,\R_+^m\times \E)$ and $(O', \R_+^{m'}\times \E')$ be two (tame) sc-retracts. A map $f:O \to O'$ is \textbf{sc-smooth} if $f\circ r: \R_+^m\times \E \supset U \to \R_+^{m'}\times \E'$ is sc-smooth, where $r$ is a sc-retraction for $(O, \R_+^m \times \E)$ and $U$ is an open subset of $\R_+^m\times \E$ such that $r(U) = O$. 
\end{definition}
This notion is well-defined by \cite[Proposition 2.3]{hofer2017polyfold}, i.e. the definition does not depend on the choice of sc-retraction $r$ and open set $U$.

\begin{definition}[{\cite[Definition 2.8, 2.19]{hofer2017polyfold}}]\label{Mpoly}
	Let $\cX$ be a topological space. A \textbf{tame M-polyfold chart} for $\cX$ is a triple $(\cO, \phi,(O,\R_+^m\times \E))$,  such that
	\begin{itemize}
		\item $(O,\R_+^m\times \E)$ is a tame sc-retract; 
		\item $\phi:\cO\to O$ is a homeomorphism from an open subset $\cO\subset \cX$.
	\end{itemize}
   Two tame M-polyfold charts $(\cO, \phi,(O,\R_+^m\times \E))$ and $(\cO', \phi',(O',\R_+^{m'}\times \E'))$ are compatible if $\phi'\circ \phi^{-1}$ resp. $\phi\circ \phi'^{-1}$ are sc-smooth map from $\phi(\cO\cap \cO')$ to $\phi'(\cO \cap \cO')$ resp. from $\phi'(\cO\cap \cO')$ to $\phi(\cO\cap \cO')$ in the sense of Definition \ref{def:smooth}. An atlas is a covering by compatible charts. A \textbf{tame M-polyfold structure} on $\cX$ is a maximal atlas of tame M-polyfold charts for $\cX$. A \textbf{tame M-polyfold} $\cX$ is a paracompact Hausdorff space with a tame M-polyfold structure. 
\end{definition}
\begin{remark}
	An \textbf{M-polyfold} \cite[Definition 2.8]{hofer2017polyfold} is a   paracompact Hausdorff space with a covering of compatible M-polyfold charts, i.e. charts modeled on general sc-retracts.
\end{remark}

The following remark defines the notion of the tangent spaces of M-polyfolds.
\begin{remark}
	The tangent space of a partial quadrant $T(\R_+^m\times \E)$ \cite[Definition 1.8]{hofer2017polyfold} is defined to be $(\R_+^m\times \E)^1 \times (\R^m \times \E)$. The tangent space of a sc-retract $(O,\R_+^m\times \E)$ is defined to be the sc-retract $(TO,T(\R_+^m\times \E))$, where $TO=Tr(TU)$ is the image of the tangent map $Tr:TU\to TU,(x,e) \mapsto (x, \rD r_x(e))$ for any choice of retraction $r:U\to U$ with $r(U)=O$. The tangent space of a tame M-polyfold $\cX$ is a M-polyfold $T\cX$ \cite[Proposition 2.5]{hofer2017polyfold} with charts $(T\cO, T\phi,(TO, T(\R_+^m\times \E)))$. Then the projection $\pi:T\cX \to \cX^1$ defines a tame M-polyfold bundle in the sense of Definition \ref{def:bund}. 
\end{remark}

\begin{definition}\label{def:infty}
	We say a tame M-polyfold $\cX$ is infinite dimensional if for every $x\in \cX_{\infty}$, the dimension of the tangent space $T_x\cX$ is infinite.
\end{definition}	

Next we review the notion of (strong) M-polyfold bundles, which generalizes the notion of vector bundles over manifolds.
\begin{definition}\label{def:bund}
		Let $\cX$ be a tame M-polyfold, $\cY$ a paracompact Hausdorff space and $p:\cY\to \cX$ a surjection with $p^{-1}(x)$ a vector space for every $x\in \cX$. A \textbf{tame bundle chart} for the bundle $p: \cY \to \cX$ is a tuple $(\cO, \Phi, (P, (\R_+^m\times\E)\times \F))$ such that the following holds:
		\begin{itemize}
			\item there exists an open subset $U\subset \R_+^m\times \E$ and a sc-smooth retraction  $R: U\times \F \to U \times \F, (x,f) \mapsto (r(x), \varrho(x)f)$, i.e. $R\circ R = R$,  where $\varrho(x)$ is a linear map from $\F$ to $\F$ and $r$ is a tame retraction;
			\item there is a homeomorphism $\phi:\cO \to r(U)$ such that $(\cO,\phi,(r(U), \R_+^m\times \E))$ is compatible chart for the M-polyfold $\cX$.
			\item $\Phi: p^{-1}(\cO) \to P:=R(U\times\F)$ is bundle isomorphism covering $\phi$, i.e. $\Phi$ is a homeomorphism such that $\pi\circ \Phi = \phi\circ p$ and $\Phi_x: p^{-1}(x) \to \Ima \varrho(\phi(x))$ is a linear isomorphism, where $\pi:r(U)\times \F \supset P\to r(U)$ is the projection.
		\end{itemize}
		Two tame bundle charts $(\cO, \Phi,(P, (\R_+^{m}\times \E) \times \F) )$ and $(\cO', \Phi', (P', (\R_+^{m'}\times \E') \times \F'))$ are compatible iff the transition maps $\Phi'\circ \Phi^{-1}: \Phi(p^{-1}(\cO \cap \cO')) \to  \Phi'(p^{-1}(\cO \cap \cO'))$ and  $\Phi\circ \Phi'^{-1}: \Phi'(p^{-1}(\cO \cap \cO')) \to  \Phi(p^{-1}(\cO \cap \cO'))$ are sc-smooth. Then $p: \cY \to \cX$ is a \textbf{tame M-polyfold bundle} iff it is equipped with a strong bundle atlas.
\end{definition}

Since fibers of tame M-polyfold bundles tend to be infinite dimensional vector spaces in applications, it will be necessary to have an extra ``strong" structure on the bundle, which will be used to formulate the notion of compact perturbation of a Fredholm section. In polyfold theory, it uses the following filtrations. Let $U\subset \R_+^m\times \E$ be a open subset and $\F$ another sc-Banach space. We define the non-symmetric product $U\lhd \F$ to be the set $U\times \F$ with extra structures of filtrations $(U\lhd \F)[i]$, which are defined to be the filtrations $((U\lhd \F)[i])_m:=U_m\oplus F_{m+i}$. In particular, $(U\lhd \F)[i]$ is an open subset of the partial quadrant $((\R_+^m\times \E) \lhd \F)[i]$.
\begin{definition}[{\cite[Definition 2.26]{hofer2017polyfold}} ]\label{strbund}
	Let $\cX$ be a tame M-polyfold, $\cY$ a paracompact Hausdorff space and $p:\cY\to \cX$ a surjection with $p^{-1}(x)$ a vector space for every $x\in \cX$. A \textbf{tame strong bundle chart} for the bundle $p: \cY \to \cX$ is a tuple $(\cO, \Phi, (P, (\R_+^m\times\E) \lhd \F))$ such that the following holds:
	\begin{itemize}
		\item there exists an open subset $U\subset \R_+^m\times \E$ and a retraction $R: U\lhd \F \to U \lhd \F, (x,f) \mapsto (r(x), \varrho(x)f)$, i.e. $R\circ R=R$, where $\varrho(x)$ is a linear map from $\F$ to $\F$ and $r$ is a tame retraction;
		\item $R[i]:=R|_{(U \lhd \F)[i]}:(U \lhd \F)[i] \to (U \lhd \F)[i]$ is sc-smooth for $i=0,1$;
		\item there is homeomorphism $\phi:\cO \to r(U)$ such that $(\cO,\phi,(r(U), \R_+^m\times \E))$ is compatible chart for the M-polyfold $\cX$.
		\item $\Phi: p^{-1}(\cO) \to P:=R((U\lhd \F)[0])$ is bundle isomorphism covering $\phi$.
	\end{itemize}
	Two strong bundle charts $(\cV, \Phi,(P, (\R_+^{m}\times \E) \lhd \F) )$ and $(\cV', \Phi', (P', (\R_+^{s'}\times \E') \lhd \F'))$ are compatible iff the transition maps $\Phi'\circ \Phi^{-1}[i]: \Phi(p^{-1}(\cO \cap \cO'))[i] \to  \Phi'(p^{-1}(\cO \cap \cO'))[i]$ and  $\Phi\circ \Phi'^{-1}[i]: \Phi'(p^{-1}(\cO \cap \cO'))[i] \to  \Phi(p^{-1}(\cO \cap \cO'))[i]$ are sc-smooth for $i=0,1$. Then $p: \cY \to \cX$ is a \textbf{tame strong M-polyfold bundle} iff it is equipped with a tame strong bundle atlas.
\end{definition}
Note that a tame strong bundle $\cY$ defines two tame M-polyfold bundles $\cY[i]$ over $\cX$ for $i=0,1$, and there is a $\sc^\infty$ bundle inclusion $\cY[1] \subset \cY[0] = \cY$ covering the identity on $\cX$. Let $p_a:\cY_a \to \cX_a$ and $p_b:\cY_b \to \cX_b$ be two M-polyfold bundles, then a sc-smooth map $F:\cY_a \to \cY_b$ is a bundle map covering a map $f:\cX_a\to \cX_b$ iff $p_b\circ F = f\circ p_a$ and $F$ is linear on each fiber $(\cY_a)_x$. If $p_a,p_b$ are strong bundles, then $F$ is a strong bundle map iff $F[i]:\cY_a[i] \to \cY_b[i]$ are sc-smooth bundle maps for $i=0,1$. A (strong) bundle map is a (strong) bundle isomorphism iff it admits a (strong) bundle map inverse.

\begin{remark}
	Sections of $\cY[1]$ are called $\sc^+$-sections \cite[Definition 2.27]{hofer2017polyfold}, which play the role of compact perturbations.
\end{remark}	

\begin{definition}\label{def:act}
	A \textbf{$\bm{\sc^\infty}$ $\bm{G}$-action} $\rho$ on a tame strong M-polyfold bundle $p:\cY\to\cX$ is a strong bundle map $\rho: G\times \cY \to \cY$, such that $\rho(h, \rho(g, y)) = \rho(hg, y)$ for $h,g\in G$, $y\in \cY$.
\end{definition}

Given a group action $\rho$, since $\rho: G\times \cY \to \cY$ is a bundle map, it induces a map $\rho_{\cX}:G\times \cX\to \cX$ on the base. To see $\rho_{\cX}$ is sc-smooth near a point $(g,x) \in G\times \cX$, we first choose $\sc^\infty$ section $s$ defined on a neighborhood of $x$. Then $\rho_{\cX}(h,z) = p(\rho(h, s(z)))$, which is sc-smooth.

\subsubsection{Slices of free group actions}
Let $p:\cY \to \cX$ be a tame strong M-polyfold bundle. Assume $\rho$ is a sc-smooth $G$ action on $p$ for a compact Lie group $G$. The core of the proof of Theorem \ref{thm:free} is finding a slice $\tilde{O}\subset \cX$ to the group action near every point $x_0\in \cX_\infty$, such that $\tilde{O}$ is transverse to the orbits of the group action. When $\cX$ is a smooth manifold, we require slices to be submanifolds. There are several notions of ``smooth subset" of M-polyfolds, which generalize the notion of submanifold, e.g. \cite[Definition 2.12]{hofer2017polyfold}. We will work with the following notion from \cite{ben2018fiber} in Definition \ref{def:slice}.
\begin{definition}[{\cite[Definition 3.2, 3.4]{ben2018fiber}}]\label{def:rn}
	Consider an open subset $U\subset \R_+^m\times \R^n \times \E$ for some $n\ge 0$, a tame sc-retraction $r:U\to U$ is called  \textbf{$\bm{\R^n}$-sliced} if it satisfies $\pi_{\R^n} \circ r = \pi_{\R^n}$. A tame strong bundle retraction $R: U\lhd \F \to U \lhd \F$ is \textbf{$\bm{\R^n}$-sliced} if it covers a $\R^n$-sliced sc-retraction on $U$.
\end{definition}

An important consequence of Definition \ref{def:rn} is that the restriction of $r$ resp. $R$ to $\R_+^m\times \{v\}\times \E$ resp. $(\R_+^m\times \{v\} \times \E) \lhd \F$ is a tame retraction resp. tame strong bundle retraction for every $v$.
\begin{lemma}[{\cite[Lemma 3.3, 3.5]{ben2018fiber}}]\label{lemma:ben}
	Given a $\R^n$-sliced tame retraction $r: \R_+^m\times\R^n\times \E \supset U \to U$, let $\tilde{U}:=(\R_+^m\times\{0\}\times \E)\cap U$. Then the restriction $\tilde{r}:= r|_{\tilde{U}}:\tilde{U}\to \tilde{U}$ is a tame retraction. Given a $\R^n$-sliced tame strong bundle retraction $R: U\lhd \F \to U \lhd \F$,  the restriction $\tilde{R}:= R|_ {\tilde{U}\lhd \F}: \tilde{U}\lhd \F \to \tilde{U} \lhd \F$ is a tame strong bundle retraction.
\end{lemma} 

\begin{definition}\label{def:slice}
	Let $\cX$ be tame M-polyfold. Then a subset $\tilde{\cO} \subset \cX$ is called a \textbf{slice} if there exists a tame chart $(\cO,\phi, (O,\R_+^m\times \R^n \times \E))$ such that $O$ is defined by a $\R^n$-sliced retraction $r:U\to U$ and $\tilde{\cO} = \phi^{-1}\circ \tilde{r}(\tilde{U}):=\phi^{-1}\circ r((\R_+^m\times\{0\}\times \E) \cap U)$. 
	
	Let $p:\cY\to \cX$ be a tame strong M-polyfold bundle and $\tilde{\cO}$ a subset of $\cX$. The subset $p^{-1}(\tilde{\cO})\subset \cY$ is called a \textbf{bundle slice} if there exists a tame strong bundle chart $(\cO,\Phi, (P,(\R_+^m\times \R^n\times E)\lhd \F))$ such that $P$ is defined by a $\R^n$-sliced bundle retraction $R:U\lhd \F \to U\lhd \F$ and $p^{-1}(\tilde{\cO}) = \Phi^{-1}\circ \tilde{R}(\tilde{U}\lhd F):= \Phi^{-1}\circ R(((\R_+^m\times\{0\}\times \E) \cap U)\lhd \F)$. 
\end{definition}
For a bundle slice $p^{-1}(\tilde{\cO})$, $p|_{p^{-1}(\tilde{\cO})}:p^{-1}(\tilde{\cO}) \to \tilde{\cO}$ is a tame strong M-polyfold bundle by Lemma \ref{lemma:ben}. In particular, slice $\tilde{\cO}$ is a tame M-polyfold. The following normal form for submersions on retracts from  \cite{ben2018fiber} will be used to construct the slices for the quotient. To state the lemma, we first recall the \textbf{reduced tangent space} from \cite{hofer2017polyfold}. Let $(O,\R_+^m\times \E)$ be a sc-retract. For every $x\in O_\infty$, the reduced tangent space \cite[Definition 2.15]{hofer2017polyfold} is the subspace:
$$T^R_xO := T_xO \cap (\R_+^m\times \E)_x \subset T_xO,$$
where $(\R_+^m\times \E)_x$ is defined in \eqref{eqn:reduced}. In particular, if $(0,0)\in O \subset \R_+^m\times \E$, we have $T^R_{(0,0)}O \subset \{0\}\times \E$.
\begin{remark}\label{rmk:reducetangent}
	$T^R_xO$ is invariant under sc-diffeomorphism \cite[Proposition 2.8]{hofer2017polyfold}. Therefore for a M-polyfold $\cX$ and $x\in \cX_\infty$, we can define $T^R_x\cX:= \rd \phi^{-1}(T^R_yO) \subset T_x\cX$ for any chart $(\cO, \phi, (O,\R_+^m\times \E))$ with $\phi(x) = y \in O$. $T^R_x\cX$ is called the \textbf{reduced tangent space} \cite[Definition 2.20]{hofer2017polyfold}.
\end{remark}
\begin{lemma}[{\cite[Lemma 4.2, Remark 4.3]{ben2018fiber}}]\label{lemma:benslice}
Consider a tame sc-retract $(O, \R_+^m\times \E)$ containing $(0,0)\in O_\infty$ and a sc-smooth map $f: O \to \R^n$. Suppose  that $f(0,0) = 0$ and the restriction of the tangent map $\rD f_{(0,0)}|_{T^R_{(0,0)} O}: T^R_{(0,0)}O \to \R^n$ is surjective. Let $\K$ denote $\ker \rD(f\circ r)_{(0,0)}\cap (\{0\}\times \E)$, where $r$ is a retraction for the sc-retract $(O,\R_+^m\times \E)$. Then we can view $\K$ as a subspace of $\E$ of codimension $n$. Assume $L\subset (T^R_{(0,0)}O)_\infty$ is a complement of $\K$ in $\E$. Then there exists neighborhoods $U\subset \R_+^m\times \E^1$ of $(0,0)$ and $U' \subset \R_+^m\times \R^n\times \K^1$ of $(0,0,0)$, such that there exists a sc-diffeomorphism $h:U\to U'$ with the following properties.
\begin{itemize}
	\item $h(0,0) = (0,0,0)$.
	\item $f\circ r \circ h^{-1}:  \R_+^m \times \R^n \times \K^1\supset U' \to \R^n$ is the projection to $\R^n$.
	\item $h\circ r \circ h^{-1}$ is a $\R^n$-sliced retraction on $U'$. In particular,  $f^{-1}(0)\cap U$ is a slice of $O^1$.
	\item\label{b4} $\rD h_{(0,0)}(\{0\}\times L) = \{0\} \times \R^n\times \{0\}$ and $\rD h_{(0,0)}(\{0\}\times \K^1) = \{0\}\times \{0\} \times \K^1$.
\end{itemize} 
\end{lemma}
\begin{remark}
	The existence of complement $L$ is guaranteed, see \cite[Lemma 2.2]{hofer2017polyfold}. 
\end{remark}
From the perspective of Lemma \ref{lemma:benslice}, in order to construct slices for the quotient, we need to construct submersive maps to $\R^{\dim G}$. Such maps will not be globally defined on $\cX$ in general. In fact, we can only construct sc-smooth submersive maps near every smooth point $x_0\in\cX_\infty$.  

\begin{definition}
A M-polyfold chart $(\cO, \phi, (O,\R_+^m\times \E))$ is around $x_0\in \cX$ if $\phi(x_0) = (0,0) \in \R_+^m\times \E$.
\end{definition}
It is clear that the existence of M-polyfold chart around $x_0$ is equivalent to $x_0\in \cX_\infty$. Let $\rho_{\cX}:G\times \cX \to \cX$ be a sc-smooth action on the M-polyfold $\cX$. Given a M-polyfold chart around $x_0$ and a neighborhood $B\subset G$ of $\Id$, we have a sc-smooth map that parametrizes the orbit through $x_0$ locally in this chart,
\begin{equation}\label{eqn:b}
\gamma: B\to \R_+^m\times \E,   \qquad g \mapsto \phi \circ \rho_{\cX}(g, x_0).
\end{equation}
The \textbf{infinitesimal directions} of the $G$-action at $x_0$ in this chart is the subspace
\begin{equation}\label{eqn:xi}
\rD \gamma_{\Id}(T_{\Id} B) \subset  (TO_{(0,0)})_\infty \subset\R^m \times \E_\infty.
\end{equation}

\begin{proposition}\label{prop:red}
	$\rD \gamma_{\Id}(T_{\Id} B) \subset (T^R_{(0,0)}O)_\infty\subset \{0\}\times \E_\infty \simeq \E_\infty$. 
\end{proposition}
\begin{proof}
	For each $g\in G$, the map $\rho_{\cX}(g,\cdot)$ is a sc-diffeomorphism. By \cite[Proposition 2.8, 2.10]{hofer2017polyfold}, $\rho(g,\cdot)$ preserves the degeneracy index. In particular, we have $\phi\circ \rho_{\cX}(g,x_0)\in \{0\} \times \E \subset \R_+^m \times \E$. Therefore $\rD \gamma_{\Id}(T_{\Id} B) \subset TO_{(0,0)}\cap (\{0\} \times \E) = T^R_{(0,0)}O\subset \{0\}\times \E  \simeq \E.$
\end{proof}
\begin{proposition}\label{prop:free}
	If $\rho_{\cX}$ is free, then infinitesimal directions at $x_0\in \cX_\infty$ in a chart $(\cO,\phi,(O,\R_+^m\times \E))$ around $x_0$ is of dimension $\dim G$.
\end{proposition}
\begin{proof}
	Assume otherwise, that is there exists $\xi\in T_{\Id}G$ such that $\rD \gamma_{\Id}(\xi) = 0$. Then $\xi(t):=\rho_{\cX}(\exp(t\xi),x_0):(-\epsilon, \epsilon) \to \cX$ is a sc-smooth map with the property that $\rD \xi(t) = 0$ because of the group property $\rho_{\cX}(\exp(t\xi),\rho_{\cX}((\exp s\xi),x_0))=\rho_{\cX}(\exp(t\xi+s\xi), x_0)$. By \cite[Proposition 1.7]{hofer2017polyfold}, $\phi\circ \xi:(-\epsilon,\epsilon)^1 = (-\epsilon,\epsilon) \to \R_+^m \times \E$ is $C^1$. Then $\xi(t) \equiv x_0$, contradicting the free action assumption.  
\end{proof}

Since $\rD \gamma_{\Id}(T_{\Id} B)\subset \E_\infty$ is finite dimensional, by \cite[Proposition 1.1]{hofer2017polyfold} there exists a sc-complement $\H$ of $\rD \gamma_{\Id}(T_{\Id} B)$ in $\E$. In Lemma \ref{lemma:slice} below, we prove that an open subset of the shifted space $\phi^{-1}(O\cap (\R_+^m\times \H^2)) \subset \cX^2$ is a slice. First, we use Lemma \ref{lemma:benslice} to prove the following technical lemma, which is used in constructing the slices to the group actions in both M-polyfold case (Lemma \ref{lemma:slice}) and polyfold case (Proposition \ref{prop:localmodel}).
\begin{lemma}\label{lemma:gamma}
	Let $p:\cY\to \cX$ be a tame M-polyfold and $(\cO, \Phi, (P, (\R_+^m\times \E)\lhd \F))$ a tame strong bundle chart around $x_0 \in \cX_\infty$ covering a tame M-polyfold chart $(\cO,\phi, (O,\R_+^m\times \E))$. Let $r$ be a tame retraction and $R$ a strong bundle retraction covering $r$, such that $O = r(U)$ and $P = R(U\lhd \F)$ for an open neighborhood $U\subset \R_+^m\times \E$ of $(0,0)$. For a neighborhood $B\subset \R^n$ of $0$, assume $\Lambda: B\times p^{-1}(\cO) \to \cY$ is a sc-smooth strong bundle map such that $\Lambda(0,\cdot)|_{p^{-1}(\cO)} = \Id_{p^{-1}(\cO)}$. 
	Let $\Gamma:B\times \cO \to \cX$ denote the induced map on the base covered by $\Lambda$. Suppose that $\rD \Gamma_{(0,x_0)}(T_{0}B\times \{0\}) \subset (T^R_{x_0}\cO)_\infty$ and is of dimension $n$. Let $\Xi:= \rD (\phi\circ\Gamma)_{(0,x_0)}(T_{0}B\times \{0\}) \subset \{0\}\times \E_\infty\simeq \E_\infty$ and $\H$ any sc-complement of $\Xi$ in $\E$. Then there exist open neighborhoods $\cO' \subset \cO^2$ of $x_0$,$V\subset B$ of $0$, such that the following hold.
	\begin{enumerate} 
		\item \label{part:4} 
		Let $\K := \left(\H \cap T^R_{(0,0)} O^1 \right)\oplus \left(\ker \rD r_{(0,0)} \cap (\{0\} \times \E^1)\right) \subset \E^1$. There exist neighborhoods $U'\subset U^2$ of $(0,0)$, $U''\subset \R_+^m\times \R^n\times \K^1$ of  $(0,0,0)$ and a sc-diffeomorphism $h:U'\to U''$, such that $h\circ r \circ h^{-1}:U''\to U''$ is a $\R^n$-sliced retraction. Moreover $\cO'= \phi^{-1} \circ r(U')$.
		\item\label{part:5}$\rD h_{(0,0)}(\Xi) = \{0\} \times \R^n \times \{0\}$, $\rD h_{(0,0)}(\{0\}\times \K^1) = \{0\}\times \{0\}\times \K^1$.
		\item\label{part:1} There exists a sc-smooth map $f:\cO' \to V$ such that $ \tilde{\cO}' := \phi^{-1}\circ r \circ h^{-1}((\R_+^m\times \{0\}\times \K^1)\cap U'') = f^{-1}(0)$ is a slice of $\cX^2$ containing $x_0$ and $p^{-1}(\tilde{\cO}')$ is a bundle slice of $p:\cY^2\to \cX^2$. 
		\item\label{part:unique} For $x\in \cO'$, $g=f(x)$ is the unique element $g\in V$, such that $\Gamma(g,x)\in \tilde{\cO}'$.
		\item\label{part:3} $\eta:\cO'\to \tilde{\cO}'$ defined by $x \mapsto \Gamma(f(x),x)$ is sc-smooth. $N:p^{-1}(\cO') \to p^{-1}(\tilde{\cO}')$ defined by $v \mapsto \Lambda(f(p(v)),v)$ is a sc-smooth strong bundle map.
		\item\label{part:6} $\tilde{\cO}' = \cO' \cap \phi^{-1}(O \cap (\R_+^m \times \H))$.
	\end{enumerate}
\end{lemma}
\begin{proof}
    Let $\pi_{\Xi}, \pi_{\H}$ be the projections to $\Xi$ and $\H$ in $\E = \Xi\oplus \H$. We define $\Sigma:= B\times U \to \cX, (g,u)\mapsto \Gamma(g, \phi^{-1}\circ r (u))$. Let $W$ be the preimage set $\Sigma^{-1}(\cO)$, which is an open neighborhood of $(0,0,0)\in B\times U$. We consider the well-defined sc-smooth map $q: U\times B \supset W \to \Xi$:
	$$q: (g,u) \mapsto \pi_{\Xi}\circ \phi \circ \Sigma(g,u)= \pi_{\Xi}\circ  \phi \circ    \Gamma(g, \phi^{-1}\circ r(u)).$$
    Since $\rD q_{(0,0,0)}|_{T_0B \times \{0\}\times \{0\}} = \pi_\Xi\circ \rD(\phi\circ \Gamma)_{(0,x_0)}|_{T_0B\times \{0\}}$, the assumption $\rD(\phi\circ \Gamma)_{(0,x_0)}(T_0B\times\{0\}) = \Xi$ implies that 
    \begin{equation}\label{eqn:iso}
    \rD q_{(0,0,0)}|_{T_0B\times \{0\}\times \{0\}} \text{ is an isomorphism on to }\Xi.
    \end{equation}
    Then we can apply Lemma \ref{lemma:help} to get open neighborhoods $U^* \subset U^1$ of $(0,0)$ and $V\subset B$ of $0$, and a sc-smooth function $t:U^* \to V$, such that 
	\begin{equation}\label{eqn:unique}
	\text{$q(v,u) = 0$ for $u\in U^*, v\in V$ iff $v = t(u)$;}
	\end{equation}
	\begin{equation}\label{eqn:linear}
	\rD t_{(0,0)}(u) = (\rD q_{(0,0,0)}|_{T_0V\times \{0\}\times \{0\}})^{-1} \circ \rD q_{(0,0,0)}(0,u),\quad \forall u\in \R^m\times \E^1.
	\end{equation}
	Note that $\Gamma(0,\cdot) = \Id_{\cO}$ by assumption. Then we have
	\begin{equation}\label{eqn:linearofq}
	\rD q_{(0,0,0)}(0,u) = \pi_{\Xi}\circ \rD r_{(0,0)}(u), \quad \forall u \in \R^m\times \E^1.
	\end{equation}
	Since $\rD r_{(0,0)}|_{\{0\}\times \Xi} = \Id_{\Xi}$, \eqref{eqn:linear} and \eqref{eqn:linearofq} imply that
	\begin{equation}\label{eqn:surj}\rD t_{(0,0)}|_{\{0\}\times\Xi} \text{ is surjective onto }T_{0} V.\end{equation}
	Note that $q(r(x), g) = q(x,g)$, hence \eqref{eqn:unique} implies that $t \circ r = t$ on $U^*\times V$. Hence $t|_{r(U^*)}$ is a sc-smooth function.
	
	Since $\Xi \subset T^R_{(0,0)} O$ by assumption, \eqref{eqn:surj} implies that $\rD t_{(0,0)}|_{T^R_{(0,0)}O}: T^R_{(0,0)} O\to T_{0}V$ is surjective. We can apply Lemma \ref{lemma:benslice} to $t$ as follows: There exist open neighborhoods $U' \subset (U^*)^1 \subset U^2$ of $(0,0)$, $U''\subset \R_+^m\times \R^n \times \K^1$ of $(0,0,0)$ and a sc-diffeomorphism $h:U' \to U''$ such that $h\circ r \circ h^{-1}$ is a $\R^n$-sliced retraction and $t \circ r \circ  h^{-1}$ is the projection to $\R^n$, where 
	\begin{equation}\label{eqn:K}
	\K := \ker \rD (t\circ r)_{(0,0)} \cap(\{0\}\times \E^1) = \ker \rD t_{(0,0)} \cap (\{0\}\times \E^1).
	\end{equation}
	Then \eqref{eqn:linearofq} implies that $\K = \left(\H \cap T^R_{(0,0)} O^1 \right)\oplus \left(\ker \rD r_{(0,0)} \cap (\{0\} \times \E^1)\right)$ as in the statement of the proposition. Let $O':= r(U'), O'':= h\circ r \circ h^{-1}(U'')$ and $\cO' := \phi^{-1}(O') = \phi^{-1}\circ h^{-1}(O'')\subset \cX^2$. We also define $H := h\lhd \Id_{\F^2}$ and $R'':= H\circ R \circ H^{-1}$ and $P''= R''(U''\lhd \F^2)$. Then $(\cO', H\circ \Phi, (P'', (\R_+^m \times \R^n\times \K^1) \lhd \F^2))$ is a tame strong bundle chart for $p:\cY^2\to \cX^2$ around $x_0$ covering a M-polyfold chart $(\cO', h\circ \phi, (O'',\R_+^m\times \R^n \times \K^1))$. Therefore by Lemma \ref{lemma:ben}, $\tilde{\cO}' :=\phi^{-1}(t^{-1}(0)\cap O')$ is a slice of $\cX^2$ and $p^{-1}(\tilde{\cO}')$ is a bundle slice. So far, we have proven Property \eqref{part:4} and \eqref{part:1}. Property \eqref{part:5} follows from Lemma \ref{lemma:benslice}.
	
	Let $f:\cO' \to V$ be the sc-smooth map defined by $x \mapsto t\circ \phi(x)$. Let $g \in V$ and $x\in \cO'$. Then $\Gamma(g,x) \in f^{-1}(0) = \tilde{\cO}'$ is equivalent to $q(0, \phi(\Gamma(g,x))) = 0$ by \eqref{eqn:unique}, that is $\pi_\Xi\circ \phi \circ \Gamma(0, \Gamma(g,x)) = \pi_\Xi\circ \phi \circ \Gamma(g,x) = 0$. Therefore $\Gamma(g,x)\in \tilde{\cO}'$ is equivalent to $(g,\phi(x))$ is a solution to $q = 0$ in $V\times U'$. Therefore by \eqref{eqn:unique}, $g=t\circ \phi(x) = f(x)$. Hence property \eqref{part:unique} holds. 
	
	To see $\eta(x) := \Gamma(f(x),x)$ is sc-smooth. By chain rule \cite[Theorem 1.1]{hofer2017polyfold}, $\eta$ is a sc-smooth map from $\cO'$ to $\cO'$. In chart $U''$, that is $h\circ \phi\circ \eta \circ \phi^{-1}\circ r \circ h^{-1}: U'' \to \R_+^m\times \R^n \times \K^1$ is sc-smooth. Since $\Ima(h\circ \phi\circ \eta \circ \phi^{-1}\circ r \circ h^{-1}) \subset \R_+^m\times \{0\} \times \K^1$, $h\circ \phi\circ \eta \circ \phi^{-1}\circ r \circ h^{-1}$ is also a sc-smooth map from $U''$ to $\R_+^m\times \{0\} \times \K^1$. That is $\eta:\cO \to \tilde{\cO}'$ is sc-smooth. The sc-smoothness of $N:=\rho(f(p(x)),x)$ follows from the same argument. This proves property \eqref{part:3}.
	
	To show property \eqref{part:6}, for $x\in O'$, by \eqref{eqn:unique} $t(x) = 0$ is equivalent to $0 = q(0,x) = \pi_{\Xi}\circ \phi\circ (0, \phi^{-1}\circ r(x)) = \pi_{\Xi}(x)$, i.e. $x\in \H$. Since $\tilde{\cO}' = \{\phi^{-1}(x)|t(x) = 0, x\in O\}$, hence $\tilde{\cO}' =\cO'\cap \phi^{-1}(O\cap(\R_+^m\times \H))$. 
\end{proof}
\begin{remark}
	Property \eqref{part:1} - \eqref{part:3} in Lemma \ref{lemma:gamma} are directly used in the construction of $G$-slices, e.g. Lemma \ref{lemma:slice}, Proposition \ref{prop:localmodel} and Proposition \ref{prop:localmodelbundle}. Property \eqref{part:4} and \eqref{part:5} is used to get $G$-slices which is also good in the sense of Definition \ref{def:good}, see Proposition \ref{prop:goodslice}.
\end{remark}

\begin{definition}\label{def:Gslice}
	Let $\rho$ be a sc-smooth action on tame strong M-polyfold bundle $p:\cY \to \cX$. For every $x_0$, a \textbf{bundle $\bm{G}$-slice of $\bm{p}$ around $\bm{x_0}$} is a tuple $(\tilde{\cO},\cO, V,f,\eta,N)$ such that the following holds.
	\begin{itemize}
		\item $\cO$ is open subset of $\cX$ and $V\subset G$ is an open neighborhood of $\Id$. 
		\item $f:\cO \to V$ is a sc-smooth map, such that $x_0\in \tilde{\cO} := f^{-1}(\Id)$ and $p^{-1}(\tilde{\cO})$ is a bundle slice.
		\item For $x\in \cO$, $g=f(x)$ is the unique element $g\in V$ such that $\rho_{\cX}(g,x) \in \tilde{\cO}$.
		\item $\eta:\cO \to \tilde{\cO}$ defined by $x\mapsto \rho_{\cX}(f(x),x)$ is sc-smooth. $N:p^{-1}(\cO) \to p^{-1}(\tilde{\cO})$ defined by $v \mapsto \rho(f(p(v)), v)$ is a sc-smooth strong bundle map.
		\item $\psi: \tilde{\cO} \to \cO \stackrel{\pi_G}{\to} \cX/G$ is injective.
	\end{itemize}
\end{definition}

\begin{lemma}\label{lemma:slice}
	Let $\rho$ be a sc-smooth action on the tame strong M-polyfold bundle $p:\cY \to \cX$ such that $\rho_{\cX}$ is free. Then there exists a bundle $G$-slice of $p:\cY^2\to \cX^2$ around every $x_0\in \cX_\infty$.
\end{lemma}
\begin{proof}
	For $x_0 \in \cX_\infty$, let $(\cO,\Phi,(P,(\R_+^m\times \E)\lhd \F))$ be a tame strong M-polyfold bundle chart around $x_0$ covering a tame M-polyfold chart $(\cO, \phi, (O,\R_+^m\times \E))$. Then we define
	$$\Lambda: B\times p^{-1}(\cO) \to \cY, \quad (g,v) \mapsto \rho(g,v),$$
	where $B\subset G$ is a neighborhood of $\Id$. Then $\Lambda(\Id, \cdot)|_{p^{-1}(\cO)} = \Id_{p^{-1}(\cO)}$ and $\Lambda$ covers the map $$\Gamma:B\times \cO \to \cX, (g,x)\mapsto \rho_{\cX}(g,x).$$
	Since the group action is free, by Proposition \ref{prop:red} and Proposition \ref{prop:free} $\rD\Gamma_{(\Id,x_0)}(T_{\Id}B\times \{0\}) = \rD \gamma_{\Id}(T_{\Id}B) \subset T^R_{x_0}\cO$ and $\dim \rD\Gamma_{(\Id,x_0)}(T_{\Id}B\times \{0\})=\dim B$. Hence we can apply Lemma \ref{lemma:gamma} to $\Lambda$. As a consequence, we have $(\tilde{\cO}',\cO', V,f,\eta,N)$ satisfying all the conditions of a bundle $G$-slice except the injective condition. Moreover, by Definition \ref{def:Gslice}, for every open neighborhood $\tilde{\cO}''\subset \tilde{\cO}'$ of $x_0$, let $\cO''= \eta^{-1}(\tilde{\cO}'')$, then $(\tilde{\cO}'',\cO'',V,f,\eta,N)$ satisfies all but the injective condition of a bundle $G$-slice.
	
	We claim that there is an open neighborhood $\tilde{\cO}''\subset \tilde{\cO}'$ of $x_0$, such that $\psi|_{\tilde{\cO}''}$ is injective. Assume otherwise, that is there exist $x_n,y_n\in \tilde{\cO}'$ converging to $x_0$ and $g_n \ne \Id\in G$ such that $\rho_{\cX}(g_n,x_n)=y_n$. Since $G$ is compact, we have $\displaystyle\lim_{i\to \infty} g_{n_i} = g_0$ for a subsequence $\{n_i\}_{i\in \N}$. The continuity of  $\rho_{\cX}$ implies that $\rho_{\cX}(g_0,x_0)=x_0$. Since the group action is free, we have $g_0=\Id$. In particular, there exists $n\in \N$ such that $g_n\in V$. Then both $x_n=\rho_{\cX}(\Id, x_n)$ and $y_n=\rho_{\cX}(g_n ,x_n)$ are in $\tilde{\cO}$, which contradicts property \eqref{part:unique} of Lemma \ref{lemma:gamma}. As a consequence, $(\eta^{-1}(\tilde{\cO}''), \tilde{\cO}'',V,f,\eta, N)$ satisfying all the conditions of a bundle $G$-slice.
\end{proof}

Although a general implicit function theorem does not exists in sc-calculus \cite{counter}, the following special form  holds. It is used in the proof of Lemma \ref{lemma:slice} and Lemma \ref{lemma:fred}.
\begin{lemma}\label{lemma:help}
	Let $q: \R^n\times \R_+^m\times \E  \to \R^n$ be a sc-smooth map. Assume $q(0,0,0) = 0$ and $\rD q_{(0,0,0)}(\R^n\times \{0\}\times \{0\}) = \R^n$. Then there exist open neighborhoods $U\subset \R_+^m \times \E^1$ of $(0,0)$, $V \subset \R^n$ of $0$ and a sc-smooth map $f:U\to V$, such that 
	\begin{enumerate}
		\item $q:\R_+^m \times \E_{k+1}\supset U_k\to V$ is $C^{k+1}$ for all $k\ge 0$.
		\item $q(v,u) = 0$ for $u\in U,v\in V$ iff $v = f(u)$;
		\item $\rD f_{(0,0)}(u) = (\rD q_{(0,0,0)}|_{\R^n\times \{0\}\times \{0\}})^{-1}\circ \rD q_{(0,0,0)}(0,u)$ for $u \in \R^m\times \E^1$.
	\end{enumerate}
\end{lemma}
\begin{proof}
	By \cite[Proposition 1.7]{hofer2017polyfold}, $q: \R^n\times \R_+^m\times \E^k \to \R^n$ is a $C^k$ map. Since $q(0,0,0) = 0$ and $\rD q_{(0,0,0)}(\R^n\times \{0\}\times \{0\}) = \R^n$, we can apply the classical implicit function theorem for partial quadrants in Banach spaces\footnote{The inverse function theorem for partial quadrant was discussed in \cite[Theorem 2.2.4]{roig1992differential}, and the implicit function theorem for partial quadrants follows as a corollary.}. That is there exist open neighborhoods $U \subset \R_+^m\times \E^1$ of $(0,0)$ and $V\subset \R^n$ of $0$, and a $C^1$ map $f:U\to V$, such that $q(v,u) = 0$ for $u\in U,v\in V$ iff $v = f(u)$. Moreover, $\rD q_{(v,u)}(\R^n\times \{0\}\times \{0\}) = \R^n.$ for every $(v,u) \in V\times U$ with $q(v,u) = 0$. It remains to prove that $f$ is sc-smooth. For every $(v,u) \in V\times U$ with $q(v,u) = 0$, if $u \in \R_+^m\times \E^k$, we can apply the classical implicit function theorem to the $C^k$ map $q:\R^n\times \R_+^m\times \E^k  \to \R^n$ near $(v,u)$ to get a $C^k$ map $\tilde{f}_k$ defined near $u$ solving $q(v,u) = 0$. By the uniqueness of implicit function, we have $f = \tilde{f}_k$ in a neighborhood of $u$ in $\R_+^m\times \E^k$. Therefore $f: U\cap (\R_+^m\times \E^k) \to \R^n$ is $C^k$. By \cite[Proposition 1.8]{hofer2017polyfold}, $f:U\to V$ is sc-smooth.  The last assertion follows from the classical implicit function theorem.
\end{proof}

\subsubsection{Free quotients of M-polyfolds and M-polyfold bundles}
We first prove that bundle G-slices give rise to a tame strong M-polyfold bundle structure on the topological quotient. In particular, part \eqref{part:freeone} of Theorem \ref{thm:free} follows directly form the following lemma.
\begin{lemma}\label{lemma:structure}
	Assume $\rho$ acts on a tame strong M-polyfold bundle $p:\cY \to \cX$ such that the induced action $\rho_{\cX}$ on $\cX$ is free. For every $x_0\in \cX_\infty$, we pick a bundle G-slice $(\tilde{\cO}_{x_0},\cO_{x_0}, V_{x_0}, f_{x_0},\eta_{x_0},N_{x_0})$ of $p$ around $x_0$. Let $\hX :=  \cup_{x_0\in \cX_\infty}\rho_{\cX}(G,\cO_{x_0})$. Then $\overline{p}:p^{-1}(\hX)/G \to \hX/G$ is a tame strong M-polyfold bundle,
	such that map 
	\begin{equation}\label{eqn:Psi}
		\Psi_{x_0} = \pi_G\circ \iota: p^{-1}(\tilde{\cO}_{x_0}) \stackrel{\iota}{\hookrightarrow} p^{-1}(\cO_{x_0}) \stackrel{\pi_G}{\to} p^{-1}(\hX)/G
	\end{equation}
	 is a strong bundle morphism. Moreover, the quotient map $\pi_G:p^{-1}(\hX) \to p^{-1}(\hX)/G$ is a sc-smooth strong bundle map. 
\end{lemma}
\begin{proof}
	By construction, $\hX$ is a $G$-invariant open subset of $\cX$ containing $\cX_{\infty}$. We claim that for every $x_0\in \cX$, the following maps 
    \begin{equation}\label{eqn:psi}
        \psi_{x_0}: \tilde{\cO}_{x_0} \hookrightarrow \hX \stackrel{\pi_G}{\to} \hX/G 
    \end{equation}
    induce  a sc-smooth structure on $\hX/G\subset \cX/G$. 
    
    To prove this claim, we first show $\psi_{x_0}$ is a homeomorphism onto the image. Since $\psi_{x_0}$ is injective, it suffices to show that $\psi_{x_0}$ is an open map. For every open subset $\tilde{\cU} \subset \tilde{\cO}_{x_0}$,  we have $\psi_{x_0}(\tilde{\cU})=\pi_G(\eta_{x_0}^{-1}(\tilde{\cU}))$.  Thus it is sufficient to prove $\pi_G^{-1}\left(\pi_G(\eta_{x_0}^{-1}(\tilde{\cU}))\right)\subset \cX$ is open. Because $\eta_{x_0}^{-1}(\tilde{\cU})\subset \cX$ is open, $ \rho_{\cX} \left(G, \eta_{x_0}^{-1}(\tilde{\cU})\right)\subset \cX$ is also open. Since $\pi_G^{-1}\left(\pi_G(\eta_{x_0}^{-1}(\tilde{\cU}))\right) = \rho_{\cX} \left(G, \eta_{x_0}^{-1}(\tilde{\cU})\right)$, $\pi_G^{-1}\left(\pi_G(\eta_{x_0}^{-1}(\tilde{\cU}))\right)$ is open.

    Next we will prove the compatibility between $\psi_{a}$ and $\psi_{b}$.
    Consider a point $q \in \hX/G$ such that $\psi_{a}(x) = \psi_{b}(y) = q$ for $x\in \tilde{\cO}_{a}, y\in \tilde{\cO}_{b}$. If we view $x,y$ as points in $\cX$, then $\pi_G(x) = \pi_G(y) = q$, i.e. there exists $g_0\in G$ such that $\rho_{\cX}(g_0, y) = x$. We claim the transition map near $y$ can be expressed as
    $$\psi_{a}^{-1}\circ \psi_{b}(z) = \eta_{a}\circ \rho_{\cX}(g_0, z).$$
    This is because $\eta_{a}$ is defined on a neighborhood of $x=\rho_{\cX}(g_0,y)$ in $\cO_{a}$ and $\pi_G\circ \eta_{a}\circ \rho_{\cX}(g_0, z) = \pi_G\circ\rho_{\cX}(g_0, z) = \pi(z)$. Since $\eta_{a}$ is sc-smooth, $\psi_{a}^{-1}\circ \psi_{b}$ is sc-smooth. Since each slice $\tilde{\cO}_{x_0}$ is a tame M-polyfold, we give $\hX/G$ a tame M-polyfold structure. 

    By \cite[Theorem 2.2]{hofer2017polyfold}, $\cX$ is metrizable. Therefore by Lemma \ref{lemma:quometric} below, $\cX/G$ is again a metrizable space and so is $\hX/G$. Hence $\hX/G$ is a paracompact Hausdorff space with tame M-polyfold structure, i.e. a tame M-polyfold.

    Finally we prove that the quotient map $\pi_G:\hX\to \hX/G$ is $\sc^\infty$. Consider $x\in \hX$, by our construction of $\hX$, there exist $g_0\in G$ and a slice $\tilde{\cO}_{x_0}$, such that $\rho_{\cX}(g_0,x)\in \tilde{\cO}_{x_0}$. Therefore the map $\pi_G$ can be locally expressed as 
    $$\pi_G: z \mapsto \eta_{x_0} \circ \rho_{\cX}(g_0, z).$$
    Since $\eta_{x_0}$ and $\rho_{\cX}$ are both sc-smooth, the quotient map $\pi_G$ is sc-smooth by the chain rule.
    
    Similarly, bundle maps $\Psi_{x_0}: p^{-1}(\tilde{\cO}_{x_0}) \hookrightarrow p^{-1}(\hX) \stackrel{\pi_G}{\to} p^{-1}(\hX)/G$ gives
    $p^{-1}(\hX)/G\to \hX/G$ a tame strong M-polyfold bundle structure such that the quotient map $\pi_G:p^{-1}(\hX) \to p^{-1}(\hX)/G$ is a sc-smooth strong bundle map. In particular, 
    \begin{equation}\label{eqn:bundiso}
    \Psi_{x_1}^{-1}\circ \Psi_{x_0} \text{ is a strong bundle isomorphism locally.}
    \end{equation}
    This fact is used in the proof of part \eqref{part:freetwo} of Theorem \ref{thm:free}.
\end{proof}

\begin{proposition}\label{prop:unique}
	Assume two different sets of choices of bundle G-slices in Lemma \ref{lemma:structure} give rise to two quotient M-polyfolds $\hX_a/G$ and $\hX_b/G$. Let $U$ be the topological quotient $(\hX_a \cap \hX_b)/G$, which is an open subset of $\cX/G$ containing $\cX_\infty/G$. Then the identity map $(\hX_a/G) \cap U \to (\hX_b/G) \cap U$ is a sc-diffeomorphism, where $(\hX_a/G) \cap U$ is $U$ with the M-polyfold structure from $\hX_a/G$ and $(\hX_b/G) \cap U$ is $U$ with the M-polyfold structure from $\hX_b/G$. Similarly, the identity map $p^{-1}(\hX_a)/G \cap \overline{p}^{-1}(U) \to p^{-1}(\hX_b)/G \cap \overline{p}^{-1}(U)$ is a strong bundle isomorphism.
\end{proposition}
\begin{proof}
	We will prove the claim on the base M-polyfold and the assertion on the bundle follows from a similar argument. The quotient construction in Lemma \ref{lemma:structure} has the following universal property: Let $W$ be a $G$-invariant open subset of $\cX$. If there is a $G$-invariant $\sc^\infty$ map $f: \hX\cap W \to \cZ$, then there exists a unique $\sc^\infty$ map $\overline{f}: (\hX\cap W)/G \to \cZ$ such that $f = \overline{f}\circ \pi_G$ on $\hX\cap W$. The existence and uniqueness of a continuous map $\overline{f}$ follows from that $(\hX\cap W)/G$ is the topological quotient of $\hX\cap W$ by $G$. To see $\overline{f}$ is sc-smooth, we have the following commutative diagram,
	$$\xymatrix{
		\tilde{\cO}_{x_0}\cap W \ar@{^{(}->}[r] \ar[rd]_{\psi_{x_0}} & \hX\cap W \ar[d]^{\pi_G} \ar[r]^{f} & \cZ \\
		& (\hX \cap W)/G. \ar[ru]_{\overline{f}} &
		}
	$$
	Recall from Lemma \ref{lemma:structure}, $\psi_{x_0}:\tilde{\cO}_{x_0}\cap W \to (\hX\cap W)/G$ is sc-diffeomorphism onto an open set. By the diagram above, $\overline{f}\circ \psi_{x_0} = f|_{\tilde{\cO}_{x_0} \cap W}$. Therefore $\overline{f}$ is sc-smooth.
	
	Given the conditions of this proposition, we have the following commutative diagram by the universal property,
	$$\xymatrix{
		& \hX_a\cap \hX_b\ar[rd]^{\pi_b}\ar[ld]_{\pi_a} & \\
		\hX_a/G \cap U \ar@<0.5ex>[rr]^{\Id=\overline{\pi_b}} &  & \hX_b/G \cap U, \ar[ll]^{\Id = \overline{\pi_a}} 		
		}
	$$
	where $\pi_a$ resp. $\pi_b$ are the quotient maps on $\hX_a$ resp. $\hX_b$ and $\overline{\pi_a}, \overline{\pi_b}$ are the sc-smooth maps induced by $G$-invariant maps $\pi_a$ resp. $\pi_b$.  Since $\overline{\pi_a} = \overline{\pi_b}=\Id_U$ as topological maps, identity map $\overline{\pi_a}:\hX_a/G\cap U \to \hX_b/G\cap U$ is a sc-diffeomorphism with sc-smooth inverse $\overline{\pi_b}$.	
\end{proof}

The slice construction only gives M-polyfold structures resp. polyfold structures, which are local in nature. Globally, M-polyfolds resp. polyfolds are also metrizable spaces. The following Lemma asserts that the metrizablity passes to the quotients.

\begin{lemma} \label{lemma:quometric}
	Let $X$ be a metrizable space, $G$ a compact Lie group and  $\rho:G\times X \to X$  a continuous group action, then the quotient topology on $X/G$ is also metrizable.
\end{lemma}
\begin{proof}
	Let $d$ be a metric on $X$. We first show that averaging over $G$ gives rise to a $G$-invariant metric $d_G$ on $X$ inducing the same topology. First we equip $G$ with a left invariant Riemannian metric and define a map $d_G:X\times X \to \R_+$ by
	$$d_G(x,y)=\int_G d(gx,gy)dg.$$
	Then $d_G$ is a $G$-invariant metric and $d_G$ is continuous. We claim that $d_G$ induces the same topology as $d$. The continuity implies any $d_G$-ball contains a $d$-ball. It suffices to prove that any $d$-ball contains a $d_G$-ball. Assume otherwise that there is a $d$-ball does not contain any $d_G$-ball, i.e. there exist $\{y_k\}_{k\in \mathbb{N}}$, such that $\lim d_G(x,y_k)=0$ and $d(x,y_k)>C$. Let $f_k:=d(gx,gy_k) \in C^0(G)$. Then $\lim_k d_G(x,y_k)= 0$ implies that the  functions $f_k$ converge to $0$ in the space of absolutely integrable functions $L^1(G)$ .  After passing to a subsequence,  $f_k$ converge to $0$ almost everywhere on $G$.  Then there is a $g_0\in G$, such that $\lim_k f_k(g_0) = \lim_k d(g_0x,g_0y_k)=0$. Since $g_0$ acts continuously on $X$, this implies $d(x,y_k)\to 0$, which contradicts the assumption $d(x,y_k)>C$. Hence $d,d_G$ induce the same topology on $X$. 
	
	Using $d_G$, we can equip $X/G$ a metric as follows. Let $\pi_G$ denote the quotient map $X \to X/G$. Then over $X/G$, we define
	$$d_{X/G}(\pi_G(x), \pi_G(y)):=\min_{g\in G}d_G(x,gy).$$
	For fixed $x,y\in X$, $d_G(x,gy)$ is a continuous function on the compact space $G$. Therefore $\min_{g\in G}d_G(x,gy)$ exists. Because $d_G$ is $G$-invariant, $d_{X/G}$ is well defined, i.e. it only depends on $\pi_G(x), \pi_G(y)$. We claim $d_{X/G}$ is a metric on $X/G$. Since $d_G \ge 0$, $d_{X/G} \ge 0$. If $d_{X/G}(\pi_G(x), \pi_G(y)) = 0$, then there exists $g \in G$ such that $d_G(x, gy) = 0$, hence $x=gy$ and $\pi_G(x) = \pi_G(y)$. Since $d_G(x,gy) = d_G(g^{-1}x,y)$, we have
	\begin{eqnarray*}
		d_{X/G}(\pi_G(x), \pi_G(y)) & = &  \min_{g\in G}d_G(x,gy) \\
		& = & \min_{g\in G}d_G(g^{-1}x,y)\\
		& = & \min_{g\in G}d_G(y,g^{-1}x)\\
	    & = & \min_{g\in G}d_G(y,gx)\\ 
		& = & d_{X/G}(\pi_G(y), \pi_G(x)),
	\end{eqnarray*}
	Therefore $d_{X/G}$ is symmetric. Finally, we check the triangle inequality. For $x,y,z \in X$, pick $g,h\in G$ such that $d_{X/G}(\pi_G(x), \pi_G(y)) = d_G(x, gy)$ and $d_{X/G}(\pi_G(y), \pi_G(z)) = d_G(y, hz)$, since we have $d_G(y,hz) = d_G(gy,ghz)$, then  
	\begin{eqnarray*}
	d_{X/G}(\pi_G(x), \pi_G(y))  + d_{X/G}(\pi_G(y), \pi_G(z))  &  = &  d_G(x, gy) + d_G(y, hz) \\
	& = & d_G(x, gy) + d_G(gy, ghz) \\
	& \ge & d_G(x, ghz) \\
	& \ge & d_{X/G}(\pi_G(x),\pi_G(z)).
	\end{eqnarray*}
	Therefore $d_{X/G}$ is a metric. 
	
	To finish the proof, it suffices to show that $d_{X/G}$ induces the quotient topology. Assume $U\subset X/G$ is open in the quotient topology. Recall that $U\subset X/G$ is open iff $\pi_G^{-1}(U)\subset X$ is open. Then for every point $x\in \pi_G^{-1}(U)$, there exists a $d_G$-ball $B^{d_G}_r(x) \subset \pi_G^{-1}(U)$ of radius $r>0$. We claim the $d_{X/G}$-ball $B^{d_{X/G}}_r(\pi_G(x))$ is contained in $U$. This is because for any point $\pi_G(y) \in B^{d_{X/G}}_r(\pi_G(x))$, there exists $g \in G$ such that $d_G(x, gy) < r$. Therefore we have $g y \in \pi_G^{-1}(U)$, hence $\pi_G(y)\in U$. Conversely, it suffices to check that every $d_{X/G}$-ball is open in the quotient topology. For a $d_{X/G}$-ball $B^{d_{X/G}}_r(\pi_G(x))$, we claim $\pi_G^{-1}(B^{d_{X/G}}_r(\pi_G(x))) = \rho(G, B^{d_G}_r(x))$. This implies $B^{d_{X/G}}_r(\pi_G(x))$ is open in the quotient topology. To prove the claim, note that $y \in \pi_G^{-1}(B^{d_{X/G}}_r(\pi_G(x))) $ iff $d_{X/G}(\pi_G(x), \pi_G(y)) < r$, that is there exists $g\in G$ such that $d_G(x, gy) < r$, which is equivalent to $y \in \rho(G, B^{d_G}_r(x))$.
\end{proof}

\subsection{Free quotients of Fredholm sections}\label{subsec:freeFred}
\subsubsection{Basics of sc-Fredholm sections}
  The most important feature of polyfold theory is that it has a Fredholm theory package \cite[Chapter 3, 5]{hofer2017polyfold}, which includes an implicit function theorem and a perturbation theory. The definition of sc-Fredholm section in polyfold is more involved than the classical Fredholm theory, since the Fredholmness of linearization is not sufficient for an implicit function theorem, also see \cite{counter}.
  
  Recall from Definition \ref{strbund}, strong M-polyfold bundles are modeled on bundle retracts, which can be very complicated subset of a sc-Banach space. When analyzing zero sets of M-polyfold sections, we need to extend sections from retracts to the underlying sc-Banach spaces so that we can apply the implicit function theorem \cite[Theorem 3.7]{hofer2017polyfold} developed for sc-Banach spaces. Moreover, such extensions should not change the zero sets. Therefore the following concept was introduced in \cite[Definition 3.4]{hofer2017polyfold}.
  
  \begin{definition}\label{def:fill} Let $R: U\lhd \F \to U\lhd \F, (x,e) \mapsto (r(x), \varrho(x)e)$ be a tame strong bundle retraction and $s: r(U) \to R(U\lhd F)$ a sc-smooth section. A filled section $s^f$ at $(0,0) \in U\subset \R_+^m\times \E$ is a section $s^f:U \to U\lhd\F$ with the following properties.
	\begin{enumerate}
		\item\label{f1} $s(x) = s^f(x)$ for $x\in r(U)$.
		\item\label{f2} If $s^f(y) = \varrho(r(y))s^f(y)$ for a point $y \in U$, then $y \in \Ima r$.
		\item\label{f3} Let $S(y):= s^f(y)- \varrho(r(y)) s^f(y)$. Then $\rD S_{(0,0)}|_{\ker \rD r_{(0,0)}}:\ker \rD r_{(0,0)} \to \F$ is an isomorphism onto $\ker \varrho(0,0)$.
	\end{enumerate}
  \end{definition}
  We first show that fillings induce fillings for any restriction of the section to a slice in the sense of Definition \ref{def:slice}.
  \begin{proposition}\label{prop:restrict}
	Let $R = (r,\varrho)$ be a $\R^n$-sliced bundle retraction on $U\lhd \F$ for an open neighborhood $U\subset \R_+^m\times \R^n\times \K$ of $(0,0,0)$ covering a $\R^n$-sliced retraction $r$. Let $\tilde{U}:=(\R_+^m\times \{0\}\times \K)\cap U$, $\tilde{R}:=R|_{\tilde{U}\lhd F}$ and $\tilde{r}:= r|_{\tilde{U}}$. Assume $s:r(U)\to R(U\lhd \F)$ be a section with filling $s^f:U\to \F$. Then the restriction to the slice $\tilde{s}:=s|_{r(\tilde{U})}:\tilde{r}(\tilde{U}) \to \tilde{R}(\tilde{U}\lhd \F)$ has a filling $\tilde{s}^f:=s^f|_{\tilde{U}}$. 
 \end{proposition}
 \begin{proof}
	It suffices to check the three properties of Definition \ref{def:fill}. Property \eqref{f1} holds because $s^f|_{\tilde{U}}$ is the restriction of $s^f$. For property \eqref{f2}, if $\tilde{s}^f(y) = \varrho(\tilde{r}(y))\tilde{s}^f(y)$ for a point $y\in \tilde{U}$, then $s^f(y) = \varrho(r(y))s^f(y)$. By the property \eqref{f2} of the filling $s^f$, we have $y\in \Ima r$. Since $\Ima \tilde{r} = \tilde{U}\cap \Ima r$, we have $y\in \Ima \tilde{r}$.
	
	For property \eqref{f3}, let $\tilde{S}(y) := \tilde{s}^f(y) - \varrho(\tilde{r}(y))\tilde{s}^f(y)$ for $y\in \tilde{U}$ and $S(y) := s^f(y) - \varrho(r(y))s^f(y)$ for $y \in U$. Then $\tilde{S} = S|_{\tilde{U}}$, hence we have $\rD \tilde{S}_{(0,0)} = \rD S_{(0,0,0)}|_{\R^m\times \{0\} \times \K}$. Since $r$ is $\R^n$-sliced, i.e. $\pi_{\R^n}\circ r = \pi_{\R^n}$, we have $\ker \rD r_{(0,0,0)} = \ker \rD \tilde{r}_{(0,0)} \subset \R^m\times \{0\} \times \K$. Therefore, we have 
	$$\rD \tilde{S}_{(0,0)}|_{\ker \rD (r_0)_{(0,0)}} = \rD S_{(0,0,0)}|_{\ker \rD r_{(0,0,0)}}.$$ 
	Therefore $\rD \tilde{S}_{(0,0)}|_{\ker \rD \tilde{r}_{(0,0)}}$ is an isomorphism onto $\ker \varrho(0,0,0)$ due to property \eqref{f3} of the filling $s^f$.
\end{proof}  
  
The key ingredient of the Fredholm property in polyfold theory is the following contracting property, so that implicit function theorem  \cite[Theorem 3.7]{hofer2017polyfold} holds. Note that for a $C^1$ map between Banach spaces, the following contracting property holds after a $C^1$ change of coordinate as long as the map has a Fredholm linearization.
  
\begin{definition}[{\cite[Definition 3.6]{hofer2017polyfold}}]\label{def:basic}
	Let $C := \R_+^m \times \R^{k}$ and let $\W$ be a sc-Banach space. A \textbf{basic germ} is a sc-smooth germ at $(0,0)$ represented by a sc-smooth section
	$$U \to U\lhd\R^N \times \W,\quad x \mapsto (x,f(x)),$$
	on an open neighborhood $U \subset C\times \W$ of $(0,0)$, such that 
	$f(0,0) = 0$ and
	$$B(c,w):= \pi_{\W} f(c,w)- w,\quad \forall (c, w) \in U$$
	has the contracting property: For every $m \in \N, \epsilon > 0$, there exists a neighborhood $U_{\epsilon, m} \subset U_m$  of $(0,0)$, so that 
	\begin{equation}\label{contr}||B(c,w)-B(c,w')||_m \le \epsilon ||w-w'||_m, \qquad \forall (c,w),(c,w') \in U_{\epsilon, m}.\end{equation}
\end{definition}

Then the definition of sc-Fredholm sections of strong M-polyfold bundles is the combination of Definition \ref{def:fill} and Definition \ref{def:basic}.
\begin{definition}[{\cite[Definition 3.8]{hofer2017polyfold}}]\label{def:fredchart}
	Let $p: \cY \to \cX$ be a M-polyfold bundle and $s:\cX\to \cY$ a sc-smooth section. A bundle chart $(\cO, \Phi, (P, (\R_+^m \times \E) \lhd \F))$ around $x_0\in \cX_\infty$ along with a strong bundle isomorphism $A:U\lhd \F \to U'\lhd (\R^N\times \W),(x,v)\mapsto(\alpha(x),\Lambda(x)v)$ is a \textbf{Fredholm chart} of $s$ around $x_0$ if the following holds.
	\begin{enumerate}
		\item $U\subset \R_+^m\times \E, U'\subset C\times \W:= \R_+^m\times \R^k \times \W$ are open neighborhoods of $(0,0)$, such that $R(U\lhd \F) = P$ for a tame strong bundle retraction $R$.
		\item $\Phi_*s$ has a filling $s^f:U \to \F$.
		\item $\alpha(0,0) = (0,0)$ and there exists a $\sc^+$-section $\gamma:U'\to U'\lhd (\R^N\times \W)$ such that $A_*s^f-\gamma$ is a basic germ at $(0,0)$.
	\end{enumerate}
	The section $s$ is \textbf{sc-Fredholm section}  iff the following two conditions hold.
	\begin{itemize}
		\item\textbf{Regularization property}: if $s(x)\in \cY[1]_m $, then $x\in \cX_{m+1}$.
		\item\textbf{Basic germ property}: for every $x_0\in \cX_\infty$, there exists a Fredholm chart around $x_0$.
		\end{itemize}
	A sc-Fredholm section is proper, if $s^{-1}(0)$ is compact in the $\cX_0$ topology. 
\end{definition}

\begin{remark}\label{rmk:compact}
	Note that the regularization property implies that $s^{-1}(0) \subset \cX_{\infty}$. By \cite[Theorem 5.3]{hofer2017polyfold}, when  $s^{-1}(0)$ is compact in the $\cX_0$ topology, then $s^{-1}(0)$ is compact in the $\cX_{\infty}$ topology\footnote{The $\cX_\infty$ topology is the inverse limit topology of $\cX_i$ topology on $\cX_\infty$.}. As a consequence, $s^{-1}(0)$ is compact in $\cX_i$ topology for every $i$. Therefore we will not specify the topology for the compactness going forward.
\end{remark}
\begin{remark}
	By \cite[Theorem 3.10]{hofer2017polyfold}, sc-Fredholm properties are preserved under $\sc^+$ perturbations after a level shift.
\end{remark}


\subsubsection{Sc-Fredholm property of quotients}
The goal of this subsection is to prove part \ref{part:freetwo} of Theorem \ref{thm:free}, namely that a proper $G$-equivariant sc-Fredholm section descends to a proper sc-Fredholm section on the free quotient of an M-polyfold bundle. The proof can be found at the end of this subsection. Most parts of this proof work with the quotient of the section in local coordinates, i.e. with restrictions of the sc-Fredholm section to $G$-slices (Definition \ref{def:Gslice}) established in Lemma \ref{lemma:slice}. Thus we will more generally consider restrictions of sc-Fredholm sections to slices of M-polyfolds, as defined in Definition \ref{def:slice}. The following proposition follows directly from the definition of regularization property in Definition \ref{def:fredchart}.
\begin{proposition}\label{prop:reg}
	Let $p:\cY\to \cX$ be a tame strong M-polyfold bundle with $s:\cX\to \cY$ a sc-Fredholm section. Assume $p^{-1}(\tilde{\cO})$ is a bundle slice (Definition \ref{def:Gslice}), then $\tilde{s}:=s|_{\tilde{\cO}}:\tilde{\cO} \to p^{-1}(\tilde{\cO})$ has the regularization property in Definition \ref{def:fredchart}. 
\end{proposition}

Next we need to address the question of whether the restriction $\tilde{s}$ has the basic germ property in Definition \ref{def:fredchart}. Since the main condition of the basic germ property is the contracting property \eqref{contr} after a bundle isomorphism $A$ in Definition \ref{def:fredchart}, the basic germ property of $\tilde{s}$ requires some compatibility between the bundle isomorphism $A$ and the slice $\tilde{\cO}$. To be more precise, we introduce the following notion of good slice to ensure the restriction $s|_{\tilde{\cO}}$ inherits the basic germ  property from $s$, see Lemma \ref{lemma:fred}.

\begin{definition}\label{def:good}
	A slice $\tilde{\cO}$ of $\cX$ around $x_0\in \cX_\infty$ is \textbf{good} with respect to a sc-Fredholm section $s:\cX \to \cY$, if we have the following.
	\begin{enumerate}
		\item There is a Fredholm chart $(\cO, \Phi, (P,(\R_+^m\times \E)\lhd \F), A)$ around $x_0$ for $s$ covering a tame M-polyfold chart $(\cO, \phi, (O,\R_+^m\times \E))$. $R$ is a strong bundle retraction covering a tame traction $r$ such that $P=R(U\lhd \F)$ and $O = r(U)$ for an open neighborhood $U\subset \R_+^m\times \E$ of $(0,0)$.
		\item The strong bundle isomorphism $A:U\lhd \F \to U'\lhd (\R^N\times \W)$ is in the form of $(x,v)\mapsto (\alpha(x),\Lambda(x)v)$, where $\alpha$ is a sc-diffeomorphism from $U$ to a neighborhood $U'\subset C\times \W := \R_+^m\times \R^k\times \W$ of $(0,0)$ and $\Lambda(x)$ is a linear isomorphism from $\F$ to $\R^N\times \W^N$ for every $x\in U$.
		\item There is a sc-diffeomorphism $h:U \to U''\subset \R_+^m\times \R^n\times \K$ such that $r'':=h\circ r \circ h^{-1}:U''\to U''$ is $\R^n$-sliced, $h(0,0)=(0,0,0)$ and $\tilde{\cO} = \phi^{-1}\circ r \circ h^{-1}((\R_+^m\times \{0\}\times \K) \cap U'')$.
		\item  $\pi_{\R^n} \circ \rD (h\circ \alpha^{-1})_{(0,0)}|_{\{0\}\times \W}:\W\to\R^n$ is surjective.
	\end{enumerate}  
\end{definition}

\begin{remark}
	Suppose the slice is constructed as in Lemma \ref{lemma:benslice} from a submersive map $f$. Then the following two equivalent conditions imply goodness. Let $E:= T^R_{x_0} \cX, F:=\ker \rD f_{x_0} \cap T^R_{x_0} \cX$.
	\begin{enumerate}
		\item\label{cond:1} There exists a $n$ dimensional subspace $L\subset  T^R_{x_0}\cX_\infty$, such that it is a complement of $F$ in $E$ and $\rD (\alpha \circ \phi)_{x_0}(L)\subset \{0\} \times \W$. This is the $s$-compatibly transversality condition in \cite[Definition 5.8]{ben2018fiber}
		\item\label{cond:2} $\rD (\alpha \circ \phi)_{x_0}(F) \cap \{0\} \times \W_\infty $ is a $n$-codimensional subspace in $\rD (\alpha \circ \phi)_{x_0}(E) \cap \{0\} \times \W_\infty$.
	\end{enumerate}
	The second condition is generically satisfied if we can perturb $F$, which is exactly the proof in Proposition \ref{prop:goodslice} of the existence of good slices.
\end{remark}
\begin{proof}
	By Lemma \ref{lemma:benslice}, we have $\rD(h\circ \alpha^{-1})_{(0,0)} \circ \rD(\alpha\circ \phi)_{x_0}(L) = \R^n$, hence \eqref{cond:1} implies Definition \ref{def:good}. Next we will prove  \eqref{cond:1}$\Leftrightarrow$\eqref{cond:2}. If $L\subset T^R_{x_0}\cX_\infty$ is a complement of $F$ in $E$,  such that $\rD (\alpha \circ \phi)_{x_0}(L)\subset \{0\} \times \W$, then $\rD (\alpha \circ \phi)_{x_0}(L)$ is the complement of $\rD (\alpha \circ \phi)_{x_0}(F) \cap \{0\} \times \W_\infty $ in $\rD (\alpha \circ \phi)_{x_0}(E) \cap \{0\} \times \W_\infty$. Hence \eqref{cond:1}$\Rightarrow$\eqref{cond:2}. If $\rD (\alpha \circ \phi)_{x_0}(F) \cap \{0\} \times \W_\infty$ is $n$-codimensional in $\rD (\alpha \circ \phi)_{x_0}(E) \cap \{0\} \times \W_\infty$. Let $V$ be the $n$-dimensional complement, then there exists $L\subset E_\infty$ such that $\rD(\alpha \circ \phi)_{x_0}(L) = V$. Then $L$ is the complement of $F$ in $E$ satisfying condition \eqref{cond:1}. 
\end{proof}

\begin{lemma}\label{lemma:fred}
	Let $x_0 \in \cX_\infty$ and $\tilde{\cO}$ a good slice of $\cX$ around $x_0$ with respect to a sc-Fredholm section $s:\cX \to \cY$. Then there is a Fredholm chart around $x_0$ for $\tilde{s}:=s|_{\tilde{\cO}^1}:\tilde{\cO}^1\to p^{-1}(\tilde{\cO}^1)$.
\end{lemma}
\begin{proof}
By the assumption of $\tilde{\cO}$ being a good slice, we have the following structures and properties:
\begin{enumerate}
	\item a strong bundle chart $(\cO,\Phi,(P,(\R_+^m\times \E)\lhd \F))$ around $x_0$ covering a tame M-polyfold chart $(\cO,\phi, (O,\R_+^m\times \E))$, where $P = R(U\lhd \F)$ and $O=r(U)$ for bundle retraction $R$ covering tame retraction $r$ and open neighborhood $U\subset \R_+^m\times \E$ of $(0,0)$;
	\item $\Phi_*s:O \to P$ has a filling $s^f:U\to U\lhd \F$ as in Definition \ref{def:fill};
	\item a strong bundle isomorphism $A:U\lhd \F \to U'\lhd(\R^N\times \W)$ for an open neighborhood $U'\subset C\times \W:=\R_+^m\times \R^k\times \W$ of $(0,0)$, where $A = (\alpha, \Lambda)$ consisting of a sc-diffeomorphism $\alpha:U\to U'$ and a $\sc^\infty$ family of linear isomorphism $\Lambda(x):\F \to \R^N\times \W$ parametrized by $x\in U$;
	\item a $\sc^+$-section $\gamma:U'\to U'\lhd(\R^N\times \W)$ such that $A_*s^f-\gamma$ is a basic germ at $(0,0)$;
	\item a sc-diffeomorphism $h:U\to U''\subset \R_+^m \times \R^n\times \K$ such that $h\circ r \circ h^{-1}:U''\to U''$ is $\R^n$-sliced, $h(0,0)=(0,0,0)$ and $\tilde{\cO}=\phi^{-1}\circ r \circ h^{-1}((\R^m_+\times \{0\}\times \K)\cap U'')$;
	\item $\pi_{\R_n}\circ \rD (h\circ \alpha^{-1})_{(0,0)}|_{\{0\}\times \W}:\W \to \R^n$ is surjective.
\end{enumerate}

Sc-smooth map $H:=h \lhd \Id_\F:U\lhd \F \to U''\lhd \F$ is a strong bundle isomorphism.  Let $R'':= H\circ R \circ H^{-1}$, which is a $\R^n$-sliced bundle retraction on $U''\lhd \F$. Then $(\cO,H\circ \Phi,(R''(U''\lhd \F), (\R_+^m\times \R^n\times \K)\lhd \F))$ is another strong bundle chart around $x_0$. Then $H_*s^f$ is a filling of $(H\circ \Phi)_*s$ by  \cite[proof of part (III) Lemma 4.2]{ben2018fiber}. Let $\tilde{U}'' := (\R_+^m \times \{0\} \times \K)\cap U''$ and $\tilde{s}^f:= H_*s^f|_{\tilde{U}''}$. By Proposition \ref{prop:restrict}, $\tilde{s}^f$ is a filling of $\tilde{s}:\tilde{\cO}\to p^{-1}(\tilde{\cO})$ in the chart induced by $(\cO,H\circ \Phi,(R''(U''\lhd \F), (\R_+^m\times \R^n\times \K)\lhd \F))$.

It remains to verify that $\tilde{s}^f$ is basic germ after a bundle isomorphism. We claim that there exist a sc-Banach subspace $\T\subset \W$ with a finite dimensional complement $\T^\perp$, a sc-diffeomorphism $q:U_\vartriangle \to U_\blacktriangle$ for open neighborhoods $U_\vartriangle\subset C\times \T^1, U_\blacktriangle \subset (\tilde{U}'')^1$ of $(0,0)$ and a strong bundle isomorphism $Q:U_\vartriangle \lhd (\R^N\times \W^1) \to U_\blacktriangle \lhd \F^1$ covering $q$ such that $Q^*\tilde{s}^f$ is a basic germ. For better visualization of all the structures, we have the following diagram.
$$
\xymatrix{
	\llap{$C\times \T^1 \supset {}$}U_\vartriangle  \ar[d]^q \ar[r]^-{Q^*\tilde{s}^f} & U_\vartriangle \lhd (\R^N\times \W^1)\ar[d]^Q \rlap{$=U_\vartriangle\lhd(\R^N\times \T^\perp \times \T^1){}$}& & &  &\\ 
	\llap{$\R_+^m\times \{0\}\times \K^1 \supset {}$}U_\blacktriangle  \ar@{}[d]|-*[@]{\cap} \ar[r]^{\tilde{s}^f} & U_\blacktriangle \lhd \F^1\ar@{}[d]|-*[@]{\cap} & & & & \\
	\llap{$\R_+^m\times \{0\}\times \K \supset {}$}\tilde{U}'' \ar@{^{(}->}[d]^{\iota} \ar[r]^{\tilde{s}^f} & \tilde{U}'' \lhd \F \ar[d]\ar@{^{(}->}[d]^{\iota\lhd \Id_\F} & \ar@{}[l]|-*[@]{\supset}\tilde{r}''(\tilde{U}'')\ar@{^{(}->}[d] \ar[r]^-{(H\circ \Phi)_*\tilde{s}} &  \tilde{R}''(\tilde{U}''\lhd \F)\ar@{^{(}->}[d] & \ar@{.>}[l]_{\qquad H\circ \Phi}\tilde{\cO} \ar@{^{(}->}[d]\ar[r]^-{\tilde{s}} & p^{-1}(\tilde{\cO})\ar@{^{(}->}[d] \\
	\llap{$\R_+^m\times \R^n\times \K \supset {}$}U''  \ar[r]^{H_*s^f} & U''\lhd \F & \ar@{}[l]|-*[@]{\supset} r''(U'') \ar[r]^-{(H\circ \Phi)_*s}& R''(U''\lhd \F) & \ar@{.>}[l]_{\qquad H\circ \Phi}\cO \ar[r]^-{s} & p^{-1}(\cO)\\
	\llap{$\R_+^m\times \E \supset {}$}U \ar[u]_{h} \ar[d]^{\alpha} \ar[r]^{s^f} & U\lhd \F\ar[u]_{H}\ar[d]^{A} & \ar@{}[l]|-*[@]{\supset} r(U) \ar[u]_h \ar[r]^-{\Phi_*s}& R(U\lhd \F)\ar[u]_H & \ar@{.>}[l]_{\qquad \Phi} \cO \ar[u] \ar[r]^-{s} & p^{-1}(\cO) \ar[u]\\
	\llap{$C\times \W \supset {}$}U' \ar[r]^-{A_*s^f} & U'\lhd (\R^N\times \W) & & & & \\
	\ar@{}[r]|-*[@]{}^{\text{\normalsize{Filled sections on partial quadrants}}}& &\ar@{}[r]|-*[@]{}^{\text{\normalsize{Sections on retracts}}} & &\ar@{}[r]|-*[@]{}^{\text{\normalsize{Sections on M-polyfolds}}} &
}
$$
Here $\iota:\tilde{U}''\to U''$ is the inclusion and $\tilde{r}'':= r''|_{\tilde{U}''}$ and $\tilde{R}'':= R''|_{\tilde{U}''}$ are the retractions on the slice. 

Let $\beta:=\alpha \circ h^{-1}$. Since $\pi_{\R^n} \circ \rD\beta^{-1}_{(0,0)}|_{\{0\}\times\W}:\W \to \R^n$ is surjective, $$\T:=\ker\left(\pi_{\R^n} \circ \rD\beta^{-1}_{(0,0)}|_{\{0\}\times\W}\right)$$ is a subspace in $\W$ of codimension $n$. Because $\pi_{\R^n} \circ\rD\beta^{-1}_{(0,0)}(\{0\}\times \W_\infty)$ is a dense subspace of $\pi_{\R^n} \circ \rD\beta^{-1}_{(0,0)}(\{0\}\times\W) = \R^n$,  we have $\pi_{\R^n} \circ \rD\beta^{-1}_{(0,0)}(\W_\infty) = \R^n$. Therefore we can pick $n$ preimages of the basis of $\R^n$ in $\W_\infty$. They form a subspace $\T^\perp\subset \W_\infty$. Then we have a splitting $\W=\T \oplus \T^\perp$ and $\pi_{\R^n} \circ \rD\beta^{-1}_{(0,0)}|_{\{0\}\times \T^\perp }:\T^\perp\to \R^n$ is an isomorphism. For $(u,e)\in \R_+^m \times \K^1$ close to $(0,0)$, we define a sc-smooth map $q'$ by
\begin{equation}\label{eqn:q}
q':(u,e)\mapsto (\pi_C\circ \beta\circ\iota(u,e), \pi_\T\circ \pi_\W\circ \beta\circ\iota(u,e)).
\end{equation}
That is $q'$ is the composition of the $\iota$, $\beta$ and the projection from $C\times \W^1$ to $C\times \T^1$. We claim $q'$ has an inverse given by
$$q:(c,t)\mapsto \beta^{-1}(c,t+\theta(c,t)),$$
where  $(c,t)\in C\times \T^1$ close $(0,0)$ and $\theta(c,t)\in \T^\perp$ solves the following equation,
\begin{equation}\label{fred}
\pi_{\R^n}\circ\beta^{-1}(c,t+\theta)=0.
\end{equation}
Since $\omega(\theta,c,t):=\pi_{\R^n}\circ\beta^{-1}(c,t+\theta)$ is a sc-smooth function from an open neighborhood of $(0,0,0)$ in $\T^\perp \times C\times \T $ to $\R^n$ and ${\rD}\omega_{(0,0,0)}(u,0,0)=\pi_{\R^n} \circ \rD_{(0,0)} \beta^{-1}(u)$, i.e. ${\rD}\omega_{(0,0,0)}$ is invertible on $\T^\perp$. By Lemma \ref{lemma:help}, there is a sc-smooth map $\theta(c,t)$ from an open neighborhood $U_\square \subset C\times \T^1$ of $(0,0)$ to $\T^\perp$, such that $\theta(c,t)$ solves the equation $\omega(\theta,c,t) = 0$ for $(c,t) \in U_\square$ and $\theta\in D$, where $D\subset \T^\perp$ is an open neighborhood of $0$. Moreover, $\theta(0,0) = 0$. This shows $q$ is sc-smooth near $(0,0)$. We now prove $q$ is the inverse to $q'$. First for $(c,t) \in U_\square$
\begin{eqnarray*}
	q'\circ q(c,t) &= &\pi_{C\times \T}\circ \beta \circ \iota \circ \beta^{-1}(c,t+\theta(c,t))\\
	&=&\pi_{C\times \T}\circ \beta \circ \beta^{-1}(c,t+\theta(c,t))\\
	&= & (c,t).
\end{eqnarray*}
On the other hand, if $(u,e)\in (\R_+^m\times \K^1)\cap q'^{-1}(U_\square) \cap \pi_{\T^\perp}\circ \beta \circ \iota^{-1}(D)$ and let $(c,t) := q'(u,e)=\pi_{C\times \T}\circ \beta \circ \iota(u,e)$. Then 
$\left(c,t+\pi_{\T^\perp}\circ \beta \circ \iota(u,e)\right)=\beta(u,e)$, which means $\theta:=\pi_{\T^\perp}\circ \beta \circ \iota(u,e)$ solves equation \eqref{fred} in $U_\square\times D$.  Then we have on $(\R_+^m\times \K^1)\cap q'^{-1}(U_\square) \cap \pi_{\T^\perp}\circ \beta \circ \iota^{-1}(D)$, 
\begin{eqnarray*}
	q\circ q' (u,e) &= & \beta^{-1}\left(q'(r,k)+(0, \pi_{\T^\perp}\circ \beta \circ \iota(u,e))\right)\\
	&\stackrel{\eqref{eqn:q}}{=} &  \beta^{-1}\circ \beta\circ \iota(u,e)\\
	&=&(u,e).
\end{eqnarray*}
Therefore $q$ is the inverse to $q'$ locally. Moreover, $q$ has the following property 
\begin{equation}\label{eqn:property}\beta\circ \iota\circ q(c,t)=\beta \circ q(c,t)=(c,t+\theta(c,t)).
\end{equation}
We choose open neighborhoods $U_\blacktriangle\subset \tilde{U}'',U_\triangle \subset U_\square$, such that $q':U_\blacktriangle\to U_\vartriangle, q:U_\vartriangle\to U_\blacktriangle$ are inverse to each other. 

The strong bundle isomorphism $Q:U_\vartriangle \lhd (\R^N\times \W^1) \to U_\blacktriangle \lhd \F^1$ is defined by $(x,v)\mapsto (q(x), \Lambda(h^{-1}\circ \iota\circ q(x))^{-1}v)$. It has an inverse $(x,v)\mapsto (q'(x), \Lambda(h^{-1}\circ \iota(x))v)$. It remains to show that $Q^*\tilde{s}^f$ up to a $\sc^+$ perturbation is a basic germ at $(0,0)$. 

First, we have
\begin{eqnarray}
Q^*\tilde{s}^f(c,t) & = & \Lambda(h^{-1}\circ \iota \circ q(c,t))\tilde{s}^f(q(c,t)) \nonumber\\
& = & \Lambda(h^{-1}\circ \iota \circ q(c,t)) s^f(h^{-1}\circ \iota \circ q(c,t)) \nonumber\\
& = & A_*s^f\circ \alpha\circ h^{-1}\circ \iota \circ q(c,t) \nonumber \\
& = & A_*s^f\circ \beta \circ \iota \circ q(c,t). \label{eqn:fredcoor}
\end{eqnarray}
Let $\tilde{\gamma}:= Q^*\iota_*H_*A^* \gamma: U_\vartriangle \to U_\vartriangle \lhd (\R^N\times \W^1)$, which is a $\sc^+$-section on $U_\vartriangle$. Then by the same reason for \eqref{eqn:fredcoor}, we have
$$\tilde{\gamma}(c,t) = \gamma \circ \beta \circ \iota \circ q(c,t).$$ 
We write $$B(c,w):=\pi_\W\circ A_*s^f (c,w)-\pi_\W\gamma(c,w)-w$$ and 
\begin{eqnarray*}
\tilde{B}(c,t) & := & \pi_\T\circ Q^*\tilde{s}^f(c,t) - \pi_\T\circ \tilde{\gamma}(t,w) - t \\
& = & \pi_{\T}\circ  A_*s^f \circ \beta \circ \iota\circ q(c,t)-\pi_{\T}\circ  \gamma \circ \beta \circ \iota \circ q(c,t)-t.
\end{eqnarray*}
To show  $Q^*\tilde{s}^f$ is a basic germ, it suffices to prove $\tilde{B}(c,t)$ has the contracting property \eqref{contr}. By \eqref{eqn:property} we have
$$\begin{array}{rcl}
\tilde{B}(c,t)&=&\pi_{\T} \circ A_*s^f\circ \beta \circ \iota \circ q(c,t)-\pi_{\T}\circ  \gamma \circ \beta \circ \iota \circ q(c,t)-t\\
&=& \pi_{\T}\circ A_*s^f(c,t+\theta(c,t))-\pi_{\T}\circ  \gamma (c,t+\theta(c,t))-t\\
&=&\pi_{\T} \circ B(c,t+\theta(c,t)).\end{array}$$
Since $B(c,w)$ has contracting property \eqref{contr} and $\theta(0,0) = 0$,  $\forall \epsilon>0, m\ge 1\in\N$ there exist an open neighborhood $U_{\epsilon, m} \subset C\times \T_m$ of $(0,0)$, over which we have:
$$||\tilde{B}(c,t_1)-\tilde{B}(c,t_2)||_{\T_m}\le \epsilon ||\theta(c,t_1)+t_1-\theta(c,t_2)-t_2||_{\W_{m}}.$$
Since $\theta(c,t)$ is $C^1$ on $C\times \T_m$ for $m\ge 1$ by Lemma \ref{lemma:help}, $||\theta(c, t_1)-\theta(c,t_2)||_{\W_{m}}$ is bounded by $C_{m} ||t_1-t_2||_{\T_m}$ on $U_{(m)}\subset C\times \T_m$. Since $||t_1 - t_2||_{\W_{m}} = ||t_1-t_2||_{\T_m}$ for $t_1,t_2\in \T_m$, we have $$||\tilde{B}(c,t_1)-\tilde{B}(c,t_2)||_{\T_m}\le \epsilon(C_m+1)||t_1-t_2||_{\T_m}, \quad \forall (c,t_1),(c,t_2) \in U_{\epsilon, m}\cap U_{(m)}.$$ This proves that $Q^*\tilde{s}^f$ is a basic germ at $(0,0)$.
\end{proof}

\begin{remark}
	In the proof of Lemma \ref{lemma:fred}, a good slice is used to get a sc-diffeomorphism $q$ such that \eqref{eqn:property} holds. The importance of \eqref{eqn:property} is that $\beta \circ \iota \circ q$ does not change the coordinate in $C$, so that we can use the contracting property of $A_*s^f-\gamma$. 
\end{remark}

Next we prove the existence of good slices. We first prove a more general result that will also be used in later in the construction of polyfold quotient. When constructing slices using Lemma \ref{lemma:gamma}, we can choose different complements $\H$ to construct different changes of coordinates $h$. When the base M-polyfold is infinite dimensional, then it is always possible to find a complement $\H$ such that the slice constructed in Lemma \ref{lemma:gamma} is good. To be more specific, we have the following proposition.
\begin{proposition}\label{prop:goodslice}
	Let  $((\cO, \Phi, (P, (\R_+^m\times \E)\lhd \F)),A)$ be a Fredholm chart (Definition \ref{def:fredchart}) of tame strong M-polyfold bundle $p:\cY\to \cX$ around $x_0\in \cX_\infty$ with respect to a sc-Fredholm section $s:\cX\to \cY$.  Suppose the covered chart on the base is $(\cO, \phi, (O, \R_+^m\times \E))$. Assume the bundle isomorphism $A$ is the form of
	$$U\lhd \F \to U'\lhd (\R^N\times \W),\quad (x,v)\mapsto (\alpha(x),\Lambda(x)v),$$
	where $U\subset \R_+^m\times \E$ and $U'\subset C\times \W := \R_+^m\times \R^k\times \W$ are neighborhoods of $(0,0)$. Under the assumptions of Lemma \ref{lemma:gamma}, that is 
	\begin{enumerate}
		\item 	for a neighborhood $B\subset \R^n$ of $0$ we have a strong bundle map $\Lambda:B\times p^{-1}(\cO)\to \cY$ such that $\Lambda(0,\cdot) = \Id_{p^{-1}(\cO)}$;
		\item  let $\Gamma:B\times \cO \to \cX$ denote the map on the base covered by $\Lambda$, then $\Xi:=\rD(\phi\circ \Gamma)_{(0,x_0)}(T_0B \times \{0\}) \subset (T^R_{x_0}(O))_\infty\subset \{0\}\times \E_\infty$ and $\dim\Xi=n$.
	\end{enumerate}
	If $\cX$ is infinite dimensional, then there exists a sc-complement $\H$ of $\Xi$ in $\E$, such that the slice $\tilde{\cO}'\subset \cX^2$ yield by Lemma \ref{lemma:gamma} is good with respect to $s$.
\end{proposition}
\begin{remark}\label{rmk:goodslice}
	Recall that in Lemma \ref{def:slice}, $G$-slices (Definition \ref{def:Gslice}) are constructed by applying Lemma \ref{lemma:gamma} to $\Lambda: B\times p^{-1}(\cO) \to \cY, (g,v)\mapsto \rho(g,v)$, where $B\subset G$ is a neighborhood $\Id$ and $\rho:G\times \cY \to \cY$ is the action. Therefore by Proposition \ref{prop:goodslice} and Lemma \ref{lemma:slice}, if the base M-polyfold $\cX$ is infinite dimensional, then there exists $G$-slice $(\tilde{\cO},\cO,V,f,\eta,N)$ around $x_0\in \cX_\infty$ such that $\tilde{\cO}\subset \cX^2$ is good.
\end{remark}
\begin{proof}
	If we pick a sc-complement $\H$ of $\Xi$ in $\E$, then Lemma \ref{lemma:gamma} yields a slice $\tilde{\cO}'\subset \cX^2$ around $x_0$. And by property \eqref{part:4} of Lemma \ref{lemma:gamma}, there is a change of coordinate $h:U'\to U''$ for neighborhoods $U'\subset U^2$ of $(0,0)$ and $U''\subset \R_+^m\times \R^n\times \K^1$ of $(0,0,0)$, such that the slice $\tilde{\cO}'=\phi^{-1}\circ r \circ h^{-1}((\R_+^m\times \{0\}\times \K^1)\cap U'')$. Recall from Definition \ref{def:good}, to show $\tilde{\cO}'$ is a good slice, it suffices to prove $\pi_{\R^n}\circ \rD(h\circ \alpha^{-1})_{(0,0)}|_{\{0\}\times \W^2}$ is surjective onto $\R^n$. Since $\rD \alpha^{-1}|_{(0,0)}(\{0\}\times \W^2)\subset \R^m\times \E^2$ is finite codimensional in $\E^2$ and $T^R_{(0,0)}O^2\subset \{0\}\times \E^2$ is infinite dimensional by assumption, then 
	$$\M := T^R_{(0,0)}O^2\cap \rD \alpha^{-1}|_{(0,0)}(\{0\}\times \W^2)\subset T^R_{(0,0)}O^2 \subset \{0\}\times \E^2$$
	is infinite dimensional and it is sufficient to prove 
	\begin{equation}\label{eqn:goal}
	\pi_{\R^n}\circ \rD h_{(0,0)}(\M) = \R^n
	\end{equation}
	By property \eqref{part:4} and \eqref{part:5} of Lemma \ref{lemma:gamma}, we have
	\begin{equation}\label{eqn:lin}
	\pi_{\R^n}\circ \rD h_{(0,0)}(\{0\}\times \Xi) = \R^n, \quad \pi_{\R^n}\circ \rD h_{(0,0)}(T^R_{(0,0)}O^2\cap (\{0\}\times \H)) = 0.
	\end{equation}
	Note that $\Xi\oplus \left(T^R_{(0,0)}O^2\cap (\{0\}\times \H)\right) = T^R_{(0,0)}O^2$, then by \eqref{eqn:lin}, \eqref{eqn:goal} is equivalent to $\M + T^R_{(0,0)}O^2\cap (\{0\}\times \H) = T^R_{(0,0)}O^2$. Let $\pi_\Xi$ denote the projection of $\E=\Xi\oplus \H$ to $\Xi$. Then it suffices to prove $\pi_{\Xi}(\M) = \Xi$. 
	
	We claim that there exists a projection $\pi_\epsilon:\E \to \Xi$ such that $\pi_\epsilon (\M) = \Xi$. Then by the discussion above $\H:= \ker \pi_{\epsilon}$ gives rise to a good slice. To prove the claim, we first pick any projection $\pi_0:\E \to \Xi$ and then we will perturb $\pi_0$ to $\pi_\epsilon$ so that $\pi_\epsilon(\M) = \Xi$. Since $\M$ is infinite dimensional, we can pick linearly independent vectors $\theta_1,\ldots, \theta_n \in \M_\infty$ such that $\la \theta_1,\ldots,\theta_n\ra \cap \Xi = \{0\}$. Then by the Hahn-Banach theorem, there are continuous linear functionals (hence sc-smooth) $l_i:\E \to \R$ such that $l_i(\theta_j) =\delta_{ij}$ and $\Xi\subset \ker l_i$. Let $\xi_1,\ldots,\xi_n$ be a basis of $\Xi$, then we define for $\epsilon := (\epsilon_1,\ldots,\epsilon_n)\in \R^n$, 
	$$\pi_{\epsilon}:\E \to \Xi, \quad e \mapsto \pi_0(e) + \sum_{i=1}^n \epsilon_i l_i(e) \xi_i.$$
	Therefore $\pi_\epsilon(\xi_i) = \xi_i$ and $\pi_\epsilon\circ \pi_\epsilon = \pi_\epsilon$ for every $\epsilon$, that is $\pi_\epsilon$ is a projection from $\E$ to $\Xi$. If we consider $\pi_{\epsilon}|_{\la\theta_1,\ldots,\theta_n\ra}$, then
	$$\det(\pi_{\epsilon}|_{\la\theta_1,\ldots,\theta_n\ra}) = \prod_{i=1}^n\epsilon_i+\text{ lower order terms,}$$
	when we use $\{\theta_1,\ldots,\theta_n\}$ and $\{\xi_1,\ldots,\xi_n\}$ as basis. Therefore there exists $\epsilon$ such that $\pi_\epsilon|_{\la \theta_1,\ldots,\theta_n\ra}$ is an isomorphism. Hence the claim is proven for such $\pi_\epsilon$.
\end{proof}

\begin{proof}[Proof of part \eqref{part:freetwo} of Theorem \ref{thm:free}]
	Let $\{(\tilde{\cO}_x, \cO_x, V_x, f_x,\eta_x,N_x)\}_{x\in \cX_\infty}$ be the bundle $G$-slices used in Lemma \ref{lemma:structure} for the construction of $\hX$. Therefore by Lemma \ref{lemma:structure}, $\Psi_x: p^{-1}(\tilde{\cO}_x) \stackrel{\iota}{\hookrightarrow} p^{-1}(\cO_x) \stackrel{\pi_G}{\to} p^{-1}(\hX)/G$ in \eqref{eqn:Psi} are strong bundle isomorphisms, where $\pi_G:p^{-1}(\hX) \to p^{-1}(\hX)/G$ is the quotient map and $\iota:p^{-1}(\tilde{\cO}_x)\to p^{-1}(\cO_x)$ is the inclusion.
	
	Since the section $s$ is $G$-equivariant, $s$ induces a continuous section $\os:\hX^1/G \to p^{-1}(\hX^1)/G$ by $s|_{\hX^1} = \pi_G^*\os$. To see $\os$ is sc-smooth and has the regularization property, since $\Psi_x^*\os = \iota^*\pi_G^*\os = \iota^*s = \tilde{s}:=s|_{\tilde{\cO}^1_x}$, Proposition \ref{prop:reg} implies that $\os$ is sc-smooth and has the regularization property.
	
	Next, we claim that $\os:\hX^1/G\to p^{-1}(\hX^1)/G$ has the basic germ property. It suffices to show that there is a Fredholm chart in $\cO^1_x$ around $x$ for $\tilde{s}$. By Remark \ref{rmk:goodslice}, for every $x\in \cX_\infty$ there exists a bundle $G$-slice $(\tilde{\cO}'_x, \cO'_x, V_x', f_x',\eta_x', N_x')$ around $x$ such that  $\tilde{\cO}'_x\subset \cX^2$ is a good slice. In particular, by Lemma \ref{fred} there is a Fredholm chart around $x$ in $(\tilde{\cO}'_x)^1$ for $\tilde{s}':=s|_{(\tilde{\cO}'_x)^1}:(\tilde{\cO}'_x)^1 \to p^{-1}((\tilde{\cO}'_x)^1)$. The good slice $\tilde{\cO}'_x$ is not necessarily $\tilde{\cO}_x$ in the construction of $\hX$. However, by \eqref{eqn:bundiso} $\Psi_x^{-1}\circ \Psi'_x$ is a strong bundle isomorphism, where $\Psi'_x$ is the map in \eqref{eqn:Psi} for the slice $\tilde{\cO}'_x$. Since $s$ is $G$-equivariant, we have
	$$\tilde{s}'= s|_{(\tilde{\cO}'_x)^1} = (\Psi'_x)^*\os = (\Psi^{-1}_x\circ \Psi'_x)^*\Psi^*_x \os =  (\Psi^{-1}_x\circ \Psi'_x)^*\tilde{s}$$ 
	in a neighborhood of $x$ where the strong bundle isomorphism $\Psi^{-1}_x \circ \Psi'_x$ is defined. Therefore there is a Fredholm chart in $\tilde{\cO}^1_x$ around $x$ for $\tilde{s}$.
	
	Since $s^{-1}(0)$ is compact in $\cX_0$ topology, by Remark \ref{rmk:compact}, $s^{-1}(0)$ is compact in $\cX^3$ topology. Because $\os^{-1}(0) = s^{-1}(0)/G$ and $G$ is compact, $\os^{-1}(0)$ is compact in $\cX^3/G$ topology, i.e. the $\hX^1/G$ topology. That is $\os:\hX^1/G \to p^{-1}(\hX^1)/G$ is a proper sc-Fredholm section.
\end{proof}

In the following, we recall the concept of general position from \cite{hofer2017polyfold}, which plays an important role in getting boundary and corner structures on the zero set of a sc-Fredholm section.
\begin{definition}[{\cite[Definition 5.9]{hofer2017polyfold}}]
	Given a sc-Fredholm section $s:\cX \to \cY$ of a tame strong M-polyfold bundle $p:\cY \to \cX$, we say $s$ is in \textbf{general position} if the linearization $\rD s_x: T_x\cX \to \cY_x$ for $x\in s^{-1}(0)$ is surjective and $\ker \rD s_x$ has a sc-complement in the reduced tangent space $T^R_x\cX$. Equivalently, $s$ is in general position if for all $x\in s^{-1}(0)$ the restricted linearization $\rD s_x|_{T^R_x\cX}: T^R_x\cX \to \cY_x$ is surjective. In other words, $s$ restricted to all the boundary and corner is transverse to $0$. 
\end{definition}		
\begin{remark}\label{rmk:general}
	Under the assumptions of Theorem \ref{thm:free}, $s$ is in general position iff $\os$ is also in general position. Let $(\tilde{\cO},\cO,V,f,\eta,N)$ be a $G$-slice around $x\in \cX_\infty$, since the infinitesimal directions $\Xi$ of the group action is contained in $T^R_x\cO$. Moreover, $T^R_x\cO = \Xi \oplus T^R_{x}\tilde{\cO}$. Since $s$ is $G$-equivariant, we have $\rD s_x(\Xi) = 0$. Therefore $\rD s_x|_{T^R_x\cO} \to \cY_x$ being surjective is equivalent to that $\rD s_x|_{T^R_x\tilde{\cO}} \to \cY_x$ is surjective, that is $\os$ is general position. 
\end{remark}

\cite[Theorem 5.6]{hofer2017polyfold} assert that there exist $\sc^+$- perturbations so that the perturbed sections are in general position. Moreover, the zero set of a sc-Fredholm section in general position is a manifold  with boundary and corner and the degeneracy index on the zero set equals to the degeneracy index of the base M-polyfold.

\begin{remark}
	The concept of good position was introduced in \cite[Definition 3.16]{hofer2017polyfold}. When a sc-Fredholm (multi)section is in good position on the boundary, the zero set still has boundary and corner structure by \cite[Theorem 3.13]{hofer2017polyfold}. It can be shown that if $s$ is in good position then $\os$ is also in good position. 
\end{remark}


\section{Finite Isotropy quotient of Polyfolds}\label{s3}
 In this section, we prove the main result Theorem \ref{thm:main} except the claim on orientation. Polyfolds are modeled on groupoids as in \cite{moerdijk2002orbifolds}, hence should be thought of as the orbifold version of M-polyfolds. Since a polyfold can have finite isotropy, the quotient theorem for polyfolds also allows the group action to have finite isotropy. If one works on general polyfolds, some pathological polyfolds, e.g Example \ref{ex:irr-one} and Example \ref{ex:irr-two}, obstruct the construction in this paper. In particular, the Hausdorffness of quotients might fail, see Remark \ref{rmk:reduced} and Example \ref{ex:three}. Therefore we introduce regular polyfolds and bundles in Definition \ref{def:reg} and Definition \ref{def:regbund}, the quotient construction works for regular polyfolds and regular polyfold bundles. All Polyfolds in known applications \cite{hofer2017applications,hofer2017application,li,wehrheim2012fredholm} are regular and regularity can be checked easily by Proposition \ref{prop:regular}, Proposition \ref{prop:hofer-regular} and Corollary \ref{cor:regular}. We also point out that orbifolds are always regular.
 
 We review the basics of polyfolds and introduce regular polyfolds in Section 3.1. Section 3.2 introduces the definition of group actions on polyfolds and Section 3.3 discusses properties of group actions. We construct the polyfold quotients in Section 3.4. Section 3.5 is about quotients of polyfold bundles and sc-Fredholm sections.

\subsection{Regular polyfolds}
 We first recall definitions and properties of polyfolds that will be used in this paper from \cite{hofer2017polyfold}. For that purpose, we first recall the definition of an ep-groupoid.
 
\subsubsection{Ep-groupoids and regular ep-groupoids}
\begin{definition}[{\cite[Def 7.3]{hofer2017polyfold}}]\label{groupoid}
	An \textbf{ep-groupoid} $(\cX,\bX)$ is a groupoid such that the object set $\cX$ and the morphism set $\bX$ are M-polyfolds. Moreover, all the structure maps: source $s$, target $t$, composition $m$, unit $u$ and inverse $i$ are $sc$-smooth and the following two properties hold,
	\begin{itemize}
		\item{\textbf{\'{e}tale}}: $s,t$ are surjective local $sc$-diffeomorphisms;
		\item{\textbf{proper}}: for every $x\in \cX$, there exists an open neighborhood $\cV$ of $x$, such that $t: s^{-1}(\overline{\cV})\to \cX$ is proper, where $\overline{\cV}$ is the closure of $\cV$ in $\cX$.\footnote{Hence for every open set $\cW$ with $\overline{\cW}\subset \overline{\cV}$, $t: s^{-1}(\overline{\cW})\to \cX$ is proper.}
	\end{itemize}	
	An ep-groupoid is \textbf{tame} if the object space $\cX$ is a tame M-polyfold.
\end{definition}
The proper condition of Definition \ref{groupoid} implies that for every $x\in \cX$, the isotorpy group $\stab_x$ is finite \cite[Proposition 7.4]{hofer2017polyfold}. The \'{e}tale condition of Definition \ref{groupoid} implies that every $\phi\in \bX$ there exists neighborhoods $\cU_{s(\phi)}\subset \cX,\cU_{t(\phi)} \subset \cX$ and $\bU_{\phi} \subset \bX$ of $s(\phi), t(\phi)$ and $\phi$, such that both $s:\bU_\phi \to \cU_{s(\phi)}$ and $t:\bU_{\phi} \to \cU_{t(\phi)}$ are diffeomorphisms. The following definition was introduced in \cite[\S 7.1]{hofer2017polyfold}.
\begin{definition}\label{def:lphi}
	For every $\phi \in \bX$, we define \begin{equation}\label{eqn:local}
	L_\phi:= t\circ s^{-1}, \quad \cU_{s(\phi)} \to \cU_{t(\phi)},
	\end{equation} which is a sc-diffeomorphism.
\end{definition} 
\begin{definition}\label{def:actgroupoid}
	Let $\Omega$ be a finite group. Assume there is a sc-smooth group action $\rho:\Omega \times \cO \to \cO$ on a M-polyfold $\cO$.  The \textbf{translation groupoid} $\Omega\ltimes \cO$ is defined to be:
	$$\obj(\Omega\ltimes \cO) := \cO, \qquad \mor(\Omega\ltimes \cO) := \Omega\times \cO.$$
	The structure maps are defined by:
	$$ s:(g,x) \mapsto x, \qquad t: (g, x) \mapsto \rho(g,x), \qquad \forall (g,x)\in \mor(\Omega\ltimes \cO)=\Omega \times \cO$$
	The composition for $(h,y)$ and $(g,x)$ satisfying $y = \rho(g,x)$ is defined to be  $(hg, x)$. The inverse map is given by $i(g,x) = (g^{-1}, \rho(g,x))$ and the unit map is given by $u(x) = (\Id, x)$.
\end{definition}
It is clear that the source and target maps of $\Omega\ltimes \cO$ are \'{e}tale and $\Omega \ltimes \cO$ is proper by Proposition \ref{prop:proper}. Therefore $\Omega \ltimes \cO$ is an ep-groupoid. The following theorem, as another consequence of the \'etale and proper properties in Definition \ref{groupoid}, essentially says that an ep-groupoid is locally isomorphic to a translation groupoid.

\begin{theorem}[{\cite[Thm 7.1]{hofer2017polyfold}}]\label{thm:natural}
	Given an ep-groupoid $(\cX, \bX)$, an object $x\in \cX$ and an open neighborhood $\cV_x\subset \cX$ of $x$. Then there exist an open neighborhood $\cU_x\subset \cV_x$ of $x$ and a group homomorphism defined by
	$$\stab_{x}\to \diff_{\sc}(\cU_x), \quad \phi\mapsto L_\phi.$$
	where $\diff_{\sc}(\cU_x)$ is the group of sc-diffeomorphism from $\cU_x$ to itself. Moreover, there is a sc-smooth map
	$$\Sigma: \stab_{x}\times \cU_x \to \bX, $$
	such that they have the following properties:
	\begin{itemize}
		\item $\Sigma(\phi,x)=\phi$; 
		\item $s(\Sigma(\phi,y))=y$ and $t(\Sigma(\phi,y))=L_\phi(y)$ for all $y\in \cU_x$ and $\phi\in \stab_{x}$;
		\item If $\psi:y\to z$ is a morphism connecting two objects $y,z\in \cU$, then there is a unique $\phi\in \stab_{x}$ such that $\Sigma(\phi,y)=\psi$.
	\end{itemize}
In other words, $\stab_x\ltimes \cU_x \stackrel{(i_{\cU_x},\Sigma)}{\longrightarrow}(\cX,\bX)$ is a fully faithful functor, where $i_{\cU_x}:\cU_x \to \cX$ is the inclusion.
\end{theorem}
The following corollary is a direct consequence of Theorem \ref{thm:natural} and Definition \ref{def:lphi}.
\begin{corollary}\label{coro:natural}
	Let $(\cX, \bX)$ be an ep-groupoid and $\cU_x \subset \cX$ a neighborhood of $x\in \cX$ such that Theorem \ref{thm:natural} holds. Let $\psi \in \bX$ such that $s(\psi),t(\psi) \in \cU_x$. Then $L_{\psi}$ is defined on $\cU_x$ and there exists $\phi\in \stab_{x}$ such that $L_{\psi} = L_\phi$ on $\cU_x$.
\end{corollary}
\begin{remark}\label{rmk:act}
	If $(\cX,\bX)$ is only an \'etale groupoid with finite isotropy. Then for every $x\in \cX$ and open neighborhood $\cV_x$ containing $x$, there exists an open neighborhood $\cU_x$ containing $x$ such that 
	$$\stab_{x}\to \diff_\sc(\cU_x),\quad \phi \mapsto L_\phi$$
	is a group homomorphism, see the proof of \cite[Thm 7.1]{hofer2017polyfold}.
\end{remark}

The fully faithful functor $\stab_x\ltimes \cU_x \stackrel{(i_{\cU_x},\Sigma)}{\longrightarrow}(\cX,\bX)$ can be thought of as the local charts on an ep-groupoid. To be more precise, the following definition was introduced in \cite{hofer2017polyfold}. For an ep-groupoid $(\cX,\bX)$, there is an equivalence relation $\sim_{\bX} $ on $\cX$ such that $x\sim_{\bX} y$ iff there exists $\phi \in \bX$ with $s(\phi) = x$ and $t(\phi) = y$. Let $|\cX|:= \cX/\sim_{\bX}$. Then $|\cX|$ is a topological space equipped with the quotient topology. We use $|\cdot|$ to denote the quotient map $\cX \to |\cX|$.

\begin{definition}[{\cite[Definition 7.9]{hofer2017polyfold}}]\label{def:uni}
	Let $(\cX,\bX)$ be an ep-groupoid and $x$ be an object in $\cX$. A \textbf{local uniformizer around $\bm{x}$} is a fully faithful functor
	$$\Psi_x: \stab_x \ltimes \cU_x \to (\cX, \bX),$$
	where $\cU_x\subset \cX$ is an open neighborhood of $x$, such that the following properties hold:
	\begin{enumerate}
		\item on the object level $\Psi_x^0:\cU_x\to \cX$ is an inclusion of open subset, on the morphism level $\Psi_x^1:\stab_x\times \cU_x \to \bX$ is sc-smooth;
		\item on the orbit space $|\Psi_x|: \cU_x/\stab_x \to |\cX|$ is homeomorphism onto an open subset of $|\cX|$, i.e. $|\cU_x|$ is open in $|\cX|$.
	\end{enumerate}
\end{definition}

Every neighborhood $\cU_x$ in Theorem \ref{thm:natural} defines a local uniformizer by $\stab_x\ltimes \cU_x\stackrel{(i_{\cU_x},\Sigma)}{\longrightarrow}(\cX,\bX)$, where the property on the orbit space was verified in   \cite[Proposition 7.6]{hofer2017polyfold}. Hence we will call $\cU_x$ a local uniformizer when there is no confusion.

In constructions of ep-groupoids, we need to verify the properness in Definition \ref{groupoid}. For that purpose, we introduce the following characterization of properness in Proposition \ref{prop:proper}
\begin{definition}\label{def:orbitset}
	An \textbf{orbit set} in an open subset $\cU\subset \cX$ of a point $x\in \cX$ is defined to be the set $ S_{x,\cU} := \{\phi \in \bX| s(\phi) = x, t(\phi) \in \cU\}\subset \bX$.
\end{definition}
Note that we do not require $x \in \cU$ and the orbit set $S_{x,\cU}$ only depends on $|x| \in |\cX|$ and $\cU$. Theorem  \ref{thm:natural} implies that if $\cU_{x}$ is a local uniformizer around $x$ and $y\in \cU_{x}$, the cardinality of the orbit set $S_{y,\cU_{x}}$ is $|\stab_{x}|$. Also note that $t:s^{-1}(\overline{\cV})\to \cX$ being proper implies that $t(s^{-1}(\overline{\cV})) \subset \cX$ is sequentially closed. Since M-polyfold $\cX$ is metrizable, $t(s^{-1}(\overline{\cV}))$ is closed. In fact, those two properties imply properness of an \'{e}tale groupoid, i.e. we have the following proposition.

\begin{proposition}\label{prop:proper}
Let $(\cX, \bX)$ be  an \'{e}tale groupoid with finite isotropy. Assume for every $x\in\cX$, there exists an open neighborhood $\cU_x\subset \cX$ of $x$, such that the following two conditions hold.
\begin{enumerate}
	\item\label{ass:1} For every $y\in \cU_x$, $|S_{y,\cU_x}| = |\stab_x|$.
	\item\label{ass:2} There exists an open neighborhood $\cV_x\subset \overline{\cV_x} \subset \cU_x$ of $x$, such that $t(s^{-1}(\overline{\cV_x}))$ is a closed subset of $\cX$.
\end{enumerate} 
Then $(\cX,\bX)$ is an ep-groupoid. 
\end{proposition}
\begin{proof}
	We claim that $t:s^{-1}(\overline{\cV_x})\to \cX$ is proper. Let $K\subset \cX$ be compact set. It suffices to prove $t^{-1}(K)\cap s^{-1}(\overline{\cV_x}) \subset \bX$ is compact. We will prove compactness by a finite cover of compact sets. Let $n:= |\stab_x|$. By assumption \eqref{ass:1}, for $y\in t(s^{-1}(\cU_x))$, there are exactly $n$ morphisms $\phi_1,\ldots,\phi_n$, such that $t(\phi_i)=y$ and $s(\phi_i)\in \cU_x$. Because the source and target maps $s,t$ are local $sc$-diffeomorphisms,  $t(s^{-1}(\cU_x))$ is open and there exist open neighborhood $\cO\subset t(s^{-1}(\cU_x))$ of $y$ and open neighborhoods $\bO(1), \ldots, \bO(n)$ of $\phi_1,\ldots,\phi_n$ in $\bX$, such that $\bO(i)\cap \bO(j)=\emptyset$ for all $i\ne j$ and both $t|_{\bO(i)}:\bO(i)\to \cO$, $s|_{\bO(i)}:\bO(i)\to s(\bO(i))\subset \cU_x$ are $sc$-diffeomorphisms for all $i$.  For every $z\in \cO$, by assumption \eqref{ass:1} there are also exactly $n$ morphisms $\psi_1,\ldots, \psi_n\in \bX$, such that $s(\psi_i)\in \cU_x$ and $t(\psi_i)=z$. Since $\left(t|_{\bO(i)}\right)^{-1}(z)$ already provides $n$ such elements, we must have $\psi_i\in \bO(i)$ (after a permutation). Therefore for every set $\cW\subset \cO$, the set of all the morphisms from $\cU_x$ to $\cW$ is completely described by:
	$$t^{-1}(\cW)\cap s^{-1}(\cU_x)= \coprod_{i=1}^n \bO(i)\cap t^{-1}(\cW)=\coprod_{i=1}^n t|_{\bO(i)}^{-1}(\cW).$$
	As a consequence, for every open neighborhood $\cO'\subset \cX$ of $y$ with $\overline{\cO'}\subset\cO$, we have
	$$ t^{-1}(\overline{\cO'}\cap K) \cap s^{-1}(\cU_x) = \coprod_{i=1}^n (t|_{\bO(i)})^{-1}(\overline{\cO'}\cap K)$$
	is compact, since $\overline{\cO'}\cap K \subset \cO$ is compact. Since $s^{-1}(\overline{\cV_x})$ is closed and $s^{-1}(\overline{\cV_x}) \subset s^{-1}(\cU_x)$, we have $t^{-1}(\overline{\cO'}\cap K) \cap s^{-1}(\overline{\cV_x})$ is compact. To sum up, we prove that for every $y\in t(s^{-1}(\cU_x))$ there exists an open neighborhood $\cO' \subset t(s^{-1}(\cU_x))$ of $y$ such that $(t|_{s^{-1}(\overline{\cV}_x)})^{-1}(\overline{\cO'}\cap K) = t^{-1}(\overline{\cO'}\cap K) \cap s^{-1}(\overline{\cV_x})$ is compact.
	
	By assumption \eqref{ass:2} that  $t(s^{-1}(\overline{\cV_x}))$ is closed,   $K\cap t(s^{-1}(\overline{\cV_x}))$ is compact. Therefore we can cover the compact set $K\cap t(s^{-1}(\overline{\cV_x}))$ by finitely many such $\cO'$. Then $(t|_{s^{-1}(\overline{\cV_x})})^{-1}(K) = t^{-1}(K)\cap s^{-1}(\overline{\cV_x})$ is a union of finitely many compact sets, hence is also compact. 
\end{proof}
The next proposition asserts that one only needs to verify the first condition in Proposition \ref{prop:proper} for very few points.
\begin{proposition}\label{prop:helper}
	Let $(\cX,\bX)$ be an \'{e}tale groupoid with finite isotropy. Suppose $\cU\subset\cX$ is an open neighborhood  of $x$ and $|\stab_{y,\cU}| = |\stab_x|$ for all $y\in \cU$. Then for every $y\in \cU$, there exists an open neighborhood $\cV \subset \cU$ of $y$, such that $|S_{z,\cV}| = \stab_y$ for $z\in \cV$.
\end{proposition}
\begin{proof}
	Let $n:= |\stab_x|$. Since $|S_{y,\cU}| = |\stab_x| = n$ for $y\in \cU$, there exists $\phi_1,\ldots,\phi_n \in \bX$ such that $s(\phi_i) \in \cU$ and $t(\phi_i) = y$.  We assume $\phi_i \in \stab_y$ iff $i\le k$.  By the \'{e}table property, there exist open neighborhood $\cO\subset\cU$ of $y$ and disjoint open neighborhoods $\bO(i)\subset \bX$ of $\phi_i$, such that $s:\bO(i) \to s(\bO(i)) \subset \cU$ and $t:\bO(i) \to \cO$ are diffeomorphisms. Moreover, we can assume $\cO\cap s(\bO(i)) = \emptyset$ and $s(\bO(j))\cap s(\bO(i)) = \emptyset$ for $i > k$ and $j \le k$.  By Remark \ref{rmk:act}, there exists open neighborhood $\cV\subset \cO \cap_{i=1}^k s(\bO(i))$ of $y$ admitting a $\stab_y$ action defined by $\phi_i \mapsto L_{\phi_i}$.  Then $|S_{z,\cV}|\ge |\stab_y| = k$ by the construction of $\cV$.  By the same argument used in Proposition \ref{prop:proper}, we have $S_{z,\cU} = \{t|_{\bO(i)}^{-1}(z) \}$ for $z\in \cO$. Note that $s(t_{\bO(j)}^{-1}(z)) \in s(\bO(j))$ and $s(\bO(j)) \cap \cV = \emptyset$ for $j>k$. Since $S_{z,\cV} \subset S_{z,\cU}$.  we have $|S_{z,\cV}|  \le k$ for every $z\in \cV$.  Hence the proposition holds.
\end{proof}

Next we introduce the concept of regular ep-groupoid, the motivation of such definition is to rule out pathological ep-groupoids in Example \ref{ex:irr-one} and Example \ref{ex:irr-two}.
\begin{definition}\label{def:regular}
	For $x\in \cX$, a local uniformizer $\stab_x\ltimes \cU$ around $x$ is \textbf{regular} if the following two conditions are met:
	\begin{enumerate}
		\item if $L_\phi|_{\cV} = \Id_{\cV}$ for some open subset $\cV \subset \cU$ and $\phi\in \stab_x$, then $L_\phi = \Id$ on $\cU$;
		\item for every connected\footnote{Since M-polyfolds are locally path connected, path connectedness is equivalent to connectedness.} uniformizer $\cW \subset \cU$ around $x$, if $\Phi:\cW \to \stab_x$ is a map such that $L_{\Phi}:\cW \to \cX, z\mapsto L_{\Phi(z)}(z)$ is sc-diffeomorphism, then there exists $\phi\in \stab_x$, such that $L_{\Phi(z)}(z) = L_{\phi}(z)$ for all $z\in \cW$. 
	\end{enumerate}
\end{definition}
By Definition \ref{def:regular}, any smaller uniformizer around $x$ inside a regular uniformizer around $x$ is also regular. 
\begin{remark}
	The second condition in Definition \ref{def:regular} can be weakened from all $\cW \subset \cU$ to a (countable) local topology basis $\{\cW_i\subset \cU\}_{i\in \N}$. 
	All the arguments used in this paper go through. For the simplicity of notation, we only work with Definition \ref{def:regular} in this paper.
\end{remark}

\begin{definition}\label{def:reg}
An ep-groupoid $(\cX,\bX)$ is \textbf{regular}, if for any point $x\in \cX$, there exists a regular local uniformizer around $x$.\footnote{The regularity of an ep-groupoid should be differed from the regularity of the  topology on the orbits space.} 
\end{definition}

\begin{remark}
	It is not clear whether the regularity of $(\cX,\bX)$ implies the regularity of $(\cX^i,\bX^i)$ for other $i > 0$. However, if the condition in Corollary \ref{cor:regular} holds, then $(\cX^i,\bX^i)$ are regular for all $i \ge 0$.
\end{remark}
 
In most of the applications, we can use one of the following propositions to confirm regularity. We first recall the definition of effective ep-groupoid. Let $\diff_{\sc}(x)$ be the group of germs of diffeomorphisms fixing $x$. Then an action $\stab_x\to \diff_{\sc}(\cU)$ in Theorem \ref{thm:natural} descends to a group homomorphism $\stab_x \to \diff_{\sc}(x)$. 
\begin{definition}
The \textbf{effective part of $\bm{\mathrm{stab}_x}$} is defined to be
\begin{equation}\label{eqn:eff}
\stab^{\eff}_x:= \stab_x/(\ker(\stab_x\to \diff_{\sc}(x))).
\end{equation}
\end{definition}

\begin{definition}[{\cite[Definition 7.11]{hofer2017polyfold}}]
	An ep-groupoid is effective if the homomorphism $\stab_x\to\diff_{sc}(x)$ is injective for any $x\in \cX$, i.e. $\stab_x = \stab^{\eff}_x$.
\end{definition}
The reason of passing to germs is that there might exist neighborhoods $\cV \subset \cU \subset \cX$ of $x$, such that $\stab_x\to \diff_{\sc}(\cU)$ is injective but $\stab_x\to \diff_{\sc}(\cV)$ is not injective, e.g.\ Example \ref{ex:irr-one}. 

\begin{proposition}\label{prop:regular}
	Let $(\cX,\bX)$ be an effective ep-groupoid. Assume for every $x\in \cX$, there exists a local uniformizer $\cU$ around $x$, such that for any connected uniformizer $\cV \subset \cU$ around $x$, we have $\cV\backslash \cup_{\phi\ne \Id \in \stab_x} \Fix(\phi)$ is connected, where $\Fix(\phi)$ is the fixed set of $L_{\phi}$. Then $(\cX,\bX)$ is regular.
\end{proposition}
\begin{proposition}\label{prop:hofer-regular}
	If $(\cX,\bX)$ has the property that for any $x\in \cX_\infty$ and $\phi\in \stab_x$ if $\rD L_\phi: T_x\cX \to T_x\cX$ is the identity map, then $L_\phi$ is the identity map near $x$. Then $(\cX,\bX)$ is regular.
\end{proposition}	
As a simple corollary of Proposition \ref{prop:hofer-regular}, we have the following.
\begin{corollary}\label{cor:regular}
	If the linearized action $\stab_x \to \Hom(T_x\cX,T_x\cX)$ defined by $\phi \mapsto (\rD L_\phi)_x$ is injective for every point $x\in \cX_\infty$, then $(\cX,\bX)$ is a regular ep-groupoid.
\end{corollary}

By Proposition \ref{prop:hofer-regular}, if the ep-groupoid is modeled on finite dimensional manifolds, i.e. an orbifold, then it is automatically regular. However, the dimension jump phenomenon in M-polyfolds may destroy the regular property, see Example \ref{ex:irr-one} and Example \ref{ex:irr-two}.  The reason of introducing regularity is related to the Hausdorffness of the quotient polyfolds, see Remark \ref{rmk:reduced} and Example \ref{ex:three}. The second property is used to get Proposition \ref{prop:lifting}, so that we can control the isotropy of the quotient ep-groupoid, see \eqref{eqn:exact}. For the other consequences of the second property, see Proposition \ref{allnatural} and Theorem \ref{thm:unique}

Proposition \ref{prop:regular} is proven in the appendix. Proposition \ref{prop:hofer-regular} was proven in \cite[Lemma 10.2]{hofer2017polyfold}. In fact, \cite[Lemma 10.2]{hofer2017polyfold} only stated the second property of a regular uniformizer. However the conditions in Proposition \ref{prop:hofer-regular} imply that the set $S:=\{x\in \cU_\infty| \rD L_\phi =\Id\}$ is both open and closed in $\cU_\infty$ for a uniformizer $\cU$. Hence when $\cU_\infty$ is connected, $S\ne \emptyset$ implies $S = \cU_\infty$. Therefore if $L_\phi = \Id$ on some $\cV \subset \cU$,  $L_\phi = \Id$ on $\cU_\infty$. Since $\cU_\infty$ is dense in $\cU$, $L_\phi = \Id$ on $\cU$. This proves the first property of a regular uniformizer.  

In applications \cite{hofer2017applications,hofer2017application,li,wehrheim2012fredholm},
Corollary \ref{cor:regular} holds. From another perspective, all the related ep-groupoids are effective and the fixed set $\Fix(\phi)$ of a nontrivial element $\phi$ has ``infinite codimension". Hence it is easy to verify the connectedness of the complement. Then Proposition \ref{prop:regular} can also be applied.

Unlike the orbifold case, irregular ep-groupoids exist. The following two examples shows that there exist ep-groupoids do not satisfy any one of the two conditions of regularity.
\begin{example}\label{ex:irr-one}
Consider the following retraction from \cite[Example 1.22]{hofer2010sc}. Let $\E$ be the sc-Hilbert space with $\E_m:=H^m_{\delta_m}(\R)$, where $0 =\delta_0 < \delta_1 < \ldots < \delta_m < \ldots$ are the exponential decay weights. Choose a smooth compactly supported positive function $\beta$, such that $\int_{\R} \beta^2=1$. Define $r_t:\E\to  \E$ to be:
$$r_t(f):=\left\{\begin{array}{lr} 0, & t\le 0;\\
\int_{\R} f(x)\beta(x+e^{\frac{1}{t}}) dx\cdot \beta(x+e^{\frac{1}{t}}), & t>0.\end{array}\right.$$ 
It was checked in \cite{hofer2010sc} that $\pi: \R\times \E\rightarrow \R\times \E, (t,f) \mapsto (t,r_t(f))$ defines a $\sc^\infty$ retraction.  Then the retract $\Ima \pi$ defines an M-polyfold $\cM$.  Pictorially, $\cM$ is a closed ray attached to an open half plane at $(0,0)$. There is  a $\Z_2$ action on the M-polyfold $\cM$ induced by multiplying $-1$ on the $\E$ component. The translation groupoid $\Z_2\ltimes \cM$ is an ep-groupoid, but it does not have the regular property at $(0,0)$, since $\phi \ne \Id \in \stab_{(0,0)}$ fixes the half line part, but does not fix the half plane part. 
\end{example}
\begin{example}\label{ex:irr-two}
	Let $\E$ and $r_t$ denote the same objects as in Example \ref{ex:irr-one}. Then we have a sc-retraction $\pi: \R \times \E  \to \R \times \E , (t,f) \mapsto (t, r_{|t|}(f))$ by the same argument in \cite[Example 1.22]{hofer2010sc}. Topologically, the retract $\Ima \pi$ is two open half spaces $(-\infty, 0)\times \R$ and $(0,\infty)\times \R$ connected at the origin. $\Z_2$ can act on $\R \times \E$ by sending $(t,x)$ to $(t,-x)$, such action induces a sc-smooth $\Z_2$ action on $\Ima \pi$. We claim the second property of regularity fails for the translation ep-groupoid $\Z_2 \ltimes \Ima \pi$. Let $\Phi:\Ima \pi \to \Z_2$ be the map such that $\Phi = \Id \in \Z_2$  when $t \le 0$ and $\Phi$ is the nontrivial element of $\Z_2$ when $t > 0$, then $\eta: \Ima \pi \to \Ima \pi, z \mapsto L_{\Phi(z)}(z)$ is sc-smooth. This is because $\eta\circ \pi$ can be expressed as
	$$\eta \circ \pi(t,f)=\left\{\begin{array}{lr} (t,r_{-t}(f)), & t\le 0;\\
	(t, -r_{t}(f)), & t>0,\end{array}\right.$$ 
	this map is sc-smooth by the same argument in \cite{hofer2010sc}. Therefore the second property of regularity does not hold for the translation groupoid $\Z_2\ltimes\Ima \pi$. It is easy to check directly that neither Proposition \ref{prop:regular} nor Proposition \ref{prop:hofer-regular} hold for $\Z_2\ltimes\Ima \pi$.
\end{example}

\subsubsection{Generalized maps between ep-groupoids}
In this subsection, we review the construction of the category of ep-groupoids. A functor $F:(\cX,\bX) \to (\cY,\bY)$ between ep-groupoids is sc-smooth if both the map on objects $F^0:\cX \to \cY$ and the map on morphisms $F^1:\bX \to \bY$ are sc-smooth. When there is no ambiguity, we will also abbreviate $F^0$ and $F^1$ to $F$.
\begin{definition}[{\cite[Definition  10.1]{hofer2017polyfold}}]\label{def:ep-equi}
A sc-smooth functor $F:(\cX,\bX) \to (\cY, \bY)$ is an \textbf{equivalence} provided it has the following properties:
	\begin{itemize}
		\item $F$ is fully faithful;
		\item the induced map $|F|:|\cX|\rightarrow |\cY|$ between the orbit spaces is a homeomorphism;
		\item $F^0$ is a local sc-diffeomorphism. 
	\end{itemize}
\end{definition}
\begin{definition}[{\cite[Definition 10.2]{hofer2017polyfold}}]
	Two sc-smooth functors $F,G$ from $(\cX, \bX)$ to $(\cY,\bY)$ are called \textbf{naturally equivalent} if there exists a sc-smooth map $\tau: \cX \to \bY$ such that $\tau(x) \in \bY$ is a morphism from $F(x)$ to $G(x)$ and the following diagram commutes for any morphism $\phi\in \bX$ from $x$ to $y$,
	$$
	\xymatrix{ 
		F(x) \ar[r]^{F(\phi)} \ar[d]^{\tau(x)} & F(y) \ar[d]^{\tau(y)} \\
		G(x) \ar[r]^{G(\phi)} & G(y).
	}
	$$
	The sc-smooth map $\tau: \cX \to \bY$ is called the natural transformation from $F$ to $G$ and symbolically referred to as $\tau: F \Rightarrow G$.
\end{definition}
It was shown in \cite[\S 10.1]{hofer2017polyfold} that ``naturally equivariant" is an equivalence relation between sc-smooth functors. If two functors $F$ and $G$ are naturally equivalent, then $|F| = |G|$. Hence two naturally equivalent functors should be thought of as the ``same map"  on ep-groupoids. However, if we only use sc-smooth functors (up to natural equivalence) as morphisms between the ep-groupoids, we will not have enough morphisms. In particular, the equivalences in Definition \ref{def:ep-equi} do not necessarily have inverses. Therefore we need to invert equivalence formally, which is the process of localization. Then the following definition was introduced in \cite{hofer2017polyfold}. 

\begin{definition}[{\cite[Definition 10.8]{hofer2017polyfold}}]
	A \textbf{diagram} from ep-groupoid $(\cX, \bX)$ to ep-groupoid $(\cY, \bY)$ is 
	$$(\cX, \bX) \stackrel{F}{\leftarrow} (\cZ, \bZ) \stackrel{\Phi}{\to} (\cY, \bY),$$
	where $F$ is an equivalence and $\Phi$ is a sc-smooth functor. A digram $(\cX, \bX) \stackrel{F'}{\leftarrow} (\cZ', \bZ') \stackrel{\Phi'}{\to} (\cY, \bY)$ is a \textbf{refinement} of $(\cX, \bX) \stackrel{F}{\leftarrow} (\cZ, \bZ) \stackrel{\Phi}{\to} (\cY, \bY)$, if there is an equivalence $H:(\cZ',\bZ') \to (\cZ, \bZ)$ such that $H\circ F \Rightarrow F'$ and $H \circ \Phi \Rightarrow \Phi'$. Equivalently, we have the following 2-commutative diagram of groupoids:
	$$\xymatrix{
		(\cX, \bX) & (\cZ,\bZ) \ar[l]_F \ar[r]^{\phi} \ar@{}[ld]^(.15){}="a"^(.5){}="b" \ar@{=>} "a";"b" \ar@{}[rd]^(.15){}="c"^(.5){}="d" \ar@{=>} "c";"d" & (\cY, \bY) \\
		&  (\cZ', \bZ'). \ar[u]_{H} \ar[lu]^{F'} \ar[ru]_{\Phi'} & 
	}$$ 
Two diagrams $(\cX, \bX) \stackrel{F}{\leftarrow} (\cZ, \bZ) \stackrel{\Phi}{\to} (\cY, \bY)$ and $(\cX, \bX) \stackrel{F'}{\leftarrow} (\cZ', \bZ') \stackrel{\Phi'}{\to} (\cY, \bY)$ are equivalent if they admit a common refinement. A \textbf{generalized map} from $(\cX, \bX)$ to $(\cY,\bY)$ is an equivalence class $[d]$ of a diagram $d$ from $(\cX,\bX)$ to $(\cY,\bY)$.
\end{definition}
Common refinement defines an equivalence relation by \cite[Lemma 10.4]{hofer2017polyfold}. Using generalized maps as the morphisms between ep-groupoids, we form a category $\mathscr{EP}(\bE^{-1})$ \cite[Theorem 10.3]{hofer2017polyfold}, where $\bE$ stands for the class of equivalences between ep-groupoids. We call isomorphisms in $\mathscr{EP}(\bE^{-1})$  generalized isomorphisms. The following important property holds for generalized isomorphisms.
\begin{theorem}[{\cite[Theorem 10.4]{hofer2017polyfold}}]\label{thm:iso}
	A generalized map $[a] = [(\cX, \bX) \stackrel{F}{\leftarrow} (\cZ, \bZ) \stackrel{G}{\to} (\cY, \bY)]$ is a generalized isomorphism iff $G$ is an equivalence. Moreover, the inverse is $[a]^{-1} = [(\cY, \bY) \stackrel{G}{\leftarrow} (\cZ, \bZ) \stackrel{F}{\to} (\cX, \bX)]$.
\end{theorem}
As a special case, a sc-smooth functor $\Phi:(\cX,\bX) \to (\cY,\bY)$ is an isomorphism in $\mathscr{EP}(\bE^{-1})$ iff $\Phi$ is an equivalence.

\subsubsection{Basics of polyfolds}
\begin{definition}[{\cite[Definition 16.1]{hofer2017polyfold}}]
	 A \textbf{polyfold structure} for a topological space $Z$ is a pair $((\cX, \bX),\alpha)$, where $(\cX, \bX)$ is an ep-groupoid and $\alpha$  is a homeomorphism from the orbits space $|\cX|$ to $Z$.
\end{definition}

\begin{definition}[{\cite[Definition 16.2,16.3]{hofer2017polyfold}}]\label{def:polyfold}
Two polyfold structures $((\cX,\bX),\alpha)$, $((\cY, \bY),\beta)$ are equivalent, if there exist a third ep-groupoid $(\cZ, \bZ)$ and equivalences $(\cX, \bX)\stackrel{F}{\leftarrow} (\cZ, \bZ )\stackrel{G}{\rightarrow}(\cY, \bY)$, such that $\alpha\circ |F|=\beta\circ |G|$. A \textbf{polyfold} $(Z,c)$ is a paracompact space $Z$ with an equivalence class of polyfold structures $c$. We will suppress $c$ when there is no confusion. A polyfold is \textbf{regular} resp. \textbf{tame} resp. \textbf{infinite dimensional} if it has a polyfold structure $((\cX,\bX), \alpha)$ such that $(\cX,\bX)$ is regular resp. tame resp. infinite dimensional \footnote{Hence all polyfold structures are regular resp. tame (\cite[Lemma 16.2]{hofer2017polyfold}) resp. infinite dimensional.}.
\end{definition}

The polyfold definition in \cite[Definition 16.3]{hofer2017polyfold} does not require paracompactness. Since partition of unit is very important in application, we only work with paracompact polyfold\footnote{The partition of unity on polyfold requires more than paracompactness, see \cite[Theoorem 7.4]{hofer2017polyfold}. If the M-polyfold structures on the object space is built on sc-Hilbert space, then we will have sc-smooth partition unity, see \cite[Theorem 5.12, Corollary 5.2, Theorem 7.4]{hofer2017polyfold}. Also note that the construction in Lemma \ref{lemma:gamma} is closed in the sc-Hilbert space category.}.
By \cite[Proposition 16.1]{hofer2017polyfold}, a topological space admitting a polyfold structure is locally metrizable and regular Hausdorff. Then being paracompact implies that the polyfold is also metrizable. 

\begin{remark}\label{rmk:levelshift}
	Given a polyfold $Z$ with a polyfold stricture $((\cX,\bX),\alpha)$, then  by \cite[Proposition 7.12]{hofer2017polyfold}, $Z^k:=\alpha(|\cX^k|)\subset Z$ is again a paracompact polyfold\footnote{$Z^k$ in not endowed with the subset topology, but the topology from $|\cX^k|$.} with polyfold structure $((\cX^k,\bX^k),\alpha)$ for any $k\in \N$.
\end{remark}

To get a \textbf{category of polyfolds $\bm{\mathscr{P}}$}, we need to define the morphisms between polyfolds. Given two polyfolds $(Z,c)$ and $(W,d)$, we can choose two polyfold structures $((\cX,\bX),\alpha)$ and $((\cY,\bY),\beta)$ respectively. Then we define a \textbf{morphism between those two polyfold structures} is a pair $(\mathfrak{f}, f)$, where $\mathfrak{f}$ is a generalized map from $(\cX, \bX)$ to $(\cY,\bY)$ and $f:Z\to W$, such that $ f \circ \alpha = \beta\circ|\mathfrak{f}|$. Suppose we have another pair of polyfold structures $((\cX',\bX'),\alpha')$ and $((\cY',\bY'),\beta')$ respectively and a morphism $(\mathfrak{f}', f')$ between them. Then $(\mathfrak{f}, f)$ is equivalent to $(\mathfrak{f}', f')$ if we have the following commutative diagram:
$$\xymatrix{
((\cX,\bX),\alpha) \ar[d]^{(\mathfrak{h}, \Id_Z)} \ar[r]^{(\mathfrak{f}, f)} & ((\cY,\bY),\beta)\ar[d]^{(\mathfrak{g}, \Id_W)} \\
((\cX',\bX'),\alpha')\ar[r]^{(\mathfrak{f}',f')} & ((\cY',\bY'),\beta'), 
}$$
where $\mathfrak{h}$ and $\mathfrak{g}$ are generalized isomorphisms and the commutativity is in the sense of $\mathfrak{g}\circ \mathfrak{f} = \mathfrak{f}'\circ \mathfrak{h}$ as generalized maps and $f = f'$. A morphism between two polyfolds is an equivalence class of morphisms between polyfold structures. 

\begin{definition}[{\cite[Definition 16.8]{hofer2017polyfold}}]
The category of polyfolds $\mathscr{P}$ is the category whose objects are polyfolds (Definition \ref{def:polyfold}) and morphisms are equivalence classes of morphisms between polyfold structures defined above. Let $Z,W$ be two polyfolds, we use 
\begin{equation}\label{eqn:polymor}
f:Z\stackrel{\mathscr{P}}{\to} W
\end{equation}
to emphasize that $f$ is morphism in the polyfold category $\mathscr{P}$. And we use $|f|:Z\to W$ to denote the underlying topological map.
\end{definition}

To simplify notations, \textbf{we will abbreviate $\bm{\alpha}$ and write $\bm{Z = |\cX|}$ from now on.} Then a morphism between polyfolds $Z = |\cX|$ and $W = |\cY|$ is represented an equivalence class of a generalized map $\mathfrak{f}:(\cX,\bX) \to (\cY,\bY)$.

\subsection{Group actions on polyfolds} Now we give a  definition of a $\sc^\infty$ group action on a polyfold that applies e.g. to Gromov-Witten polyfolds in the equivariant setting and Hamiltonian-Floer cohomology polyfolds in the autonomous setting. Let $G$ be a group, we use $BG$ to denote the following category:
$$\obj(BG) = \{\star\},\qquad \mor(BG) = G.$$
\begin{definition}\label{def:preaction}
	Let $Z$ be a polyfold and $G$ be a Lie group. A \textbf{$\bm{G}$-action} on $Z$ is a functor $\mathfrak{P}:BG \to \mathscr{P}$ and $\mathfrak{P}$ sends the unique object of $BG$ to $Z$.
\end{definition}	
Definition \ref{def:preaction} only describes the group property, i.e. for every $g,h\in G$, we have two polyfold morphisms $\mathfrak{P}(g),\mathfrak{P}(h):Z\to Z$ such that $\mathfrak{P}(g)\circ \mathfrak{P}(h) = \mathfrak{P}(g\circ h)$ as polyfold morphisms. To describe the sc-smoothness of an action, we need the following local condition.
\begin{definition}\label{def:action}
	Let $Z$ be a polyfold and $G$ be a Lie group. A \textbf{$\bm{\sc^\infty}$ $\bm{G}$-action} on $Z$ is a $G$-action $\mathfrak{P}$ with the following extra property: For all $g\in G$ and $p,q \in Z$ such that $|\mathfrak{P}(g)|(p) = q\in Z$, there exist a neighborhood $U$ of $g$, two equivalent polyfold structures $(\cX,\bX)$ and $(\cY, \bY)$, two points $x\in \cX, y\in\cY$ with $|x| = p$, $|y| = q$ and a local uniformizer $\cU\subset \cX$ around $x$.  So that the $U$-family of polyfold maps $U \to \mor_\mathscr{P}(|\cU|,Z)$ defined by sending $h$ to the restriction $\mathfrak{P}(h)|_{|\cU|}$ is represented by a sc-smooth functor
	$$\rho_{g,x}: U\times (\stab_x \ltimes \cU )\to (\cY,\bY),$$
	with $\rho_{g,x}(g,x)=y$.
\end{definition}
Hence the sc-smoothness of an action is characterized by the property that the functor $\mathfrak{P}$ locally (both in the group and in the polyfold)  can be expressed as sc-smooth functors. As an easy corollary of Definition \ref{def:action}, the underlying map $\rho:=|\mathfrak{P}|:G\times Z \to Z$ on the orbit space is a continuous action, where the continuity follows from the continuity of $\rho|_{U\times |\cU|} = |\rho_{g,x}|$.

\begin{remark}
Although the continuous action  $\rho:=|\mathfrak{P}|$ is determined by $\mathfrak{P}$, we will use $(\rho,\mathfrak{P})$ to denote a sc-smooth group action for easy reference.
\end{remark}

Being a sc-smooth group action imposes strong properties on morphisms $\mathfrak{P}(g)$. Indeed, the following remarks shows that the local representative $\rho_{g,x}$ in Definition \ref{def:action} defines a family of fully-faithful functors, which are also embedding of open sets on the object level.
\begin{remark}\label{rmk:emb}
Assume $G$ acts on $Z$ sc-smoothly by $(\rho,\mathfrak{P})$. Let $g \in G$ and $\rho_{g,x}:U\times (\stab_x\ltimes \cU) \to (\cY,\bY)$ be a local sc-smooth representation as in Definition \ref{def:action}. Since for each $h\in G$, the morphism $\mathfrak{P}(h)$ is an isomorphism of polyfold $Z$, i.e. $\mathfrak{P}(h)$ is represented by generalized isomorphisms. Then we have the local representative $\rho_{g,x}(h,\cdot)$ is fully faithful by Theorem \ref{thm:iso} for $h\in U$. By Proposition \ref{prop:uniform}, when $\cU$ is connected and regular, on the object level $\rho_{g,x}(h,\cdot):\cU \to \cY$ is a sc-diffeomorphism onto the open image set for every $h\in U$. 
\end{remark}

The following proposition provides a class of sc-smooth group actions.
\begin{proposition}\label{prop:action}
Let $\mathfrak{p}:G\times Z \stackrel{\mathscr{P}}{\to} Z$ be a polyfold morphism such that 
\begin{enumerate}
	\item $\mathfrak{p}(\Id, \cdot) = \Id_Z$ as polyfold morphism;
	\item $\mathfrak{p}(gh,\cdot) = \mathfrak{p}(g,\cdot)\circ \mathfrak{p}(h,\cdot)$ as polyfold morphism.
\end{enumerate} 
Then $\mathfrak{p}$ induces a $G$-action $(\rho, \mathfrak{P})$, where $\rho = |\mathfrak{p}|$ and $\mathfrak{P}(g) = \mathfrak{p}(g,\cdot)$. 
\end{proposition}
\begin{proof}
	It suffices to prove the existence of sc-smooth local representations of $\mathfrak{P}$. Note that $\mathfrak{p}$ can be represented the following digram $G\times (\cX,\bX) \stackrel{F}{\leftarrow} (\cW,\bW) \stackrel{\Phi}{\to}(\cY,\bY)$ with $F$ an equivalence. For every $(g,x)\in G\times \cX$, there are local uniformizers of $G\times (\cX,\bX)$ around $(g,x)$ in the form of $U\times (\stab_x\ltimes \cU)$, where $\cU$ is a local uniformizer for $(\cX,\bX)$ around $x$ and $U\subset G$ is a neighborhood of $g$. Since $F$ is an equivalence, for small enough $U$ and $\cU$, we can assume a local uniformizer of $\cW$ also in the form of $U\times (\stab_x\ltimes \cU)$. Then $\mathfrak{P}$ is locally represented by $\Phi|_{U\times (\stab_x\ltimes \cU)}$. 
\end{proof}

\begin{remark}
	It is likely that Proposition \ref{prop:action} is equivalent to Definition \ref{def:action}. In principle, Definition \ref{def:action} is easier to check. However, Proposition \ref{prop:action} is often the case in applications.
\end{remark}

\begin{example}
	Assume $(M,\omega)$ be closed symplectic manifold and let $J$ be a $\omega$-compatible almost structure. Suppose a compact group $G$ acts on $M$ preserving $J$, then $G$ acts on the Gromov-Witten polyfolds $Z_{A,g,m}$ in \cite[Theorem 1.7]{hofer2017applications}. On the topology side, the action of $g\in G$ is sending a stable curve $[(S,j,M,D,u)]$ \cite[Definition 1.4]{hofer2017applications} to $[(S,j,M,D,g\cdot u)]$. The sc-Fredholm section induced by the Cauchy-Riemann operator \cite[Theorem 1.11]{hofer2017applications} is $G$-equivariant, since $G$ perverse the almost complex structure $J$.  A detailed verification will be carried out in \cite{equi}.
\end{example}

Another major example of group action is the $S^1$-action on the Hamiltonian-Floer cohomology polyfolds when the Hamiltonian is autonomous, see Section \ref{s6}.


\subsection{Properties of group actions}\label{subsec:lift}
In this subsection, we study the properties of group actions on polyfolds. In the following,  we fix the following notations.
\begin{itemize}
	\item $Z$ is a polyfold with a sc-smooth $G$-action $(\rho,\mathfrak{P})$.
	\item Let $g\in G$ and $x_a\in \cX_a$, $\rho_{g,x_a}:U_a\times(\stab_{x_a}\ltimes \cU_a) \to (\cX_b,\bX_b)$ is a local representative in Definition \ref{def:action} for two equivalent polyfold structures $(\cX_a,\bX_a),(\cX_b,\bX_b)$ of $Z$, where $U_a\subset G$ is an open neighborhood of $g$ and $\cU_a\subset \cX_a$ is a local uniformizer around $x_a$. 
\end{itemize}
The local representative $\rho_{g,x_a}$ is not very useful from the perspective of the group properties of the action, since its domain and codomain are in different polyfold structures. In the following, we will show how to fix it and get a nice family of local representatives that have group properties, see Proposition \ref{prop:property}. We first set up more notations.
\begin{itemize}
	\item $(\cX,\bX)$ is a  \textbf{fixed} polyfold structure for $Z$ throughout this subsection.
	\item $x,y\in \cX$ are two objects such that $\rho(g,|x|) = |y|$.
\end{itemize}
Since $(\cX,\bX), (\cX_a,\bX_a),(\cX_b,\bX_b)$ are equivalent polyfold structures of $Z$, then we have the following sequence of equivalences of ep-groupoids, note that we suppress the morphism spaces since there are no ambiguities,
\begin{equation}\label{gammadia}
\xymatrix {& & \cX_a\ar[r]^{\rho_{g,x_a}(g,\cdot)}  & \cX_b  &\ar[l]_{F_b} \cW_b \ar[r]^{G_b} &  \cX,\\
\cX & \ar[l]_{F_a} \cW_a \ar[r]^{G_a} &  \cU_a \ar[u]_{\subset} & & &
}\end{equation}
where $F_a,G_a, F_b, G_b$ are all equivalences and $|G_b|\circ |F_b|^{-1}\circ |\rho_{g,x_a}(g,\cdot)|\circ |G_a| \circ |F_a|^{-1} = \rho(g, \cdot)$, i.e. over the orbit space $|\cU_a|\subset Z$ the composition of the sequence is the topological action $\rho(g,\cdot)$. Since $F_a,G_a, F_b, G_b$ are all equivalences, there exist $w_a\in \cW_a, w_b\in \cW_b, \phi, \psi\in \bX, \delta \in \bX_a, \eta\in \bX_b$, such that 
$$\begin{array}{rclrcl}
F_a(w_a)&=&\phi(x), & \delta(G_a(w_a)) &=& x_a, \\
\eta(\rho_{g,x_a}(g,x_a))&=&F_b(w_b),& G_b(w_b)&=&\psi(y).
\end{array}$$
That is 
\begin{equation}\xymatrix{
	\cX  & \ar[l]_{F_a}\cW_a\ar[r]^{G_a}&  \cX_a\ar[r]^{\rho_{g,x_a}(g,\cdot)} & \cX_b & \ar[l]_{F_b}\cW_b\ar[r]^{G_b} & \cX	\\	
	x\ar[d]_{\phi} & & x_a\ar[r] & \rho_{g,x_a}(g,x_a)\ar[d]_{\eta} & & y\ar[d]_{\psi}\\
	F_a(w_a) & \ar[l] w_a \ar[r] & G_a(w_a) \ar[u]^{\delta} & F_b(w_b) & \ar[l] w_b\ar[r] & G_b(w_b) & & \\
}  \nonumber \end{equation}
Here $\rho_{g,x_a}$ is only partially defined on $\cX_a$. The first row here just indicates where the elements are from and the directions of equivalences. We point out that arrows in the lower half of the diagram have two different meanings, the vertical arrows are morphisms in one ep-groupoid and the horizontal arrows are the maps on objects of equivalences. Let $F_{a,w_a}$ and $F_{b, w_b}$ denote the local sc-diffeomorphisms $F_a$ and $F_b$ near $w_a$ in $\cW_a$ and $w_b$ in $\cW_b$ respectively. With these data, we can define a map for $h$ close to $g$
\begin{equation}\label{eqn:gamma}
\Gamma(h,\cdot): z\mapsto L_\psi^{-1}\circ G_b\circ F_{b,w_b}^{-1}\circ L_\eta \circ \rho_{g,x_a}(h,\cdot)\circ L_\delta \circ G_a \circ F_{a,w_a}^{-1}\circ L_\phi(z)\end{equation}
where $z$ is close to $x$.

\begin{proposition}\label{prop:emb}
Following the notation above, there exists an open neighborhoods $V\subset U\subset G$ of $g$ and $\cV\subset  \cX$ of $x$, such that for every $h\in V$, $\Gamma(h,\cdot)$ in \eqref{eqn:gamma} is defined on $\cV$ and $\Gamma(h,\cdot)$ is sc-diffeomorphism from $\cV$ to an open set of $\cX$.
\end{proposition}
\begin{proof}
	Since  $L_{\psi}^{-1}, G_b, F^{-1}_{b,w_b},L_\eta, L_\delta, G_a, F^{-1}_{a,w_a}, L_{\phi}$ in \eqref{eqn:gamma} are all local sc-diffeomorphisms and $\rho_{g,x_a}(h,\cdot)$ is a sc-diffeomorphism  from a fixed open set onto the open image set for any $h$ close to $g$ by Remark \ref{rmk:emb}, then the claim holds.
\end{proof}

 Since on the orbit space, we have $|\Gamma(h,\cdot)| = \rho(h,\cdot)$. We call $\Gamma$ a \textbf{local lift of the action at $\bm{(g,x,y)}$}. Given $(g,x,y)\in G\times \cX\times \cX$, $\Gamma$ is not unique. It depends on choices of $w_a,w_b,\phi,\psi,\delta, \eta$ as well as the polyfold structures $\cW_a,\cW_b,\cX_a,\cX_b$ in the construction of \eqref{eqn:gamma}.
 
 \begin{definition}\label{def:germ}
 	Let $\cU,\cV$ be two M-polyfolds and $f:\cU \supset \cO \to \cV$  a sc-smooth map such that $f(p) = q$ for $p\in \cU,q\in \cV$. We use $[f]_p$ to denote the germ of $f$ at $p$. Then $[f]_p = [g]_p$ if there exists neighborhood $\cO'\subset \cU$ of $p$ such that $f|_{\cO'} = g|_{\cO'}$.
 \end{definition}
  
\begin{definition}\label{def:liftset}
	Let $Z$ be a polyfold with a sc-smooth $G$-action $(\rho, \mathfrak{P})$ and $(\cX,\bX)$ a polyfold structure of $Z=|\cX|$. Let $x,y\in \cX, g\in G$ such that $\rho(g,|x|) = |y|$.  We define the \textbf{local lifts set} at $(g,x,y)$ to be
	$$L_{\cX}(g,x,y) := \{[\Gamma]_{(g,x)}| \Gamma \text{ is a local lift of the action at } (g,x,y)\}.$$	
\end{definition}

Although there are infinitely many different choices involved in the construction of a local lift $\Gamma$, the following proposition asserts that the local lifts set $L_{\cX}(g,x,y)$ in Definition \ref{def:liftset} is a finite set for regular polyfolds.

\begin{proposition}\label{prop:lifting}
	 Let $Z$ be a regular polyfold with a sc-smooth $G$-action $(\rho, \mathfrak{P})$ and $(\cX,\bX)$ a polyfold structure of $Z=|\cX|$. Then for $x,y\in \cX$ and $g\in G$ with $\rho(g,|x|) = |y|$, the isotropy $\stab_x$ acts on $L_{\cX}(g,x,y)$ by pre-composing the local sc-diffeomorphism $L_\phi$ for $\phi \in \stab_x$, i.e. 
	 $$(\phi, [\Gamma]_{(g,x)}) \mapsto [\Gamma\circ L_\phi]_{(g,x)},$$
	 and this action is transitive. The isotropy $\stab_y$ also acts on $L_{\cX}(g,x,y)$ by post-composing $L_\phi$ for $\phi\in \stab_y$, i.e. 
	$$(\phi, [\Gamma]_{(g,x)}) \mapsto [L_\phi\circ \Gamma]_{(g,x)},$$
	 this action is also transitive. 
\end{proposition}	
Roughly speaking, this proposition is a consequence of the fact that all the functors in \eqref{eqn:gamma} are equivalences. We prove Proposition \ref{prop:lifting} in Appendix \ref{A1}. Regularity plays an important role in the proof of Proposition \ref{prop:lifting}. Indeed, the irregular polyfold in Example \ref{ex:irr-two} provides an example in which Proposition \ref{prop:lifting} does not hold, see Example \ref{ex:three}.

\begin{corollary}
	Under the same assumptions in Proposition \ref{prop:lifting}, we have $|L_{\cX}(g,x,y)|=|\stab^{\eff}_x|=|\stab^{\eff}_y|$, where $\stab^{\eff}_x,\stab^{\eff}_y$ are the effective parts defined in \eqref{eqn:eff}.
\end{corollary}
Another corollary of Proposition \ref{prop:lifting} in the following, asserts that the local lifts have the unique continuation property on  connected regular uniformizers.
\begin{corollary}\label{coro:unique}
	Under the same assumptions in Proposition \ref{prop:lifting}, let $\Gamma_1,\Gamma_2$ be two local lifts at $(g,x,y)$. Assume $\Gamma_1,\Gamma_2$ are both defined on $U\times\cU_x$, where $U\subset G$ is a connected open neighborhood of $g$ and $\cU_x$ is a connected regular uniformizer around $x$.If $[\Gamma_1]_{(g,x)} = [\Gamma_2]_{(g,x)}$, then $\Gamma_1|_{U\times \cU_x} = \Gamma_2|_{U\times \cU_x}$. 
\end{corollary}
\begin{proof}
	Let $S:= \{(g',x')\in U\times \cU_x|[\Gamma_1]_{(g',x')} = [\Gamma_2]_{(g',x')} \}\subset U\times \cU_x$, then $S$ is open by the definition of germ (Definition \ref{def:germ}). We claim $S$ is sequentially closed, hence closed. Assume otherwise that there exist $(g_i,x_i)\in S$ and $\lim_{i}(g_i,x_i) = (g_\infty, x_\infty) \in U\times \cU_x$ such that $(g_\infty,x_\infty)\notin S$. Since $\Gamma_1(g_i,x_i) = \Gamma_2(g_i,x_i)$ for $i<\infty$, we have $y_\infty:=\Gamma_1(g_\infty, x_\infty) = \Gamma_2(g_\infty, x_\infty)$. Therefore $[\Gamma_1]_{(g_\infty,x_\infty)}$ and $[\Gamma_2]_{(g_\infty,x_\infty)}$ are both in $L_{\cX}(g_\infty, x_\infty, y_\infty)$. By Proposition \ref{prop:lifting}, there exists $\psi \in \stab_{x_\infty}$ such that $[\Gamma_1]_{(g_\infty,x_\infty)} = [\Gamma_2\circ L_\psi]_{(g_\infty,x_\infty)}$, that is $\Gamma_1=\Gamma_2\circ L_\psi$ in a neighborhood of $(g_\infty, x_\infty)$. On the other hand, $\Gamma_1 = \Gamma_2$ in neighborhoods of $(g_i,x_i)$ by assumption. Since $\Gamma_1,\Gamma_2$ are sc-diffeomorphisms onto the image sets, we have $L_\psi$ fixes a neighborhood of $x_i$ for some $i\gg 0$. By Corollary \ref{coro:natural}, there exists $\phi \in \stab_x$ such that $L_\psi = L_\phi$. Since the $\cU_x$ is a regular uniformizer (Definition \ref{def:regular}), $L_\phi = \Id_{\cU_x}$. Hence $L_\psi = \Id$, then $[\Gamma_1]_{(g_\infty,x_\infty)} = [\Gamma_2\circ L_\psi]_{(g_\infty,x_\infty)} = [\Gamma_2]_{(g_\infty, x_\infty)}$, which contradicts the assumption. Then $S$ is an open and closed subset of the connected space $U\times \cU_x$. Since by assumption $(g,x) \in S$, we have $S = U\times \cU_x$ and the claim in the proposition follows. 
\end{proof}
By Corollary \ref{coro:unique}, passing to the germs does not loss information, i.e. the representative $\Gamma$ for a germ $[\Gamma]_{(g,x)}\in L_{\cX}(g,x,y)$ is unique if we specify a connected domain $U\times \cU_x$. 

The local lifts set $L_{\cX}(g,x,y)$ also have the following group properties, which is important for us to construct the quotient ep-groupoid.
 \begin{proposition}\label{prop:property}
 	Let $Z$ be a regular polyfold with a $G$-action $(\rho,\mathfrak{P})$ and $(\cX,\bX)$ a polyfold structure of $Z=|\cX|$. Then the local lifts set $L_{\cX}(g,x,y)$ has the following properties.
\begin{enumerate}
	\item  There is a well-defined multiplication $\circ: L_{\cX}(g,y,z)\times L_{\cX}(h,x,y)\to L_{\cX}(gh,x,z)$ with the property that if $[\Gamma_1]_{(g,y)}\in L_{\cX}(g,y,z)$ and $[\Gamma_2]_{(h,x)} \in L_{\cX}(h,x,y)$, then there is a local lift $\Gamma_{12}$ at $(gh,x,z)$ such that $[\Gamma_{12}]_{(gh,x)} = [\Gamma_1]_{(g,y)} \circ [\Gamma_2]_{(h,x)} $ and 
	\begin{equation}\label{comp} \Gamma_{12}(\epsilon gh,u)= \Gamma_1(\epsilon g, \Gamma_2(h,u))=\Gamma_1(g, \Gamma_2(g^{-1}\epsilon gh,u))\end{equation}
	for $\epsilon$ in a neighborhood of $\Id \in G$ and $u$ in a neighborhood of $x$. The sizes of the neighborhoods depend on the representative local lifts $\Gamma_1,\Gamma_2,\Gamma_{12}$. 
	\item There is a unique identity element $[\ID_x]_{(\Id,x)}\in L_{\cX}(\Id,x,x)$, such that the identity is both left and right identity in the multiplication structure. Any representative $\ID_x$ has the property that $\ID_x(\Id, u) = u$ for $u$ in a neighborhood of $x$.\footnote{By Proposition \ref{prop:lifting}, the identity element is characterized by this property.} As before, the neighborhood depends on the choice of the representative. 
	\item There is an inverse map $L_{\cX}(g,x,y)\to L_{\cX}(g^{-1},y,x)$ with respect to the multiplication and identity structures above.
	\item\label{property:5}
	For $x, y\in \cX,g\in G$ with $\rho(g,|x|) = |y|$, there exists an open neighborhood $V\times \cU\times \cO$ of $(g,x,y)$ in $G\times \cX\times \cX$ such that there exist a representative local lift $\Gamma_\alpha$ defined on $V\times \cU$ with image in $\cO$ for each element $\alpha \in L_{\cX}(g,x,y)$. Moreover, for any $(g',x')\in V\times \cU$ and $y'\in \cO$ such that $\rho(g',|x'|)=|y'|$, every element $\beta \in L_{\cX}(g',x',y')$ is represented by $[\Gamma_\alpha]_{(g',x')}$ by a unique element $\alpha \in L_{\cX}(g,x,y)$. 
\end{enumerate}
\end{proposition}
We prove Proposition \ref{prop:property} in Appendix \ref{A1}.

\subsection{Quotients of polyfolds}
In this subsection, we construct quotients of polyfolds. Given a sc-smooth $G$-action $(\rho,\mathfrak{P})$ on a polyfold $Z$ as in Definition \ref{def:action}, we first construct local uniformizers of the quotient ep-groupoids in Proposition \ref{prop:localmodel}. Then we assemble such local uniformizers in Theorem \ref{quopoly} to form  quotient ep-groupoids, which give polyfold structures on an open dense set of the topological quotient $Z^2/G$.

\begin{definition}\label{def:isotorpy}
	Let $Z$ be a  polyfold with a sc-smooth $G$-action $(\rho,\mathfrak{P})$ and $(\cX,\bX)$ a polyfold structure. For each $z\in Z$, the \textbf{isotropy} $G_z$ is defined to be the isotropy of $\rho$ at $z$, i.e. 
	$$G_z:=\{g\in G|\rho(g,z)=z\}.$$
	For each $x\in \cX$, the \textbf{isotropy} $G_x$ is defined to be $G_{|x|}$. An action has \textbf{finite isotropy} iff $G_z$ is a finite group for every $z \in Z$.
\end{definition}
Since the isotropy $G_z$ will contribute to the isotropy of the quotient ep-groupoid, see \eqref{eqn:exact}, to get a proper quotient groupoid we need to at least assume the group action has finite isotropy. 

\begin{definition}
	For $x_0\in \cX_\infty$, $[\ID_{x_0}]\in L_{\cX}(\Id, x_0,x_0)$ is the identity element as asserted in Proposition \ref{prop:property}. Then we can define the \textbf{infinitesimal directions} of the group action at $x_0$ to be
	$$\mathfrak{g}_{x_0}:=\rD(\ID_{x_0})_{(\Id,x_0)}(T_{\Id}G\times \{0\})\subset T_{x_0}\cX.$$
\end{definition}
Then we have analogues of Proposition \ref{prop:red} and Proposition \ref{prop:free} as follows.
\begin{proposition}\label{prop:redfree}
	If the action has finite isotropy, then $\mathfrak{g}_{x_0}\subset (T^R_{x_0}\cX)_\infty$ (Remark \ref{rmk:reducetangent}) and $\dim \mathfrak{g}_{x_0} = \dim G$.
\end{proposition}
\begin{proof}
	By Proposition \ref{prop:property}, we have $\ID_{x_0}(\epsilon, \ID_{x_0}(\delta, x))= \ID_{x_0}(\epsilon \delta, x)$ for $\epsilon,\delta$ close to $\Id$ and $x$ close to $x_0$. Then the proof of Proposition \ref{prop:free} applies here and $\dim \mathfrak{g}_{x_0} = \dim G$. Since $\ID_{x_0}(\epsilon, \cdot)$ is a local sc-diffeomorphism for $\epsilon$ close to $\Id$ by Proposition \ref{prop:emb}, the arguments in Proposition \ref{prop:red} prove that $\mathfrak{g}_{x_0}\subset (T^R_{x_0}\cX)_\infty$.
\end{proof}
In fact, the infinitesimal directions $\mathfrak{g}_{x_0}$ have a nice invariant property, which will be used in Remark \ref{rmk:quoquo}. 
\begin{proposition}\label{prop:gammachoice}
	Let $g_0 \in G_{x_0}\subset G$ be an element in the isotropy  and a local lift $\Gamma$ at $(g_0,x_0,x_0)$, then
	\begin{equation}\label{gammachoice}
	\rD \Gamma_{(g_0,x_0)}(\{0\} \times \mathfrak{g}_{x_0}) = \mathfrak{g}_{x_0}.
	\end{equation}
\end{proposition}
\begin{proof}
By \eqref{comp}, for $\epsilon \in G$ close to $\Id$, we have
$$\Gamma(g_0, \ID_{x_0}(g_0^{-1}\epsilon g_0, x_0)) = \Gamma(\epsilon g_0, x_0) = \ID_{x_0}(\epsilon, \Gamma(g_0,x_0)) = \ID_{x_0}(\epsilon, x_0).$$
Then \eqref{gammachoice} is proven by taking derivatives in $\epsilon$.
\end{proof}

To construct quotient polyfolds, we need to introduce the analogue of $G$-slices in the ep-groupoid setup.  Slices for ep-groupoids should be translation groupoids (Definition \ref{def:actgroupoid}), and will eventually be local uniformizers (Definition \ref{def:uni}) for quotient ep-groupoids. We first discuss the isotropy of the quotients. For every $x\in \cX$, we define $$\stab^Q_{x}:=\cup_{g\in G_{x}} L_{\cX}(g,x,x).$$ 
Then $\stab^Q_{x}$ is a group by Proposition \ref{prop:property}, and it is supposed to be the isotropy of the quotient ep-groupoid, hence the superscript ``$Q$" stands for quotient. By Proposition \ref{prop:lifting} and Proposition \ref{prop:property}, we have an exact sequence of groups:
\begin{equation}\label{eqn:exact}
1 \to \stab_{x}^{\eff} \to \stab^Q_{x} \to G_{x} \to 1,
\end{equation}
where the first map is the inclusion $\stab^{\eff}_{x} \simeq L_{\cX}(\Id,x,x) \hookrightarrow \stab^Q_{x}$ and the second map is the projection. As a consequence, $\stab^Q_{x}$ is a finite group with $|\stab^Q_{x}| = |\stab^{\eff}_{x}|\cdot|G_{x}|$. 
\begin{definition}\label{def:polyGslice}
		Let $Z$ be a regular tame polyfold with a sc-smooth $G$-action with finite isotropy (Definition \ref{def:isotorpy}) and $(\cX,\bX)$ a tame polyfold structure. For every $x_0\in \cX_{\infty}$, a \textbf{$\bm{G}$-slice for the polyfold structure $\bm{(\cX,\bX)}$} around $x_0$ is a tuple $(\cU,\tilde{\cU},V,f,\eta)$ such that the following holds.
		\begin{enumerate}
			\item\label{polyGc1} $\cU\subset \cX$ is an open neighborhood of $x_0$ and $V\subset G$ is an open neighborhood of $\Id$. 
			\item There is a representative $\ID_{x_0}$ of the identity of $L_{\cX}(\Id,x_0,x_0)$, such that  $\ID_{x_0}$ is defined on $V\times \cU$.
			\item $f:\cU\to V$ is a sc-smooth map and $\tilde{\cU}:=f^{-1}(\Id)$ is a slice (Definition \ref{def:slice}) of $\cX$ containing $x_0$.
			\item \label{slice4}  For $x \in \cU$, $g=f(x)$ is the unique element $g\in V$ such that $\ID_{x_0}(g,x) \in \tilde{\cU}$.
			\item\label{polyGc5} $\eta:\cU \to \tilde{\cU}$ defined by $x\mapsto \ID_{x_0}(f(x),x)$ is sc-smooth.
			\item\label{polyGc6} There exist representatives $\Gamma$ for each element in $\stab^Q_{x_0}$ such that $\Gamma(g_0,\cdot)$ is defined on $\cU$  and there is a sc-smooth action $\stab^Q_{x_0} \times \tilde{\cU} \to \tilde{\cU}$ defined by
			\begin{equation}\label{eqn:locaction}
			([\Gamma]_{(g_0,x_0)}, y)\mapsto \Gamma \odot y : = \eta\circ \Gamma(g_0, y).
			\end{equation}
			\item\label{polyGc7}  For every $y\in \tilde{\cU}$, the set $S_{y,\tilde{\cU}}:=\{(z,g,[\Gamma]_{(g,y)})|z\in \tilde{\cU}, g\in G, [\Gamma]_{(g,y)} \in L_{\cX}(g,y,z)\}$ has exactly $|\stab^Q_{x_0}|$ elements.
		\end{enumerate}  
\end{definition}

We now prove the existence of $G$-slices. Just like Lemma \ref{lemma:slice}, this will require a level shift in the polyfold as in Remark \ref{rmk:levelshift}. 
\begin{proposition}\label{prop:localmodel}
		Let $Z$ be a regular tame polyfold with a sc-smooth $G$-action with finite isotropy and $(\cX,\bX)$  a tame polyfold structure. For every $x_0\in \cX_{\infty}$, there exists a $G$-slice for the polyfold structure $(\cX^2,\bX^2)$ of shifted polyfold $Z^2$ around $x_0$.
\end{proposition}
\begin{proof}
	By Proposition \ref{prop:redfree}, we can apply Lemma \ref{lemma:gamma} to the identity representative $\ID_{x_0}$ and the trivial bundle $\Id:\cX\to \cX$. Therefore we have a tuple $(\cW,\tilde{\cW}, V, f,\eta)$ such that the following holds.
	\begin{enumerate}[(i)]
		\item\label{part:loc1} $\cW\subset \cX^2$ is an open neighborhood $x_0$ and $V\subset G$ is an open neighborhood $\Id$. 
		\item$\ID_{x_0}$ is defined on $V\times \cW$.
		\item $f:\cW \to V$ is sc-smooth and $\tilde{\cW}:=f^{-1}(\Id)$ is slice of $\cX$ containing $x_0$.
		\item\label{part:loc3}  For $x \in \cW$, $g=f(x)$ is the unique element $g\in V$ such that $\ID_{x_0}(g,x) \in \tilde{\cW}$.
		\item $\eta:\cW \to \tilde{\cW}$ defined by $x\mapsto \ID_{x_0}(f(x),x)$ is sc-smooth.
		\item\label{part:loc6}  We can fix representatives $\Gamma_\alpha$ for each element  $\alpha\in \stab^Q_{x_0}$, such that Property \eqref{property:5} of Proposition \ref{prop:property} holds for those chosen representatives $\Gamma_\alpha$ on $V\times \cW$.
	\end{enumerate}
	Moreover, for every open neighborhood  $\tilde{\cW}'\subset \tilde{\cW}$ of $x_0$, let $\cW'=\eta^{-1}(\tilde{\cW}')$. Then the tuple $(\cW', \tilde{\cW}', V, f|_{\cW'}, \eta|_{\cW'})$ also has property \eqref{part:loc1}-\eqref{part:loc6}. That is $(\cW', \tilde{\cW}', V, f|_{\cW'}, \eta|_{\cW'})$ satisfies condition \eqref{polyGc1}-\eqref{polyGc5} of Definition \ref{def:polyGslice}.
	
	Next we will construct an neighborhood $\tilde{\cV}\subset \tilde{\cW}$ of $x_0$ so that \eqref{eqn:locaction} defines a sc-smooth group action on $\tilde{\cV}$. First, we pick a small enough neighborhood $\tilde{\cW}'\subset \tilde{\cW}$ of $x_0$, such that \eqref{eqn:locaction} is well-defined for the chosen representative $\Gamma_\alpha$ for every $\alpha \in \stab^Q_{x_0}$, i.e. $\Gamma(g_0,\tilde{\cW}') \subset \cW$. Since 
	\begin{equation}\label{expand}
	\Gamma_\alpha\odot y = \eta\circ \Gamma_\alpha(g_0,y) = \ID_{x_0}(f\circ \Gamma_\alpha(g_0,y),\Gamma_\alpha(g_0,y)),
	\end{equation} 
	and the inverse to $\ID_{x_0}$ is itself,  we can require the neighborhood $\tilde{\cW}'$ is small enough such that we can invert \eqref{expand} to get
	\begin{equation}\label{repre}
	\Gamma_\alpha(g_0, y) =\ID_{x_0}( f\circ \Gamma_\alpha(g_0,y)^{-1}, \Gamma_\alpha \odot y), \quad \forall \alpha \in \stab^Q_{x_0}, y \in \tilde{\cW}'.
	\end{equation}
	Moreover, we can shrink $\tilde{\cW}'$ further such that the following equations are defined and hold for $\Gamma_\alpha,\Gamma_\beta$, where $\alpha \in L_{\cX}(g_\alpha, x_0, x_0), \beta \in L_{\cX}(g_\beta, x_0, x_0)$.
	\begin{eqnarray}
	(\Gamma_\alpha \circ \Gamma_\beta) \odot y & = &   \eta\circ \Gamma_\alpha\circ \Gamma_\beta(g_\alpha g_\beta,y) \nonumber \\
	&\stackrel{\eqref{comp}}{=} &\eta\circ \Gamma_\alpha(g_\alpha, \Gamma_\beta(g_\beta,y)) \nonumber \\
	&\stackrel{\eqref{repre}}{=} &\eta\circ\Gamma_\alpha(g_\alpha,  \ID_{x_0}( f\circ \Gamma_\beta(g_\beta,y)^{-1}, \Gamma_\beta\odot y)) \nonumber \\
	&\stackrel{\eqref{comp}}{=} &\eta\circ \Gamma_\alpha (g_\alpha f\circ \Gamma_\beta(g_\beta,y)^{-1}, \Gamma_\beta\odot y) \nonumber \\
	&\stackrel{\eqref{comp}}{=}&\eta\circ \ID_{x_0}(g_\alpha f\circ\Gamma_\beta(g_\beta,y)^{-1} g_\alpha^{-1}, \Gamma_\alpha(g_\alpha,\Gamma_\beta\odot y)) \nonumber\\
	& \stackrel{\eqref{part:loc3}}{=} &\eta\circ \Gamma_\alpha(g_\alpha, \Gamma_\beta\odot y) \nonumber\\
	&= & \Gamma_\alpha \odot(\Gamma_\beta\odot y) \label{eqn:assoc}
	\end{eqnarray}
    Since  we have
    \begin{equation}\label{eqn:id}
    \ID_{x_0}\odot y = \eta\circ \ID_{x_0}(\Id, y) = \eta(y) = y, \quad \forall y\in \tilde{\cW}', 
    \end{equation}
    \eqref{eqn:assoc} implies that $\Gamma_{\alpha^{-1}}\odot(\Gamma_\alpha \odot y) = y$ for every $y\in \tilde{\cW}'$. Hence $\Gamma_\alpha\odot $ is injective on $\tilde{\cW}'$. Pick an open neighborhood $\tilde{\cO} \subset \tilde{\cW}'$, such that $\Gamma\odot\tilde{\cO} \subset \tilde{\cW}'$ for all $\Gamma$. Then the injectivity of $\Gamma_{\alpha^{-1}}\odot$ on $\tilde{\cW}'$ and $\Gamma_{\alpha^{-1}}\odot(\Gamma_\alpha \odot y) = y$ on $\tilde{\cO}$ imply that $\Gamma_\alpha\odot \tilde{\cO} = (\Gamma_{\alpha^{-1}}\odot)|_{\tilde{\cW}'}^{-1}(\tilde{\cO})$. As a consequence, $\Gamma_\alpha\odot\tilde{\cO}$ is open in $\tilde{\cW}'$ and contains $x_0$ for every $\alpha \in \stab_{x_0}^Q$. We define 
    $$\tilde{\cV}: = \cap_{\alpha \in \stab^Q_{x_0}} \Gamma_\alpha \odot \tilde{\cO},$$
    then $\tilde{\cV}\subset \tilde{\cW}'$ is an open neighborhood of $x_0$.  
    Therefore \eqref{eqn:assoc} implies that \eqref{eqn:locaction} defines a sc-smooth group action of $\stab^Q_{x_0}$ on $\tilde{\cV}$.  Moreover, for every open neighborhood $\tilde{\cV}' \subset \tilde{\cV}$ of $x_0$, the $\stab^Q_{x_0}$ action can be restricted to $\tilde{\cV}'':=  \cap_{\alpha \in \stab^Q_{x_0}} \Gamma_\alpha \odot \tilde{\cV}' \subset \tilde{\cV}'$. So far, we prove that $(\cV'':=\eta^{-1}(\tilde{\cV}''), \tilde{\cV}'', V, f|_{\cV''}, \eta|_{\cV''})$ satisfies condition \eqref{polyGc1}-\eqref{polyGc6} of Definition \ref{def:polyGslice}.
    
    Finally, we verify condition \eqref{polyGc7} of Definition \ref{def:polyGslice}. We claim that we can find a neighborhood $\tilde{\cU}\subset \tilde{\cV}$ of $x_0$ invariant under the $\stab^Q_{x_0}$ action, such that $|S_{y,\tilde{\cU}}| = |\stab^Q_{x_0}|$ for every $y\in \tilde{\cU}$. First we pick a smaller neighborhood $\tilde{\cV}' \subset \tilde{\cV}$ such that for all $y \in \tilde{\cV}'$ and $\alpha \in L_{\cX}(g_\alpha,x_0,x_0)$, the following holds:
	\begin{equation}\label{orbit}
	\Gamma_\alpha \odot y =\eta\circ \Gamma_\alpha(g_\alpha,y)=\ID_{x_0}(f\circ \Gamma_\alpha(g_\alpha,y)),\Gamma_\alpha(g_\alpha,y)) \stackrel{\eqref{comp}}{=}\Gamma_\alpha(f\circ \Gamma_\alpha (g_\alpha,y)g_\alpha, y).\end{equation}
	For every open neighborhood $\tilde{\cU}' \subset \tilde{\cV}'$ of $x_0$, let $\tilde{\cU}:=  \cap_{\alpha \in \stab^Q_{x_0}} \Gamma_\alpha \odot \tilde{\cU}' \subset \tilde{\cU}'$. Then \eqref{orbit} implies that $$[\Gamma_\alpha]_{(f\circ \Gamma_\alpha (g_\alpha,y)g_\alpha,y)} \in L_{\cX}(f\circ \Gamma_\alpha (g_\alpha,y)g_\alpha, y,\Gamma_\alpha\odot y).$$ 
	Then $(\Gamma_\alpha\odot y, f\circ \Gamma_\alpha (g_\alpha,y)g_\alpha, [\Gamma_\alpha]_{(f\circ \Gamma_\alpha (g_\alpha,y)g_\alpha, y)})\in S_{y,\tilde{\cU}}$. By the uniqueness of the restriction in the property \eqref{property:5} of Proposition \ref{prop:property}, different $\alpha\in \stab^Q_{x_0}$ gives different elements in $S_{y,\tilde{\cU}}$, thus $|S_{y,\tilde{\cU}}|\ge |\stab^Q_{x_0}|$. 
	
	We claim we can find $\tilde{\cU}$ small enough, such that $|S_{y,\tilde{\cU}}| = |\stab^Q_{x_0}|$ for every $y\in \tilde{\cU}$. Assume otherwise, that is there exists a shrinking sequence of open neighborhoods $\tilde{\cU}_{(k)} \subset \tilde{\cV}'$ invariant under the $\stab^Q_{x_0}$ action, such that there exist $y_k \in \tilde{\cU}_{(k)}$ with $\lim_{k}y_k = x_0$ and for each $k$ there are at least $|\stab^Q_{x_0}|+1$ tuples $\{(z^i_k, g_k^i,[\Gamma^i_k]_{(g_k^i, y_k)})\}_{1\le i \le |\stab^Q_x|+1}$ with $\Gamma_k^i(g_k^i, y_k)=z^i_k \in \tilde{\cU}(k)$ and $\lim_{k}z^i_k= x_0$.  After passing to a subsequence, we can assume $\lim_k g^i_k\in G_{x_0}$ for all $i$. Therefore there is a set $I\subset\{1,\ldots, |\stab^Q_{x_0}|+1\}$ with at least $|\stab^Q_{x_0}|/|G_{x_0}|+1=|\stab^{\eff}_{x_0}|+1$ elements, such that there is an element $g_0\in G_{x_0}$ and for all $i\in I$, $\lim_k g^i_k=g_0\in G_x$. By property \eqref{part:loc6}, $\Gamma_k^i$ is a restriction of some $\Gamma_\alpha$ for $\alpha \in L_{\cX}(g_0,x_0,x_0)$. After passing to a subsequence, we can assume $\Gamma^i_k$ is the restriction of a fixed $\Gamma_{\alpha_i}$ for any fixed $i\in I$. Since $|L_{\cX}(g_0,x_0,x_0)|=|\stab^{\eff}_{x_0}|$, there is a subset $J\subset I$ containing at least two elements such that for $j\in J$, $\Gamma_k^j$ is the restriction of a common $\Gamma_\alpha$. Then we have $z^{j_1}_k=\Gamma_\alpha(g^{j_1}_k, y_k)$, $z^{j_2}_k=\Gamma_\alpha(g^{j_2}_k, y_k)$ for $j_1,j_2\in J$.  Therefore for $k\gg0$, we have
	$$z^{j_1}_k  = \Gamma_\alpha(g^{j_1}_k, y_k) =  \ID_{x_0}(g^{j_1}_k g_0^{-1}, \Gamma_\alpha(g_0,y_k)).$$
	Since $z^{j_1}_k \in \tilde{\cU}_{(k)}\subset \tilde{\cW}$, by property \eqref{part:loc3} we have $g^{j_1}_k g_0^{-1} = f\circ \Gamma_\alpha(g_0,y_k)$. Therefore $z^{j_1}_k = \Gamma_\alpha\odot y_k$. Similarly, we have $g^{j_2}_k g_0^{-1} = f\circ\Gamma_\alpha(g_0,y_k))$ and $z^{j_2}_k = \Gamma_\alpha\odot y_k$. This contradicts that $(z^{j_1}_k,g_k^{j_1},[\Gamma^{j_1}_k]_{(g_k^{j_1},y_k)})$ and $(z^{j_2}_k, g_k^{j_2},[\Gamma^{j_2}_k]_{(g_k^{j_2}, y_k)})$ are different. 
\end{proof}

We will see that $S_{y,\tilde{\cU}}$ in Property \eqref{polyGc7} of Definition \ref{def:polyGslice} is the orbit set in Definition \ref{def:orbitset} for quotient ep-groupoids. Therefore Property \eqref{polyGc7} of Definition \ref{def:polyGslice} will be used to prove the properness of the quotient ep-groupoids by Proposition \ref{prop:proper}. 

Now we have all the prerequisites for the proof of the following theorem. 
\begin{theorem}\label{quopoly}
	Let $Z$ be a regular tame polyfold and $(\cX,\bX)$ a polyfold structure. Let $(\rho,\mathfrak{P})$ is a sc-smooth $G$-action on $Z$ as in Definition \ref{def:action}, such that $G$ is a compact Lie group and the $G$-action has finite isotropy as in Definition \ref{def:isotorpy}. For every $x_0\in \cX_{\infty}$, let $(\cU_{x_0},\tilde{\cU}_{x_0}, V_{x_0},f_{x_0},\eta_{x_0})$ be a $G$-slice for $(\cX^2,\bX^2)$ around $x_0$. 	Let $\hZ:= \cup_{x_0\in \cX_{\infty}} \rho(G, |\cU_{x_0}|) \subset Z^2$. Then $\hZ$ is a $G$-invariant open set of $Z^2$ containing $Z_\infty$. We define sets
	$$\begin{array}{rcl}
	\cQ & := & \bigsqcup_{x_0\in \cX_{\infty}} \tilde{\cU}_{x_0}, \\
	\bQ & := & \{(x,y,g,[\Gamma]_{(g,x)})| x,y \in \cQ, g\in G, [\Gamma]_{(g,x)}\in L_{\cX}(g,x,y)\},
	\end{array}$$
	and maps   
	$$\begin{array}{rrclrcl}
	s_Q:&\bQ &\to & \cQ: & (x,y,g,[\Gamma]) &\mapsto & x,\\
	t_Q:&\bQ &\to & \cQ: & (x,y,g,[\Gamma]) &\mapsto& y,\\
	m_Q:&\bQ_s\times_t \bQ& \to & \bQ: &((y,z, g, [\Gamma]),(x,y,h,[\Pi])) &\mapsto& (x,z,gh,[\Gamma]\circ[\Pi]),\\
	u_Q:&\cQ &\to& \bQ: & x &\mapsto& (x,x,\Id, [\ID_x]),\\
	i_Q: &\bQ &\to & \bQ:, & (x,y,g,[\Gamma]) &\mapsto& (y,x,g^{-1},[\Gamma]^{-1}).\end{array}$$
	Then $\cQ,\bQ$ can be equipped with tame M-polyfold structures so that $(\cQ,\bQ)$ is a tame ep-groupoid. $(\cQ,\bQ)$ defines a polyfold structure on $\hZ/G$ such that the topological quotient map $\pi_G:\hZ \to \hZ/G$ is realized by a sc-smooth polyfold map $\mathfrak{q}:\hZ \to \hZ/G$.
\end{theorem}
\begin{proof}
	From the definition of $\cQ,\bQ$ and structure maps, $(\cQ,\bQ)$ is a groupoid. $\cQ$ is a tame M-polyfold since it is a disjoint union of tame M-polyfolds. The remaining part of the proof is divided into several steps. 
	 \begin{step}
	 	$\bQ$ has a tame M-polyfold structure, such $s_Q,t_Q$ are \'{e}tale and all the structure maps are sc-smooth. 
	 \end{step}
	 \begin{proof}	
	 	We first give $\bQ$ a topology. Let $\cU\subset \cQ,\cW\subset \cQ,V\subset G$ be open subsets and $\Gamma$ a local lift defined on $V\times \cU$. Then we define 
	 	\begin{equation}\label{eqn:basis}
	 	\bQ_{\cU,\cW,V,\Gamma}:= \{(x,y,g,[\Gamma]_{(g,x)})|x\in \cU,y\in\cW,g\in V, \Gamma(g,x) = y\}\subset \bQ.
	 	\end{equation}
	 	We claim all such $\bQ_{\cU,\cW,V,\Gamma}$ form a topological basis. For any $(x,y,g,[\Gamma]_{(g,x)}) \in \bQ_{\cU',\cW',V',\Gamma'}\cap \bQ_{\cU'',\cW'',V'',\Gamma''} $, we have $[\Gamma']_{(g,x)} = [\Gamma'']_{(g,x)} = [\Gamma]_{(g,x)}$. Hence we can find smaller neighborhood $\cU \subset \cU' \cap \cU'', \cW \subset \cW' \cap \cW''$ and $V \subset V'\cap V''$, such that $(x,y,g,[\Gamma]_{(g,x)}) \in \bQ_{\cU,\cW,V,\Gamma} \subset \bQ_{\cU',\cW',V',\Gamma'}\cap \bQ_{\cU'',\cW'',V'',\Gamma''}$. This proves the claim and we can use this topological basis to put a topology on $\bQ$.
	 
	    Next we equip $\bQ$ with a tame M-polyfold structure. Let $(x,y,g,[\Gamma]_{(g,x)})\in \bQ$. If $y\in \tilde{\cU}_w$ for $w\in \cX_\infty$, then we can find a neighborhood $\cV\subset \cQ$ of $x$ such that the following hold.
	    \begin{enumerate}
	    	 \item\label{c1} $\Gamma(g,z)$ is defined for $z\in \cV$ and $\Gamma(g,z) \in \cU_w$.
	    	 \item Let $g_z := f_w\circ \Gamma(g,z)g$, then $g_z$ is in a neighborhood $V'\subset G$ of $g$, such that $V'g^{-1}\subset V_w$.
	    	 \item\label{c3} For $g'\in V'$ and $z\in \cV$, $\Gamma(g', z) = \ID_{w}(g'g^{-1}, \Gamma(g,z))$.
	   \end{enumerate}
	    Since $\eta_w\circ \Gamma(g,z) = \ID_{w}(f_w\circ \Gamma(g,z), \Gamma(g,z)) = \ID_{w}(g_zg^{-1}, \Gamma(g,z)) \stackrel{\eqref{c3}}{=} \Gamma(g_z,z)$, we can define the map
	    \begin{equation} \label{source}
	    s_{Q}^{-1}:   z \mapsto (z, \eta_w\circ \Gamma(g,z), g_z ,[\Gamma]_{(g_z,z)})\in \bQ, \qquad \forall z\in \cV. \end{equation}
	    We claim $s_Q^{-1}$ is the local inverse to $s_Q$ near $(x,y,g,\Gamma)$ as the notation indicates. Since $s_Q \circ s_Q^{-1} = \Id_{\cV}$,  it is sufficient to prove $s_Q^{-1}$ is an open map. For every open subset $\cO \subset \cV$, we will show that $s_Q^{-1}(\cO) = \bQ_{\cO,\tilde{\cU}_w, V', \Gamma}$. By \eqref{source}, $s_Q^{-1}(\cO) \subset \bQ_{\cO,\tilde{\cU}_w, V', \Gamma}$.  Let $(x',y',g',[\Gamma]_{(g',x')}) \in \bQ_{\cO,\tilde{\cU_w},V',\Gamma}$. Since $\Gamma(g',x') = y'$, by \eqref{c3} $y'=\ID_w(g'g^{-1},\Gamma(g,x'))$. By property \eqref{slice4} of Definition \ref{def:polyGslice}, we have $g'g^{-1} = f_w\circ \Gamma(g,x')$ and $y'=\eta_w\circ \Gamma(g,x')$. That is $(x',y',g',[\Gamma]_{(g',x')}) \in s_Q^{-1}(\cO)$.  Therefore $s_Q^{-1}(\cO)$ is open, hence a local sc-diffeomorphism. Then $s_Q^{-1}$ gives $\bQ$ a tame M-polyfold structure locally. The transition maps are the identity map on $\cQ$, thus $\bQ$ has a tame M-polyfold structure and the source map $s_Q:\bQ\to \cQ$ is a local sc-diffeomorphism. The target map $t_Q:\bQ\to \cQ$ is  $\sc^\infty$, since the composition of target map $t_Q$ with \eqref{source} is 
	    $$z\to  \eta_w\circ \Gamma(g,z),$$
	    which is a local sc-diffeomorphism, as one can write down a local $\sc^\infty$ inverse just like \eqref{source}. Similarly the unit map $u_Q$, the inverse map $i_Q$ and the multiplication $m_Q$ are $\sc^\infty$. This proves the \'{e}tale property. 
	\end{proof}
	\begin{step}
		$\bQ$ is Hausdorff and paracompact and $(\bQ,\cQ)$ is an \'etale groupoid.
	\end{step}
	\begin{proof}
	    By our definition of the topology on $\bQ$, we have a continuous projection:
	    $$\Pi: \bQ \to \cQ \times \cQ \times G.$$
	    Because $\cQ \times \cQ \times G$ is Hausdorff, any two points in $\bQ$ with different projections are separated by open sets. If two points have same projection, i.e. $(x,y,g,[\Gamma_1]_{(g,x)})$ and $(x,y,g,[\Gamma_2]_{(g,x)})$ such that $[\Gamma_1]_{(g,x)}\ne [\Gamma_2]_{(g,x)}$, then by Proposition \ref{prop:lifting}, $[\Gamma_2]_{(g,x)}=[\Gamma_1\circ L_\phi]_{(g,x)}$ for some $\phi\in \stab_x$. Since $[\Gamma_1]_{(g,x)}\ne [\Gamma_2]_{(g,x)}$ and the ep-groupoid $(\cX,\bX)$ is regular, thus $L_\phi$ does not fix any open subset near $x$. Then for $z$ in a neighborhood of $x$ and $h$ in a neighborhood of $g$, we have $[\Gamma_1]_{(h,z)} = [\Gamma_2\circ L_\phi]_{(h,z)}\ne [\Gamma_2]_{(h,z)}$. Thus $(x,y,g,[\Gamma_1]_{(g,x)})$ and $(x,y,g,[\Gamma_2]_{(g,x)})$ can be separated by $\bQ_{\cU,\cW,V,\Gamma_1}$ and $\bQ_{\cU,\cW,V,\Gamma_2}$ for some open neighborhoods $\cU,\cW,V$ of $x,y,g$. 
	
	    We will use \cite[Proposition 2.17]{hofer2017polyfold} to prove that $\bQ$ is paracompact. We first prove that the map $\Pi: \bQ \to \cQ \times \cQ \times G$ is a closed map. Let $\bC$ be a closed subset of $\bQ$. Note that if we have $(x_k,y_k,g_k) \in \cQ \times \cQ \times G$ such that $\lim_k(x_k,y_k,g_k) = (x,y,g)$ and $(x_k,y_k,g_k,[\Gamma_k]_{(g_k,x_k)}) \in \bC$, then by Proposition \ref{prop:property}, there is a local lift $\Gamma$ with $[\Gamma]_{(g,x)} \in L_{\cX}(g,x,y)$ such that after passing to a subsequence $[\Gamma_{k}]_{(g_k,x_k)} = [\Gamma]_{(g_k,x_k)}$. This shows that $(x,y,g,[\Gamma]_{g_k,x_k})\in \bQ$ is a limit point of $\bC$, hence $(x,y,g,[\Gamma]_{g_k,x_k})\in \bC$ and $(x,y,g) \in \Pi(\bC)$. Therefore $\Pi$ is a closed map. Next we claim that $\bQ$ is a regular space. That is for $p = (x,y,g,[\Gamma]_{(g,x)}) \in \bQ$ and a closed set $\bC \subset \bQ$ not containing $p$, we can separate them by non-intersecting open sets. If $\Pi(p) \notin \Pi(\bC)$, since $\cQ \times \cQ \times G$ is regular, we can separate $p$ and $\bC$. If $\Pi(p) \in \Pi(\bC)$,  by Proposition \ref{prop:lifting} $\bC\cap \Pi^{-1}(\Pi(p))$ is a finite set. Since $\bQ$ is Hausdorff, we can find an open neighborhood $\bU\subset \bQ$ of $p$ and $\bV\subset \bQ$ of $\bC\cap \Pi^{-1}(\Pi(p))$ separating $p$ and $\bC\cap \Pi^{-1}(\Pi(p))$. Note that $p$ and $\bC \backslash \bV$ can be separated since $\Pi(p) \notin \Pi(\bC \backslash \bV)$.  Therefore $p$ and $\bC$ can be separated hence $\bQ$ is regular. 
	    
	    Next we apply \cite[Proposition 2.17]{hofer2017polyfold} as follows. Observe that for every $(x,y,g) \in \cQ \times \cQ \times G$, the preimage $\Pi^{-1}(x,y,g) \subset \bQ$ is a finite set. Then by the argument in Step 1,  we can find small neighborhood $\cU\times \cW \times V \subset \cQ \times \cQ \times G$ of $(x,y,g)$, such that $\Pi^{-1}(\cU\times \cW \times V)$ has finitely many components and over each component $s_Q$ is a sc-diffeomorphism onto its image. Therefore $\Pi^{-1}(\cU\times \cW \times V)$ is a paracompact space hence metrizable. Since $\cQ \times \cQ \times G$ is a metrizable space and is covered by such open sets $\cU \times \cW \times V$, by \cite[Lemma 2.7]{hofer2017polyfold}, we can find a locally finite refinement $\{C_i\}_{i\in I}$ consisting closed sets. Then $\{\Pi^{-1}(C_i)\}_{i\in I}$ is a locally finite covering of $\bQ$ by closed sets.  Since $\Pi^{-1}(C_i)$ is a closed subset of some metrizable space $\Pi^{-1}(\cU\times \cW \times V)$, $\Pi^{-1}(C_i)$ is also metrizable and hence paracompact. We have $\bQ$ is also paracompact by \cite[Proposition 2.17]{hofer2017polyfold}. Therefore $\bQ$ is a tame M-polyfold and $(\cQ,\bQ)$ is a tame \'{e}table groupoid.
	\end{proof}
	
	\begin{step}
	 $(\cQ,\bQ)$ is proper.
	\end{step}
	\begin{proof}
		By definition, $\stab^Q_x$ in \eqref{eqn:exact} is the isotropy group of $(\cQ,\bQ)$ for $x\in \cQ$. By \eqref{polyGc7} of Definition \ref{def:polyGslice}, the orbit set $S_{y,\tilde{\cU}_{x}}$ has $|\stab^Q_x|$ elements for a $G-$slice $\cU_x$ and $y\in \tilde{\cU_x}$. Then by Proposition \ref{prop:proper} and Proposition \ref{prop:helper}, it is sufficient to prove that for any $z\in \cQ$ there exists an arbitrary small open neighborhood $\cV \subset \cQ$ of $z$ with the property that $t_Q(s_Q^{-1}(\overline{\cV}))$ is a closed set of $\cQ$. For every open neighborhood $\cO \subset \tilde{\cU}_x$ of $z$, we can pick $\cV\subset \cO$, such that $\cV$ is contained in an open set $\cW \subset \cU_x \subset \cX^2$ which has the property that $t:s^{-1}(\overline{\cW })\to \cX^2$ is proper. Moreover, for small enough $\cV$, we can assume there is neighborhood $V\subset G$ of $\Id$, such that $\cW = \ID_x(V,\cV)$. We claim that \begin{equation} \label{eqn:w} \overline{\cW} = \ID_x(\overline{V},\overline{\cV}).\end{equation} This is because a sequence $x_i\in \cW$ converge $x_\infty$, iff $\eta_x(x_i) \to \eta_x(x_\infty)$ in $\overline{\cV}$ and $f_x(x_i) \to f_x(x_\infty)$, that is $x_\infty \in \ID_x(\overline{V},\overline{\cV})$. By \eqref{eqn:w}, $\overline{\cW} \subset \eta_x^{-1}(\overline{\cV})$. Since $\cQ$ is metrizable, to show $t_Q(s_Q^{-1}(\overline{\cV}))$ is a closed set, it is enough to show it is sequentially closed. We pick a sequence $y_k\in t_Q(s_Q^{-1}(\overline{\cV}))$ converging to $y_\infty\in \tilde{\cU}_y$, it suffices to show that $y_\infty \in t_Q(s_Q^{-1}(\overline{\cV}))$. Since $y_k \in t_Q(s_Q^{-1}(\overline{\cV}))$, there is $x_k\in \overline{\cV} \subset \tilde{\cU}_x$ and $g_k \in G$ such that $\rho(g_k, |x_k|)=|y_k|$. Since $G$ is compact, we can assume $\lim_k g_k\to g_0\in G$ by choosing a subsequence.  We choose any $\tilde{y}_\infty\in \cX^2$ such that $\rho(g_0^{-1}, |y_\infty|)=|\tilde{y}_\infty|$ and pick any $[\Gamma]_{(g_0^{-1},y_\infty)}\in L(y_\infty, \tilde{y}_\infty,g_0^{-1})$. Let $\tilde{y}_k:=\Gamma(g_0^{-1}, y_k)$, then $\lim_{k}\tilde{y}_k=\tilde{y}_\infty$ in $\cX^2$. We also define $\tilde{x}_k:=\ID_x(g_0^{-1}g_k, x_k)$. Then for $k$ big enough, $\tilde{x}_k\in \overline{\cW}$ by \eqref{eqn:w}. Note that in the orbits space $|\cX^2|$, we have
		$$\rho(g_0^{-1}g_k, |x_k|)=|\tilde{x}_k|, \quad \rho(g_0^{-1}, |y_k|)=|\tilde{y}_k|.$$
		Therefore $|\tilde{x}_k|=|\tilde{y}_k|$ in $|\cX^2|$. By the properness of $t:s^{-1}(\overline{\cW})\to\cX^2$, there exists $\tilde{x}_\infty\in \overline{\cW}$ and a morphism $\phi: \tilde{x}_\infty\to \tilde{y}_\infty$. Since $\eta_x(\tilde{x}_\infty)\in \overline{\cV}$, we have a morphism in $\textbf{Q}$ from $\eta_x(\tilde{x}_\infty)$ to $y_\infty$ defined by
		$$(\eta_x(\tilde{x}_\infty), y_\infty, g_0f_x(\tilde{x}_\infty),  [\Gamma^{-1}\circ L_\phi \circ \ID_x]_{(g_0f_x(\tilde{x}_\infty),\eta_x(\tilde{x}_\infty))}).$$
		Therefore $y_\infty$ is in $t_Q(s_Q^{-1}(\overline{\cV}))$. Thus we prove that $t_Q(s_Q^{-1}(\overline{\cV}))$ is a closed set of $\cQ$ and the properness follows. Hence $(\cQ,\bQ)$ is a tame ep-groupoid.
	\end{proof}
	\begin{step}
		$\hZ/G$ is a tame polyfold.
	\end{step}
	\begin{proof}
		 By the definition of $\hZ$, $\hZ\subset Z^2$ is a $G$-invariant open neighborhood $Z_\infty$. By Lemma \ref{lemma:quometric}, $\hZ/G$ is metrizable space. Let $\beta$ denote the composition of maps
	     $$\beta: \cQ \stackrel{i}{\to} \cX^2 \stackrel{\pi_{\cX}}{\to} |\cX^2| = Z^2 \stackrel{\pi_G}{\to} Z^2/G,$$
	     where $i:\cQ\to \cX^2$ is defined by inclusion of the slices $\tilde{\cU}_x \hookrightarrow \cX^2$. We claim $\beta$ induces a homeomorphism $|\beta|:|\cQ| \to \hZ/G$. First,  for $x,y\in \cQ$,  the existence of $(x,y,g,[\Gamma]_{(g,x)})\in \bQ$ is equivalent to that $|x|$ and $|y|$ are in the same $G$-orbit in $Z$.  As a consequence $|\beta|$ is well-defined and is a bijection. Since $\beta$ is continuous, $|\beta|$ is also continuous. Therefore to show that $|\beta|$ is homeomorphism, it suffices to show that $|\beta|$ is an open map. Let $U$ be an open set of $|\cQ|$.  Let $\pi_{\cQ}:=\cQ \to |\cQ|$, $\eta: \coprod_{x_0\in \cX_{\infty}} \cU_{x_0} \to \cQ$ defined by $\eta|_{\cU_{x_0}} = \eta_{x_0}$ and $\iota: \coprod_{x_0\in \cX_{\infty}} \cU_{x_0} \to \cX^2$ defined by the inclusions $\cU_{x_0} \hookrightarrow \cX^2$. Then 
	     $$|\beta|(U) = \pi_G\circ \pi_{\cX}\circ i (\pi_{\cQ}^{-1}(U))= \pi_G
	    (\rho(G,\pi_{\cX}\circ \iota(\eta^{-1}(\pi_{\cQ}^{-1}(U))))).$$
	     Since $\iota$ and $\pi_{\cX}$ are open maps by \cite[Proposition 7.6]{hofer2017polyfold}. Then $\pi_{\cX}\circ \iota(\eta^{-1}(\pi_{\cQ}^{-1}(U)))$ is open, therefore $\rho(G,\pi_{\cX}\circ    \iota(\eta^{-1}(\pi_{\cQ}^{-1}(U))))$ is a $G$-invariant open set of $Z^2$. Hence $|\beta|(U)$ is open. 
	\end{proof}
	\begin{step}
		There is a sc-smooth polyfold map $\mathfrak{q}$ realizing the topological quotient $\pi_G:\hZ \to \hZ/G$.
	\end{step}
	\begin{proof}
		 Let $(\hX, \widehat{\bX})$ be the full subcategory $(\cX^2, \bX^2)$ with orbit space $\hZ$. Then $(\hX,\widehat{\bX})$ is a polyfold structure on $\hZ$. But we will construct another equivalent polyfold structure on $\hZ$ to write down the quotient map. Let $z\in \hX$, then there is an open neighborhood $\cV_z\subset \hX$ with the following structures. 
		 \begin{enumerate}
		 	\item There exists a $G$-slice $\tilde{\cU}_{u_z}$ for $u_z\in \cX_\infty$ and $g_z\in G$, such that $v_z\in \tilde{\cU}_{u_z}$ and $\rho(g_z, |z|)=|v_z|$.
		 	\item There is $[\Gamma_z]_{(g_z,z)} \in L(g_z,z,v_z)$  and 
		 	$$\Gamma_z(f_{u_z}\circ\Gamma_z(g_z,x)g_z,x)=\ID_{u_z}(f_{u_z}\circ\Gamma_z(g_z,x),  \Gamma_z(g_z,x)) = \eta_{u_z}(\Gamma_z(g_z,x))\in \tilde{\cU}_{u_z}$$
		 	holds for $x\in \cV_z$.
		 \end{enumerate}
		 Then we define an ep-groupoid $(\cT,\bT)$ as follows:
	     $$ \cT :=\coprod_{z\in \hX}  \cV_z, \quad \bT:=\{ (x,\phi,y)| x\in \cV_z, y\in \cV_w, \phi\in \bX, \phi(x)=y \} $$
	     The structures maps are:
	     $$\begin{array}{rrcl}
	     s: & (x,\phi,y)&\mapsto& x\\
	     t: & (x,\phi,y)&\mapsto & y\\
	     m: & ((y,\psi,z), (x,\phi,y))& \mapsto & (x,\psi\circ\phi,z)\\
	     u: & (x,\phi,y)&\mapsto &(y,\phi^{-1},x)\\
	     i: &x&\mapsto &(x,\Id,x).
	     \end{array}$$
	     The natural functor $(\cT,\bT) \to  (\hX,\widehat{\bX})$ by sending  $x\in \cV_z$ to $x\in \hX$ and $(x,\phi,y)$ to $\phi$ is an equivalence. 
	     
	     We claim there is a functor $Q: (\cT,\bT) \to (\cQ,\bQ)$ realizing the topological quotient map $\pi_G:\hZ\to \hZ/G$ as follows. For $x\in \cV_z$, we define $\rho_x:=f_{u_z}\circ \Gamma_z(g_z,x)g_z$ and $[\Lambda_x] = [\Gamma_z]_{(\rho_x,x)}$. Let $(x,\phi,y)\in \bT$ be a morphism,  we define
	     \begin{equation}\label{quotient}Q: 
	     \begin{aligned}
	     x &  \mapsto \ID_{u_z}(\rho_x, x)\in \tilde{\cU}_{w};\\
	     (x,\phi,y) & \mapsto (Q(x), Q(y), \rho_y\rho_x^{-1},  [\Lambda_y] \circ [L_{\phi}]_x\circ [\Lambda_x]^{-1}).
	     \end{aligned}
	     \end{equation}
	     It is a functor by the definition of structure maps. By Chain rule, $Q:\cT\to \cQ$ is sc-smooth. The \'{e}tale property implies that $Q:\bT \to \bQ$ is sc-smooth.  The functor $Q$ descends to the quotient map $\pi_G:\hZ \to \hZ/G$ on the orbit space by definition. 
	 \end{proof}
 Thus finishes the proof of the theorem.
\end{proof}

The construction in Theorem \ref{thm:quopoly} yields a polyfold structure on the quotient, once we fix a polyfold structure $(\cX,\bX)$ for $Z$ and choose $G$-slices. The following remark and proposition explain that in what sense the polyfold structures on the quotient are unique. 
\begin{remark}\label{rmk:polyunique}
	If for every $x_0\in \cX_{\infty}$, we pick two different $G$-slices around $x_0$ to construct two quotient ep-groupoids $(\cQ_a,\bQ_a),(\cQ_b,\bQ_b)$ and give $\hZ_a/G, \hZ_b/G$ polyfold structures.  Then we can take union of two sets of slices to form an ep-groupoid $(\cQ_{ab},\bQ_{ab})$, which gives $(\hZ_a\cup \hZ_b)/G$ a polyfold structure. The natural inclusions $(\cQ_a,\bQ_a) \to (\cQ_{ab},\bQ_{ab})$ and $(\cQ_b,\bQ_b) \to (\cQ_{ab},\bQ_{ab})$ induce the open inclusions of polyfolds $\hZ_a/G\to(\hZ_a\cup \hZ_b)/G$ and $\hZ_b/G\to(\hZ_a\cup \hZ_b)/G$. As a consequence, $\hZ_a/G$ and $\hZ_b/G$  restrict to the same polyfold structure on $(\hZ_a\cap \hZ_b)/G$.
\end{remark}
\begin{proposition}\label{prop:polyunique}
		Let $Z$ be a regular tame polyfold and $(\cX_a,\bX_a),(\cX_b,\bX_b)$ two polyfold structures. Assume compact Lie group $G$ acts on $Z$ sc-smoothly by $(\rho,\mathfrak{P})$ and the action only has finite isotropy. Let $\hZ_a, \hZ_b\subset Z^2$ be two open $G$-invariant open sets containing $Z_\infty$ such that $\hZ_a/G, \hZ_b/G$ are the quotient polyfolds constructed from $(\cX_a,\bX_a)$ resp. $(\cX_b,\bX_b)$ using Theorem \ref{quopoly}. Then there exists a $G$-invariant open set $\hZ \subset \hZ_a\cap \hZ_b$ containing $Z_\infty$, such that $\hZ_a/G, \hZ_b/G$ restricted to $\hZ/G$ are the same polyfold. 
\end{proposition}
\begin{proof}
	Since $(\cX_a,\bX_a), (\cX_b,\bX_b)$ are equivalent polyfold structures, we have equivalences
	$(\cX_a,\bX_a) \stackrel{F}{\leftarrow} (\cW,\bW) \stackrel{G}{\to}(\cX_b,\bX_b)$. We apply Theorem \ref{quopoly} to $(\cW,\bW)$, i.e. we pick $G$-slices $(\cU_{x_0}, \tilde{\cU}_{x_0}, V_{x_0}, f_{x_0}, \eta_{x_0})$ for $(\cW^2,\bW^2)$ around $x_0 \in \cW_\infty$ to construct a quotient ep-groupoid $(\cQ_{\cW},\bQ_{\cW})$ and gives $\hZ_{\cW}/G$ a polyfold structure, where $\hZ_{\cW}:=\cup_{x_0\in \cW_\infty} \rho(G, |\cU_{x_0}|)$. Since $F,G$ are equivalences,  we can require that $\cU_{x_0}$ is small enough such that $F|_{\cU_{x_0}}, G|_{\cU_{x_0}}$ are sc-diffeomorphisms. Then $(F(\cU_{x_0}), F(\tilde{\cU}_{x_0}), V_{x_0}, \allowbreak F_*f_{x_0}, F_*\eta_{x_0})$ resp. $(G(\cU_{x_0}), G(\tilde{\cU}_{x_0}), V_{x_0}, G_*f_{x_0}, G_*\eta_{x_0})$ are $G$-slices for $(\cX^2_a,\bX^2_a)$ resp. $(\cX^2_b, \bX^2_b)$. By Remark \ref{rmk:polyunique},  polyfold structure constructed from $G$-slices $(F(\cU_{x_0}), F(\tilde{\cU}_{x_0}), V_{x_0}, F_*f_{x_0}, F_*\eta_{x_0})$ and $\hZ_a/G$ induce  equivalent polyfold structures on $(\hZ_{\cW}\cap \hZ_a)/G$. Similarly, polyfold structure constructed from $G$-slices $(G(\cU_{x_0}), G(\tilde{\cU}_{x_0}), V_{x_0}, G_*f_{x_0}, G_*\eta_{x_0})$ and $\hZ_b/G$ induce the equivalent polyfold structures on $(\hZ_{\cW}\cap \hZ_b)/G$.  Therefore $\hZ_a/G, \hZ_b/G$ restricted to $(\hZ_{\cW} \cap \hZ_a\cap \hZ_b)/G$ are the same polyfold.
\end{proof}

\begin{remark}\label{rmk:reduced}
	The challenge of constructing a quotient ep-groupoid is the construction of the morphism space. 
	\begin{enumerate}
		\item 	Our approach in Theorem \ref{quopoly} is geometric in the sense that two morphisms are different iff they are represented by geometrically different data $(x,y,g, \Gamma)$. The regular property of polyfold is to the control local lifts set $L_{\cX}(g,x,y)$. It is not clear to us how to construct a quotient ep-groupoid $(\cQ,\bQ)$ with isotropy $\stab^Q_x$ satisfying an exact sequence $1\to \stab_x \to \stab^Q_x\to G_x\to 1$.
		\item 	When $G$ is the trivial group, the construction in Theorem \ref{quopoly} gives an effective polyfold structure on $\hZ \subset Z^2$. 
		\item One can still apply the construction in Theorem \ref{quopoly} to the irregular polyfold in Example \ref{ex:irr-one}, where the group $G$ is the trivial group. Then the new morphism space contains two open half planes and one half-line with the attaching point is a double point, thus not Hausdorff. This illustrates the role of the first property of regularity\footnote{However, to get an effective ep-groupoid out of a general ep-groupoid, one only needs the first condition in the regular property (Definition \ref{def:regular}), see the proof of \cite[Proposition 7.9]{hofer2017polyfold}.}.
	\end{enumerate}
\end{remark}
The following example shows the importance of the second property of regularity (Definition \ref{def:regular}) in the proof of Theorem \ref{quopoly}.
\begin{example}\label{ex:three}
	One can still  apply the construction in Theorem \ref{quopoly} to Example \ref{ex:irr-two} with the trivial group action. Note that Proposition \ref{prop:lifting} does not hold for $\Z_2\ltimes \Ima \pi$ in Example \ref{ex:irr-two}, this is because the sc-diffeomorphism $\eta:\Ima \pi \to \Ima \pi$ in Example \ref{ex:irr-two} can be completed to an autoequivalence of $\Z_2\ltimes \Ima \pi$, but $\eta \ne L_\phi$ for any $\phi \in \stab_{(0,0)}=\Z_2$.  Note that $\eta\circ \eta = \Id$, then the isotropy group $ \stab^Q_{(0,0)}$ for the quotient construction is $\Z_2\times \Z_2$ generated by $\eta$ and $L_\phi$. Therefore \eqref{eqn:exact} does not holds anymore. 
	
	In this example, one can still carry out the remaining construction. If we use $\Ima \pi$ as the only slice for the construction, then the resulted morphism space $\bQ$ is two copies of $\Ima \pi$ with two extra points $p,q$ corresponding to $\eta$ and $\eta \circ L_{\phi}$, such that $p$ connects the negative side of one copy with the positive side of the other copy and $q$ connects the remaining two sides. As a result, $\bQ$ is no longer Hausdorff due to $p$ and $q$.
\end{example}

\begin{remark}\label{rmk:quoquo}
	If we need to take another quotient on top of a quotient, we need to check whether the quotient is regular. The situation depends on the external group action. For example, if we think of the translation groupoids in Example \ref{ex:irr-one} and Example \ref{ex:irr-two} as quotients of $\Z_2$ actions, then the quotient polyfolds are not regular. In light of Corollary \ref{cor:regular}, we give a sufficient condition for the quotient to be regular. By Proposition \ref{prop:gammachoice}, $\rD \Gamma(g_0,\cdot)_{x_0}$ preserves the infinitesimal direction $\mathfrak{g}_{x_0}$ for $g_0\in G_{x_0}$ and $\Gamma \in L_{\cX}(g_0,x_0,x_0)$. Then for every $x_0\in \cX_\infty$, we can define the map $\stab^Q_{x_0} \to \Hom(T_{x_0}\cX/\mathfrak{g}_{x_0}, T_{x_0}\cX/\mathfrak{g}_{x_0})$ by $\Gamma \in L_{\cX}(g_0,x_0,x_0) \mapsto \rD \Gamma(g_0,\cdot)_{x_0}$. If such map is injective for every $x_0\in \cX_\infty$, then the quotient polyfold is regular. 
\end{remark}

\subsection{Quotients of polyfold bundles and sections}
In this section, we prove Theorem \ref{thm:main} up to the claim orientation. The construction of quotient polyfold bundles is analogous to the construction of quotient polyfolds in Theorem \ref{quopoly}, once we set up the counterpart of Proposition \ref{prop:localmodel} in the bundle case.

\subsubsection{Regular strong polyfold bundles and sc-Fredholm sections}
We first recall the basics of polyfold bundles and sc-Fredholm sections that will be used here. Like the polyfold case, we review the bundles over ep-groupoids first.
\begin{definition}[{\cite[Definition 8.4]{hofer2017polyfold}} ]\label{def:strPbund} 
	A sc-smooth functor $P:(\cE,\bE)\to (\cX,\bX)$ between two ep-groupoids is a \textbf{tame strong ep-groupoid bundle}, if the following properties hold.
	\begin{enumerate}
		\item On the object space $P^0:\cE\to \cX$ is a tame strong M-polyfold bundle.
		\item The morphism space $\bE$ is the fiber product $\bX_s\times_{P^0}\cE$. The functor on the morphism space $P^1:\bE \to \bX$ is the projection to the $\bX$ component.
		\item The structure maps on $\bE$ are given by 
		$$ \begin{array}{rcl}s(g,e) &=&e;\\
		t(g,e)&=&\mu(g,e);\\
		m((g,e),(h,a))&=&(g\circ h, a);\\
		u(e)&=&(\Id,e);\\
		i(g,e)&=&(g^{-1}, \mu(g,e)).\end{array}$$
		\item\label{cond:4} The target map $\mu$ has following properties:
		\begin{itemize}
			\item $t_{\cX}\circ P^1=P^0\circ \mu$, i.e. $\mu$ is a strong bundle map $\bE \to \cE$ covering the target map $t_{\cX}:\bX \to \cX$;
			\item $\mu$ is surjective local sc-diffeomorphism;
			\item $\mu(\Id_x,e)=e$, for $e\in \cE_x$;
			\item $\mu(g\circ h, e)=\mu(g, \mu(h,e)).$
		\end{itemize}	
	\end{enumerate}
\end{definition}

Note that for $\phi \in \bX$, there exist open neighborhoods $\cU_{s_{\cX}(\phi)}, \cU_{t_{\cX}(\phi)} \subset \cX$ and $\bU_{\phi} \subset \bX$ of $s_{\cX}(\phi), t_{\cX}(\phi)$ and $\phi$, such that $t_{\cX} : \bU_\phi \to \cU_{t_{\cX}(\phi)}$ and $s_{\cX}:\bU_\phi \to \cU_{s_{\cX}(\phi)}$ are sc-diffeomorphisms. Then $t:(P^1)^{-1}(\bU_\phi) \to (P^0)^{-1}(\cU_{t_{\cX}(\phi)})$ and $s:(P^1)^{-1}(\bU_\phi) \to (P^0)^{-1}(\cU_{s_{\cX}(\phi)})$ are both strong bundle isomorphisms. In particular, we can define the following strong bundle isomorphism covering $L_\phi$ (Definition \ref{def:lphi}), 
$$R_\phi:= t\circ s^{-1}:(P^0)^{-1}(\cU_{s(\phi)}) \to (P^0)^{-1}(\cU_{t(\phi)}).$$
Then $R_\phi$ is a special case of local action $L_{(\phi,0)}$ on the ep-groupoid $(\cE,\bE)$ for $(\phi,0) \in \bE$. 

To introduce the definition of regular strong bundle, we first define a local uniformizer for a tame strong polyfold bundle.
\begin{proposition}\label{prop:uniformizer}
	Let $P:(\cE,\bE) \to (\cX,\bX)$ be a tame strong ep-groupoid bundle and $x\in \cX$. Assume $\psi_x:\stab_x\ltimes\cU \to (\cX,\bX)$ is a local uniformizer of $(\cX,\bX)$ around $x$. Then there exists a diagram of sc-smooth functors
	$$
	\xymatrix{
		\stab_x \ltimes (P^0)^{-1}(\cU) \ar[d]^P \ar[r]^{\qquad \Psi_x} &  (\cE,\bE) \ar[d]^P \\
		\stab_x \ltimes \cU \ar[r]^{\psi_x} & (\cX,\bX),
	}
	$$
	such that the following holds.
		\begin{enumerate}
			\item $\stab_x$ acts on $(P^0)^{-1}(\cU)$ by strong bundle isomorphisms.
			\item $\Psi_x$ is fully faithful and on the object level $\Psi_x^0 :(P^0)^{-1}(\cU)  \to \cE $ is the inclusion.
			\item On the orbit space $|\Psi_x|: (P^0)^{-1}(\cU) /\stab_x \to |(P^0)^{-1}(\cU)|\subset |\cE| $ is a homeomorphism. 
		\end{enumerate} 
\end{proposition}
\begin{proof}
 $\stab_x$ acts on $(P^0)^{-1}(\cU)$ by $\phi \in \stab_x \mapsto R_\phi$.  	We write $\psi_x = (\Id_{\cU}, \Sigma):\stab_x\ltimes \cU \to (\cX,\bX)$. Then we  can define $\Psi_x:\stab_x\ltimes(P^0)^{-1}(\cU) \to (\cE,\bE)$ by
	$$\Psi^0_x:(P^0)^{-1}(\cU) \hookrightarrow \cE, $$
	$$ \Psi^1_x: \stab_x\times (P^0)^{-1}(\cU)   \to \bE, \quad (\phi, e) \mapsto (\Sigma(\phi, P^0(e)), e)\in \bE = \bX_s\times_{P^0} \cE.$$
	Then $\Psi$ is fully faithful and on object level $|\Psi_x|: (P^0)^{-1}(\cU) /\stab_x \to |(P^0)^{-1}(\cU)|\subset |\cE| $ is a homeomorphism.
\end{proof}

From the proof above, we see that $(P^0)^{-1}(\cU)$ is a local uniformizer of $(\cE,\bE)$ around $(x,0)\in \cE$ and is also a strong bundle over $\cU$. We will call $\Psi_x:\stab_x\ltimes (P^0)^{-1}(\cU) \to (\cE,\bE)$ a \textbf{local uniformizer for tame strong ep-groupoid bundle $\bm{P}$ around $\bm{x}$}. When there is no ambiguity, we will call $(P^0)^{-1}(\cU)$ a local uniformizer of $P$ around $x$. Let $\diff^p_\sc(x)$ denotes the group of germs of bundle isomorphism fixing $x$. Then the action in Proposition \ref{prop:uniformizer} defines a homomorphism $\stab_x \to \diff^p_{\sc}(x)$. We define the \textbf{effective part} 
$$\stab^{p, \eff}_x:=\stab_x /\ker(\stab_x \to \diff^p_\sc(x)).$$
Note that $\stab^{p, \eff}_x$ might be different from $\stab^{\eff}_x$ defined in \eqref{eqn:eff} for the base.  

\begin{definition}
	Let $P:(\cE,\bE) \to (\cX,\bX)$ be a tame strong ep-groupoid bundle.  For $x\in \cX$, a local uniformizer  $\stab_x\ltimes P^{-1}(\cU)$ around $x$ is \textbf{regular} if the following two conditions are met:
	\begin{enumerate}
		\item if $R_\phi|_{P^{-1}(\cV)} = \Id_{P^{-1}(\cV)}$ for some open subset $\cV \subset \cU$, then $R_\phi = \Id$ on $P^{-1}(\cU)$;
		\item for every connected uniformizer $\cW \subset \cU$ of $\cX$ around  $x$, if $\Phi:P^{-1}(\cW) \to \stab_x$ is a a map such that $R_{\Phi}:P^{-1}(\cW) \to P^{-1}(\cW), w\mapsto R_{\Phi(w)}(w)$ is sc-smooth strong bundle isomorphism, then there exists $\phi\in \stab_x$, such that $R_{\Phi(w)}(w) = R_{\phi}(w)$ for all $w\in P^{-1}(\cW)$. 
	\end{enumerate}
\end{definition}

\begin{definition}\label{def:regbund}
	A strong ep-groupoid bundle $P:(\cE,\bE) \to (\cX,\bX)$ is \textbf{regular}, if for any point $x\in \cX$, there exists a regular local uniformizer of $P$ around $x$.
\end{definition}

Note that the a strong ep-groupoid bundle being regular is different from the underlying ep-groupoid being regular. It seems that being regular as a bundle is weaker than being regular as an ep-groupoid, since a regular bundle uniformizer around $x$ is a regular uniformizer around $(x,0)$. The regularity of a bundle may only imply the regularity around the zero section. 

The counterparts of Proposition \ref{prop:regular}, Proposition \ref{prop:hofer-regular} and Corollary \ref{cor:regular} hold for strong ep-groupoid bundles for the same reason.  That is we have the following.
\begin{proposition}
	Let $P:(\cE,\bE) \to (\cX,\bX)$ be a tame strong ep-groupoid bundle such that $\stab_x^{p,\eff} = \stab_x$ for all $x\in \cX$. Assume for every $x\in \cX$, there exists a local uniformizer $\cU$ around $x$, such that for any connected uniformizer $\cV\subset \cU$ around $x$, we have $P^{-1}(\cV)\backslash \cup_{\phi \ne \Id \in \stab_x} \Fix(\phi)$ is connected, where $\Fix(\phi)$ is the fixed set of $R_\phi$. Then $P$ is regular. 
\end{proposition}
\begin{proposition}
	If the strong ep-groupoid bundle $P:(\cE,\bE) \to (\cX,\bX)$ has the property that for any $x\in \cX_{\infty}$ and $\phi\in \stab_x$ if $\rD R_{\phi}:T_{(x,0)} \cE \to T_{(x,0)}\cE$ is the identity map, then $R_\phi = \Id$ locally. Then $P$ is regular.  
\end{proposition}
\begin{corollary}\label{coro:bundreg}
	If the linearized action $\stab_x \to \Hom(T_x\cX,T_x\cX)$ defined by $\phi \mapsto (\rD L_\phi)_x$ is injective for every point $x\in \cX_\infty$, then  $P:(\cE,\bE) \to (\cX,\bX)$ is a regular strong ep-groupoid bundle.
\end{corollary}
We point out that the regularity of a bundle is not related to the regularity of the base. For example we can think of the ep-groupoid $\Z_2\ltimes \cM$ in Example \ref{ex:irr-one} as an ep-groupoid bundle over the trivial action groupoid $\Z_2 \ltimes \R$. Then  $\Z_2\ltimes \cM$ is not regular bundle, but the base is regular. On the other hand,  $\Z_2\ltimes (\cM \times \R)$ is an ep-groupoid bundle over $\Z_2\ltimes \cM$, where $\Z_2$ acts on the $\R$ coordinate by multiplying $-1$.  Then $\Z_2\ltimes (\cM \times \R)$ is a regular bundle over the irregular ep-groupoid $\Z_2\ltimes \cM$. 

A strong bundle functor $F:(\cE,\bE) \to (\cF,\bF)$ between two strong ep-groupoid bundles $(\cE,\bE) \to (\cX,\bX)$ and $(\cF,\bF) \to (\cY,\bY)$ is a sc-smooth functor, such that $F^0:\cE \to \cF$ and $F^1:\bE\to \bF$ are both strong bundle maps. As a consequence, $F$ induces a sc-smooth functor $f:(\cX,\bX) \to (\cY,\bY)$ on the bases.

\begin{definition}[{\cite[Definition 10.10]{hofer2017polyfold}}]
	A \textbf{strong bundle equivalence}
	$$F:(\cE,\bE) \to (\cF,\bF)$$
	is a strong bundle functor, moreover the following two conditions hold.
	\begin{enumerate}
		\item $F[i]:(\cE[i],\bE[i]) \to (\cF[i],\bF[i])$ are equivalences between ep-groupoids for $i=0,1$, covering an equivalence $f:(\cX,\bX) \to (\cY,\bY)$.
		\item $F^0, F^1$ are local strong bundle isomorphisms. 
	\end{enumerate}
\end{definition}

Just like the ep-groupoid situation, one can define generalized strong bundle maps between strong ep-groupoid bundles \cite[\S 10.5]{hofer2017polyfold},  i.e. an equivalence of diagram of strong bundle functors $\cE \stackrel{F}{\leftarrow}\cH \stackrel{\Phi}{\to }\cF$, where $F$ is a strong bundle equivalence. Then one can define the category of strong ep-groupoid bundles $\mathscr{SEP}(\bF^{-1})$ \cite[Definition 10.4]{hofer2017polyfold}, where $\bF$ is the class of strong bundle equivalences. The objects of $\mathscr{SEP}(\bF^{-1})$ are strong ep-groupoid bundles and the morphisms of $\mathscr{SEP}(\bF^{-1})$ are generalized strong bundle maps. Sections of strong ep-groupoids can be pulled back through generalized strong bundle maps and can be pushed forward if the generalized bundle map is an isomorphism \cite[Theorem 10.9]{hofer2017polyfold}. 

\begin{definition}[{\cite[Definition 8.7]{hofer2017polyfold}}]
	A \textbf{sc-Fredholm section} of a strong bundle $P: (\cE,\bE) \to (\cX, \bX)$ is a sc-smooth functor $S: (\cX, \bX) \to (\cE, \bE)$, such that $P\circ S = \Id_{(\cX,\bX)}$ and $S^{0}: \cX \to \cE$ is a sc-Fredholm section of the strong M-polyfold bundle $P^{0}:\cE \to \cX$. The section is called proper if $|(S^{-1}(0)| \subset |\cX|$ is compact in the $|\cX_0|$ topology\footnote{As a consequence, $|S^{-1}(0)|$ is compact in all $|\cX_i|$ topology and $|\cX_\infty|$ topology, by the same proof of \cite[Theorem 5.3]{hofer2017polyfold}}.
\end{definition}

\begin{definition}[{\cite[Definition 16.16]{hofer2017polyfold}}]
	Let $p: W \to Z$ be a surjection between paracompact Hausdorff spaces, a \textbf{strong polyfold bundle structure} $((\cE,\bE) \stackrel{P}{\to} (\cX,\bX), \Gamma, \gamma)$  is the following structure:
	\begin{enumerate}
		\item $P:(\cE,\bE) \to (\cX,\bX)$ is a tame strong ep-groupoid bundle;
		\item $\Gamma: |\cE| \to W$ and $\gamma:|\cX| \to Z$ are homeomorphisms, such that $\pi\circ \Gamma  = \gamma \circ |P|$. 
	\end{enumerate}
\end{definition}
The equivalence between strong polyfold bundle structures are defined similarly as the equivalence between polyfold structures, see \cite[Definition 16.17]{hofer2017polyfold}. A strong polyfold bundle $p:W\to Z$ a surjection between paracompact Hausdorff spaces with an equivalence class of strong polyfold bundle structures, see \cite[Definition 16.18]{hofer2017polyfold}. Then one can define the category of strong polyfold bundles, see \cite[\S 16.3]{hofer2017polyfold} for more details. Let $\mathscr{SPB}$ denotes the category of strong polyfold bundles \cite[Definition 16.18, 16.22]{hofer2017polyfold}. The sections of strong polyfold bundle \cite[Definition 16.27]{hofer2017polyfold} are represented by sections of a strong polyfold bundle structure up to the  equivalences induced by the pullbacks through strong polyfold bundle structure equivalences. A section is sc-Fredholm \cite[Definition 16.40]{hofer2017polyfold} if one of the representative sections (hence all) is sc-Fredholm. By \cite[Proposition 10.7]{hofer2017polyfold}, a section of a strong polyfold bundle has a unique representative for a fixed strong polyfold bundle structure. To simplify notations, we will \textbf{ abbreviate a strong polyfold bundle structure $\bm{((\cE,\bE) \stackrel{P}{\to} (\cX,\bX), \Gamma, \gamma)}$ to $\bm{(\cE,\bE) \stackrel{P}{\to} (\cX,\bX)}$} in the remaining part of this paper, that is $Z = |\cX|$ and $W = |\cE|$ and a strong polyfold bundle map is represented by an equivalence class of generalized strong bundle maps. 

\subsubsection{Group action on strong polyfold bundles}
\begin{definition}\label{def:actbundle}
		A \textbf{$\bm{\sc^\infty}$ $\bm{G}$-action} on a strong polyfold bundle $p: W\to Z$ is a functor  $\mathfrak{P}: BG \to \mathscr{SPB}$ sending the unique object of $BG$ to $W$. Moreover,  for every $g\in G$ and $p,q \in Z$ such that $|\mathfrak{P}(g)|(p,0) = (q,0)$, there exist a neighborhood $U$ of $g$, two equivalent polyfold bundle structures $(\cE,\bE) \stackrel{P}{\to} (\cX,\bX), (\cF, \bF) \stackrel{Q}{\to} (\cY, \bY)$, two points $x\in \cX, y\in\cY$ with $|x| = p, |y|= q$ and a local uniformizer $\cU$ around $x$. So that the $U$-family of polyfold bundle maps $U \to \mor_\mathscr{SPB}(|P^{-1}(\cU)|,W)$ defined by $h \mapsto \mathfrak{P}(h)$ is represented by the following sc-smooth bundle functor
		$$\rho_{g,x}: U\times (\stab_x \ltimes P^{-1}(\cU) )\to (\cF,\bF),$$
		satisfying $\rho_{g,x}(g,(x,0))=(y,0)$.
		
		Assume $G$ acts on $p:W\to Z$ by $ \mathfrak{P}$. A section $s:Z\to W$ is \textbf{$\bm{G}$-equivariant} iff it is invariant under the pullback of strong bundle isomorphism $\mathfrak{P}(g)$ for all $g\in G$.
\end{definition}
Just like the polyfold case, the sc-smooth group action $\mathfrak{P}$ induces a continuous action $\rho:G\times W \to W$. We will include this induced data, hence we will refer to a group action by $(\rho,\mathfrak{P})$. It is clear that a group action $(\rho,\mathfrak{P})$ on $p:W\to Z$ induces a group action $(\rho_Z, \mathfrak{P}_Z)$ on the base $Z$ in the sense of Definition \ref{def:action}. 
\begin{remark}\label{rmk:embbundle}
	By Remark \ref{rmk:bundleemb}, for connected and regular uniformizer $\cU$, $\rho_{g,x}(h,\cdot)$ is a local strong bundle isomorphism from $P^{-1}(\cU)$ for every $h\in V$.
\end{remark}

We can define local lifts for an action on a strong polyfold bundle like before. Let $(\rho, \mathfrak{P})$ acts on $p: W \to Z$ and $(\cE,\bE) \stackrel{P}{\to} (\cX,\bX)$ a fixed strong bundle structure on $p$. Let $x,y\in \cX$ and assume $\rho_Z(g, |x|) = |y|$ in $Z=|\cX|$.  Then by assumption, the action can be locally represented by 
\begin{equation}\label{localrepbundle}
\rho_{g,x_a}: U_a\times (\stab_{x_a} \ltimes P_a^{-1}(\cU))\to (\cE_b,\bE_b),
\end{equation}
where  $(\cE_a,\bE_a) \stackrel{P_a}{\to} (\cX_a,\bX_a)$, $(\cE_b,\bE_b) \stackrel{P_b}{\to} (\cX_b,\bX_b)$ are two strong polyfold bundles and  $\stab_{x_a}\ltimes P_a^{-1}(\cU_a)$ is a local uniformizer of $P_a$ with $|x_a| = |x| \in Z$. Then we have the following sequence of equivalences of ep-groupoids bundles, note that we suppress the morphism spaces since there are no ambiguities,
\begin{equation}\label{gammadiabundle}
\xymatrix {& & \cE_a\ar[r]^{\rho_{g,x_a}(g,\cdot)}  & \cE_b  &\ar[l]_{F_b} \cH_b \ar[r]^{G_b} &  \cE\\
	\cE & \ar[l]_{F_a} \cH_a \ar[r]^{\hspace{-2em}G_a} &  P_a^{-1}(\cU_a) \ar[u]_{\subset} & & &
}\end{equation}
where $F_a,G_a, F_b, G_b$ are all strong bundle equivalences and $|G_b|\circ |F_b|^{-1}\circ |\rho_{g,x_a}(g,\cdot)|\circ |G_a| \circ |F_a|^{-1} = \rho(g, \cdot)$, i.e. over the orbit space $|P_a^{-1}(\cU)|$ the composition of the sequence covers the action $\rho(g,\cdot)$. Let $f_a,g_a,f_b,g_b$ be the induced equivalences on the bases by $F_a,G_a,F_b,G_b$.  There exist $w_a\in \cW_a, w_b\in \cW_b, \phi, \psi\in \bX, \delta \in \bX_a, \eta\in \bX_b$, such that $f_a(w_a)=\phi(x)$, $\delta(g_a(w_a)) = x_a$, $\eta(\rho_{Z,g,x_a}(g,x_a))=f_b(w_b)$ and $g_b(w_b)=\psi(y)$, i.e.
\begin{equation*}\xymatrix{
	\cE \ar[d]_P  & \ar[l]_{F_a}\cH_a \ar[d]_{Q_a} \ar[r]^{G_a}&  \cE_a \ar[d]_{P_a} \ar[r]^{\rho_{g,x_a}(g,\cdot)} & \cE_b \ar[d]^{P_b} & \ar[l]_{F_b}\cH_b\ar[d]^{Q_b} \ar[r]^{G_b} & \cE\ar[d]^P	\\
	\cX  & \ar[l]_{f_a}\cW_a\ar[r]^{g_a}&  \cX_a\ar[r]^{\rho_{Z,g,x_a}(g,\cdot)} & \cX_b & \ar[l]_{f_b}\cW_b\ar[r]^{g_b} & \cX	\\	
	x\ar[d]_{\phi} & & x_a\ar[r] & \rho_{Z,g,x_a}(g,x_a)\ar[d]_{\eta} & & y\ar[d]_{\psi}\\
	f_a(w_a) & \ar[l] w_a \ar[r] & g_a(w_a) \ar[u]^{\delta} & f_b(w_b) & \ar[l] w_b\ar[r] & g_b(w_b) & & } 
\end{equation*}
Here $\rho_{g,x_a}$ is only partially defined on $\cE_a$. Let $F_{a,w_a}$ and $F_{b, w_b}$ denote the local bundle diffeomorphisms $F_a$ and $F_b$ near $w_a$ in $\cW_a$ and $w_b$ in $\cW_b$ respectively. Then by Remark \ref{rmk:embbundle}, there exists an open neighborhood $\cV \subset \cX$ of $x$, such that for $h$ close to $g$ and $w\in P^{-1}(\cV)$ we have
\begin{equation}\label{eqn:gammabundle}
\Lambda(h,w):=R_\psi^{-1}\circ G_b\circ F_{b,w_b}^{-1}\circ R_\eta \circ \rho_{g,x_a}(h,\cdot)\circ R_\delta \circ G_a \circ F_{a,w_a}^{-1}\circ R_\phi(w)\end{equation}
is a local bundle isomorphism from $P^{-1}(\cV)$. We call $\Lambda$ a \textbf{local lift of the bundle action at $\bm{(g,x,y)}$}. Moreover, every local lift $\Lambda$ induces a local lift of the action on the base:
\begin{equation}\label{eqn:gammabase}
\Gamma(h,z):=L_\psi^{-1}\circ g_b\circ f_{b,w_b}^{-1}\circ L_\eta \circ \rho_{Z,g,x_a}(h,\cdot)\circ L_\delta \circ g_a \circ f_{a,w_a}^{-1}\circ L_\phi(z).\end{equation}
We will use $(\Lambda,\Gamma)$ to denote a local lift of the bundle actions, although $\Gamma$ is determined by $\Lambda$. $(\Lambda, \Gamma)$ is not unique, but the only ambiguity comes from isotropy. That is we have the following analogue of Proposition \ref{prop:lifting} with an identical proof.
\begin{definition}
	Let $P:\cE \to \cX$ and $P':\cE' \to \cX'$ be two tame strong M-polyfold bundle and $\Lambda:P^{-1}(\cV) \to \cE'$ a strong bundle map for open set $\cV \subset \cX$. For $x\in \cV$, we define $[\Lambda]_x$ to be the base germ at $x$, which is an equivalence class of strong bundle maps with the equivalence is defined by 
	$$[\Lambda_1]_x=[\Lambda_2]_x \text{ iff } \Lambda_1|_{P^{-1}(\cV')} = \Lambda_2|_{P^{-1}(\cV')} \text{ for some open neighborhod $\cV'\subset \cX$ of $x$}.$$
\end{definition}
\begin{definition}\label{def:bundleliftset}
	Assume $G$ acts on a strong polyfold bundle $p:W\to Z$ and $(\cE,\bE)\stackrel{P}{\to}(\cX,\bX)$ is a bundle structure, then \textbf{local lifts set} is defined as
	$$L^p_{\cE}(g,x,y)=\{([\Lambda]_{(g,x)}, [\Gamma]_{(g,x)})|\Lambda \text{ is a local lift of the action at }(g,x,y) \text{ covering } \Gamma. \}.$$
\end{definition} 
\begin{proposition}\label{prop:liftingbundle}
Let $p:W\to Z$ be a regular strong polyfold bundle with a $G$-action $(\rho,\mathfrak{P})$ and $(\cE,\bE)\stackrel{P}{\to}(\cX,\bX)$ a strong polyfold bundle structure. Then for $x,y\in \cX$ and $g\in G$ with $\rho_Z(g,|x|)=|y|$, the isotropy $\stab_x$ acts on $L^p_{\cE}(g,x,y)$ by pre-composing with the local bundle diffeomorphism $R_\phi$ for $\phi \in \stab_x$, and this action is transitive. $\stab_y$ also acts on $L^p_{\cE}(g,x,y)$ by post-composing $R_\phi$ for $\phi\in \stab_y$, this action is also transitive. 
\end{proposition} 
\begin{proof}
	The claim follows from a similar diagram chasing in the proof of Proposition \ref{prop:lifting} by replacing Lemma \ref{lemma:localnatural} with Lemma \ref{lemma:localnaturalbundle}. 
\end{proof}
We point out here that there exist $([\Lambda_1],[\Gamma_1])$, $([\Lambda_2],[\Gamma_2])$ such that $[\Lambda_1]\ne [\Lambda_2]$ but $[\Gamma_1] = [\Gamma_2]$. For example, let $p: \R \to \text{pt}$ be a trivial bundle over a point. $\Z_2$ acts on $p$ by multiplying $-1$. Then the translation groupoid $Z_2\ltimes \R$ is an ep-groupoid bundle over $\Z_2\ltimes \text{pt}$.  Then there are two local lifts $(\Lambda_1,\Gamma_1),(\Lambda_2,\Gamma_2)$ of the action of the trivial group with $[\Gamma_1]=[\Gamma_2]$. The observation here is that the effective isotropy $\stab^{p, \eff}_x$ of the bundle might be larger than the effective isotropy $\stab^{\eff}_{x}$ of the base.

The local lifts set $L^p_{\cE}(g,x,y)$ also has the group properties like Proposition \ref{prop:property}. 
\begin{proposition}\label{prop:propertybundle}
 Let $p:W\to Z$ be a regular strong polyfold bundle with a $G$ action $(\rho,\mathfrak{P})$ and $(\cE,\bE)\stackrel{P}{\to}(\cX,\bX)$ a strong polyfold bundle structure. Then for $x,y \in \cX$ and $g\in G$ with $\rho_Z(g,|x|)=|y|$, $L^p_{\cE}(g,x,y)$ has the following structures.
    \begin{enumerate} 
 		\item  There is a well-defined multiplication $\circ: L^p_{\cE}(g,y,z)\times L^p_{\cE}(h,x,y)\to L^p_{\cE}(gh,x,z)$. Moreover, if $[\Lambda_1]_{(g,y)}\in L^p_{\cE}(g,y,z)$ and $[\Lambda_2]_{(h,x)} \in L^p_{\cE}(h,x,y)$ with $\Lambda_{12}$ represent $[\Lambda_1]_{(g,y)}\circ [\Lambda_2]_{(h,x)}$ locally, then 
 		\begin{equation*}\Lambda_{12}(\epsilon gh,w)= \Lambda_1(\epsilon g, \Lambda_2(h,w))=\Lambda_1(g, \Lambda_2(g^{-1}\epsilon gh,w))\end{equation*}
 		for $\epsilon$ in a neighborhood of $\Id \in G$, and $P(w)$ in a neighborhood of $x$. Therefore the induced local lifts on the base also have the property
 		\begin{equation}\label{compbundle}\Gamma_{12}(\epsilon gh,u)= \Gamma_1(\epsilon g, \Gamma_2(h,u))=\Gamma_1(g, \Gamma_2(g^{-1}\epsilon gh,u)),\end{equation}
 		for $\epsilon$ in a neighborhood of $\Id \in G$ and $u$ in a neighborhood of $x$.
 		\item There is a unique identity element $[\ID^p_x]_{(\Id,x)}\in L^p_{\cE}(\Id,x,x)$, such that $\ID^p_x(\Id,w)=w$ on $P^{-1}(\cV)$ for some neighborhood $\cV\subset \cX$ of $x$. This element is both left and right identity in the multiplication structure.
 		\item There is an (both right and left) inverse map $L^p_{\cE}(g,x,y)\to L^p_{\cE}(g^{-1},y,x)$ with respect to the multiplication and identity structures above.
 		\item There exists an open neighborhood $V\times \cU\times \cO$ of $(g,x,y)$ in $G\times \cX\times \cX$ such that there exist a representative local lift $\Lambda_\alpha$ defined on $V\times P^{-1}(\cU)$ with image in $P^{-1}(\cO)$, for each element $\alpha \in L^p_{\cE}(g,x,y)$. Moreover, for any $(g',x')\in V\times \cU$ and $y'\in \cO$ such that $\rho_Z(g',|x'|)=|y'|$, every element in $L^p_{\cE}(g',x',y')$ is represented by $[\Gamma_\alpha]_{(g',x')}$ by a unique element $\alpha \in L^p_{\cE}(g,x,y)$.
 	\end{enumerate}
 \end{proposition}
\begin{proof}
  The claim follows from an analogues proof of Proposition \ref{prop:property} by replacing Proposition \ref{prop:lifting} with Proposition \ref{prop:liftingbundle}. 
\end{proof} 

Note that for every $x\in \cX$, $\stab_x^{BQ}:=\cup_{g\in G_{x}} L^p_{\cE}(g,x,x)$ forms a group by Proposition \ref{prop:propertybundle}. It is clear that we have an exact sequence of groups:
\begin{equation}\label{eqn:exactbundle}
1 \to \stab_{x}^{p, \eff} \to \stab^{BQ}_{x} \to G_{x} \to 1,
\end{equation}
where the first map is the inclusion $\stab^{p, \eff}_{x} \simeq L^p_{\cE}(\Id,x,x) \hookrightarrow \stab^{BQ}_{x}$ and the second map is the projection. In particular, $\stab^{BQ}_{x}$ is a finite group with $|\stab^{BQ}_{x}| = |\stab^{p,\eff}_{x}|\cdot|G_{x}|$. For $x_0\in \cX_\infty$, let $(\ID^p_{x_0},\ID_{x_0})$ be the identity element in $L^p_{\cE}(\Id,x_0,x_0)$. Then we define the \textbf{infinitesimal directions} to be
\begin{equation}
\mathfrak{g}_{x_0}=\rD(\ID_x)_{(\Id,x_0)}(T_{g}G\times \{0\})
\end{equation}
By the same argument in Proposition \ref{prop:redfree}, we have $\mathfrak{g}_{x_0} \subset (T^R_{x_0}\cX)_\infty$ and $\dim \mathfrak{g}_{x_0}  = \dim G$. As a consequence, we have the existence of $G$-slice defined below, which is a analogue of Proposition \ref{prop:localmodel}.
\begin{definition}\label{def:bundleGslice}
	Let $p: W\to Z$ be a regular strong polyfold bundle with a sc-smooth $G$-action with finite isotropy and $P:(\cE,\bE) \to (\cX, \bX)$ a polyfold bundle structure. A \textbf{$\bm{G}$-slice for $\bm{P}$ around $\bm{x_0 \in \cX_{\infty}}$} is a tuple $(\cU, \tilde{\cU},V,f,\eta,N)$ such that the following holds.
	\begin{enumerate}
			\item $\cU\subset \cX$ is an open neighborhood $x_0$ and $V\subset G$ is an open neighborhood $\Id$. 
			\item $\ID^p_{x_0}$ is defined on $V\times P^{-1}(\cU)$.
			\item \label{bund:unique} $f:\cU \to V$ is sc-smooth and $\tilde{\cU}:=f^{-1}(0)$ is slice  of $\cX$ containing $x_0$ and $P^{-1}(\tilde{\cU})$ is a bundle slice.
			\item  For $x \in \cU$, $g=f(x)$ is the unique element $g\in V$ such that $\ID_{x_0}(g,x) \in \tilde{\cU}$.
			\item \label{bund:return}$\eta:\cU \to \tilde{\cU}$ defined by $x\mapsto \ID_{x_0}(f(x),x)$ is sc-smooth and $N:P^{-1}(\cU) \to P^{-1}(\tilde{\cU})$ defined by $v\mapsto \ID^p_{x_0}(f(P(v)),v)$ is a sc-smooth strong bundle map.
			\item  There exist representatives $\Lambda$ for each element in $\stab^{BQ}_{x_0}$ such that $\Lambda(g_0,\cdot)$ defined on $P^{-1}(\cU)$ and there is sc-smooth action of $\stab^{BQ}_{x_0}$ acts on $P^{-1}(\tilde{\cU})$ defined by
			\begin{equation}\label{eqn:bundlocaction}
			([\Lambda]_{(g_0,x_0)},v) = \Lambda\odot v : = N(\Lambda(g_0, v)), \quad \forall \Lambda \in L^p_{\cE}(g_0,x_0,x_0), v \in P^{-1}(\tilde{\cU}).
			\end{equation}
			\item  For every $v\in P^{-1}(\tilde{\cU})$, the set $$S_{v,P^{-1}(\tilde{\cU})}:=\{(u,g,[\Lambda]_{(g,P(v))})| u\in P^{-1}(\tilde{\cU}), g\in G, [\Lambda]_{(g,P(v))} \in L^p_{\cE}(g,P(v),P(u)), \Lambda(g,v)=u\}$$ has exactly $|\stab^{BQ}_{x_0}|$ elements. As a consequence, for every $y\in \tilde{\cU}$, the set 
			$$S_{y,\tilde{\cU}}:=\{(z,g, [\Lambda]_{(g,y)})| z\in \tilde{\cU}, g\in G, [\Lambda]_{(g,y)} \in L^p_{\cE}(g,y,z)\}$$ also has exactly $|\stab^{BQ}_{x_0}|$ elements.
	\end{enumerate}
\end{definition}

The following proposition follows from the same proof of Proposition \ref{prop:localmodel} and the assertion on good slices follows from Proposition \ref{prop:goodslice}.
\begin{proposition}\label{prop:localmodelbundle}
	Let $p: W\to Z$ be a regular strong polyfold bundle with a sc-smooth $G$-action such the induced action on $Z$ has finite isotropy. Let $P:(\cE,\bE) \to (\cX, \bX)$ be a polyfold bundle structure. Then for every $x_0\in \cX_{\infty}$,  there exists a $G$-slice $(\cU, \tilde{\cU},V,f,\eta,N)$ for $P:\cE^2\to \cX^2$ around $x_0$. 
	
	If $\cX$ is infinite dimensional and there exists a sc-Fredholm section $s:\cX \to \cE$, then $\tilde{\cU}$ can be chosen to be good with respect to $s$ in the sense of Definition \ref{def:good}.
\end{proposition} 

\subsubsection{Quotient of strong polyfold bundles and sections}
 \begin{theorem}\label{thm:quopoly}
	Let $p:W\to Z$ be a regular tame strong polyfold bundle such that $Z$ is infinite dimensional. A compact Lie group $G$ acts on $p$ sc-smoothly. If the $G$-action on $Z$ only has finite isotropy, then there is a $G$-invariant open dense set $\hZ\subset Z^2$ containing $Z_\infty$ such that $p^{-1}(\hZ)/G \to \hZ/G$ can be equipped with a strong tame polyfold bundle structure. Moreover, the topological quotient map $\pi_G:p^{-1}(\hZ)\to p^{-1}(\hZ)/G$ is realized by sc-smooth strong polyfold bundle map $\mathfrak{q}:p^{-1}(\hZ)\to p^{-1}(\hZ)/G$.
	
	If $s:Z\to W$ is a $G$-equivariant proper sc-Fredholm section. Then $s$ induces a proper sc-Fredholm section $\os:\hZ^1/G\to p^{-1}(\hZ^1)/G$ by $\mathfrak{q}^*\os=s|_{\hZ^1}$.
\end{theorem}
\begin{proof}
	The proof is analogous to the proof of Theorem \ref{quopoly} by replacing Proposition \ref{prop:localmodel} with Proposition \ref{prop:localmodelbundle}.  Let $P:(\cE,\bE) \to (\cX,\bX)$ be strong polyfold bundle structure of $p:W\to Z$. If we pick a $G$-slice $(\cU_x,  \tilde{\cU}_x, V_x, f_x,\eta_x, N_x)$ for $P:\cE^2\to \cX^2$ around every $x\in \cX_{\infty}$, then we can construct two ep-groupoids by 
	$$\mathcal{BQ} := \coprod_{x\in \cX_\infty} P^{-1}(\tilde{\cU}_x), \quad \cQ := \coprod_{x\in \cX_\infty} \tilde{\cU}_x,$$
	and
	$$\bm{BQ} := \{(u,v,g,([\Lambda]_{(g,P(u))},[\Gamma]_{(g,P(u))}))|(u,v,g) \in \mathcal{EQ}\times \mathcal{EQ}\times G, (\Lambda,\Gamma) \in L^p_{\cE}(g,P(u),P(v)), \Lambda(g,u) = v\},$$
	$$\bQ := \{(x,y,g,([\Lambda]_{(g,x)},[\Gamma]_{(g,x)}))|(x,y,g) \in \cQ\times \cQ\times G, (\Lambda,\Gamma) \in L^p_{\cE}(g,x,y)\}.$$
	The structure maps are defined similarly as in Theorem \ref{quopoly}. By the same argument of Theorem \ref{quopoly}, $(\mathcal{BQ},\bm{BQ})$ and $(\cQ,\bQ)$ are ep-groupoids. The obvious projection $P_Q:(\mathcal{BQ},\bm{BQ})\to (\cQ,\bQ)$ defines a strong ep-groupoid bundle and gives $p^{-1}(\hZ)/G \to \hZ/G$ a strong polyfold bundle structure, where $\hZ = \cup_{x \in \cX_{\infty}}\rho_Z(G,|\cU_x|) \subset Z^2$.
	
	Let $S:(\cX,\bX) \to (\cE,\bE)$ be the representative of the section $s$, such representative is unique by \cite[Proposition 10.7]{hofer2017polyfold}. Let $([\Lambda]_{(g,x)},[\Gamma]_{(g,x)}) \in L^p_{\cE}(g,x,y)$ and assume $\Lambda$ is defined on $V\times P^{-1}(\cU)$ for neighborhoods $V\subset G$ of $g$ and $\cU \subset \cX$ of $x$. Since $s$ is $G$-equivariant and by the uniqueness of representative $S$ on $\cX$ (also on $\cU$), we have:
	\begin{equation}\label{eqn:equi}
	S^0\circ \Gamma(h,z) = \Lambda(h,S^0(z)), \quad \forall (h,z)\in V\times \cU.
	\end{equation}
	Let $S^0_Q: \cQ \to \mathcal{BQ}$ denote the restriction $S|_{\tilde{\cU}_x}:\tilde{\cU}_x\to P^{-1}(\tilde{\cU}_x)$. We define $S^{1}_Q:\bQ \to \bm{BQ}$ to be 
	\begin{equation}\label{eqn:morsec}
	(x,y,g,([\Lambda]_{(g,x)},[\Gamma]_{(g,x)})) \mapsto (S_Q^{0}(x), S_Q^{0}(y), g, ([\Lambda]_{(g,x)},[\Gamma]_{(g,x)})).
	\end{equation}
	\eqref{eqn:morsec} is well-defined by \eqref{eqn:equi}. We claim $S_Q:=(S^{0}_Q,S^{1}_Q)$ is a functor from $(\cQ,\bQ)$ to $(\mathcal{BQ},\bm{BQ})$. It is clear that $s_{\mathcal{BQ}}\circ S^{1}_Q = S^{0}_Q\circ s_{\cQ}$ and $t_{\mathcal{BQ}}\circ S^{1}_Q = S^{0}_Q\circ t_{\cQ}$, where $s_{\mathcal{BQ}},s_{\cQ}$ are the source maps for $(\mathcal{BQ},\bm{BQ})$ and $(\cQ,\bQ)$, $t_{\mathcal{BQ}},t_{\cQ}$ are the target maps for $(\mathcal{BQ},\bm{BQ})$ and $(\cQ,\bQ)$. It remains to prove that the compatibility with composition. Let $m_{\mathcal{BQ}}$ and $m_\cQ$ denote the composition in 
	$(\mathcal{BQ},\bm{BQ})$ and $(\cQ,\bQ)$. Then we have
	\begin{eqnarray*}
		& & S^{1}_Q(m_Q(y,z,g, ([\Lambda_1],[\Gamma_1])), (x,y,h, ([\Lambda_2],[\Gamma_2]))) \\
		& = & S^{1}_Q((x,z,gh,([\Lambda_1]\circ [\Lambda_2],[\Gamma_1]\circ [\Gamma_2]))) \\
		& = & (S^{0}_Q(x), S^{0}_Q(z), gh, ([\Lambda_1]\circ [\Lambda_2],[\Gamma_1]\circ [\Gamma_2])) \\
		& =  & m_{\mathcal{BQ}}((S^{0}_Q(y),S^{0}_Q(z),g, ([\Lambda_1],[\Gamma_1])), (S^{0}_Q(x),S^{0}_Q(y),h, ([\Lambda_2],[\Gamma_2]))) \\
		& = &   m_{\mathcal{BQ}}( S^{1}_Q(y,z,g, ([\Lambda_1],[\Gamma_1])), S^{1}_Q(x,y,h, ([\Lambda_2],[\Gamma_2])).
	\end{eqnarray*}
	This finishes the proof of the claim. By Proposition \ref{prop:localmodelbundle}, for every $x\in \cX_\infty$, there exists $G$-slice $\tilde{\cU}'_x$ such that $S^{0}_Q|_{(\tilde{\cU}'_x)^1}$ has a Fredholm chart around $x$ by Lemma \ref{lemma:fred}. Although $\tilde{\cU}_x$ may be different from $\tilde{\cU}'_x$, $S^{0}_Q|_{\tilde{\cU}_x}$ is equivalent to  $S^{0}_Q|_{\tilde{\cU}'_x}$ by a strong bundle isomorphism like the proof of Theorem \ref{thm:free}. Hence $S^{0}_Q|_{\tilde{\cU}^1_x}$  has a Fredholm chart around $x$. By Proposition \ref{prop:reg}, $S_Q$ has the regularizing property.  Therefore $S_Q:\cQ^1 \to \mathcal{BQ}^1$ is a sc-Fredholm functor, hence the induced section $\os:\hZ^1/G \to \pi^{-1}(\hZ^1)/G$ on the quotient polyfold is sc-Fredholm. $\os^{-1}(0)$ is compact because $\os^{-1}(0) = s^{-1}(0)/G$ and both $s^{-1}(0)$ and $G$ are compact.
\end{proof}
\begin{remark}
	By the same argument in Proposition \ref{prop:polyunique}, if one uses two sets of different choices of $G$-slice from two different polyfold bundle structures to construct two quotient polyfold bundle $p^{-1}(\hZ_a)/G$ and $p^{-1}(\hZ_b)/G$, then there exists a $G$-invariant open set $\hZ\subset \hZ_a\cap \hZ_b$ containing $Z_\infty$ such that the restrictions of $p^{-1}(\hZ_a)/G,p^{-1}(\hZ_b)/G$ give equivalent polyfold structures on $p^{-1}(\hZ)/G$.
\end{remark}

\begin{remark}
Theorem \ref{thm:quopoly} provides a quotient of the base polyfold $Z$ even if $Z$ is not regular. Assume the base is  also regular, the constructions in Theorem \ref{thm:quopoly} and Theorem \ref{quopoly} may yield different quotients. The reason is that $\stab^{p,\eff}_x$ and $\stab^{\eff}_x$ may be different. However, when the base is effective, Theorem \ref{thm:quopoly} and Theorem \ref{quopoly} yield the same quotient on the base.
\end{remark}

\section{Orientation}\label{s4}
Orientations of sc-Fredholm sections were discussed in \cite[Chapter 6, Section 9.3 and 18.5]{hofer2017polyfold}. We first review briefly the basic concepts and properties of the orientation theory that will be used in this paper. 

\subsection{Orientations of sc-Fredholm sections}
In the case of M-polyfold, let $s:\cX\to \cY$ be a sc-Fredholm section of a tame strong M-polyfold bundle $p:\cY\to \cX$. Then there exists a continuous $\Z_2$-bundle $\mathscr{O}_s$ over $\cX_\infty$ called the \textbf{orientation bundle}, which is essentially the $\Z_2$ reduction of the ``determinant bundle" of the linearization $\rD s$, see \cite[Chapter 6]{hofer2017polyfold} for details. One of the basic property of $\mathscr{O}_s$ is the following.

\begin{proposition}\cite[Chapter 6, Theoerem 12.11]{hofer2017polyfold}\label{prop:oriiso}
	Let $p:\cY \to \cX,p':\cY' \to \cX'$ be two tame strong M-polyfold bundles and $s:\cX \to \cY, s':\cX'\to \cY'$ two sc-Fredholm sections. Assume there is a strong bundle isomorphism $\Phi:\cY\to \cY'$ such that $\Phi_*s = s'$, then there is an induced continuous $\Z_2$-bundle isomorphism $\Phi_*:\mathscr{O}_s \to \mathscr{O}_{s'}$ satisfying the obvious functorial properties.
\end{proposition}

In the case of ep-groupoids, $(\cX_\infty,\bX_\infty)$ is a groupoid such that the source and target maps are local homeomorphisms. The proper property in Definition \ref{groupoid} holds by Proposition \ref{prop:proper}. We will call such groupoid a \textbf{topological ep-groupoid}. We can define bundles over topological ep-groupoid, in particular line bundles or $\Z_2$-bundles, in a similar way to Definition \ref{def:strPbund}.

\begin{proposition}\label{prop:oribundle}
	Let $S:(\cX,\bX) \to (\cE,\bE)$ be a sc-Fredholm functor of a tame strong ep-groupoid bundle $P:(\cE,\bE) \to (\cX,\bX)$, then $\mathscr{O}_S := (\mathscr{O}_{S^0}, \mathscr{O}_{S^1})$ defines a topological $\Z_2$-bundle over $(\cX_\infty,\bX_\infty)$. 
\end{proposition}
\begin{proof}
	Let $\phi \in \bX_\infty$, then the source map $s$ and target map $t$ of $(\cE,\bE)$ induce isomorphisms $s_*:(\mathscr{O}_{S^1})_\phi \to (\mathscr{O}_{S^0})_{s_{\cX}(\phi)}$ and $t_*:(\mathscr{O}_{S^1})_\phi \to (\mathscr{O}_{S^0})_{t_{\cX}(\phi)}$ by Proposition \ref{prop:oriiso}, where $s_{\cX},t_{\cX}$ are source and target maps of $(\cX,\bX)$. We define $\phi_*:=t_*\circ s^{-1}_*: (\mathscr{O}_{S^0})_{s_{\cX}(\phi)}\to  (\mathscr{O}_{S^0})_{t_{\cX}(\phi)}$. The functorial property of Proposition \ref{prop:oriiso} implies that $\phi_*\circ \psi_* = (\phi\circ \psi)_*$. Then we have $\mathscr{O}_{S^1} \simeq \mathscr{O}_{S^0} \prescript{}{p}{\times}_{s_{\cX}}\bX_\infty$ defined by $(\mathscr{O}_{S^1})_\phi \ni o \mapsto (s_*o, \phi)$. Then $\mu:\mathscr{O}_{S^0} \prescript{}{p}{\times}_{s_{\cX}} \bX_\infty \to  \mathscr{O}_{S^0}, (o,\phi) \mapsto \phi_*o$ defines a $\Z_2$-bundle over $(\cX_\infty,\bX_\infty)$ following Definition \ref{def:strPbund} by direct check. 
\end{proof}

\begin{definition}[{\cite[Definition 12.12]{hofer2017polyfold}}]
	An orientation of $S:(\cX,\bX) \to (\cE,\bE)$ is a continuous section of $\mathscr{O}_S$.
\end{definition}

\begin{proposition}[{\cite[Proposition 12.4]{hofer2017polyfold}}]\label{prop:o}
	A sc-Fredholm section $S:(\cX,\bX) \to (\cE,\bE)$ is orientable iff the following two conditions holds.
	\begin{enumerate}
		\item $|\mathscr{O}_S| \to |\cX_\infty|$ is a continuous $\Z_2$-bundle.
		\item $|\mathscr{O}_S| \to |\cX_\infty|$ admits a continuous section.
	\end{enumerate}
\end{proposition}

Finally in the polyfold case, by \cite[Theorem 12.12]{hofer2017polyfold}, generalized strong bundle isomorphisms $\mathfrak{f}$ induce fiber preserving homeomorphisms of the orbit spaces of orientation bundles. Then an orientation for a polyfold sc-Fredholm section of a \textit{regular} strong polyfold bundle is an orientation on one strong polyfold bundle structure. It is well-defined by the bundle counterpart of Theorem \ref{thm:unique}.  When orientation exists, the transverse solution set is oriented by \cite[Theorem 6.3,Theorem 15.6]{hofer2017polyfold}.

\begin{definition}\label{def:admitori}
	We say a sc-Fredholm $s:Z\to W$ of a regular strong polyfold bundle $p:W\to Z$ \textbf{admits orientation bundle} if there exists representative (hence for any representative) $S:(\cX,\bX) \to (\cE,\bE)$ such that $|\mathscr{O}_S| \to |\cX_\infty|$ is a continuous $\Z_2$ bundle. As a consequence, there is a continuous $\Z_2$-bundle $\mathscr{O}_s := |\mathscr{O}_S| \to Z_\infty$.
\end{definition}
\begin{remark}\label{rmk:admitori}
It is clear that a sc-Fredholm section admits orientation bundle iff there is a representative $S:(\cX,\bX) \to (\cE,\bE)$ such that for all $\phi \in \stab_{x}$ with $x\in \cX_\infty$, $\phi_*:(\mathscr{O}_{S^0})_x \to  (\mathscr{O}_{S^0})_x$ defined in the proof of Proposition \ref{prop:oribundle} is identity. 
\end{remark}

\subsection{Quotients of orientations}
The main results of this subsection are that group action of strong bundle induce group action on the orientation bundle, and the quotient of that group action can be identified the orientation bundle of the quotient, when the latter exists. 

Let $p:W\to Z$ be a regular strong polyfold bundle with a sc-smooth $G$-action $(\rho,\mathfrak{P})$ (Definition \ref{def:actbundle}). Assume $s:Z\to W$ is a $G$-equivariant sc-Fredholm section and $s$ admits orientation bundle as in Definition \ref{def:admitori}. If we fix a polyfold bundle structure $(\cE,\bE) \stackrel{P}{\to} (\cX,\bX)$, then for any $x,y\in \cX_\infty,g\in G$ such that $\rho_{\cX}(g,|x|) = |y|$, we have a set of local lifts $L^p_{\cE}(g,x,y)$ in Definition \ref{def:bundleliftset}. Then for any $[\Lambda]_{(g,x)} \in L^p_{\cE}(g,x,y)$, we have a pushforward of orientation $\Lambda_*: (\mathscr{O}_{S^0})_x \to  (\mathscr{O}_{S^0})_y$. By Proposition \ref{prop:liftingbundle} and Remark \ref{rmk:admitori}, the pushforward does not depend on  $[\Lambda]_{(g,x)} \in L^p_{\cE}(g,x,y)$. Hence we can simply write the pushforward as $g_*:(\mathscr{O}_{S^0})_x \to  (\mathscr{O}_{S^0})_y$. Therefore the group property in Proposition \ref{prop:propertybundle} implies that the pushforward has a similar group property $(gh)_* = g_*\circ h_*$. As a result, we have the the following proposition. 

\begin{proposition}\label{prop:oriaction}
	Under the same setup as above, the $G$-action on $p$ induces a continuous $G$-action on the orientation bundle on $\mathscr{O}_s \to Z_\infty$.  
\end{proposition}
\begin{proof}
	To see that the pushforward descends to the orbit space, we have the following commutative diagram
	$$
	\xymatrix{
		(\mathscr{O}_{S^0})_x \ar[d]^{\phi_*} \ar[r]^{\Lambda_*} & (\mathscr{O}_{S^0})_y \ar[d]^{\psi_*}\\
		(\mathscr{O}_{S^0})_z \ar[r]^{\Lambda'_*} & (\mathscr{O}_{S^0})_w
		}
	$$
	where $[\Lambda]_{(g,x)} \in L^p_{\cE}(g,x,y), [\Lambda']_{(g,z)} \in L^p_{\cE}(g,z,w)$ and $\phi \in \mor(x,z),\psi \in \mor(y,w)$. This is because $  [R_{\psi^{-1}}\circ \Lambda'\circ R_\phi]_{(g,x)} \in L^p_{\cE}(g,x,y)$. Assume $\Lambda$ is defined $U\times P^{-1}(\cU) \to \cE$, then we have a local strong bundle isomorphism $\widetilde{\Lambda}:U\times P^{-1}(\cU) \to U\times \cE, (h,v) \mapsto (h,\Lambda(h,v))$. Then the continuity of the group action follows from the continuity of $\widetilde{\Lambda}_*$ by Proposition \ref{prop:oriiso}.      
\end{proof}

\begin{definition}\label{def:preori}
	Let $p:W\to Z$ be a regular strong polyfold bundle with a sc-smooth $G$-action. Assume $s:Z\to W$ is a $G$-equivariant sc-Fredholm section and $s$ is oriented, we say $G$ preserves the orientation iff the $G$-action in Proposition \ref{prop:oriaction} preserves the orientation section.
\end{definition}
As a corollary of Proposition \ref{prop:oriaction}, when $G$ is connected, $G$ always preserves the orientation if $s$ is oriented. 

\begin{proposition}\label{orient}
	Under the assumptions of Theorem \ref{quopoly}, assume the sc-Fredholm section $s$ is oriented and the $G$-action preserves the orientation. Then there is an isomorphism $\mathscr{O}_{s}/G \to \mathscr{O}_{\os}$, in particular, $\os$ is orientable.
\end{proposition}
\begin{proof}
	Let $(\cU_a,\tilde{\cU}_a,V_a,f_a,\eta_a,N_a)$ and  $(\cU_b,\tilde{\cU}_b,V_b,f_b,\eta_b,N_b)$ be two $G$-slice of a polyfold bundle structure $(\cE,\bE) \stackrel{P}{\to} (\cX,\bX)$ used in the construction of the quotient in Theorem \ref{thm:quopoly}. For $x \in (\tilde{\cU}_a)_\infty,y\in (\tilde{\cU}_b)_\infty$ such that there is a local lift $\Lambda$ with $[\Lambda]_{(g,x)} \in L^p_{\cE}(g,x,y)$ covering $\Gamma$ on the base. Then we have the following commutative diagram of local strong bundle isomorphisms
	\begin{equation}\label{eqn:com}
	\xymatrix{
		P^{-1}(\cU_a) \ar[rr]^{(f_a,N_a)} \ar[d]^{\Lambda(g,\cdot)} && V_a \times P^{-1}(\tilde{\cU}_a) \ar[d]^{\Phi}\\
		P^{-1}(\cU_b) \ar[rr]^{(f_b,N_b)} && V_b\times P^{-1}(\tilde{\cU}_b).
			}
	\end{equation}
	Note that the inverse to $(f_a,N_a)$ is given by $\ID^p_a$, hence here $\Phi: V_a \times P^{-1}(\tilde{\cU}_a) \to V_b\times P^{-1}(\tilde{\cU}_b)$ is defined by $\Phi:(h,v) \mapsto (f_b\circ P \circ \Lambda(g, \Id_a^p(h,v)), N_b\circ \Lambda(g, \Id_a^p(h,v)))$. By Proposition \ref{prop:propertybundle}, when $h$ is close to $\Id$, we have  $\Lambda(g, \Id_a^p(h,v)) = \Lambda(gh, v) = \ID_b(ghg^{-1},\Lambda(g,v))$. By condition \eqref{bund:unique} and \eqref{bund:return} of Definition \ref{def:bundleGslice}, $N_b\circ \Lambda(g, \Id_a^p(h,v)) = N_b\circ \Lambda(g,v)$. Hence we have
	\begin{equation}\label{eqn:formula}
	\Phi(h,v) = (f_b \circ \Gamma(gh,P(v)), N_b\circ \Lambda(g,v)).
	\end{equation}
	Let $\tilde S_a:= S|_{\tilde{\cU}_a}$ and  $\tilde S_b:= S|_{\tilde{\cU}_b}$, where $S:(\cX,\bX) \to (\cE,\bE)$ is the preventative of $s$. Then $s$ being $G$-equivariant implies that the diagram \eqref{eqn:com} commutes with sections, i.e. $\Lambda(g,\cdot)_*S = S$, $(f_a,N_a)_* S = \tilde{S}_a$, $(f_b,N_b)_* S = \tilde{S}_b$ and $\Phi_* \tilde{S}_a = \tilde{S}_b$. From the construction in Theorem \ref{thm:quopoly}, $N_b\circ \Lambda(g,v) = R_{(x,y,g,[\Lambda])}(v)$ on the quotient ep-groupoid $(\mathcal{BQ},\textbf{BQ})$ for $(x,y,g,[\Lambda]) \in \bQ$, where $R_{(x,y,g,[\Lambda])}$ is local action of $(x,y,g,[\Lambda])\in \bQ$ on $\mathcal{BQ}$. Moreover, when $P(v) = x$, by \eqref{eqn:formula} we have $\pi_{V_b}\circ \Phi(h,v) = f_b\circ \Gamma(gh,x) = f_b \circ \Id_b(ghg^{-1}, \Gamma(g,x)) = f_b\circ \ID_b(ghg^{-1},y)$, which by condition \eqref{bund:unique} of Definition \ref{def:bundleGslice}, is $ghg^{-1}$.  In particular, $\rD(\pi_{V_b}\circ \Phi)_{(\Id,(x,0))}:T_{\Id} V_a \times \{0\} \to T_{\Id} V_b$ preserve the orientation if we fix an orientation of $G$. Then we have a commutative diagram of orientation bundles,
	$$
	\xymatrix{
		\mathscr{O}_S|_{\cU_a} \ar[d]^{\Lambda(g,\cdot)_*} \ar[rr]^{(f_a,N_a)_*} && \mathscr{O}_{\tilde{S}_a}|_{V_a\times \tilde{\cU_a}} \ar[d]^{\Phi_* = (x,y,g,[\Lambda])_*} \\
		\mathscr{O}_S|_{\cU_b} \ar[rr]^{(f_b,N_b)_*} && \mathscr{O}_{\tilde{S}_b}|_{V_b\times \tilde{\cU_b}}. 
		}
	$$
	The orientation of $\mathscr{O}_S$ pushed forward to an orientation of $\mathscr{O}_{\tilde{S}_a}|_{V_a\times \tilde{\cU_a}}$ by $(f_a, N_a)$. After fixing an orientation of $G$ (hence $V_a$),  an orientation of $\mathscr{O}_{\tilde{S}_a}|_{V_a\times \tilde{\cU_a}}$ determines an orientation of $\mathscr{O}_{\tilde{S}_a}|_{\tilde{\cU_a}}$. By assumption, the orientation is preserved by the group action, hence $\Lambda(g,\cdot)_*$ preserves the orientation. The commutative diagram above shows that the pushforward orientation is well-defined on the quotient. 
\end{proof}
\begin{remark}
	Using the orientation scheme in the proof of Proposition \ref{orient}, when $s$ is transverse to $0$, $T_x\os^{-1}(0)$ is oriented by basis $\la \theta_1,\ldots,\theta_m\ra$ in such a way that $\la \xi_1,\ldots,\xi_n,\theta_1,\ldots, \theta_m\ra$ gives the orientation of $T_xs^{-1}(0)$, where $\la \xi_1,\ldots, \xi_n\ra$ is the infinitesimal directions of the group action and gives the fixed orientation of $G$.  
\end{remark}

Combining Theorem \ref{thm:quopoly} and Proposition \ref{orient},  we get a proof of the main Theorem \ref{thm:main}.

\section{Equivariant Transversality}\label{s5}
As corollaries of Theorem \ref{thm:main} and the perturbation theory \cite[Theorem 5.6, Theorem 15.4]{hofer2017polyfold} on polyfolds, we have the following basic equivariant transversality results.

\begin{corollary}{\label{coro:p1}}
	Let $p:\cY \to \cX$ be a tame strong M-polyfold bundle equipped with a sc-smooth $G$-action and a $G$-equivariant proper sc-Fredholm section $s$, such that the induced action on $\cX$ is free. Assume $\cX$ is infinite dimensional and supports sc-smooth bump functions. Then there exists an $G$-invariant open neighborhood $\hX\subset \cX^3$ containing $\cX_\infty$ and an equivariant $\sc^+$-perturbation $\gamma$ on $\hX$, such that $s+\gamma$ is proper and in general position. Moreover, if $s$ is already in general position on a $G$-invariant closed set $C\subset \hX$, then $\gamma$ can be chosen to satisfy $\gamma|_C = 0$. 
\end{corollary}

\begin{proof}
	By Theorem \ref{thm:free}, $\os$ is a sc-Fredholm section of $p^{-1}(\hX)/G\to \hX/G$ for a $G$-invariant open set $\hX \subset \cX^3$ containing $\cX_\infty$. By Corollary \ref{coro:equi}, $\hX/G$ also supports sc-smooth bump functions. By \cite[Theorem 5.6]{hofer2017polyfold}, there exists a $\sc^+$-perturbation $\overline{\gamma}$ on $\hX/G$ such that  $\os+\overline{\gamma}$ is proper and in general position.  Let $\gamma:=\pi^*_G\overline{\gamma}$, where $\pi_G:\hX \to \hX/G$ is the quotient, then $s+\gamma$ is $G$-equivariant, proper and transverse. By Remark \ref{rmk:general}, $s+\gamma$ is in general position. For the last assertion, we can chose $\overline{\gamma}|_{C/G} = 0$ by the proof of \cite[Theorem 5.6]{hofer2017polyfold}, then $\gamma|_{C} = 0$. 
\end{proof}

\begin{remark}
	Since $\hX$ contains of a neighborhood of $s^{-1}(0)$ in $\cX^3$, if we choose the $\supp \gamma$ small enough, then $\gamma$ can be extended to $\cX^3$ by $0$.  However, if there are infinite M-polyfolds with free group actions, keeping this property on the supports in a coherent way might be challenging.  Since in applications, we only need regular moduli spaces, which are contained in $\cX_\infty$, having $\gamma$ defined on a neighborhood of $\cX_\infty$ is sufficient.
\end{remark}

By \cite[Theorem 15.4]{hofer2017polyfold}, we have the following polyfold version with an identical proof.

\begin{corollary}{\label{coro:p2}}
	Let $p: W\to Z$ be a regular tame strong polyfold bundle equipped with a sc-smooth $G$-action and a $G$-equivariant proper sc-Fredholm section $s$, such that the induced action on $Z$ has finite isotropy. Assume $Z$ is infinite-dimensional and supports sc-smooth bump functions. Then there exists a $G$-invariant neighborhood $\hZ$ of $Z^3$ containing $Z_\infty$ and an equivariant $\sc^+$-multisection perturbation $\kappa$ defined on $\hZ$, such that $s+\kappa$ is proper and in general position. Moreover, if $s$ is already in general position on a $G$-invariant closed set $C\subset \hZ$, then $\kappa$ can be chosen to satisfy $\kappa|_C = 0$. 
\end{corollary}
\begin{remark}
	\cite[Theorem 5.6, Theorem 15.4]{hofer2017polyfold} assert the abundance of transverse perturbations and the perturbation can be chosen to be supported in arbitrarily small neighborhoods of $s^{-1}(0)$. One can get similar properties for equivariant perturbations under the conditions in Corollary \ref{coro:p1} and Corollary \ref{coro:p2}.
\end{remark}

For more general group actions, equivariant transversality is often obstructed. The obstructions usually arise from points with larger isotropy. Therefore the first step to analyze equivariant transversality would be understanding those with biggest isotropy group, namely the fixed locus. Moreover, the study of fixed locus is important for the localization theorem in \cite{equi}. We will first discuss the finite dimensional case to motivate the discussion in the polyfold case. 

\subsection{Manifold case}
In this subsection, we discuss equivariant transversality near the fixed locus in the case of manifolds. We first show that there is a standard local model for the equivariant transversality problem near the fixed locus. We call a vector bundle $p:E\to M$ a \textbf{$\bm{G}$-vector bundle} iff $G$ acts on $p$ such that the induced $G$-action on $M$ is trivial. Then the fibers of $G$-vector bundle are $G$-representations.  Let $\{V^\lambda\}_{\lambda\in \Lambda}$ denote the set of all the nontrivial irreducible $G$-representations.
\begin{proposition}\label{prop:dec}
	Let $p:E\to M$ be a $G$-vector bundle. Then there is a decomposition of $G$-vector bundles $E = E^G\oplus_{\lambda\in \Lambda} E^\lambda$ such that $E^G$ is fixed by the $G$-action and the fibers of $E^\lambda$ are sums of the irreducible representation $V^\lambda$.
\end{proposition}
\begin{proof}
	Since $G$ is compact, every finite dimensional representation can be decomposed into direct sum of irreducible representations. That is each fiber $E_x$ can be decomposed into irreducible representations, it suffices to show that the $V^\lambda$ component of $E_x$  forms a smooth subbundle. Let $\chi^{\lambda}:G\to \R$ denote the character of the irreducible representation $V^\lambda$ and $\rho$ denote the group action on $E$. We equip $G$ with a Haar measure $\mu$ such that $\mu(G) = 1$, then 
	$$P^\lambda: E \to E, \quad v \mapsto       \frac{\dim V^\lambda}{\dim_\R \End(V^\lambda)} \int_G \chi^\lambda(g) \rho(g,v) \rd\mu$$
	defines a smooth projection on $E$, e.g see \cite{brocker2013representations}. On each fiber $P^\lambda_x:E_x\to E_x$ is the projection to the $V^\lambda$ component. Therefore $E^\lambda := \ker (\Id - P^\lambda)$ is a smooth subbundle of $E$. 
\end{proof}

\begin{proposition}\label{prop:Glocal}
	Let $p:E\to M$ be a vector bundle over a closed manifold with a $G$-action $\rho$ for a compact group $G$. Suppose $M^G\subset M$ is the fixed point set of the induced $G$-action $\rho_M$ on $M$.  Then we have the following.
	\begin{enumerate}
		\item $M^G$ is a submanifold of $M$.
		\item Let $N$ denote the normal bundle of $M^G\subset M$, then the linearization of $\rho_M$ induces a $G$-action on $N$ such that $N$ is $G$-vector bundle. 
		\item There exists a $G$-invariant open neighborhood $U\subset M$ of $M^G$, such that there is a $G$-equivariant bundle isomorphism
		$$\Phi: E|_U \to \pi^*(E|_{M^G}),$$
		where $\pi$ is the projection from $N$ to $M^G$. 
	\end{enumerate}
\end{proposition}
\begin{proof}
We equip $M$ with a $G$-invariant metric $d$.  If $x\in M^G$, then there is a linear representation:
$$\rD(\rho_M)_x: G\to O(T_xM),$$
where $O(T_xM)$ is the orthogonal group. By the uniqueness of geodesic, we have
$$\rho_M(g, \exp(\xi))=\exp(\rD(\rho_M)_x(g)\cdot \xi), \quad \forall \xi \in T_xM.$$
Therefore $M^G$ is a submanifold and the tangent $T_xM^G$ is the fixed subspace $(T_xM)^G$. Let $\pi:N\to M^G$ denote the normal bundle of $M^G\subset M$. Linearization $\rD(\rho_M)_x$ induces a $G$-action on each fiber $N_x$. Then by the equivariant tubular neighborhood theorem \cite[Theorem 3.18]{emerson2010equivariant}, there is a $G$-invariant neighborhood of $M^G$ equivariantly diffeomorphic to the normal bundle $N$. With a little abuse of notation, we will call this tubular neighborhood $N$. 

Fixing a connection, we have a bundle isomorphism $\Psi: E|_N\to \pi^* (E|_{M^G})$ by parallel transportation, such that $\Psi|_{E|_{M^G}} = \Id$ is equivariant. We can write $\Psi$ as $(x,v)\mapsto (x,\psi(x,v))$, where $\psi(x,v)\in E_{\pi(x)}$. We equip $G$ with a Haar measure $\mu$ such that $\mu(G) = 1$, then we define
$$\Phi: E|_N\to \pi^* (E|_{M^G}), (x,v)\to \left(x,\int_G \rho(g, \psi\circ \rho(g^{-1},(x,v))) \rd \mu \right).$$
Then $\Phi$ is an equivariant bundle map and $\Phi|_{E|_{M^G}} = \Psi|_{E|_{M^G}}=\Id$. Hence $\Phi$ is an equivariant bundle isomorphism for a smaller $G$-invariant tubular neighborhood of $M^G$ in $N$.  
\end{proof}

As a corollary of Proposition \ref{prop:Glocal}, the equivariant transversality problem on the fixed locus $M^G$ is equivalent to the equivariant transversality problem on the local model $\pi^*(E|_{M^G})\to N$. Since $N,E|_{M^G}$ are $G$-vector bundles over $M^G$, by Proposition \ref{prop:dec} we have decompositions of $G$-vector bundles:
\begin{equation}\label{eqn:bundledec}
N = \oplus_{\lambda \in \Lambda} N^\lambda,\qquad E|_{M^G} := E^G\oplus_{\lambda \in \Lambda} E^\lambda.
\end{equation}
Let $s:N\to \pi^*(E|_{M^G})$ be a $G$-equivariant section. If $s(x)=0$ for $x\in M^G$, by Schur's lemma the linearized operator $\rD s_x$ can be decomposed into
\begin{equation}\label{eqn:sectiondec}
\rD s_{x}=\rD^Gs_x \oplus_{\lambda \in \Lambda} \rD^\lambda s_x,
\end{equation}
where $\rD^Gs_x: T_xM^G\to E_x^G$ is the linearization of the fixed part $s^G:=s|_{M^G}:M^G\to E^G$ and $\rD^\lambda s_x: N^\lambda_x \to E^\lambda_x $ is $G$-equivariant. Thus $G$-equivariant transversality on $M^G$ means that $s|_{M^G}:M^G\to E^G$ is transverse to $0$ and all $\rD^\lambda s_x$ are surjective for all $x\in (s^G)^{-1}(0)$.  If $\rank E^\lambda>\rank N^\lambda$ for some $\lambda$, then we can not expect equivariant transversality, unless $(s^{G})^{-1}(0)=\emptyset$. Even if $\rank E^\lambda \le \rank N^\lambda $ for all $\lambda\in \Lambda$ appearing in the decomposition, there are still some other obstructions. 
\begin{definition}\label{def:singular}
	The spaces of equivariant linear maps $\Hom_G(N_x^\lambda,E_x^\lambda)$ form a vector bundle over manifold $M^G$. We define 
	\begin{equation}\label{eqn:singular}
	S^\lambda : = \{ (x,h)\in\Hom_G(N^\lambda,E^\lambda)|x\in M^G, h\in\Hom_G(N^\lambda_x, E^\lambda_x) \text{ is not surjective.}\}.
	\end{equation}
	Then $S^\lambda \to M^G$ is a fiber bundle
\end{definition}

\begin{proposition}\label{prop:codim}
	Assume $\ind \rD^\lambda s \ge 0$, then the fiber of $S^\lambda$ is a real variety and the following two properties hold.
	\begin{enumerate}
		\item The real codimension of $S^\lambda_x$ in $\Hom_G(N_x^\lambda,E_x^\lambda)$ is $(\frac{\ind D^\lambda s}{\dim V^\lambda} + 1)\dim_\R \End_G(V^\lambda)$.
		\item The real codimension of the singularities of $S^\lambda_x$ in $S^\lambda_x$ is $(\frac{\ind D^\lambda s}{\dim V^\lambda} + 3)\dim_\R \End_G(V^\lambda)\ge 3$.  
	\end{enumerate} 
\end{proposition}
\begin{proof}
	For the irreducible representation $V^\lambda$, by Schur's Lemma the endomorphism ring $\End_{G}(V^\lambda)$ is a finite-dimensional division ring over $\R$. Hence by Frobenius Theorem, there are only three possibilities for $\End_{G}(V^\lambda)$ , namely $\R,\C$ and $\H$. The not-onto maps in $\Hom_G((V^\lambda)^n, (V^\lambda)^m)=\Hom_{\End_G(V^\lambda)}(\End_G(V^\lambda)^n,\End_G(V^\lambda)^m)$, i.e. $\End_G(V^\lambda)$-coefficient matrices with rank smaller than $m$. Such matrices form a determinantal variety $S$ and $S$ has real codimension $(n-m+1)\dim_{\R} \End_G(V^\lambda)$ in $\Hom_{\End_G(V^\lambda)}(\End_G(V^\lambda)^n,\End_G(V^\lambda)^m)$ \cite[Proposition 12.2]{harris2013algebraic}. The singularities come from those matrices with rank smaller than $m-1$, whose real codimension in $S$ is $(m-n+3)\dim_{\R} \End_G(V^\lambda)$. Therefore when $m\ge n$, this singularities are of real codimension bigger than $3\dim_{\R}\End_G(V^\lambda) \ge 3$.
\end{proof}

Now we can rephrase equivariant transversality near fixed locus as follows.
\begin{enumerate}
	\item $s|_{M^G}:M^G\to E^G$ is transverse to $0$;
	\item $\rD^\lambda_x \notin S^\lambda$ for all $x\in s|^{-1}_{M^G}(0)$ and $\lambda \in \Lambda$.
\end{enumerate}
By Proposition \ref{prop:codim}, since the singularities of $S^\lambda$ have high enough codimension, there is a well-defined Euler class $e(S^\lambda)\in H^{\codim S^\lambda}(M^G)$ if $M^G$ is closed. Such class can be represented by a pseudo-cycle $t^{-1}(S^\lambda)$ for a generic section $t$ of $\Hom_G(N^\lambda, E^\lambda) \to M^G$ \cite[Section 6.5]{mcduff2012j}. Then if $e(E^G) \cup e(S^\lambda)\ne 0$ for some $\lambda$, it is impossible to find equivariant transverse sections on $M^G$, here $e(E^G)$ is the Euler class of $E^G\to M^G$. However in the following simple case, we can use a dimension argument to get equivariant transversality near the fixed locus.
\begin{proposition}\label{prop:dim}
	Following the notation in Proposition \ref{prop:Glocal}, \eqref{eqn:bundledec} and \eqref{eqn:sectiondec}, let $s:N\to \pi^*(E|_{M^G})$ be a $G$-equivariant section. Let $S^\lambda$ be the fiber bundle defined in Definition \ref{def:singular}. If $$\dim M^G - \rank E^G < \codim S^\lambda = (\frac{\ind \rD^\lambda s}{\dim V^\lambda} + 1)\dim_{\R} \End_G(V^\lambda) $$ for all $\lambda$ such that $\rank E^\lambda > 0$, then there exists $G$-equivariant section $\gamma$ such that $s+\gamma$ is transverse to $0$ on $M^G$. 
\end{proposition}
\begin{proof}
	Let $\gamma:M^G\to E^G$ be a perturbation such that $s^G+\gamma:M^G\to E^G$ is transverse to $0$. Then $s+\pi^*\gamma$ is a $G$-equivariant section such that $(s+\pi^*\gamma)|_{M^G}:M^G\to E^G$ is transverse to $0$. Hence without loss of generality, we can assume $s^G:M^G\to E^G$ is transverse to $0$. In particular $(s^G)^{-1}(0)$ is a smooth manifold with dimension $\dim M^G - \rank E^G$.
	
	For each $\lambda$ such that $\rank E^\lambda > 0$, we can find a section $t^\lambda: (s^G)^{-1}(0) \to \Hom_G(N^\lambda,E^\lambda)|_{(s^G)^{-1}(0)}$, such that $\rD^\lambda s_x+t^\lambda (x) \notin S^\lambda_x $ for every $x\in (s^G)^{-1}(0)$. This is because $\dim (s^G)^{-1}(0)< \codim S^\lambda$. 
	We then extend $t^\lambda$ to a section of $\Hom_G(N^\lambda,E^\lambda)\to M^G$. For all other $\lambda$ such that $\dim E^\lambda = 0$, we define $t^{\lambda} = 0$. Since $\oplus_{\lambda\in \Lambda} t^\lambda$ can be understood as an equivariant section of $\pi^*(E|_{M^G}) \to N$ vanishing on $M^G$, then $s+\oplus_{\lambda \in \Lambda} t^\lambda$ is a $G$-equivariant transverse section on $M^G$.  
\end{proof}	
\begin{remark}
	Assume $s^G: M^G\to E^G$ is transverse, then one can study the equivariant transversality from obstruction theory, where the obstructions lies in $H^{k+1}((s^G)^{-1}(0),\allowbreak \pi_k(\Hom_G(N^\lambda_x,E^\lambda_x)\backslash S^\lambda_x))$.
\end{remark}

\subsection{Polyfold case}
Let $p:W\to Z$ be a strong polyfold bundle with $G$-action and $s:Z\to W$ an equivariant proper sc-Fredholm section. We need an analogue of Proposition \ref{prop:Glocal} in the polyfold case, thus we introduce the following assumption (Definition \ref{def:tube}). Let $\Lambda$ be the index set for all the nontrivial irreducible representations $V^\lambda$ of $G$.

\begin{definition}\label{def:lambundle}
	We define \textbf{$\bm{G}$-tame (strong) M-polyfold bundle} to be a tame (strong) M-polyfold bundle $\pi: \cN \to \cX$ with a sc-smooth $G$-action $\rho$, such that the induced action on the base $\rho_{\cX}$ is trivial. Let $\lambda \in \Lambda$, we call a $G$-tame (strong) M-polyfold bundle $\cN \to \cX$ a \textbf{$\bm{\lambda}$-tame (strong) M-polyfold bundle} iff for any vector $v\in \cN_x$, the $G$-invariant space spanned by $v$ is $G$-equivariantly isomorphic to $V^\lambda$. 
\end{definition}
To guarantee the existence of bump functions and partition of unity, \textbf{we assume from now on that all the M-polyfolds are modeled on sc-Hilbert spaces\footnote{In fact, it is sufficient to assume the first level spaces are Hilbert spaces, see \cite[Corollary 5.2]{hofer2017polyfold}.}}. Such assumption also plays a role in Proposition \ref{prop:complement}.

\begin{definition}\label{def:tube}
	We say the $G$-action on $p:W\to Z$ satisfies the \textbf{tubular neighborhood assumption on the fixed locus without isotropy} if the following conditions hold.
	\begin{enumerate}
		\item The fixed part $Z^G$ and $W^G$ are tame subpolyfolds \cite[Definition 16.10]{hofer2017polyfold} of $Z$ and $W$, such that $W^G$ is a strong polyfold bundle over $Z^G$. Since $s$ is equivariant, we have $s^G:=s|_{Z^G}: Z^G \to W^G$. We also require that $s^G$ is a proper sc-Fredholm section.
		\item\label{noiso} $Z^G$ has no isotropy, i.e. $Z^G$ is a tame M-polyfold. Hence $W^G$ is a  tame strong M-polyfold bundle over $Z^G$.
		\item\label{normal} There exists a $G$-invariant neighborhood $N\subset Z$ of $(s^G)^{-1}(0):=s^{-1}(0)\cap Z^G$, such that there is a projection $\pi: N \to N\cap Z^G$ making $N$ a $G$-tame M-polyfold bundle. We have a decomposition of $G$-tame M-polyfold bundles  $N = \widehat{\oplus}_{\lambda \in \Lambda} N^\lambda$, such that $N^\lambda$ is a $\lambda$-tame M-polyfold bundle. Here $\widehat{\oplus}$ stands for some completion of the direct sum, see \eqref{complete} of Remark \ref{rmk:tube}.
		\item\label{strong} The $G$-tame strong M-polyfold bundle $W|_{N\cap Z^G}$ has a similar decomposition, i.e. we have $\lambda$-tame strong M-polyfold bundles $W^\lambda$ over $Z^G$ for $\lambda \in \Lambda$, such that $W[i]|_{N\cap Z^G} = W^G[i] \widehat{\oplus}_{\lambda \in \Lambda} W^\lambda [i]$ for $i=0,1$. Moreover, there is a $G$-equivariant strong bundle isomorphism $W|_N \to \pi^*(W|_{N\cap Z^G})$. 
		\item\label{dim} We have $\rank W_x^\lambda=\infty$ for any $x\in N\cap Z^G$ unless $W^\lambda$ is the rank zero bundle.
	\end{enumerate} 
\end{definition}
In the following, we will abbreviate tubular neighborhood assumption on the fixed locus without isotropy (Definition \ref{def:tube}) to \textbf{tubular neighborhood assumption}.

\begin{remark}\label{rmk:tube}
	Several remarks on Definition \ref{def:tube} are in order.
	\begin{enumerate}
		\item A less restrictive assumption is dropping \eqref{noiso} in Definition \ref{def:tube}, so that the fixed locus $Z^G$ is a polyfold and all $W^G, N^\lambda, W^\lambda$ are polyfold bundles. Condition \eqref{noiso} guarantees the existences of a global stabilization in Proposition \ref{prop:stab}, which is important for us to actually get the equivariant transversality when possible. However, for the localization theorem in \cite{equi},  we can drop \eqref{noiso} and work with polyfold fixed locus.
		\item 	The tubular neighborhood assumption does not always hold for general polyfolds. In fact, the tubular neighborhood Theorem does not hold for M-polyfolds. For example, there is no neighborhood of the origin with linear structures in the object M-polyfold in Example \ref{ex:irr-one}. Such phenomenon is due to that M-polyfolds are modeled on retracts. Although one may expect a tubular neighborhood of retract fiberation, we need  linear structures in the proof of Theorem  \ref{thm:equitran}. 
		\item \label{complete} Let $\E^\lambda$ be a family of sc-Hilbert spaces, $\widehat{\oplus}_{\lambda \in \Lambda} \E^\lambda$ is sc-Hilbert space such that $\oplus_{\lambda \in \Lambda} \E^\lambda \hookrightarrow \widehat{\oplus}_{\lambda \in \Lambda} \E^\lambda$ is a dense inclusion.
		\item	The condition \eqref{dim} in Definition \ref{def:tube} is used in the proof of Proposition \ref{prop:stab}. Such condition is satisfied in all known applications.
	\end{enumerate}	
\end{remark}


Definition \ref{def:tube} without \eqref{noiso} is usually satisfied in applications. The no-isotorpy assumption holds for the $S^1$-action on the autonomous Hamiltonian-Floer cohomology polyfold,  see Section \ref{s6} for details.

\begin{proposition}\label{prop:fix}
	Suppose the tubular neighborhood assumption (Definition \ref{def:tube}) holds, then there exists a $G$-equivariant $\sc^+$ perturbation $\gamma$, such that $(s+\gamma)|_{Z^{G}}:Z^G\to W^G$ is transverse to $0$ and $s+\gamma$ is proper. In particular, $(s+\gamma)^{-1}(0)\cap Z^G$ is a compact manifold. Moreover, one can choose $\gamma$ such that $\supp \gamma$ is contained in arbitrarily small neighborhood of $s^{-1}(0)\cap Z^G$ in $Z$. 
\end{proposition}
\begin{proof}
	By \cite[Theorem 5.6]{hofer2017polyfold}, for any neighborhood $U\subset Z$ of $s^{-1}(0)\cap Z^G$, we can find a $\sc^+$-perturbation $\tau$ such that $s^G+\tau:Z^G\to W^G$ is a proper transverse section and $\supp \tau \subset N\cap U \cap  Z^G$. Using the tubular neighborhood assumption, $\pi^*\tau$ is a $G$-equivariant $\sc^+$-perturbation on $N$. By Corollary \ref{coro:equi} and sc-smooth partition of unit, there exists a $G$-invariant sc-smooth function $f: N\to [0,1]$ such that  $f|_{(s^G+\tau)^{-1}(0)} = 1$ and $\supp f \subset N \cap U$. Then $\gamma:=f\pi^*\tau$ is $G$-equivariant $\sc^+$-perturbation supported in $N\cap U$ and $(s+\gamma)|_{Z^G}:=s^G+\tau:Z^G\to W^G$ is transverse to $0$. 
	Since $\gamma$ can be chosen to be arbitrarily small in some auxiliary norm (\cite[Definition 5.1]{hofer2017polyfold}), then $s+\gamma$ is proper by \cite[Theorem 5.1]{hofer2017polyfold}
\end{proof}

Let $x\in Z^G_\infty$ such that $s(x) = 0$. Since $Z^G$ and $N$ are M-polyfolds by assumption, we can take the linearization of $s$ on $Z^G$ instead of on a polyfold structure. Since $s$ is equivariant, we have 
$$\rD s_x=\rD^Gs_x\oplus_{\lambda \in \Lambda} \rD^\lambda s_x,$$ 
where $\rD^Gs_x$ is the linearization of $s^G:Z^G\to W^G$ at $x$ and $\rD^\lambda s$ are $G$-equivariant linear Fredholm operators from $N^\lambda_x\to W^\lambda_x$. To get equivariant transversality near $Z^G$, we need at least $\ind \rD^\lambda s_x \ge 0$ for all $x\in (s^G)^{-1}(0)$ and $\lambda \in \Lambda$. To translate the discussion in the finite dimensional case to the polyfold world, we need to understand $\Hom_G(N^\lambda, W^\lambda)$.  In the following, we show that we can find a finite dimensional substitute of $\Hom_G(N^\lambda, W^\lambda)$. The following two propositions are proven in Appendix \ref{C}.
\begin{proposition}\label{prop:stab}
Under the tubular neighborhood assumption (Definition \ref{def:tube}), assume $(s^G)^{-1}(0)$ is a compact manifold. Then there is a $G$-invariant finite-dimensional trivial subbundle $\overline{W}^\lambda\subset W^\lambda_\infty$ over $(s^G)^{-1}(0)$ with constant rank, such that $\overline{W}^\lambda$ covers $\coker \rD^\lambda s$ over $(s^G)^{-1}(0)$. 
\end{proposition}

For every $x\in N\cap Z^G_\infty$, there exists $\sc^+$ perturbation $\tau:Z^G\to W^G$ such that $s(x)+\tau(x) = 0$. Then $f\pi^*\tau$ is a $G$-equivariant perturbation defined on $Z$ for an equivariant bump function $f$ supported in a neighborhood $x$.  Then we can define $\rD^\lambda s_x$ to be  $\rD^\lambda(s+f\pi^*\tau)_x$, i.e. the $\lambda$-component of $\rD (s+f\pi^*\tau)_x$.  Since $\rD^\lambda(f_1\pi^*\tau_1-f_2\pi^*\tau_2)_x = 0$ for different choices of $f$ and $\tau$, $\rD^\lambda s_x$ is well-defined for every $x\in N\cap Z^{G}_\infty$. Although $\rD^\lambda s_x$ may not be continuous in the norm topology in a local trivialization \cite[Theorem 4.1]{counter}, the following key properties of $\rD^\lambda s_x$ hold.

\begin{proposition}\label{prop:tang}
Under the tubular neighborhood assumption (Definition \ref{def:tube}), we have the following.
\begin{enumerate}
	\item $\ind \rD^\lambda s$ is locally constant on $N\cap Z^G_\infty$ using $Z^G_i$ topology for any $0\le i \le \infty$.
	\item For every $k \in \N$, $\{x\in N\cap Z^G_\infty|\dim \coker \rD^\lambda s_x \le k\}$ is an open subset of $Z^G_\infty$ with $Z^G_i$ topology for any $2\le i \le \infty$. 
	\item Suppose that $s^G:Z^G\to W^G$ is transverse to $0$. Let $\overline{W}^\lambda$ be a trivial bundle asserted in Proposition \ref{prop:stab}, then $\overline{N}^\lambda : = (\rD^\lambda s)^{-1}(\overline{W}^\lambda)$ is smooth subbundle in $N^\lambda_\infty$ over $(s^G)^{-1}(0)$. Moreover, $\rank \overline{N}^\lambda_x-\rank \overline{W}^\lambda_x=\ind \rD^\lambda s_x$.
\end{enumerate}
\end{proposition}
After applying Proposition \ref{prop:tang}, we can choose a $G$-invariant decomposition $N^\lambda=\overline{N}^\lambda \oplus \underline{N}^\lambda$ by Proposition \ref{prop:complement}.  Take $\tau^\lambda \in \Hom_G(\overline{N}^\lambda, \overline{W}^\lambda)$, then $\rD^\lambda s+\tau^\lambda\circ \pi_{\overline{N}^\lambda}: N^\lambda \to W^\lambda$ is surjective iff $\rD^\lambda s|_{\overline{N}^\lambda}+\tau^\gamma:\overline{N}^\lambda \to \overline{W}^\lambda$ is surjective. Moreover,  $\tau^\lambda \circ \pi_{\overline{N}^\lambda}$ can be viewed as a perturbation in the $\lambda$-direction.

\begin{theorem}\label{thm:equitran}
	Suppose the tubular neighborhood assumption (Definition \ref{def:tube}) holds. If for all $\lambda \in \Lambda$ and $x\in (s^G)^{-1}(0)$, we have 
	\begin{equation}\label{eqn:cond}
	\ind s_x^G < (\frac{\ind \rD^\lambda s_x}{\dim V^\lambda} + 1)\dim_{\R} \End_G(V^\lambda), \end{equation}
	then there exists a $G$-invariant $\sc^+$-perturbation $\gamma$ supported in $N$, such that $s+\gamma$ is proper and transverse to $0$ on $Z^G$. 
\end{theorem}
\begin{proof}
	By Proposition \ref{prop:tang}, the set of points where \eqref{eqn:cond} holds is open in $N\cap Z^G_0$ topology. Therefore, by Proposition \ref{prop:fix} we can assume $s^G:Z^G\to W^G$ is transverse to $0$ and \eqref{eqn:cond} holds for every $x\in ({s^G})^{-1}(0) \subset N\cap Z^G$. By Proposition \ref{prop:stab} and Proposition \ref{prop:tang}, there exists $G$-equivariant trivial subbundle $\overline{W}^\lambda\subset W^\lambda_\infty$ over $(s^G)^{-1}(0)$ covering the cokernel of $\rD^\lambda s$, and $\overline{N}^\lambda:=D^\lambda s^{-1}(\overline{N}^\lambda)\subset N^\lambda_\infty$ is a finite rank bundle over $(s^G)^{-1}(0)$.  Since the fiber bundle $S^\lambda$ (Definition \ref{def:singular}) in $\Hom_G(\overline{N}^\lambda,\overline{W}^\lambda)$ has codimension $(\frac{\ind \rD^\lambda s_x}{\dim V^\lambda} + 1)\dim_{\R} \End_G(V^\lambda)$. By the same argument in Proposition \ref{prop:dim}, we can find a section $\tau^\lambda$ in $\Hom_G(\overline{N}^\lambda,\overline{W}^\lambda)\to (s^G)^{-1}(0)$, such that $\rD^\lambda s_x|_{\overline{N}_x^\lambda}+\tau^\lambda(x):\overline{N}^\lambda_x\to \overline{W}^\lambda_x$ is surjective for every point in $x\in (s^G)^{-1}(0)$. Let $\pi_{\overline{N}^\lambda}$ denote the projection $N^\lambda\to \overline{N}^\lambda$ guaranteed by Proposition \ref{prop:complement}, then $\tau^\lambda\circ\pi_{\overline{N}^\lambda}$ is a $\sc^\infty$ bundle map $N^{\lambda}\to \overline{W}^\lambda\subset W^\lambda[1]$ over $(s^G)^{-1}(0)$.  Since we assume $s^G$ is transverse to zero, by \cite[Theorem 3.13]{hofer2017polyfold} $(s^G)^{-1}(0)\subset Z^{G}$ is sub-M-polyfold in the sense of \cite[Definition 2.12]{hofer2017polyfold}. That is for every $x\in (s^G)^{-1}(0)$, there exists a neighborhood $V \subset N\cap Z^{G}$ of $x$ and a sc-smooth retraction $r:V\to V$ such that $r(V) = (s^G)^{-1}(0)\cap V$.  Then $r^*(\tau^\lambda\circ \pi_{\overline{N}^\lambda}|_{(s^G)^{-1}(0)\cap V}):N^\lambda|_V \to W^\lambda[1]|_V$ defines a sc-smooth bundle map  extension of $\tau^\lambda\circ \pi_{\overline{N}^\lambda}|_{(s^G)^{-1}(0)\cap V}$ by Proposition \ref{prop:pull}.  Applying a partition of unity argument to all such $r^*(\tau^\lambda\circ \pi_{\overline{N}^\lambda}|_{(s^G)^{-1}(0)\cap V})$, we can find a sc-smooth bundle map $\theta^\lambda: N^\lambda|_U \to W^\lambda[1]|_U$ for a neighborhood $U\subset Z^G$ of $(s^G)^{-1}(0)$ such that $\theta^\lambda|_{N|_{(s^G)^{-1}(0)}}=\tau^\lambda \circ \pi_{\overline{N}^\lambda}$. By Corollary \ref{coro:ave},  we can assume $\theta^\lambda: N^\lambda|_U \to W^\lambda[1]|_U$ is a $G$-equivariant bundle map and $\theta^\lambda|_{N^\lambda|_{(s^G)^{-1}(0)}} = \tau^\lambda \circ \pi_{\overline{N}^\lambda}$ by averaging.  Using the tubular neighborhood assumption, $\theta^{\lambda}$ can be treated as a $\sc^+$-section of $W\to Z$ on $N|_U\subset Z$. We apply the same procedure to all $\lambda$, such that $\rD^\lambda s_x$ is not surjective for some $x\in(s^G)^{-1}(0)$.  By Proposition \ref{prop:tang} and the compactness of $(s^G)^{-1}(0)$, there are only finitely many such $\lambda$. Then $s+f\sum_{\lambda\in \Lambda}\theta^\lambda$ is an equivariant section and transverse on $Z^G$, where $f$ is a $G$-invariant bump function on $N$ such that $f|_{(s^G)^{-1}(0)} = 1$ and $\supp f \subset N|_U$. We can choose $f$ such that $f\sum_{\lambda\in \Lambda}\theta^\lambda$ is supported in arbitrarily small neighborhood of $(s^G)^{-1}(0)$ in $Z$ and arbitrarily small in some auxiliary norm \cite[Definition 12.2]{hofer2017polyfold}, then $s+f\sum_{\lambda\in \Lambda}\theta^\lambda$ is proper by \cite[Theorem 12.9]{hofer2017polyfold}.
\end{proof}

In the case of $S^1$-action, every nontrivial irreducible representation is two-dimensional and classified by a non-negative integer weight. The endomorphism ring is isomorphic to $\C$. 

\begin{corollary}[Corollary \ref{s1}]\label{S1}
	Suppose $p:W\to Z$ is regular tame strong polyfold bundle with a proper sc-Fredholm section $s$ and $Z$ is infinite dimensional. Assume $S^1$ acts on $p$ such that $s$ is equivariant and the tubular neighborhood assumption (Definition \ref{def:tube}) holds. Provided that for all weights $\lambda \in \N^+$ and $x\in s^{-1}(0)\cap Z^{S^1}$, we have $\ind \rD^\lambda s_x+2>\ind \rD^{S^1}s_x$. Then there exists a $S^1$-invariant neighborhood $\hZ\subset Z^3$ containing $Z_\infty$ and an equivariant $\sc^+$-multisection perturbation $\kappa$ on $\hZ$, such that 
	\begin{enumerate}
		\item $s+\kappa$ is transverse to zero and proper;
		\item there exist a $S^1$-invariant neighborhood $U\subset Z^3$ of $(Z^{S^1})^3$, such that $\kappa|_U$ is single valued.
	\end{enumerate}
	If $s$ is transverse to $0$ on an invariant closed set $D$, then we can require $\kappa = 0$ on $D$. 
\end{corollary}
\begin{proof}
	By Theorem \ref{thm:equitran}, there is a $S^1$-equivariant $\sc^+$-perturbation $\gamma$ supported in $N$ such that $s+\gamma$ is transverse to $0$ on $Z^{S^1}$. By \cite[Corollary 3.3]{counter}, there is a $S^1$-invariant neighborhood $U\subset Z^3$ of $(Z^{S^1})^3$, such that $s+\tau$ is proper and transverse to $0$ on $\overline{U} \subset Z^3$. Let $C := \overline{U}\backslash Z^{S^1}$, then $C$ is a closed subset in $Z^3\backslash Z^{S^1}$.  Since the $S^1$-action on $Z\backslash Z^{S^1}$ has finite isotropy, by Corollary \ref{coro:p2} there exists an equivariant transverse $sc^+$-multisection perturbation $\tau$ on $V \subset Z^3\backslash Z^{S^1}$ containing $Z_\infty\backslash Z^{S^1}_\infty$ such that $\tau|_{C \cap V} = \gamma|_{C\cap V}$ and $s+\tau$ is proper on $Z^3\backslash U$. Then 
	$$\kappa(z) = \left\{ \begin{array}{rl} 
	\gamma(z), & z \in U;\\
	\tau(z), & z \in V\backslash U;
	\end{array}
	\right.$$
	is an equivariant transverse perturbation defined on the $G$-invariant open set $\hZ := V\cup U\subset Z^3$ containing $Z_\infty$ and is single valued on $U$. The claim on the closed set $D$ follows from a similar argument in Corollary \ref{coro:p1}.
\end{proof}

The following Proposition is a direct consequence of the chain rule \cite[Theorem 1.1]{hofer2017polyfold}, which is used in the proof of Theorem \ref{thm:equitran}.
\begin{proposition}\label{prop:pull}
	Let $(K,(\R_+^m\times \E)\times \F)$ and $(P,(\R_+^m\times \E)\times \H)$ be two tame M-polyfold bundle retracts over the same tame M-polyfold retract $(O, \R_+^m\times \E)$. Let $R_1,R_2$ be the bundle retractions defining $K$ and $P$ and assume both $R_1$ and $R_2$ cover the tame retractions $r$ on the base.  Suppose we have a sc-smooth retraction $\pi:O\to O$ and a sc-smooth bundle map $f:K|_{\pi(O)} \to P|_{\pi(O)}$, then 
	$$\pi^*f: K \to P, \quad (x,v) \mapsto R_2(x,\pi_\H\circ f\circ R_1(\pi(x), v)) $$
	is a sc-smooth bundle map extending $f$.
\end{proposition}

\section{Application: Hamiltonian-Floer Cohomology for Autonomous Hamiltonians}\label{s6}
Let $(M,\omega)$ be a closed symplectic manifold, the weak Arnold conjecture asserts that for any non-degenerate Hamiltonian $H_t:S^1\times M \to \R$, then the number of periodic orbits of the Hamiltonian vector field $X_{H_t}$ is bounded below by $\sum_{i} \dim_{\R} H^i(M;\R)$. Floer \cite{floer1989symplectic} developed an invariant, which is now referred as Hamiltonian-Floer cohomology, to prove the conjecture. Given a non-degenerate Hamiltonian, the cochain complex of the Hamiltonian-Floer cohomology is generated by periodic orbits. A typical strategy initiated by Floer for symplectic aspherical manifolds and later generalized to all symplectic manifolds \cite{fukaya1999arnold,liu1998floer} is to show that the Hamiltonian-Floer cohomology is isomorphic to the regular cohomology. In order to relate Hamiltonian-Floer cohomology with regular cohomology, one can either use a $C^2$-small time independent Hamiltonian \cite{floer1989symplectic,fukaya1999arnold,liu1998floer} or the PSS construction \cite{piunikhin1996symplectic}. In both case, the proof of isomorphism boils down to existence of $S^1$-equivariant transversality. Another strategy for the weak Arnold conjecture, which does not prove the isomorphism, can be found in \cite{filippenko2018polyfold}. In particular, $S^1$-equivariant transversality is not required there. The purpose of this section is to explain $S^1$-equivariant transversality can be achieved in the polyfold setup for Hamiltonian-Floer cohomology, assuming Hamiltonian-Floer cohomology is generally constructed and proven to be invariant using polyfolds. As a result, we have an isomorphism between the Hamiltonian-Floer cohomology and regular cohomology, which yields the weak Arnold conjecture.

\subsection{Hamiltonian-Floer cohomology}
We first review the Hamiltonian-Floer cohomology setup from \cite[Chapter 12]{mcduff2012j}. Let $\widetilde{\cL M}$ be the space of equivalence classes $[(x,u)]$, where $x:S^1\to M$ is a contractible loop in $M$ and $u$ is extension of $x$ to disk. Here $(x,u)$ is equivalent to $(x,v)$ iff the boundary sum $u\# (-v)$ is homologous to zero. Let $H_t:S^1\times M \to \R$ be a Hamiltonian. Then the symplectic action functional $\cA_H:\widetilde{\cL M} \to \R$ is defined by
$$\cA_H([(x,u)]) = -\int_D u^*\omega - \int_{S^1} H_t(x(t))dt.$$
Critical points of $\cA_H$ correspond to $[(x,u)]$ where $x$ is a contractible periodic orbit of $X_{H_t}$. Let $\widetilde{\cP}(H_t)\subset \widetilde{LM}$ be the set of critical points of $\cA_H$. For $\tilde{x}\in \widetilde{\cP}(H_t)$, we use $x$ to denote the underlying periodic orbit. Assume all periodic orbits are non-degenerated. Let $J_t$ be a $S^1$ family of compatible almost complex structures on $M$. For $\tilde{x}_-,\tilde{x}_+ \in \widetilde{\cP}(H_t)$, solutions to the Floer equation below can be viewed as downwards gradient flow lines of $\cA_H$ from $\tilde{x_-}$ to $\tilde{x}_+$,
\begin{equation}\label{eqn:HF}\partial_s u+J_t(u)(\partial_tu-X_t(u))=0,\end{equation}
such that $\lim_{s\to \pm \infty} u(s,t) = x_\pm(t)$ and $\tilde{x}_-\# u = \tilde{x}_+$. Then the index of \eqref{eqn:HF} is given by $\mu_{CZ}(\tilde{x}_-)-\mu_{CZ}(\tilde{x}_+)$, where $\mu_{CZ}$ is the Conley-Zehnder index \cite{robbin1995spectral}. Here we only point out that 
$$\mu_{CZ}(A\#\tilde x) = \mu_{CZ}(\tilde{x})-2c_1(A)$$
for $\tilde{x}\in \widetilde{\cP}(H_t)$ and $A\in H_2^S(M):=\Ima(\pi_2(M) \to H_2(M))$. Let $\cM_{\tilde{x}_-,\tilde{x}_+}$ be the compactified moduli space of solutions to \eqref{eqn:HF} module the $\R$-translation by adding the broken flow lines. Assume transversality holds, i.e. $\cM_{\tilde{x}_-,\tilde{x}_+}$ are compact smooth manifold with boundary of dimension $\mu_{CZ}(\tilde{x}_-)-\mu_{CZ}(\tilde{x}_+)-1$ whenever $\mu_{CZ}(\tilde{x}_-)-\mu_{CZ}(\tilde{x}_+) \le 2$. For $\mu_{CZ}(\tilde{x}_-)-\mu_{CZ}(\tilde{x}_+) = 2$, boundary of $\cM_{\tilde{x}_-,\tilde{x}_+}$ comes from broken flow lines, i.e. 
$$\partial\cM_{\tilde{x}_-,\tilde{x}_+} = \bigcup_{\substack{\tilde{y}\in \widetilde{\cP}(H_t)\\ \mu_{CZ}(\tilde{x}_-)-\mu_{CZ}(\tilde{y})=1}} \cM_{\tilde{x}_-,\tilde{y}}\times \cM_{\tilde{y},\tilde{x}_+}.$$
They are oriented in a coherent way by \cite{floer1993coherent}. Then we can define a cochain complex $C_{HF}^*$ consisting of elements $\sum_{\tilde{x}\in \widetilde{\cP}(H_t)} a_{\tilde{x}}\tilde{x}$ such that for every $C$
$$\#\{ \tilde{x}\in \widetilde{\cP}(H_t)| a_{\tilde{x}}\ne 0, A_{H}(\tilde{x})<C\}<\infty.$$
The differential $\delta_{HF}$ is defined by
\begin{equation}\label{eqn:differential}
\delta_{HF} \tilde{y} := \sum_{\tilde{x}\in \widetilde{\cP}(H_t)} (\int_{\cM_{\tilde{x},\tilde{y}}} 1)\tilde{x}.
\end{equation}
Let 
$$\Lambda:=\left\{\textstyle{\sum_{A\in H_2^S(M)}}f_A q^A \st  \forall C \in \R, \#\left\{A\in H^S_2(M)\st f_A\ne 0,\omega(A) < C\right\}<\infty\right\}$$ be the Novikov field with $q^A$ graded by $2c_1(A)$. Then $\Lambda$ acts on $CF^*$ by $A\cdot \tilde{x} = \tilde{x}\#(-A)$. Since we have $\cM_{\tilde{x}_-, \widetilde{x}_+} = \cM_{\tilde{x}_-\#A, \widetilde{x}_+\#A}$, we have $\delta_{HF}$ is a $\Lambda$-module map. With such preparation, the Hamiltonian-Floer cohomology is defined as
$$HF^*(H,J) := H^*(C_{HF}^*,\delta_{HF}),$$
which is a $\Lambda$-module. The key property is that $HF^*(H,J)$ is independent of $H,J$ by a continuation map construction. In order to achieve transversality we assumed above, one can use a generic almost complex structure $J_t$ when sphere bubbles can ruled out \cite{audin2014morse,floer1989symplectic} or one need to apply virtual techniques \cite{fukaya1999arnold,liu1998floer,wehrheim2012fredholm} for general symplectic manifolds.

\subsubsection{Hamiltonian-Floer cohomology for autonomous Hamiltonians}
To relate Hamiltonian-Floer cohomology to regular cohomology, we can use a very special Hamiltonian. Let $H:M\to \R$ be a Morse function. We can assume the $C^2$ norm of $H$ is small enough such that all time-$1$ period orbits of the Hamiltonian vector field $X_H$ are critical points of $H$.  Let $\cC(H)$ denote the set of critical points of $H$. We can assume $H$ is also self-indexing in the sense that 
\begin{equation}
H(x)>H(y) \text{ iff }\ind x > \ind y, \forall x,y\in \cC(H),
\end{equation}
where $\ind x,\ind y$ are the Morse indexes. A critical point $x$ can be viewed as $[(x,0)] \in \widetilde{\cP}(H)$ and $\mu_{CZ}([x,0])=n-\ind x$. In this special case, $C_{HF}^*$ is the free $\Lambda$-module generated by $\cC(H)$ and grading for $q^A x $ for $x\in \cC(H)$ and $A\in H^S_2(M)$ is given by $n-\ind x + 2c_1(A)$. If we pick an almost complex structure $J$ such that $g := \omega(\cdot, J\cdot)$ satisfies the Sard-Smale condition for the Morse function $H$, i.e. the stable and unstable manifolds of the gradient vector field $\nabla_g H$ intersect transversely. For $x,y\in \cC(H)$, let $\cM^0_{x,y}$ denote the compactified moduli space of gradient flow lines for $(H,g)$ from $x$ to $y$. Then we have $\cM^0_{x,y}\subset \cM_{x,y}$. Assume we can achieve transversality for all $\cM_{q^Ax,q^By}$ of dimension smaller than $2$ using this time independent almost complex structure $J$. Then if $2c_1(A)+n-\ind x = 2c_1(B)+n-\ind y + 1$, we have
\begin{equation}\label{eqn:rel}
\cM_{q^Ax,q^By} = \left\{\begin{array}{lr}
\emptyset, & \text{ if }A\ne B,\\
\cM^0_{x,y}, & \text{ if }A = B.
\end{array}\right.
\end{equation}
This is because that the rotation of the $S^1$-coordinate acts on $\cM_{q^Ax,q^By}$. Since  $\dim \cM_{q^Ax,q^By} = 0$, the action must be trivial. In other words, any solution in $\cM_{q^Ax,q^By}$ must be independent of the $S^1$ coordinate. Hence the claim follows.

Using $\cM^0_{x,y}$, one can define the Morse cochain complex $C_M^*=\oplus_{x\in \cC(H)}\Lambda \la x \ra$ with grading $|x|=2n-\ind x$, i.e. the Morse index of $-H$ at $x$. The differential $\delta_M$ is defined by
$$\delta_M y  = \sum_{x\in \cC(H)}(\int_{\cM^0_{x,y}}1 )  x .$$
The cohomology of $(C_M^*, \delta_M)$ is denoted by $HM^*(-H,g)$ and it is isomorphic to the regular cohomology with Novikov coefficient $H^*(M;\Lambda)$ \cite{audin2014morse}. Then by \eqref{eqn:rel}, we have $HF^*(H,J)=HM^{*+n}(-H,J) = H^{*+n}(M;\Lambda)$. Hence the weak Arnold conjecture follows. In absence of sphere bubbles, transversality using time independent $J$ was proven in \cite{floer1989symplectic,salamon1992morse}.
\subsubsection{Polyfold setup}\label{s6s1}
For general symplectic manifold, perturbing the almost complex is not sufficient. Since the key property we need in addition to usual transversality is that $S^1$ acts on $\cM_{q^Ax,q^Bx}$, we need to argue that  $S^1$-equivariant transversality is possible in framework of the virtual technique used to implement Hamiltonian-Floer cohomology. Such arguments were carried out for Kuranishi type constructions in \cite{fukaya1999arnold,liu1998floer}. The purpose of this section is to resolve those difficulties transparently in a manner that is applicable to any abstract perturbation approach. To be specific and rigorous, we work in the technical context of polyfold theory and formulate the setup given for Floer theory as assumptions blew. Note that this is the same setup as in \cite{fukaya1999arnold,liu1998floer}, with exception of the precise smooth nature of the spaces $Z_{x,y,A}$.

In the following, we will consider Hamiltonian-Floer polyfolds for an autonomous $C^2$ small non-degenerate Hamiltonian $H$ and time independent almost complex structure. We denote by $\hbar > 0$ the minimal symplectic area of nontrivial $J$-holomorphic curve. For simplicity, we can choose $H$ such that $\max H -\min H < \hbar$. We first list some properties of Hamiltonian-Floer cohomology polyfolds as assumptions.
\begin{assumption}\label{ass:poly}
	\hspace{1em}
	\begin{enumerate}
		\item\label{a1} For every pair of critical points $x,y \in \cC(H)$ and a homology class $A\in H^S_2(M)$ such that $H(y)-H(x)+\omega(A) > 0$, there is an effective regular tame strong polyfold bundle $p_{x,y,A}:W_{x,y,A} \to Z_{x,y,A}$ modeled on sc-Hilbert spaces and an oriented sc-Fredholm section $s_{x,y,A}:Z_{x,y,A}\to W_{x,y,A}$ with compact zero set. We have $$\ind s_{x,y,A} = \mu_{CZ}(x,-A) - \mu_{CZ}(y,0)-1=\ind y - \ind x + 2c_1(A)-1.$$ 
		$s_{x,y,A}$ are oriented in an coherent way \cite{floer1993coherent}. The maximal degeneracy index on each polyfold $Z_{x,y,A}$ is finite. When $A=0$, the polyfold $Z_{x,y,0}$ has trivial isotropy, i.e. an M-polyfold. 
		\item\label{a2} There are natural sc-smooth identifications
		\begin{eqnarray*}
		\partial^k Z_{x,y,A} & = & 	\bigcup_{\substack{x_1,\ldots, x_{k-1} \in C(H)\\ \sum_{i=1}^k A_k = A}}  \partial^0 Z_{x,x_1,A_1}\times \ldots  \times \partial ^0Z_{x_{k-1},y,A_k},\\
		W_{x,z,A}|_{\partial^k Z_{x,y,A}} & = & \bigcup_{\substack{x_1,\ldots, x_{k-1} \in C(H)\\ \sum_{i=1}^k A_k = A}}  W_{x,x_1,A_1}|_{\partial^0 Z_{x,x_1,A_1}}\times \ldots  \times W_{x_{k-1},y,A_k}|_{\partial ^0Z_{x_{k-1},y,A_k}},\\
		s_{x,y,A}|_{\partial^0 Z_{x,x_1,A_1}\times \ldots  \times \partial ^0Z_{x_{k-1},y,A_k}} & = & s_{x,x_1,A_1}\times \ldots \times s_{x_{k-1},y,A_k}.
		\end{eqnarray*}
		Here $\partial^k Z\subset Z$ is the set of points with degeneracy index $k$.
		 \item\label{a3} The reparametrization in the $S^1$-coordinate is a $S^1$-action on $p_{x,y,A}$, such that $s_{x,y,A}$ is $S^1$-equivariant. The $S^1$-action on $Z_{x,y,A}$ has finite isotropy unless $A = 0$. 
		 \item\label{a4}The fixed part $s^{S^1}_{x,y,0}:=s_{x,y,0}|_{Z^{S^1}_{x,y,0}}:Z^{S^1}_{x,y,0}\to W^{S^1}_{x,y,0}$ is the sc-Fredholm section on the Morse homology M-polyfold sketched in \cite{fabert2016polyfolds} for $(H, g)$ and is in general position by the Sard-Smale assumption. 
		 \item \label{a5} When $\ind y - \ind x = 1$ and $s_{x,y,0}+\kappa_{x,y,0}$ is transverse to $0$ for a $S^1$-equivariant $\sc^+$-multisection $\kappa_{x,y,0}$ vanishing on $Z^{S^1}_{x,y,0}$, then for every $q\in (s_{x,y,0}+\kappa_{x,y,0})^{-1}(0)\cap Z^{S^1}_{x,y,0}$, the orientation of $q$ induced from $s_{x,y,0}+\kappa_{x,y,0}$ is the same as the orientation of $q$ induced from $s^{S^1}_{x,y,0}$.
	\end{enumerate}
\end{assumption}		 
\begin{remark}
	Since $\max H - \min H <\hbar$, there is no sphere bundles in $Z_{x,y,0}$, hence $Z_{x,y,0}$ is a M-polyfold. 
\end{remark}
Roughly speaking, the polyfold $Z_{x,y,A}$ is the space of cylinders with bubbles from $x$ to $y$ representing homology class $A$, fibers of $W_{x,y,A}$ are spaces of sections of $(0,1)$ form and the section $s_{x,y,A}$ is the Floer equation. One way to construct polyfolds satisfying Assumption \ref{ass:poly} is through the SFT polyfolds \cite{hofer2017application}. Alternatively, one can generalize the construction in \cite{wehrheim2012fredholm} to general symplectic manifolds by a fiber product construction \cite{ben2018fiber} with the Gromov-Witten polyfolds \cite{hofer2017applications}. Constructions using Kuranishi type theory with similar properties were developed in \cite{fukaya1999arnold,liu1998floer}.

To define the Hamiltonian-Floer cohomology from a system of polyfolds as described in Assumption \ref{ass:poly}, we need perturbations satisfying the following conditions.
\begin{definition}\label{def:coh}
	A family of $\sc^+$-multisection perturbations $\{\kappa_{x,y,A}\}$ is \textbf{transverse and coherent on level $\bm{l\in \N}$} iff we have following.
	\begin{enumerate}
	\item For all $x,y,A$, $\kappa_{x,y,A}$ is $\sc^+$-multisection of $p_{x,y,A}:W^l_{x,y,A}\to Z^l_{x,y,A}$ over an open neighborhood of $(Z_{x,y,A})_\infty$ in $Z^l_{x,y,A}$, such that $s_{x,y,A}+\kappa_{x,y,A}$ is proper. 
	\item For all $x,y,A$ such that $\ind s_{x,y,A} \le 1$, we have $s_{x,y,A}+\kappa_{x,y,A}$ is transverse to $0$.
	\item For $x_0,\ldots, x_k \in \cC(H)$, $A_1,\ldots, A_k \in H_s^S(M)$ and $A:=\sum_{i=1}^k A_i$, we have
	\begin{equation}\label{eqn:coh}
	\kappa_{x_0,x_k,A}|_{\partial^k Z_{x_0,x_k,A}} = \kappa_{x_0,x_1,A_1}\times \ldots \times \kappa_{x_{k-1},x_k,A_k},
	\end{equation}
	on a neighborhood of ($\partial^0Z_{x_0,x_1,A_1})_\infty \times\ldots \times (\partial^0Z_{x_{k-1},x_k,A_k})_\infty$ in $\partial^0Z^l_{x_0,x_1,A_1} \times\ldots \times \partial^0Z^l_{x_{k-1},x_k,A_k} \subset \partial^k Z^l_{x_0,x_k,A}$.
	\end{enumerate}
\end{definition}

\begin{proposition}\label{prop:trans}
	There exists a transverse and coherent perturbation family ${\kappa_{x,y,A}}$ for any level $l\in \N$.
\end{proposition}
\begin{proof}[Sketch of proof]
	To construct transverse and coherent perturbations as in Definition \ref{def:coh}, one usually needs some inductive arguments. The induction is typically based on the maximal degeneracy index of the polyfold. We first find perturbations on those polyfolds without boundary, i.e. those polyfolds with maximal degeneracy index $0$,  and we choose the perturbations to be transverse if the index is not bigger than $1$ using \cite[Theorem 15.4]{hofer2017polyfold}.  Then by coherent condition \eqref{eqn:coh}, we have perturbations on the boundaries of polyfolds with maximal degeneracy index $1$. We extend the perturbations into the interior. When we use multisection in the perturbation, the existence of extension of a multisection from boundary requires the perturbation to be structurable \cite[Definition 13.17, Definition 14.4]{hofer2017polyfold}. \cite[Theorem 15.5]{hofer2017polyfold} guarantees the extended perturbation is again structurable. We require transversality if the index is smaller than $1$. To see such requirement is feasible, note that on $Z_{x,y,A_1}\times Z_{y,z,A_2}\subset \partial^1 Z_{x,z,A_1+A_2}$ we have $\ind_{x,y,A_1}+\ind_{y,z,A_2} = \ind_{x,z,A_1+A_2}-1$. Hence if $\ind_{x,z,A_1+A_2}\le 1$, then either both $\ind_{x,y,A_1}$ and $\ind_{y,z,A_2}$ are zero or one of them is negative. Therefore by induction hypothesis, if $\ind_{x,y,A_1} = \ind_{y,z,A_2} = 0$ we have $\kappa_{x,y,A_1}\times \kappa_{x,y,A_2}$ is a transverse perturbation to $s_{x,y,A_1}\times s_{y,z,A_2}$. If one of the indexes is negative, then  $(s_{x,y,A_1}\times s_{y,z,A_2} + \kappa_{x,y,A_1}\times \kappa_{x,y,A_2})^{-1}(0) = \emptyset$. In either case, $\kappa_{x,y,A_1}\times \kappa_{y,z,A_2}$ is a transverse perturbation on $Z_{x,y,A_1}\times Z_{y,z,A_2}$. Therefore we can find transverse extension to the interior using \cite[Theorem 15.5]{hofer2017polyfold}. The induction argument goes on with respect to the maximal degeneracy index. The claim on properness of $s_{x,y,A}+\kappa_{x,y,A}$ follows from \cite[\S 12.4]{hofer2017polyfold}.
\end{proof}

Given a transverse and coherent perturbation family $\{\kappa_{x,y,A}\}$,  we can define a differential $\delta_{HF}$ on $C_{HF}^*$,
\begin{equation}\label{eqn:diff}
\delta_{HF}  y  := \sum_{\substack{y\in C(H),A \\ \ind s_{x,y,A} = 0}} (\int_{(s_{x,y,A}+\kappa_{x,y,A})^{-1}(0)} 1)q^{A}  x .
\end{equation}
To get a cohomology, we need the following, which is a part of the polyfold construction for Hamiltonian-Floer cohomology.  Since they are not the topic of this paper, we list it as assumptions.
\begin{assumption}\label{ass:coh}
	The coherent orientations of $s_{x,y,A}$ in Assumption \ref{ass:poly} imply that $\delta_{HF} \circ \delta_{HF} = 0$.
\end{assumption}

\subsection{$S^1$-equivariant coherent perturbations}\label{s6s2}
In the following, we will show that there exists a transverse and coherent perturbation family $\{\kappa_{x,y,A}\}$, such that we have an isomorphism $HF^*(H,J,\{\kappa_{x,y,A}\}):=H^*(C_{HF}^*,\delta_{HF})\cong H^{*+n}(M)\otimes \Lambda$. In addition to the isomorphism, to get a proof of the weak Arnold conjecture, we still need the invariance of the Hamiltonian-Floer cohomology not only for autonomous Hamiltonians sketched above but also for general time dependent Hamiltonians. We list it as an assumption.
\begin{assumption}\label{ass:ind}
	For every non-degenerate Hamiltonian $H_t:S^1\times M \to \R$, there is a cochain complex generated by $\widetilde{\cP}(H_t)$ whose cohomology is isomorphic to $HF^*(H,J,\{\kappa_{x,y,A}\})$ above.
\end{assumption}
Assumption \ref{ass:ind} is verified by constructing Hamiltonian-Floer cohomology polyfolds for general Hamiltonians and a continuation map construction. They are parts of the Hamiltonian-Floer cohomology polyfolds package, hence are expected to be derived from \cite{hofer2017application}. The counterparts using classical construction and Kuranishi construction can be found in   \cite{audin2014morse,floer1989symplectic,fukaya1999arnold,liu1998floer}.

To relate the Hamiltonian-Floer cohomology with Morse cohomology, it suffices to find transverse and coherent perturbations, which are also $S^1$-equivariant. We will use Corollary \ref{S1} to show the $S^1$-equivariant transversality. Therefore we need the following.

\begin{proposition}\label{prop:equi}
	For the polyfolds in Assumption \ref{ass:poly}, the following two properties hold.
	\begin{enumerate}
		\item\label{e1}The $S^1$-action on $p_{x,y,0}:W_{x,y,0}\to Z_{x,y,0}$ satisfies the tubular neighborhood assumption in Definition \ref{def:tube}.
		\item\label{e2}For every $x,y\in \cC(H)$ such that $\ind y - \ind x \in \{ 1,2 \}$ and $\lambda\in \N^+$, we have $\ind \rD^\lambda s_{x,y,0} = 0$ on $s_{x,y,0}^{-1}(0)\cap Z^{S^1}_{x,y,0}$.\footnote{In fact, the claim holds for every $x,y \in \cC(H)$, see Subsection \ref{subsec:ver}.}
	\end{enumerate}
\end{proposition}
We prove Proposition \ref{prop:equi} in Subsection \ref{subsec:ver} for the direct Hamiltonian-Floer cohomology polyfolds described in \cite{wehrheim2012fredholm}. Similar properties are expected to hold for Hamiltonian-Floer cohomology polyfolds from SFT \cite{hofer2017application}.

\begin{theorem}\label{thm:ham}
	Suppose Assumption \ref{ass:poly} holds. There exists a transverse and coherent perturbation family $\{\kappa_{x,y,A}\}$ on level $3$ such that $\kappa_{x,y,A}$ is $S^1$-equivariant for all $x,y,A$.
\end{theorem}
As a corollary of Theorem \ref{thm:ham}, one have the following.
\begin{corollary}\label{coro:iso}
	Using perturbations in Theorem \ref{thm:ham}, we have $\delta_{HF} = \delta_M$. In particular, when Assumption \ref{ass:ind} holds, the weak Arnold conjecture holds.
\end{corollary}
\begin{proof}
	For every $x,y,A$ such that $\ind_{x,y,A} = 0$, since $S^1$-action has finite isotropy if $A\ne 0$， then equivariant transversality implies that $(s_{x,y,A}+\kappa_{x,y,A})^{-1}(0) = \emptyset$ when $A \ne 0$. If $A = 0$ then equivariant transversality implies that $(s_{x,y,0}+\kappa_{x,y,0})^{-1}(0) = (s_{x,y,0}^{S^1})^{-1}(0)$ with the same orientation. Hence $\delta_{HF}  =  \delta_M$ and the remaining claim follows.
\end{proof}
	
Before proving Theorem \ref{thm:ham}, we first gather some properties of structurable multisections from \cite[Chapter 14]{hofer2017polyfold} that will be crucial for our argument.
\begin{proposition}\label{prop:str}
	Let $p:W\to Z$  and $p':W'\to Z$ be two strong tame polyfold bundles, then we have the following properties. 
	\begin{enumerate}
		\item If $\sc^+$-multisection $\kappa$ of $p$ is single-valued, then $\kappa$ is structurable.
		\item Let $\{U,V\}$ be an open cover of $Z$. If $\sc^+$-multisection $\kappa$ is single-valued on $U$ and structurable on $V$, then $\kappa$ is structurable on $Z$. \footnote{This follows from the proof of \cite[Corollary 13.1]{hofer2017polyfold}}
		\item Let $\kappa,\kappa'$ be structurable $\sc^+$-multisections of $p$ resp. $p'$, then $\kappa\oplus \kappa'$ is a structurable $\sc^+$-multisection of $p\oplus p': W\oplus W' \to Z$.
		\item Let $f:X \to Z$ be a polyfold map and $\kappa$ a structurable $\sc^+$-multisection of $p$, then $f^*\kappa$ is a structurable $\sc^+$-multisection of $f^*p: f^*W \to X$. \footnote{This follows from \cite[Theorem 14.3]{hofer2017polyfold}.}
	\end{enumerate}
\end{proposition}

To prove Theorem \ref{thm:ham}, we need to make an induction argument sketched in Proposition \ref{prop:trans} in a $S^1$-equivariant way. On a single polyfolds, the $S^1$-equivariant transverse perturbation is constructed using Corollary \ref{coro:p2} and Corollary \ref{S1}. The key point of the proof is that we can maintain both the $S^1$-equivariant property and the structurable property in the induction. 

\begin{proof}[Proof of Theorem \ref{thm:ham}]
	We apply the induction argument sketched in Proposition \ref{prop:trans} to find such perturbations. To make sure that we can stay in the category of structurable multisections, we claim that there exist $S^1$-equivariant transverse and coherent perturbation family such that the following two conditions holds 
	\begin{enumerate}
		\item $\overline{\kappa_{x,y,A}}$ is structurable for  $A \ne 0$;
		\item $\kappa_{x,y,0}$ is single-valued near fixed locus and $\overline{\kappa_{x,y,0}}$ is structurable on $(Z_{x,y,0}\backslash Z^{S^1}_{x,y,0})/S^1$.
	\end{enumerate}
	where $\overline{\kappa_{x,y,A}}$ and $\overline{\kappa_{x,y,0}}$ are induced sections on the quotients. Strictly speaking, the quotients are defined on open dense sets with a level shift by three. We will not spell them out for simplicity.
	
	We prove the claim by induction on the maximal degeneracy index. The base case is maximal degeneracy index $0$, i.e. polyfolds without boundary. On $Z_{x,y,A}$ with $A\ne 0$, then by Assumption \ref{ass:poly}, the $S^1$-action has finite isotropy. Hence there exists $S^1$-equivariant transverse $\sc^+$ multisection perturbation $\kappa_{x,y,A}$ and $\overline{\kappa_{x,y,A}}$ can be chosen to be structurable on $Z_{x,y,A}/S^1$ by \cite[Theorem 15,4]{hofer2017polyfold}. On $Z_{x,y,0}$ with $0 \le \ind s_{x,y,0} \le 1$, then by Proposition \ref{prop:equi}, Corollary \ref{S1} can be applied to find $S^1$-equivariant perturbation $\kappa_{x,y,0}$ such that $\kappa_{x,y,0}$ is single-valued in a neighborhood of $Z^{S^1}_{x,y,0}$ and $\overline{\kappa_{x,y,0}}$ is structurable on $(Z_{x,y,0}\backslash Z^{S^1}_{x,y,0})/S^1$.
	On $Z_{x,y,0}$ with $\ind s_{x,y,0} > 1$, we simply require $\kappa_{x,y,0} = 0$. Finally, the self-index condition implies that $Z_{x,y,0} =\emptyset$ for every $x,y$ such that $\ind x \ge \ind y$. Therefore the claim holds for maximal degeneracy index $0$. 
	
	Assume the claim holds on polyfolds with maximal degeneracy index at most $k$. We will prove that there exist perturbations satisfying the same properties on all polyfolds with maximal degeneracy index $k+1$. We divide the proof into two cases. 
	
	\textit{Case 1: $Z_{x,y,A}$ with $A\ne 0$.} By Assumption \ref{ass:poly}, the $S^1$-action has finite isotropy. We claim that induced perturbation on $\partial Z_{x,y,A}/S^1$ is structurable, hence we can extend it to a structurable perturbation $\overline{\kappa_{x,y,A}}$ on $Z_{x,y,A}/S^1$. On the boundary face $Z_{x,z,A_1}\times Z_{z,y,A_2} \subset \partial Z_{x,y,A}$ with $A = A_1+A_2$, the quotient bundle $(W_{x,z,A_1}\times W_{z,y,A_2})/S^1$ has a decomposition $(W_{x,z,A_1} \times \{0\})/S^1\oplus (\{0\} \times W_{z,y,A_2})/S^1$. Let $\pi_1,\pi_2$ denote the two projections. If $A_1 \ne 0$ then $$\pi_1 \circ \overline{\kappa_{x,z,A_1}\times \kappa_{z,y,A_2}} =\mathrm{pr}_1^*\overline{\kappa_{x,z,A_1}},$$
	where $\mathrm{pr}_1:(Z_{x,y,A_0}\times Z_{y,z,A_1})/S^1\to Z_{x,y,A_0}/S^1$ is the quotient of the projection. By induction hypothesis, $\overline{\kappa_{x,z,A_1}}$ is structurable. Hence by Proposition \ref{prop:str}, $\pi_1 \circ \overline{\kappa_{x,y,A_0}\times \kappa_{y,z,A_1}}$ is structurable.  If $A_1 = 0$, then $\pi_1 \circ \overline{\kappa_{x,z,0}\times \kappa_{y,z,A}}$ on $(U\times Z_{y,z,A_1})/S^1$ is single-valued by induction hypothesis, where $U$ is a $S^1$-invariant neighborhood of $Z^{S^1}_{x,y,0}$. Similar to the $A_1\ne 0$ case, we have $\pi_1 \circ \overline{\kappa_{x,z,0}\times \kappa_{z,y,A}}$ is structurable on $((Z_{x,z,0}\backslash Z^{S^1}_{x,z,0})\times Z_{z,y,A})/S^1$. Therefore by Proposition \ref{prop:str}, $\pi_1 \circ \overline{\kappa_{x,z,0}\times \kappa_{z,y,A}}$ is structurable. So far, we show that $\pi_1 \circ \overline{\kappa_{x,z,A_1}\times \kappa_{z,y,A_2}}$ is structurable for any $A_1$, similarly  $\pi_2 \circ \overline{\kappa_{x,z,A_1}\times \kappa_{z,y,A_2}}$ is structurable. Then by Proposition \ref{prop:str}, $\overline{\kappa_{x,z,A_1}\times \kappa_{z,y,A_2}} = \pi_1\circ \overline{\kappa_{x,z,A_1}\times \kappa_{z,y,A_2}}\oplus \pi_2\circ \overline{\kappa_{x,z,A_1}\times \kappa_{z,y,A_2}}$ is structurable.  Therefore we have an extension $\overline{\kappa_{x,y,A}}$ from $\partial Z_{x,y,A}/S^1$ to $Z_{x,y,A}/S^1$, and the pullback $\kappa_{x,y,A}$ to $Z_{x,y,A}$ is $S^1$-equivariant. When $\ind s_{x,y,A} \le 1$, we have $\overline{s_{x,y,A}}$ is transverse on $\partial Z_{x,y,A}/S^1$ by the same argument in Proposition \ref{prop:trans}. Therefore using \cite[Theorem 15.5]{hofer2017polyfold}, we can require $\overline{\kappa_{x,y,A}}$ is a transverse perturbation on $Z_{x,y,A}/S^1$, hence $\kappa_{x,y,A}$ is an equivariant transverse perturbation on $Z_{x,y,A}$.
	
	\textit{Case 2: $Z_{x,y,0}$.} By the same argument above, $\overline{\kappa_{x,z,A}\times \kappa_{z,y,-A}}$ is structurable if $A \ne 0$, and  $\overline{\kappa_{x,z,0}\times \kappa_{z,y,0}}$ is structurable on $(Z_{x,z,0}\times Z_{z,y,0} \backslash Z_{x,z,0}^{S^1}\times Z^{S^1}_{z,y,0})/S^1$. By induction hypothesis, $\kappa_{x,z,0}\times \kappa_{z,y,0}$ is single valued on $U\times V$ where $U\subset Z_{x,z,0},V\subset Z_{z,y,0}$ are neighborhoods of $Z^{S^1}_{x,z,0},Z^{S^1}_{z,y,0}$. Hence there exists extension $\kappa_{x,y,0}$ such that $\kappa_{x,y,0}$ is single valued in a neighborhood of $Z^{S^1}_{x,y,0}$ and $\overline{\kappa_{x,y,0}}$ is structurable on $(Z_{x,y,0}\backslash Z^{S^1}_{x,y,0})/S^1$. When  $0 \le \ind s_{x,y,0} \le 1 $, by Corollary \ref{S1} we can require that $s_{x,y,0}$ is transverse to $0$. 
    Therefore the claim holds for all polyfolds by induction, and Theorem \ref{thm:ham} follows.  
\end{proof}

\subsection{Analyzing fixed locus in ``the direct construction"}\label{subsec:ver}
There are mainly two constructions of Hamiltonian-Floer cohomology polyfolds satisfying Assumption \ref{ass:poly} and Assumption \ref{ass:ind}. One way is deriving it from the SFT polyfolds \cite{hofer2017application}, i.e. the polyfolds of cylinders in $\R \times S^1 \times M$. The other construction is more direct by considering cylinders in $M$. When there are no sphere bubbles, the later construction was sketched in \cite{wehrheim2012fredholm}. In this subsection, we prove Proposition \ref{prop:equi} for the direct construction of Hamiltonian-Floer cohomology polyfolds. 
\subsubsection{The Morse homology M-polyfolds and Hamiltonian-Floer cohomology polyfolds}
We first give a brief description of the Morse homology M-polyfolds which contain compactified Morse trajectory spaces following \cite{fabert2016polyfolds}. For $x,y \in \cC(H)$, a flow line $[\gamma]$ from $x$ to $y$ is an equivalence class of curves $\gamma:\R \to M$ with $\lim_{t\to -\infty} \gamma(t) = x$ and $\lim_{t\to \infty} \gamma(t) = y$ with exponential decay at two ends. Here $\gamma_1$ is equivalent to $\gamma_2$ iff $\gamma_1(s+\cdot) = \gamma_2(\cdot)$ for some $s \in \R$. A broken flow line is a tuple $([\gamma_1],\ldots,[\gamma_k])$ such that $\lim_{t\to \infty} \gamma_i(t) = \lim_{t\to -\infty} \gamma_{i+1}(t)$ for all $1 \le i \le k-1$. The M-polyfold $\cX_{x,y}$ consists of both unbroken and broken flow lines from $x$ to $y$. A chart of $\cX_{x,y}$ near an unbroken flow linw $[\gamma]$ is modeled on a neighborhood of $0$ in 
\begin{equation}\label{eqn:Eg}
\E_{\gamma}:=\{v\in H^{3+i}_{\delta_i}(\R,\gamma^*TM)| v(0)\in H \}_{i\in \N},
\end{equation}
with chart map $v\in \E_{\gamma} \to [\exp_\gamma(v)]$, 
where $H$ is a hyperplane in $T_{\gamma(0)}M$ transverse to $\gamma$ at $\gamma(0)$, and $\delta_i>0$ are exponential weights in the Sobolev spaces\footnote{We require that $\delta_i$ is increasing to guarantee that $\E_\gamma$ is a sc-Hilbert space. We also require $\delta_i$ is small than minimum of the absolute values of eigenvalues of $\rm{Hess}(H)$ on all critical points to include every Morse trajectory. }. A chart near a broken flow line $([\gamma_1],\ldots, [\gamma_n])$ is modeled on a tame sc-retraction constructed from pregluing, see \cite{fabert2016polyfolds} for details.

We can similarly define the sc-Hilbert space $\H_{\gamma} := \{ H^{2+i}_{\delta_i}(\R, \gamma^*TM)\}_{i\in \N}$. Then there is a strong M-polyfold bundle $\cY_{x,y}$ over $\cX_{x,y}$, such that the fiber over $([\gamma_1],\ldots, [\gamma_k])$ isomorphic to $\H_{\gamma_1} \times \ldots \times \H_{\gamma_k}$ if we choose representatives $\gamma_1,\ldots,\gamma_k$. Moreover, there is a proper sc-Fredholm section $s_{x,y}:\cX_{x,y} \to \cY_{x,y}$ given by
$$s_{x,y}([\gamma_1],\ldots, [\gamma_n]) = (\gamma'_1 - \nabla H(\gamma_1), \ldots, \gamma'_n - \nabla H(\gamma_n) ).$$
This section is in general position if we assume the Sard-Smale condition. 

Similarly, for $x,y \in \cC(H)$ such that $\ind y > \ind x$, a cylindrical flow line $[u]$ from $x$ to $y$ is an equivalence class of curves $u:\R \times S^1 \to M$ with $\lim_{t\to -\infty} u(t,\cdot) = x$ and $\lim_{t\to \infty} u(t,\cdot) = y$ with exponential decay at two ends. Here $u_1$ is equivalent to $u_2$ iff $u_1(s+\cdot,\cdot) = u_2(\cdot,\cdot)$ for some $s \in \R$. We have a M-polyfold description of cylindrical flows lines from $x$  to $y$ including broken flow lines, which serves as the Hamiltonian-Floer polyfold $Z_{x,y,0}$. Let $u:\R\times S^1 \to M$ be a flow from $x$ to $y$ with exponential decay at two ends and $[u]$ denote the equivalence class in $Z_{x,y,0}$. Then a chart of $Z_{x,y,0}$ near $u$ is modeled on a neighborhood of $0$ in 
\begin{equation}\label{eqn:Eu}
\E_{u}:=\{v\in H^{3+i}_{\delta_i}(\R\times S^1,u^*TM)| v(0,0)\in H \}_{i\in \N},
\end{equation}
with chart map $v\in \E_{u} \to [\exp_u(v)]$, where $H$ is a hyperplane in $T_{u(0,0)}M$ transverse to $u(s,0)$ at $u(0,0)$, and $\delta_i > 0$ are the exponential weights as before. A chart near a broken flow $([u_1],\ldots, [u_k])$ is modeled on a tame sc-retraction constructed from pregluing, see \cite{wehrheim2012fredholm} for details. Let $\H_u := \{H^{2+i}_{\delta_i}(\R\times S^1,u^*TM) \}_{i\in \N}$. The strong M-polyfold bundle $W_{x,y,0}$ is constructed in a similar way, and the fiber over $([u_1],\ldots, [u_k])$ isomorphic to $\H_{u_1}\times \ldots \times \H_{u_k}$, once we choose representatives $u_1,\ldots, u_k$. Then the Floer equation gives a proper section $s_{x,y,0}:Z_{x,y,0} \to W_{x,y,0}$ such that
\begin{equation}\label{eqn:floer}
s_{x,y,0}([u_1],\ldots, [u_k]) = (\partial_s u_1 + J\partial_t u_1 - \nabla H(u_1), \ldots, \partial_s u_k + J\partial_t u_k - \nabla H(u_k) ).
\end{equation}
It was proven in \cite{wehrheim2012fredholm} that this section is sc-Fredholm.

In the following, we prove Proposition \ref{prop:equi}.

\begin{proof}[Proof of Proposition \ref{prop:equi} \eqref{e1}]
 The reparametrization in the $S^1$-coordinate acts sc-smoothly on $p_{x,y,0}:W_{x,y,0} \to Z_{x,y,0}$ by
 $$\theta\cdot (\overline{u}(\cdot,\cdot),\overline{v}(\cdot,\cdot)) := (\overline{u}(\cdot,\theta+\cdot), \overline{v}(\cdot,\theta+\cdot)),$$
 where $\overline{u}= ([u_1],\ldots, [u_k])$ and $\overline{v}=(v_1,\ldots,v_k)$ for $v_i\in \H_{u_i}$. Then the section $s_{x,y,0}$ is $S^1$-equivariant. The fixed points of the $S^1$-action on $W_{x,y,0} \to Z_{x,y,0}$ are curves independent of the $S^1$-coordinate, i.e. the Morse homology M-polyfold bundle $\cY_{x,y} \to \cX_{x,y}$. Thus to verify Definition \ref{def:tube}, we need to prove the following on $W_{x,y,0}\to Z_{x,y,0}$.
 \begin{enumerate}
 		\item There exists a $S^1$-invariant neighborhood $N_{x,y,0}\subset Z_{x,y,0}$ of $(s^{S^1}_{x,y,0})^{-1}(0):=s_{x,y,0}^{-1}(0)\cap Z^{S^1}$, such that there is a projection $\pi: N_{x,y,0} \to N_{x,y,0}\cap Z^{S^1}_{x,y,0}$ making $N_{x,y,0}$ a $S^1$-tame M-polyfold bundle. We have a decomposition of $S^1$-tame M-polyfold bundles  $N_{x,y,0} = \widehat{\oplus}_{\lambda \in \N^+} N_{x,y,0}^\lambda$, such that $N_{x,y,0}^\lambda$ is a $\lambda$-tame M-polyfold bundle. 
 		\item The $S^1$-tame strong M-polyfold bundle $W_{x,y,0}|_{N_{x,y}\cap Z^{S^1}_{x,y,0}}$ has a similar decomposition, i.e. we have $\lambda$-tame strong M-polyfold bundles $W^\lambda_{x,y,0}$ over $Z^{S^1}_{x,y,0}$ for $\lambda \in \N^+$, such that $W[i]_{x,y,0}|_{N_{x,y,0}\cap Z^{S^1}_{x,y,0}} = W^{S^1}[i]_{x,y,0} \widehat{\oplus}_{\lambda \in \N^+} W^\lambda [i]_{x,y,0}$ for $i=0,1$. Moreover, there is a $G$-equivariant strong bundle isomorphism $W_{x,y,0}|_{N_{x,y}} \to \pi^*(W_{x,y,0}|_{N_{x,y,0}\cap Z^{S^1}_{z,y,0}})$. 
 		\item We have $(\rank W^\lambda_{x,y,0})_z=\infty$ for any $z\in N_{x,y,0}\cap Z^{S^1}_{x,y,0}$ unless $W^\lambda_{x,y,0}$ is the rank zero bundle.
 \end{enumerate}
To construct a tubular neighborhood, we construct the normal bundle of $\cX_{x,y}\subset Z_{x,y,0}$. There is a M-polyfold bundle $\cN_{x,y} \to \cX_{x,y}$ where the fiber over $([\gamma_1],\ldots, [\gamma_k])$ is $\F_{\gamma_1}\times \ldots \times \F_{\gamma_k}$ for chosen representatives $\gamma_1,\ldots,\gamma_k$, where 
$$\F_{\gamma}:= \{ \xi \in H^{3+i}_{\delta_i}(\R \times S^1,\gamma^*TM) | \textstyle{\int}_{S^1}\xi(\cdot,t) dt = 0\}_{i\in \N}.$$
Then there is a sc-smooth map $\mathrm{EXP}: \cN_{x,y} \to Z_{x,y,0}$ given by
$$(([\gamma_1],\ldots, [\gamma_k]), (\xi_1,\ldots, \xi_k))\mapsto ([\exp_{\gamma_1} \xi_1], \ldots, [\exp_{\gamma_k} \xi_k]),$$
where $(\exp_{\gamma_i}\xi_i)(s,t) = \exp_{\gamma_i(s)}\xi_i(s,t)$. To see that $\mathrm{EXP}$ is a local sc-diffeomorphism near $\cX_{x,y}$, for example near a unbroken flow line $[\gamma] \in \cX_{x,y} \subset Z_{x,y,0}$,  instead of using \eqref{eqn:Eu} as the local model, we can use 
$$\E'_\gamma:= \{v\in H^{3+i}_{\delta_i}(\R\times S^1,u^*TM)| \textstyle \int_{S^1}v(0,t) dt\in H \}_{i\in \N}$$
as the local model. Then we have $\E'_\gamma = \E_\gamma \oplus \F_\gamma$ and $\mathrm{EXP}$ is the local sc-diffeomorphism near $\gamma$ in $\cN_{x,y}$. The situation near a broken flow line is similar. Therefore there is neighborhood $V$ of the compact set $(s^{S^1}_{x,y,0})^{-1}(0) \subset \cN_{x,y}$, such that $\mathrm{EXP}|_V$ is a sc-diffeomorphism. To describe $\mathrm{EXP}(V)$ as a tubular neighborhood, we first construct a sc-smooth metric $g:\cN_{x,y}\otimes \cN_{x,y} \to \R$ by partition of unity, see the proof of Proposition \ref{prop:complement}. The rotation in the $S^1$-direction makes $\cN_{x,y}$ into a $S^1$-M-polyfold bundle. By Corollary \ref{coro:ave}, the metric $g$ can be chosen to be $S^1$-invariant. Then there is a $\epsilon > 0$ and an open neighborhood $U$ of $(s^{S^1}_{x,y,0})^{-1}(0)$ in $\cX_{x,y}$, such that the $\epsilon$-ball bundle $B_{\epsilon}(U)\subset \cN_{x,y}$ is contained in $V$. Then $\mathrm{EXP}_\epsilon :\cN_{x,y}|_U \to Z_{x,y,0}$ defined by 
\begin{equation}\label{eqn:normalmap}
(([\gamma_1],\ldots, [\gamma_k]), (\xi_1,\ldots, \xi_k))\mapsto ([\exp_{\gamma_1} (f_\epsilon \circ g(\xi_1,\xi_1) \cdot \xi_1)], \ldots, [\exp_{\gamma_n} (f_\epsilon \circ g(\xi_k,\xi_k)\cdot \xi_k)]),
\end{equation}
is a sc-diffeomorphism onto its image, where $f_\epsilon:\R_+ \to [0,1]$ is smooth map such that $r\mapsto f_\epsilon(r^2)r$ is diffeomorphism from $\R_+$ to $[0,\epsilon)$ and $f_\epsilon$ is identity near $0$. Since the metric $g$ is $S^1$-invariant, $\mathrm{EXP}_\epsilon$ is $S^1$-invariant. Hence there is a tubular neighborhood $\mathrm{EXP}_\epsilon(B_\epsilon(U))$ containing $(s^{S^1}_{x,y,0})^{-1}(0)$, which is sc-diffeomorphic to $\cN_{x,y}|_U$.

Next we show that we have a decomposition of $\cN_{x,y}$ into different representations using Fourier series. Note that the nontrivial representation of $S^1$ is classified by a positive integer weight. For every weight $\lambda  \in \N_+$, similar to the construction of $\cN_{x,y}$, we have a M-polyfold subbundle $\cN_{x,y}^{\lambda}\subset \cN_{x,y}$ such that the fiber over 
$([\gamma_1],\ldots, [\gamma_k])$ is given by $\F_{\gamma_1}^{\lambda}\times \ldots \times \F_{\gamma_k}^{\lambda}$, where
$$\F_{\gamma}^{\lambda}:= \left\{a(s)\sin(2\lambda \pi t) + b(s)\cos(2\lambda \pi t)\st a,b \in H^{3+i}_{\delta_i}(\R,\gamma^*TM)\right\}_{\i\in \N} \subset \F_\gamma.$$
Then $\cN_{x,y}^{\lambda}$ is a $\lambda$-M-polyfold bundle in the sense of Definition \ref{def:lambundle}. In addition, $\oplus_{\lambda > 0} \cN_{x,y}^{\lambda} \hookrightarrow \cN_{x,y}$ is a dense inclusion by the properties of Fourier series. Moreover, $W_{z,y,0}|_U$ has a similar decomposition, and each component of the decomposition has infinite rank. 

Finally, we verify that $W_{z,y,0}|_{B_\epsilon(U)}$ is equivariantly isomorphic to $\pi^*W_{z,y,0}|_U$ as strong bundles, where $\pi$ is the projection $\cN_{x,y} \to U$. Let $[\overline{\gamma}] = ([\gamma_1],\ldots,[\gamma_k]) \in U$, $\overline{\xi}=(\xi_1,\ldots, \xi_k) \in (\cN_{x,y})_{([\gamma_1],\ldots,[\gamma_k])}$ and $\overline{\eta}=(\eta_1,\ldots,\eta_k) \in (W_{x,y,0})_{([\gamma_1],\ldots,[\gamma_k])}$. The equivariant strong bundle isomorphism $\Phi: \pi^*W_{z,y,0}|_U \to  W_{x,y,0}|_{B_\epsilon(U)}$ is given by
\begin{equation}
\Phi:([\overline{\gamma}],\overline{\xi},\overline{\eta})\mapsto(\mathrm{EXP}_\epsilon([\overline{\gamma}],\overline{\xi}), (P_{\xi_1}\eta_1, \ldots, P_{\xi_n}\eta_n)),
\end{equation}
where $(P_{\xi_i}\eta_i)(s,t)$ is the parallel transportation of $\eta_i(s,t)$ along the path $[0,1]\rightarrow M, v\mapsto \exp_{\gamma_i(s)} (v \cdot f_\epsilon \circ g(\xi_i,\xi_i) \cdot \xi_i(s,t))$. Then $\Phi$ is $S^1$-equivariant strong bundle isomorphism. This finishes the proof of Proposition \ref{prop:equi} \eqref{e1}.
\end{proof}

\begin{proof}[Proof of Proposition \ref{prop:equi} \eqref{e2}]
The sc-Fredholm section $s_{x,y,0}$ is given by \eqref{eqn:floer}. The linearization of $s_{x,y,0}$ at an unbroken flow line $\gamma \in (s^{S^1}_{x,y,0})^{-1}(0)$ is given by
$$\rD (s_{x,y,0})_\gamma \xi=\nabla_s \xi +J\nabla_t \xi - \nabla_\xi \nabla H,$$
where $\nabla$ is the Levi-Civita connection the for metric $g=\omega(\cdot, J\cdot)$. After a complex trivialization $\Psi:\C^n\to \gamma^*{TM}$, the section $\xi$ can be viewed as a map from $\R\times S^1$ to $\C^n$ and 
$$\rD (s_{x,y,0})_\gamma \xi=\partial_s \xi+i\partial_t \xi+A\xi,$$ 
where $A:= \Psi^{-1}(\nabla_s \Psi -\nabla_\Psi \nabla H) \in C^\infty(\R,\Hom_\R(\C^n,\C^n))$, see \cite[\S 2.2]{salamon1997lectures}. Since $\gamma'(s)$ decays exponentially and $H$ is $C^2$-small, $A(\pm \infty):=\lim_{s\to \pm \infty}A(s)=-\lim_{s\to \pm \infty}\Psi^{-1}\nabla_{\Psi}\nabla H$ is $C^0$ small. To prove Proposition \ref{prop:equi} \eqref{e2}, we need to analyze $\rD^\lambda (s_{x,y,0})_\gamma$. Let $\xi \in \cN^\lambda_{x,y}$, under the trivialization $\Psi$, we can write
$$\xi(s,t) = a_\lambda(s) \sin 2\lambda \pi t + b_\lambda(s) \cos 2\lambda \pi t,$$
for $a_\lambda,b_\lambda \in \H^{3}_\delta(\R,\C^n):=\{H^{3+i}_{\delta_i}(\R,\C^n)\}_{i\in\N}$. Therefore we have $D^\lambda (s_{x,y,0})_\gamma$ given by
$$ 
 \xi \mapsto (a_\lambda'(s)-i2\lambda \pi b_\lambda(s) +A(s)a_\lambda(s))\sin 2 \lambda \pi t+( b_\lambda'(s)+i2\lambda \pi a_\lambda(s) +A(s)b_\lambda(s))\cos 2\lambda \pi t.
$$
Using $\sin 2\lambda \pi t$ and $\cos 2\lambda \pi t$ as basis, the linearization above can be rewritten as 
\begin{eqnarray} 
D^\lambda (s_{x,y,0})_\gamma (a_\lambda(s),b_\lambda(s)) & = &(a_\lambda'(s)-i2\lambda \pi b_\lambda(s) +A(s)a_\lambda(s), b_\lambda'(s)+i2\lambda \pi a_\lambda(s) +A(s)b_\lambda(s)) \nonumber \\
& = & (\textstyle\frac{\rd}{\rd s} - B(s))(a_\lambda(s),b_\lambda(s)). \label{eqn:lambda} 
\end{eqnarray}
Here $B(s) \in \Hom_\R(\C^n\times \C^n,\C^n\times \C^n)$ is defined by $B(s)(a,b) = (i2\lambda b +A(s)a, -2\lambda i a + A(s)b)$ for $a,b\in \C^n$. Let $B^{\pm} := \lim_{s\to \pm\infty} B(s)$. The index of \eqref{eqn:lambda} can be computed by spectral flow and is given by $\dim E^u(B^-) - \dim E^u(B^+)$ by \cite[Theorem 2.1]{robbin1995spectral}, where $E^u(B)$ is the negative eigenspace of the matrix $B$. Since $A(\pm \infty)$ is small in $C^0$ norm, we have $\dim E^u(B^-) = \dim E^u(B^+) = 2n$. Therefore $\ind D^\lambda s_{x,y,0} = 0$. This verifies Proposition \ref{prop:equi} \eqref{e1} for unbroken flow lines. The proof for broken flow lines is similar, since the $\lambda$-direction linearization at a broken flow line is the product of the $\lambda$-direction linearization on every component, i.e. the gluing parameter does not interact with any of the $\lambda$-directions.
\end{proof}

\appendix
\section{Properties of Local Lifts}\label{A1}
First, we set up the following properties of the local action $L_\phi$, which follows directly from Definition \ref{def:lphi}.
\begin{proposition}\label{prop:fact}
	 Let $F:(\cX,\bX)  \to (\cY,\bY)$ be a fully faithful sc-smooth functor, such that $F^0:\cX \to \bY$ is a local sc-diffeomorphism. Then the following holds.
	\begin{enumerate}
		\item For $\phi\in \bX$, then we have the following equation as germs around $s(\phi)$ (Definition \ref{def:germ}).
		\begin{equation}\label{commute}
		[L_{F(\phi)}\circ F]_{s(\phi)}=[F\circ L_\phi]_{s(\phi)}.\end{equation} 
		\item For composable $\phi,\psi\in \bX$, we have
		\begin{equation}\label{assoc}
		[L_\phi \circ L_\psi]_{s(\psi)}=[L_{\phi\circ \psi}]_{s(\psi)}.
		\end{equation} 
		\item If $L_\phi$ is defined over $\cV$, then there exists a neighborhood $\bV$ of $\phi\in \bX$ such that for every $\psi\in \bV$, $L_\psi$ is defined on $\cV$ and
		\begin{equation}\label{nearby}
		L_\phi=L_\psi \text{ on } \cV.
		\end{equation}
	\end{enumerate}	
\end{proposition}

We first prove a direct corollary of regularity (Definition \ref{def:reg}). 
\begin{lemma}\label{lemma:localnatural}
	Let $(\cX,\bX)$ be a regular ep-groupoid. Let $x,y$ be two points in $\cX$ and $\cU_x,\cU_y$ two local uniformizers around $x$ resp. $y$. If there is a sc-smooth map $F:\cU_x\to \cU_y$ such that 
	\begin{enumerate}
		\item $F$ is a sc-diffeomorphism from $\cU_x$ to an open subset of $\cX$;
		\item $|F(z)|=|z|$ for every $z\in \cU_x$ and $F(x) = y$.
    \end{enumerate}	
	Then there exists a $\theta \in \mor(x,y)$ and a local uniformizer $\cV$ around $x$, such that $F|_{\cV}= L_\theta$.
\end{lemma}	
\begin{proof}
Since $|y|=|F(x)|=|x|$,  there exists $\theta_0\in \mor(x,y)$.  Then there is an open neighborhood $\cV$ of $x$ such that $L_{\theta_0}$ is defined on $\cV$ with images in $\cU_y$, and $\cU:=F(\cV)\subset \cU_y$ is a connected regular uniformizer around $y$. For every $z\in \cV$, we have $|L_{\theta_0}(z)|=|z|=|F(z)|$. Since $L_{\theta_0}(z)$ and $F(z)$ are in the local uniformizer $\cU_y$, then for any $z\in \cV$, there exists a $\phi_z\in \stab_{y}$, such that
\begin{equation}\label{eqn:eqn}
	L_{\theta_0}(z)=L_{\phi_z}\circ F(z).
\end{equation}
Then $L_{\phi_{F^{-1}(u)}}(u)=L_{\theta_0}\circ F^{-1}(u)$ for $u\in \cU$. Since $\cU$ is a connected regular uniformizer, we have
	$$L_{\phi_{F^{-1}(u)}}(u)=L_\phi(u)$$
	for some $\phi\in \stab_{y}$. Then \eqref{eqn:eqn} induces:
	$$L_{\phi^{-1}\circ \theta_0}(z)=F(z), \quad \forall  z\in \cV.$$
\end{proof}	

A typical scenario of applying this lemma is when we have two equivalences of ep-groupoids $F,G:(\cW,\bW)\to (\cX,\bX)$ such that $|F| = |G|$. Then for $w_1,w_2\in \cW$ and $\phi \in \mor(w_1,w_2)$, then there exists $\theta \in \mor(F(w_1),G(w_2))$ such that locally:
$$L_\theta=G_2\circ L_\phi \circ F^{-1}_{w_1}.$$

In the bundle case, we have the following  analogue of Lemma \ref{lemma:localnatural} with an identical proof.
\begin{lemma}\label{lemma:localnaturalbundle}
		Let $(\cE,\bE)\stackrel{P}{\to}(\cX,\bX)$ be a regular strong ep-groupoid bundle. Let $x,y$ be two points in $\cX$ and $\cU_x,\cU_y$ two local uniformizers around $x$ resp. $y$. If there is a sc-smooth strong bundle map $F:P^{-1}(\cU_x)\to P^{-1}(\cU_y)$ such that 
		\begin{enumerate}
			\item $F$ is a sc-diffeomorphism from $P^{-1}(\cU_x)$ to an open subset of $\cE$;
			\item $|F(v)|=|v|$ for every $v\in P^{-1}(\cU_x)$ and $F(x,0) = (y,0)$.
		\end{enumerate}	
		Then there exists a $\theta \in \mor(x,y)$ and a local uniformizer $\cV$ around $x$, such that $F|_{P^{-1}(\cV)}= R_\theta$.
\end{lemma}	

To prove Proposition \ref{prop:lifting}, we need to argue that the local representation $\rho_{g,x}$ of a group action in Definition \ref{def:action} can be chosen as a family of embeddings and the inverses of $\rho_{g,x}(h,\cdot)$ can be defined on a uniform open set for all $h$ close to $g$, i.e. we have the following.
	
\begin{proposition}\label{prop:uniform}
	Let $\rho_{g,x}: U\times (\stab_{x}\ltimes \cU) \to (\cX,\bX)$ be a local representation of a sc-smooth group action $(\rho, \mathfrak{P})$ near $(g,x)$ on a connected regular uniformizer $\stab_{x}\ltimes \cU$.  Then we have
	\begin{enumerate}
		\item for each $h\in U$, $\rho_{g,x}(h,\cdot):\stab_{x}\ltimes \cU \to (\cX,\bX)$ is a fully faithful functor and on object space and $\rho_{g,x}(h,\cdot)$ is a sc-diffeomorphism onto an open subset;
		\item\label{p2:uniform} there exists a neighborhood $U'\subset U$ of $g$, such that $\cap_{h\in U'}\rho_{g,x}(h, \cU)$ contains an open set $\cV \subset \cX$, which contains $\{\rho_{g,x}(h,x)|h\in U'\}$.
	\end{enumerate}
	As a consequence of \eqref{p2:uniform}, $\rho_{g,x}(h,\cdot)^{-1}$ is defined on $\cV$ for all $h\in U'$ and contains $x$ in its image. 
\end{proposition}
\begin{proof}
	For the first assertion, for each $h\in G$, $\mathfrak{P}(h):Z\to Z$ is an isomorphism of polyfolds.  As a local representation of such map, $\rho_{g,x}(h,\cdot): \stab_{x} \ltimes \cU\to (\cX,\bX)$ is fully faithful and local sc-diffeomorphism on objects. 
	To show it is a sc-diffeomorphism on objects, we only need to show it is injective on objects. Assume otherwise, i.e. there exist $h\in U, y,z\in \cU$, such that $\rho_{g,x}(h,y) = \rho_{g,x}(h,z)$.  Since $\rho_{g,x}(h,\cdot)$ is local diffeomorphism, we have $\rho_{g,x}(h,\cdot)_{z}^{-1} \circ \rho_{g,x}(h,\cdot)$ is well defined in a neighborhood of $y$, where $\rho_{g,x}(h,\cdot)_{z}^{-1}$ is the inverse of $\rho_{g,x}(h,\cdot)$ near $z$. By Lemma \ref{lemma:localnatural}, $[\rho_{g,x}(h,\cdot)_{z}^{-1} \circ \rho_{g,x}(h,\cdot)]_y = [L_\psi]_y$  for some $\psi \in \mor(y,z)$. Since $\cU$ is a local uniformizer, $L_\psi = L_\phi$ for some $\phi \in \stab_x$ by Corollary \ref{coro:natural}. Such $\phi$ is unique and nontrivial in $\stab^{\eff}_x$, because if we have another $\phi'$ such that $[L_\phi]_y = [L_{\phi'}]_y$,  then $L_{\phi^{-1} \circ \phi'}$ is identity near $y$. By the regular property (Definition \ref{def:regular}), $L_{\phi^{-1} \circ \phi'} = \Id$ on $\cU$, i.e. $\phi =\phi'\in \stab^{\eff}_x$.  Next we consider the following set
	$$S: = \{p\in\cU|   [\rho_{g,x}(h,\cdot)]_p = [\rho_{g,x}(h,\cdot) \circ L_\phi]_p\}\subset \cU.$$
	$S$ is open by definition and $S$ is not empty by our assumption. We claim $S$ is also closed. Suppose we have $\lim_{i}x_i = x_\infty$ for $x_i \in S$ and $x_\infty \in \cU$. Since $\rho_{g,x}(h,x_i) = \rho_{g,x}(h,L_\phi(x_i))$, we have $\rho_{g,x}(h,x_\infty) = \rho_{g,x}(h,L_\phi(x_\infty))$. Therefore by the same argument before, there exists a unique $\eta \in \stab^{\eff}_x$, such that $[\rho_{g,x}(h,\cdot)_{L_{\phi}(x_\infty)}^{-1} \circ \rho_{g,x}(h,\cdot)]_{x_\infty} = [L_\eta]_{x_\infty}$. Since $\rho_{g,x}(h,\cdot)_{L_{\phi}(x_\infty)}^{-1}  = \rho_{g,x}(h,\cdot)|_{L_{\phi}(x_i)}^{-1} $ for $i\gg 0$, we have $[\rho_{g,x}(h,\cdot)_{L_{\phi}(x_i)}^{-1} \circ \rho_{g,x}(h,\cdot)]_{x_i} = [L_\eta]_{x_i}$ for $i\gg 0$. By assumption $[\rho_{g,x}(h,\cdot)_{L_{\phi}(x_i)}^{-1} \circ \rho_{g,x}(h,\cdot)]_{x_i} = [L_\phi]_{x_i}$. We must have $\eta = \phi$ in $\stab^{\eff}_x$. Therefore $S$ is closed. Since $\cU$ is connected and $S$ is not empty by assumption, we have $S = \cU$. Since $L_\phi \ne \Id$ on any open subset $\cU$ by regularity, there exists $y_i \to x$ such that $L_\phi(y_i) \ne y_i$. This contradicts $\rho_{g,x}(h,\cdot)$ is a local diffeomorphism near $x$.
	
	To prove the second assertion, we first pick a local uniformizer $\cW$ around $y := \rho_{g,x}(g,x)$. Then there exist an open neighborhood $U'\subset U$ of $g$ and a uniformizer $\cU'\subset \cU$ around $x$ such that $\rho_{g,x}(U',\cU')\subset \cW$. Since a polyfold is a metrizable space,  by the proof of Lemma \ref{lemma:quometric}, we may choose the metric $d$ to be $G$-invariant. Because open set $|\cU'|\subset Z$ contains a $r$-ball around $|x|$. Then we can choose the neighborhood $U'\subset U$ of $g$ small enough, such that the following two conditions hold:
	\begin{itemize}
		\item $U'$ is connected;
		\item $\cup_{h\in U'}\rho(h,|x|)$ has diameter smaller than $\frac{r}{2}$.
	\end{itemize} 
    Since $\rho(h, B_r(|x|))$ is a $r$-ball centered at $\rho(h,|x|)$, we have $B_{\frac{r}{2}}(|y|) \subset \bigcap_{h \in U'} \rho(h, B_r(|x|))$.  Let $\cV \subset \cW$ be the $\stab_y$-invariant neighborhood of $y$ such that $|\cV| = B_{\frac{r}{2}}(|y|)$. We claim that $\cV\subset \rho_{g,x}(h,\cU')$ for every $h\in U'$. Since on object level we have $|\cV| \subset \rho(h,|\cU'|)$ and $\rho_{g,x}(h,\cU')\subset \cW$, for every $z\in \cV$ there exists a $\phi \in \stab_y$, such that $L_\phi(z)=\rho_{g,x}(h,w)$ for some $w\in \cU'$. Since $\rho_{g,x}(h,\cdot)$ is a fully faithful functor, we have $\stab_w \simeq \stab_{L_\phi(z)}$. Note that the $\stab_x$-orbit of $w$ in $\cU'$ has size $|\stab_x|/|\stab_w|$ and the $\stab_y$-orbit of $L_\phi(z)$ in $\cV$ has size $|\stab_y|/|\stab_{L_\phi(z)}| = |\stab_x|/|\stab_w|$. Then $\rho_{g,x}(h,\cdot)$ identifies the $\stab_x$-orbit of $w$ and the $\stab_y$-orbit of $L_\phi(z)$ because $\rho_{g,x}(h,\cdot)$ is injective by the first assertion. We must have $z \in \rho_{g,x}(h,\cU')$, since $z$ is in the $\stab_y$-orbit of $L_{\phi}(z)$. Therefore $\cV\subset \rho_{g,x}(h,\cU')\subset \rho_{g,x}(h,\cU)$ for every $h\in U'$.
\end{proof}  

\begin{remark}\label{rmk:bundleemb}
	By a similar argument, when $\rho_{g,x}:U\times (\stab_x\ltimes P^{-1}(\cU)) \to (\cE,\bE)$ is a local representation of a sc-smooth action on a strong polyfold bundle. If $\cU$ is connected and regular, then $\rho_{g,x}(h,\cdot)$ is a sc-diffeomorphism from $P^{-1}(\cU)$ onto an open image for every $h\in U$. 
\end{remark}	
	
\begin{proposition}[Proposition \ref{prop:lifting}]\label{prop:Alifting}
	Let $Z$ be a regular polyfold with a sc-smooth $G$-action $(\rho, \mathfrak{P})$ and $(\cX,\bX)$ a polyfold structure of $Z=|\cX|$. Then for $x,y\in \cX$ and $g\in G$ with $\rho(g,|x|) = |y|$, the isotropy $\stab_x$ acts on $L_{\cX}(g,x,y)$ by pre-composing the local sc-diffeomorphism $L_\phi$ for $\phi \in \stab_x$, i.e. 
	$$(\phi, [\Gamma]_{(g,x)}) \mapsto [\Gamma\circ L_\phi]_{(g,x)},$$
	and this action is transitive. The isotropy $\stab_y$ also acts on $L_{\cX}(g,x,y)$ by post-composing $L_\phi$ for $\phi\in \stab_y$, i.e. $$(\phi, [\Gamma]_{(g,x)}) \mapsto [L_\phi\circ \Gamma]_{(g,x)},$$
	this action is also transitive.
\end{proposition}
\begin{proof}
	We first recall the definition of local lifts from Subsection \ref{subsec:lift}. For $x,y\in \cX$  and $g\in G$ such that $\rho(g,|x|)=|y|$, we have two equivalent polyfold structures $(\cX_a,\bX_a)$,$(\cX_b,\bX_b)$ and a local uniformizer $\cU_a\subset \cX_a$ around $x_a$, such that $|x_a|=x$ and the action can be locally represented by 
	$$\rho_{g,x_a}:U\times (\stab_{x_a}\ltimes \cU_a) \to (\cX_b,\bX_b),$$
	where $U\subset G$ is a neighborhood of $g$. Since $(\cX_a,\bX_a),(\cX_b,\bX_b)$ and $(\cX,\bX)$ are equivalent polyfold structures, we have the following diagram of equivalences as in \eqref{gammadia} (we suppress the morphism space from now on):
	\begin{equation*}
	\xymatrix{
	\cX & \ar[l]_{F_a} \cW_a\ar[r]^{G_a}& \cX_a\ar[r]^{\rho_{g,x_a}(\cdot, \cdot)} & \cX_b & \ar[l]_{F_b}  \cW_b\ar[r]^{G_b} & \cX} 
    \end{equation*}
	A local lift $\Gamma$ at $(g,x,y)$ is constructed by first finding  $w_a\in \cW_a,w_b\in \cW_b,\phi,\psi\in \bX, \delta \in \bX_a, \eta\in \bX_b$, such that $F_a(w_a)=\phi(x)$,  $\delta(G_a(w_a)) = x_a$, $\eta(\rho_{g,x_a}(g,x_a))=F_b(w_b)$ and $G_b(w_b)=\psi(y)$. Let $F_{a,w_a}$ resp. $F_{b,w_b}$ denote the local diffeomorphisms $F_a$ resp. $F_b$ near $w_a$ resp. $w_b$. We define for $z$ close to $x$ and $h$ close to $g$,
	\begin{equation}\label{eqn:def}
	\Gamma(h,z):= L_{\psi^{-1}}\circ G_b\circ F_{b,w_b}^{-1}\circ L_\eta\circ \rho_{g,x_a}(h,\cdot)\circ L_{\delta}\circ G_a\circ F_{a,w_a}^{-1}\circ L_\phi(z).
	\end{equation}
	
	\textit{Case one - different choices of $w_a,w_b,\phi,\psi,\delta,\eta$: }  Assume we have two sets of different choices inducing two local lifts $\Gamma,\Gamma'$.
	\begin{equation}\xymatrix{
		\cX  & \ar[l]_{F_a}\cW_a\ar[r]^{G_a}&  \cX_a\ar[r]^{\rho_{g,x_a}(g,\cdot)} & \cX_b & \ar[l]_{F_b}\cW_b\ar[r]^{G_b} & \cX	\\	
		x\ar[d]_{\phi} & & x_a \ar[r]& \rho_{g,x_a}(g,x_a)\ar[d]^\eta & & y\ar[d]^{\psi}\\
		F_a(w_a) & \ar[l] w_a\ar[dd]_{\theta_a} \ar[r] & G_a(w_a) \ar[u]^\delta & F_b(w_b)  & w_b\ar[l]\ar[r] \ar[dd]^{\theta_b} & G_b(w_b)\\
		x\ar[d]_{\phi'} & & x_a \ar[r]& \rho_{g,x_a}(g,x_a)\ar[d]^{\eta'} & & y\ar[d]^{\psi'}\\
		F_a(w_a') & \ar[l] w_a'\ar[r] & G_a(w_a') \ar[u]^{\delta'} & F_b(w_b')  & w_b'\ar[l]\ar[r] & G_b(w_b')\\
	}  \nonumber \end{equation}

Since $F_a$ is an equivalence, $F_a$ induces a homeomorphism on the orbit spaces, so we can find $\theta_a\in \bW_a$, such that $\theta_a(w_a)=w_a'$. Moreover $F_b:\mor(w_b,w_b')\to\mor(F_b(w_b),F_b(w'_b))$ is a bijection, thus there exists $\theta_b\in \mor(w_b,w_b')$, such that 
\begin{equation}\label{theta2}
F_b(\theta_b)= \eta'\circ \rho_{g,x_a}(g,\delta'\circ G_a(\theta_a)\circ \delta^{-1})\circ \eta^{-1}.
\end{equation}
Here $\delta'\circ G_a(\theta_a)\circ \delta^{-1} \in \mor(x_a, x_a)$ and $\rho_{g,x_a}(g,\delta'\circ G_a(\theta_a)\circ \delta^{-1})\in \mor(\rho_{g,x_a}(g,x_a),\rho_{g,x}(g,x_a))$ is the functor $\stab_{x_a}\ltimes \cU_b \to (\cX_b,\bX_b)$. The following computation is made for $z$ close for $x$ and $h$ close to $g$, or equivalently one can treat them as germs. In the following, the underlined part is where we apply changes and the bold part is the resulted formula after change. 
\begin{eqnarray}
 \Gamma(h,z) & \stackrel{\eqref{eqn:def}}{=} & L_{\psi^{-1}}\circ G_b\circ F_{b,w_b}^{-1}\circ L_{\eta}\circ \rho_{g,x_a}(h,\cdot)\circ L_{\delta} \circ G_a\circ \underline{F_{a,w_a}^{-1}\circ L_{\phi}(z)} \nonumber \\
 &\stackrel{\eqref{commute}+\eqref{assoc}}{=}& L_{\psi^{-1}}\circ G_b\circ F_{b,w_b}^{-1}\circ L_{\eta}\circ \rho_{g,x_a}(h,\cdot) \underline{\circ L_{\delta} \circ G_a \circ \bm{L_{\theta^{-1}_a}}}\bm{\circ F_{a,w'_a}^{-1}\circ L_{ F_a(\theta_a)\circ \phi}(z)} \nonumber \\
 &\stackrel{\eqref{commute}+\eqref{assoc}}{=} & L_{\psi^{-1}}\circ G_b\circ F_{b,w_b}^{-1}\circ \underline{L_{\eta}\circ \rho_{g,x_a}(h,\cdot)\circ \bm{L_{\delta \circ G_a(\theta^{-1}_a)\circ \delta'^{-1}}}} \bm{\circ L_{\delta'}\circ G_a}\circ F_{a,w'_a}^{-1}\circ L_{ F_a(\theta_a)\circ \phi}(z) \nonumber \\
&\stackrel{\eqref{commute}}{=} & L_{\psi^{-1}}\circ G_b\circ F_{b,w_b}^{-1}\circ \bm{L_{\eta \circ \rho_{g,x_a}(\underline{h},\delta \circ G_a(\theta^{-1}_a)\circ \delta'^{-1})}\circ \rho_{g,x_a}(h,\cdot)}\circ L_{\delta'} \circ G_a\circ F_{a,w'_a}^{-1}\circ L_{F_a(\theta_a)\circ \phi}(z) \nonumber\\
&\stackrel{\eqref{nearby}}{=} & L_{\psi^{-1}}\circ G_b\circ F_{b,w_b}^{-1}\circ  \underline{L_{\eta \circ \rho_{g,x_a}(\bm{g},\delta \circ (G_a(\theta^{-1}_a)\circ \delta'^{-1})}}\circ \rho_{g,x_a}(h,\cdot)\circ L_{\delta'} \circ G_a\circ F_{a,w'_a}^{-1}\circ L_{F_a(\theta_a)\circ \phi}(z) \nonumber\\
&\stackrel{\eqref{theta2}}{=}& L_{\psi^{-1}}\circ G_b\circ F_{b,w_b}^{-1}\circ \underline{\bm{L_{F_b(\theta_b^{-1})\circ \eta'}}}\circ \rho_{g,x_a}(h,\cdot)\circ  L_{\delta'}\circ G_a\circ F_{a,w'_a}^{-1}\circ L_{F_a(\theta_a)\circ \phi}(z) \nonumber\\
& \stackrel{\eqref{assoc}}{=} & L_{\psi^{-1}}\circ G_b\circ \underline{F_{b,w_b}^{-1}\circ \bm{L_{F_b(\theta_b^{-1})}}}\bm{\circ L_{\eta'}}\circ \rho_{g,x_a}(h,\cdot)\circ  L_{\delta'} \circ G_a\circ F_{a,w'_a}^{-1}\circ L_{ F_a(\theta_a)\circ \phi}(z) \nonumber\\
&\stackrel{\eqref{commute}+\eqref{assoc}}{=}&\underline{L_{\psi^{-1}}\circ G_b\circ \bm{L_{\theta_b^{-1}}}}\bm{\circ F_{b,w'_b}^{-1}}\circ L_{\eta'}\circ \rho_{g,x_a}(h,\cdot)\circ  L_{\delta'} \circ G_a\circ F_{a,w'_a}^{-1}\circ L_{ F_a(\theta_a)}(z)\circ \phi \nonumber\\
&\stackrel{\eqref{commute}+\eqref{assoc}}{=} &\bm{L_{\psi^{-1}\circ G_b(\theta_b^{-1})}\circ G_b}\circ F_{b,w'_b}^{-1}\circ L_{\eta'}\circ \rho_{g,x_a}(h,\cdot)\circ L_{\delta'} \circ  G_a\circ F_{a,w'_a}^{-1}\circ L_{ F_a(\theta_a)\circ \phi}(z) \nonumber\\
&= &L_{\psi^{-1}\circ G_b(\theta_b^{-1})\circ \psi'}\circ \Gamma'(h,\cdot)\circ L_{\phi'^{-1}\circ F_a(\theta_a)\circ \phi}(z). \nonumber  
\end{eqnarray}	
So far, we prove that:
$$[\Gamma]_{(g,x)}=[L_{\beta}\circ \Gamma'\circ L_{\alpha}]_{(g,x)}$$
for $\alpha:=\phi'^{-1}\circ F_a(\theta_a)\circ \phi \in \stab_x, \beta:=\psi^{-1}\circ G_b(\theta_b^{-1})\circ \psi'\in \stab_y$. By the similar arguments as above, one can move $L_{\beta}$ inside to get
$$[\Gamma]_{(g,x)}=[\Gamma'\circ L_{\beta'}\circ  L_{\alpha}]_{(g,x)}$$
for $$\beta':=\phi^{-1}\circ F_a(G_a^{-1}(\delta^{-1}\circ\rho_{g,x_a}(g,\eta^{-1}\circ F_b(G_b^{-1}(\psi\circ \beta\circ \psi^{-1}))\circ \eta) \circ \delta))\circ \phi \in \stab_x.$$
Therefore different choices of $w_a,w_b,\phi,\psi,\delta,\eta$ give different germs of local lifts up to pre-composing the $\stab_x$-action. Post-composing with $\stab_y$-action can be rewritten as pre-composing the $\stab_x$ action. The transitivity for the pre-composition action is equivalent to the transitivity for the post-composition action.

\textit{Case two-different choices of equivalent polyfold structures $\cW_a$ and $\cW_b$:}  If we have two ep-groupoids $\cW_a, \cW_a'$ along with equivalences $F_a:\cW_a\to \cX$, $G_a:\cW_a\to \cX_a$, $F_a':\cW'_a \to \cX$ and $G'_a:\cW'_a \to \cX_a$. Let $\cW_a\times_{\cX} \cW_a'$ denote the weak fiber product \cite[Definition 10.3]{hofer2017polyfold} of $F_a$ and $F'_a$, then by \cite[Theorem 10.2]{hofer2017polyfold}, we have two equivalences $\pi:\cW_a\times_{\cX}\cW'_a \to \cW_a$ and $\pi':\cW_a\times_{\cX}\cW'_a \to \cW'_a$. Therefore, we have the following (not necessarily commutative) diagram:
$$
\xymatrix{
 & \ar[ld]_{F_a}\cW_a\ar[rd]^{G_a}&  & & &\\
 \cX &  \ar[u]_{\pi} \cW_a\times_{\cX}\cW'_a  \ar[d]^{\pi'} & \cX_a \ar[r]^{\rho_{g,x_a}(g,\cdot)}& \cX_b & \ar[l]_{F_b} \cW_b\ar[r]^{G_b} & \cX.\\
 & \ar[lu]^{F_a'}\cW'_a\ar[ru]_{G'_a}&	& & &
}
$$
We can choose $w\in \cW_a\times_{\cX} \cW'_a$ such that $|w|=|x|$ in $Z$, let $w_a:=\pi(w)\in \cW_a,w'_a:=\pi'(w)\in \cW'_a$ Then we can find $\phi,\phi'\in \bX$, $\delta,\delta'\in \bX_a$, $\eta\in \bX_b$, $w_b\in \cW_b$ and $\psi\in \bX$, such that we have the following diagram:
$$
\xymatrix{
	F_a(w_a) & \ar[l]_{F_a} w_a \ar[r]^{G_a} & G_a(w_a)\ar[d]_{\delta} &  &  & \\
	x\ar[u]^\phi\ar[d]_{\phi'} & w\ar[u]^\pi\ar[d]_{\pi'} & x_a \ar[r]^{\rho_{g,x_a}(g,\cdot) \quad} & \rho_{g,x_a}(g,x_a)  \ar[d]_{\eta} & &  y \ar[d]_{\psi}\\
	F'_a(w'_a) & \ar[l]_{F'_a} w'_a \ar[r]^{G'_a} & G'_a(w'_a)\ar[u]^{\delta'} & F_b(w_b) & \ar[l]_{F_b} w_b \ar[r]^{G_b} & G_b(w_b), \\
}
$$
which define two different local lifts $\Gamma$ and $\Gamma'$. We only need to prove $\Gamma$ and $\Gamma'$ are related by a pre-composition and post-composition of the $\stab_x$-action.

By Lemma \ref{lemma:localnatural}, there exists $\theta_1\in \mor(F_a(w_a),F'_a(w_a'))$, $\theta_2\in \mor(G_a(w_a), G'_a(w'_a))$, such that
\begin{eqnarray}
[L_{\theta_1}]_{F_a(w_a)} & = & [F'_a\circ\pi'\circ \pi^{-1}_{w}\circ F^{-1}_{a,w_a}]_{F_a(w_a)}\label{theta1'} \\
{}[L_{\theta_2}]_{G_a(w_a)} & = & [G'_a\circ \pi'\circ \pi^{-1}_w \circ G^{-1}_{a,w_a}]_{G_a(w_a)}\label{theta2'} 
\end{eqnarray}
 If we write 
\begin{equation}\label{theta3'}
\theta_3:=F_b^{-1}(\eta\circ \rho_{g,x_a}(g, \delta' \circ\theta_2 \circ \delta^{-1} )\circ \eta^{-1}) \in \stab_{w_b}
\end{equation}
Then $h$ close to $g$ and $z$ close to $x$ we have
\begin{eqnarray*}
	\Gamma(h,z) &= & L_{\psi^{-1}}\circ G_b\circ F^{-1}_{b,w_b}\circ L_{\eta} \circ \rho_{g,x_a}(h,\cdot) \circ L_{\delta} \circ G_a\circ \underline{F^{-1}_{a,w_a}\circ L_{\phi}(z)}\\
	&\stackrel{\eqref{theta1'}}{=} & L_{\psi^{-1}}\circ G_b\circ F^{-1}_{b,w_b}\circ L_{\eta} \circ \rho_{g,x_a}(h,\cdot) \circ L_{\delta} \circ \underline{G_a\circ   \bm{\pi \circ   \pi'^{-1}_{w}}}\bm{\circ  F'^{-1}_{a,w'_a}\circ L_{\theta_1\circ \phi}(z)}\\
	&\stackrel{\eqref{theta2'}}{=} & L_{\psi^{-1}}\circ G_b\circ F^{-1}_{b,w_b}\circ L_{\eta} \circ \underline{\rho_{g,x_a}(h,\cdot)\circ L_{\delta}\circ \bm{L_{\theta_2^{-1}}}} \bm{\circ G'_a}\circ F'^{-1}_{a,w'_a}\circ L_{\theta_1\circ \phi}(z)\\
	&\stackrel{\eqref{commute}}{=} &L_{\psi^{-1}}\circ G_b\circ F^{-1}_{b,w_b}\circ L_{\eta} \circ \bm{ L_{\rho_{g,x_a}(\underline{h},\delta\circ \theta_2^{-1} \circ \delta'^{-1})}\circ \rho_{g,x_a}(h,\cdot) \circ L_{\delta'}} \circ G'_a\circ F'^{-1}_{a,w'_a}\circ L_{\theta_1\circ \phi}(z)\\
	&\stackrel{\eqref{nearby}}{=} &L_{\psi^{-1}}\circ G_b\circ F^{-1}_{b,w_b}\circ \underline{L_{\eta}  \circ  L_{\rho_{g,x_a}(\bm{g}, \delta\circ \theta_2^{-1} \circ \delta'^{-1})}} \circ \rho_{g,x_a}(h,\cdot) \circ L_{\delta'} \circ G'_a\circ F'^{-1}_{a,w'_a}\circ L_{\theta_1\circ \phi}(z)\\
	&\stackrel{\eqref{theta3'}}{=}&L_{\psi^{-1}}\circ G_b\circ \underline{F^{-1}_{b,w_b}\circ \bm{L_{F_b(\theta_3^{-1}) \circ \eta}}} \circ \rho_{g,x_a}(h,\cdot) \circ L_{\delta'} \circ G'_a\circ F'^{-1}_{a,w'_a}\circ L_{\theta_1\circ \phi}(z)\\
	&= &\underline{L_{\psi^{-1}}\circ G_b\circ \bm{L_{\theta_3^{-1}}}}\bm{\circ F^{-1}_{b,w_b}\circ L_{ \eta}} \circ \rho_{g,x_a}(h,\cdot) \circ L_{\delta'} \circ G'_a\circ F'^{-1}_{a,w'_a}\circ L_{\theta_1\circ \phi}(z)\\
	&= &\bm{L_{\psi^{-1}\circ G_b(\theta_3^{-1})}\circ G_b}\circ  F^{-1}_{b,w_b}\circ L_{ \eta} \circ \rho_{g,x_a}(h,\cdot)\circ L_{\delta'} \circ G'_a\circ F'^{-1}_{a,w'_a}\circ L_{\theta_1\circ \phi}(z)\\
	&= & L_{\psi^{-1}\circ G_b(\theta_3^{-1})\circ \psi}\circ \Gamma'\circ L_{\phi'^{-1}\circ \theta_1 \circ \phi}
\end{eqnarray*}
We can move $L_{\psi^{-1}\circ G_2(\theta_3^{-1})\circ \psi}$ inside just like before, thus $[\Gamma]_{(g,x)}$ and $[\Gamma']_{(g,x)}$ are in the same $\stab_x$-orbit. Similarly, different polyfold structures for $\cW_b$ also give local lifts in the same orbit.

\textit{Case three-different local representations of the actions:}  Suppose we have polyfold structures $\cX_a,\cX'_a,\cX_b,\cX'_b$ and two connected regular local uniformizers $\cU_a,\cU_a'$ around $x_a\in \cX_a$ and $x'_a \in \cX'_a$ respectively, such that the action is locally represented by both $\rho_{g,x_a}: U\times(\stab_{x_a} \ltimes \cU_a) \to \cX_b$ and $\rho'_{g,x'_a}:U\times (\stab_{x'_a} \ltimes \cU'_a) \to \cX'_b$. A priori, the polyfold structures $\cW_a,\cW'_a$ might be different. However, since they are equivalent, there is a common refinement. Similarly for $\cW_b,\cW'_b$, there is a common refinement. Since we have shown the case of changing $\cW_a,\cW_b$. Therefore we can assume the following diagram of equivalences:
$$
\xymatrix{
& & \cX_a \ar[r]^{\rho_{g,x_a}} & \cX_b & &\\
\cZ & \ar[l]_{F_a} \cW_a \ar[ru]^{G_a}\ar[rd]_{G'_a} &  &  & \ar[lu]_{F_b} \ar[ld]^{F'_b} \cW_b \ar[r]^{G_b} &\cX\\
& & \cX'_a \ar[r]^{\rho'_{g,x'_a}} & \cX'_b & &
}
$$	
To define the local lifts for the two representations, we fix elements $w_a\in \cW_a,w_b\in \cW_b$, and morphisms $\phi\in \mor(x,F_a(w_a))$, $\psi \in \mor(y,G_b(w_b))$, $\delta \in \mor(G_a(w_a),x_a)$, $\delta' \in \mor(G'_a(w_a), x'_a)$, $\eta \in \mor(\rho_{g,x_a}(g,x_a), F_b(w_b)) $, $\eta' \in \mor(\rho'_{g,x'_a}(g,x'_a), F'_b(w_b)) $, such that we have the following diagram:
$$
\xymatrix{ 
	& & x_a \ar[r] & \rho_{g,x_a}(g,x_a) \ar[d]^{\eta}& &  \\
	x\ar[d]_{\phi} & & G_a(w_a)\ar[u]^{\delta} & F_b(w_b) & & y\ar[d]_{\psi}\\ 
	F_a(w_a) & \ar[l]_{\quad F_a}\ar[ru]^{G_a}\ar[rd]_{G'_a} w_a & & & \ar[lu]_{F_b}\ar[ld]^{F'_b} w_b \ar[r]^{G_b\quad} & G_b(w_b)\\
	& & G'_a(w_a)\ar[d]_{\delta'}  & F_b'(w_b) &  & \\
    & & x'_a \ar[r] & \rho'_{g,x'_a}(g,x'_a)\ar[u]_{\eta'}& &}
$$
For every $h$ close to $g$, we define the local diffeomorphism around $w_a$:
$$\mu_h:=G'^{-1}_{a,w_a}\circ L_{\delta'^{-1}}\circ \rho'_{g,x'_a}(h,\cdot)^{-1} \circ L_{\eta'^{-1}}\circ F'_b\circ F^{-1}_{b,w_b} \circ L_{\eta}\circ \rho_{g,x_a}(h,\cdot)\circ L_{\delta} \circ G_a.$$
By Proposition \ref{prop:uniform}, there is neighborhood $U'\subset U$ of $g$ and regular local uniformizers $\cV,\cW$ of $w_a$, such that for every $h\in U'$, $\mu_h$ is a diffeomorphism from $\cV$ with images in $\cW$. We can assume $\cV$ is a connected local uniformizer around $w_a$. By Lemma \ref{lemma:localnatural}, we can find $\theta_g\in \stab_{w_a}$, such that near $w_a$ we have 
$$[L_{\theta_g}]_{w_a}=[\mu_g]_{w_a}.$$
We claim $\mu_g = L_{\theta_g}$ on $\cV$. To see this, we define
$$S:=\{p\in \cV| [\mu_g]_p = [L_{\theta_g}]_p \}.$$
Then $S$ is open and non-empty by definition. Since $\cV$ is connected, it suffices to prove $S$ is closed. Let $p_i\in S$ such that $\lim_i p_i = p_\infty \in \cV$. Then by Lemma \ref{lemma:localnatural}, $[\mu_g]_{p_\infty} = [L_{\lambda}]_{p_\infty}$ for $\lambda \in \mor(p_\infty,\mu_g(p_\infty))$. By Corollary \ref{coro:natural}, $L_\lambda = L_{\theta'_g}$ on $\cW$ for $\theta'_g\in \stab_{w_a}$.  Since $L_{\theta'_g} = L_{\theta_g}$ on a neighborhood of $p_i$ for $i\gg 0$, the regular property of $\cW$ implies that $\theta'_g = \theta_g$ in $\stab^{\eff}_{w_a}$. Hence $p_\infty$ is also in $S$ and $S$ is closed.

Similarly for every $h\in U'$, we can find $\theta_h \in \stab_{w_a}$, such that 
$$L_{\theta_h}=\mu_h \text{ on } \cV.$$
If there is a sequence of $h_i$ converging to $g$, such that $L_{\theta_{h_i}}\ne L_{\theta_g}$ on $\cV$. Since $|\stab_{w_a}|<\infty$, we can pick a subsequence such that $L _{\theta_{h_i}}=L_{\theta}$ for $\theta \in \stab_{w_a}$. Since $L_{\theta_{h_i}}\ne L_{\theta_g}$, there is an object $p\in \cV$, such that $L_{\theta^{-1}}\circ L_{\theta_g}(p)\ne p$. Then 
\begin{equation}\label{eqn:contr}
\mu_{h_i}(p)=L_{\theta}(p)\ne L_{\theta_g}(p)=\mu_g(p).
\end{equation}
Note that $|\mu_{h_i}(p)| = |\mu_g(p)|$, $\mu_{h_i}(p)$ only has finite possibilities. Then \eqref{eqn:contr} contradicts that $\mu_h$ is continuous in $h$. Therefore there exists a neighborhood $U''\subset U'$ of $g$, such that $\mu_h=L_{\theta_g}$ for $h\in U''$. Therefore the two local lifts $\Gamma$, $\Gamma'$ are related as follows.
\begin{eqnarray*}
\Gamma(h,z) &= & L_{\psi^{-1}}\circ G_b\circ F^{-1}_b\circ L_{\eta}\circ \rho_{g,x_a}(h,\cdot)\circ L_{\delta}\circ G_a\circ F^{-1}_a \circ L_{\phi}(z)\\
&= & L_{\psi^{-1}}\circ G_b\circ F'^{-1}_b \circ L_{\eta'}\circ \rho_{g,x'_a}'(h,\cdot) \circ L_{\delta'}\circ G'_a\circ L_{\theta_g}\circ F_a^{-1} \circ L_{\phi}(z) \\
&= & \Gamma'(h,\cdot)\circ L_{\phi^{-1}\circ F_a^{-1}(\theta_g)\circ \phi}(z).
\end{eqnarray*} 
Therefore, we have the group action by $\stab_x$ is transitive. 	
\end{proof}	

As a corollary of Proposition \ref{prop:lifting}, we have the following.
\begin{corollary}\label{center}
	Let $\cU$ be a connected regular local uniformizer of $x$, such that every element $\alpha$ in $L_{\cX}(g,x,y)$ is represented local lift $\Gamma_\alpha(g,\cdot)$ defined over $\cU$, then
	$$
	\begin{aligned}
	L_{\cX}(g,x,y) & \to \Map_{\sc^\infty}(\cU,\cX)\\
	 \alpha & \mapsto \Gamma_\alpha(g,\cdot)
	\end{aligned}
	$$
	is injective. 
\end{corollary}	
\begin{proof}
	Assume otherwise that two different $\Gamma_\alpha, \Gamma_\beta$ satisfy $\Gamma_\alpha(g,\cdot) = \Gamma_\beta(g,\cdot)$. By Proposition \ref{prop:lifting}, $[\Gamma_\alpha]_{(g,x)} = [\Gamma_\beta\circ L_\phi]_{(g,x)}$ for $\phi \in \stab_x$. Therefore we have$[\Gamma_\alpha(g,\cdot)]_x = [\Gamma_\beta(g,\cdot)]_x = [\Gamma_\beta(g,L_\phi(\cdot))]_x$. Since $\Gamma_\alpha(g,\cdot),\Gamma_\beta(g,\cdot)$ are local sc-diffeomorphisms by Proposition \ref{prop:emb}, we have $L_\phi = \Id$ near $x$. This contradicts $[\Gamma_\alpha]_{(g,x)} \ne [\Gamma_\beta]_{(g,x)}$. 
\end{proof}

\begin{proposition}[Proposition \ref{prop:property}]
 	 	Let $Z$ be a regular polyfold with a $G$-action $(\rho,\mathfrak{P})$ and $(\cX,\bX)$ a polyfold structure of $Z=|\cX|$. Then the local lifts set $L_{\cX}(g,x,y)$ has the following properties.
 	 	\begin{enumerate}
 	 		\item\label{pp2}  There is a well-defined multiplication $\circ: L_{\cX}(g,y,z)\times L_{\cX}(h,x,y)\to L_{\cX}(gh,x,z)$ with the property that if $[\Gamma_1]_{(g,y)}\in L_{\cX}(g,y,z)$ and $[\Gamma_2]_{(h,x)} \in L_{\cX}(h,x,y)$, then there is a local lift $\Gamma_{12}$ at $(gh,x,z)$ such that $[\Gamma_{12}]_{(gh,x)} = [\Gamma_1]_{(g,y)} \circ [\Gamma_2]_{(h,x)} $ and 
 	 		\begin{equation*} \Gamma_{12}(\epsilon gh,u)= \Gamma_1(\epsilon g, \Gamma_2(h,u))=\Gamma_1(g, \Gamma_2(g^{-1}\epsilon gh,u))\end{equation*}
 	 		for $\epsilon$ in a neighborhood of $\Id \in G$ and $u$ in a neighborhood of $x$. The sizes of the neighborhoods depend on the representative local lifts $\Gamma_1,\Gamma_2,\Gamma_{12}$. 
 	 		\item\label{pp3} There is a unique identity element $[\ID_x]_{(\Id,x)}\in L_{\cX}(\Id,x,x)$, such that the identity is both left and right identity in the multiplication structure. Any representative $\ID_x$ has the property that $\ID_x(\Id, u) = u$ for $u$ in a neighborhood of $x$. As before, the neighborhood depends on the choice of the representative \footnote{By Proposition \ref{prop:lifting}, the identity element is characterized by this property.}. 
 	 		\item\label{pp4} There is an inverse map $L_{\cX}(g,x,y)\to L_{\cX}(g^{-1},y,x)$ with respect to the multiplication and identity structures above.
 	 		\item\label{pp5}
 	 		For $x, y\in \cX,g\in G$ with $\rho(g,|x|) = |y|$, there exists an open neighborhood $V\times \cU\times \cO$ of $(g,x,y)$ in $G\times \cX\times \cX$ such that there exist a representative local lift $\Gamma_\alpha$ defined on $(V\times \cU)$ with image in $\cO$, for each element $\alpha \in L_{\cX}(g,x,y)$. Moreover, for any $(g',x')\in V\times \cU$ and $y'\in \cO$ such that $\rho(g',|x'|)=|y'|$, every element $\beta \in L_{\cX}(g',x',y')$ is represented by $[\Gamma_\alpha]_{(g',x')}$ by a unique element $\alpha \in L_{\cX}(g,x,y)$. 
 	 	\end{enumerate}
\end{proposition}

\begin{proof}
	For Property \eqref{pp2}, let $\rho(h,|x|)=|y|$ and $\rho(g,|y|)=|z|$, fix two three local representations of the actions: 
	$$
	\begin{aligned}\rho_a: &U\times (\stab_{x_a} \ltimes \cU_a) & \to \cX_b\\
	\rho_c: &V\times (\stab_{x_c} \ltimes \cU_c) & \to \cX_d\\
	\rho'_a:&W\times (\stab_{x'_a}\ltimes \cU'_a) & \to \cX'_d
	\end{aligned}$$
	for open neighborhoods $U\ni h$, $V\ni g$ and $W\ni gh$, and $\cU_a,\cU_c,\cU'_a$ are local uniformizers in $\cX_a,\cX_c,\cX'_a$ around $x_a,x_c,x'_a$.  We have the following diagram of equivalences:
	\begin{equation}\label{dia:prod}
	\xymatrix{
	& & \cX_a\ar[r]^{\rho_a(h,\cdot)} & \cX_b & \ar[l]_{F_b} \cW_b \ar[r]^{G_b} & \cX & \ar[l]_{F_c} \cW_c \ar[r]^{G_c} & \cX_c \ar[r]^{\rho_c(g,\cdot)} & \cX_d & &\\
	\cX & \ar[l]_{F_a} \ar[ru]^{G_a} \ar[rd]_{G'_a} \cW_a & & & & & & &  & \ar[lu]_{F_d} \ar[ld]^{F'_d} \ar[r]^{G_d} \cW_d & \cX\\
	& & \cX'_a \ar[rrrrrr]^{\rho'_a(gh,\cdot)} & &&&& &\cX'_d &
    }
	\end{equation}
	Let $[\Gamma_1]_{(g,y)} \in L_{\cX}(g,y,z),[\Gamma_2]_{(h,x)}\in L_{\cX}(h,x,y)$. By the same argument in the case three of the proof of Proposition \ref{prop:Alifting}, for $\epsilon$ close to $\Id$ and $u$ close to $x$ we have
	$$\Gamma_1(\epsilon g, \Gamma_2(h,u))=\Gamma_1(g, \Gamma_2(g^{-1}\epsilon g h,u)).$$ 
	Applying the same argument in the case three of the proof of Proposition \ref{prop:Alifting} to the diagram \eqref{dia:prod}, we can find $[\Gamma_{12}]_{(gh,x)} \in L(x,z,gh)$ such that
	$$\Gamma_1(\epsilon g, \Gamma_2(h,u))=\Gamma_1(g, \Gamma_2(g^{-1}\epsilon g h,u))=\Gamma_{12}(\epsilon gh, u).$$
	We define $[\Gamma_1]_{(g,y)}\circ [\Gamma_2]_{(h,x)}$ to be $[\Gamma_{12}]_{(gh,x)}$. To verify that the definition does not depend on the specific digram, note that the map in Corollary \ref{center} transfers the product structure defined by any diagram \eqref{dia:prod} to the composition of maps. By injectivity result in Corollary \ref{center} and that the composition of maps is a fixed structure, the multiplication structure is well-defined.
	
	For Property \eqref{pp3}, note that $\mathfrak{P}(\Id,\cdot)$ is an equivalence of polyfold structure. The local representation $\rho_{\Id,x_a}:U\times (\stab_{x_a}\ltimes \cU_a) \to \cX_b$ has the property that $\rho_{\Id,x_a}(\Id,\cdot)$ is an embedding of local uniformizer and induces the inclusion of $|\cU| \to Z$ on the orbit space. Then by the argument in proof of the quotient map in Theorem \ref{quopoly} or \cite[\S 17.1]{hofer2017polyfold}, we can build a new polyfold structure $\cX'_a$ containing $\cU$ as a component\footnote{Explicitly, $\cX'_a$ can be chosen as $\cU \cup \cX_b$.}, such that there is an equivalence $\rho_0:\cX'_a \to \cX_b$ covers the identity map on $Z$ extending $\rho_{\Id,x_a}(\Id,\cdot)$. Then we have the following diagram:
	$$
	\xymatrix{ \cX & \ar[l]_{F_a} \ar[r]^{G_a} \cW_a & \cX'_a \ar[r]^{\rho_{\Id,x_a}(\Id,\cdot)} & \cX_b & \ar[l]_{\rho_0\circ G_a} \cW_a \ar[r]^{F_a} &\cX \\
	x\ar[d]^\phi &   & x_a \ar[r] &\rho_{\Id,x_a}(\Id,x_a) \ar[d]^\eta  &  & x \ar[d]^{\phi}\\
	F_1(w_a) & \ar[l] w_a \ar[r] &  G_a(w_a)\ar[u]^{\delta} & \rho_0\circ G_a(w_a) & \ar[l] w_a \ar[r] & F_a(w_a). 
	}
	$$
	Since $\rho_{\Id,x_a}(\Id,x_a)=x_a=\rho_0(x_a)$, we can choose $\eta = \rho_0(\delta^{-1})$. Then the local lift from this digram is
	\begin{eqnarray*}
	\Gamma(\Id, z) & =  &  L_{\phi^{-1}} \circ F_a\circ G_a^{-1}\circ \rho_0^{-1} \circ L_{\rho_0(\delta^{-1})}\circ \rho_{\Id,x_a}(\Id,\cdot)\circ L_{\delta} \circ G_a\circ F_a^{-1}\circ L_{\phi}(z) \\
	& \stackrel{\eqref{commute}}{=} & L_{\phi^{-1}} \circ F_a\circ G_a^{-1}\circ L_{\delta^{-1}}\circ \rho_0^{-1} \circ  \rho_{\Id,x_a}(\Id,\cdot) \circ L_\delta\circ G_a\circ F_a^{-1}\circ L_{\phi}(z)\\
	& = & z
	\end{eqnarray*}
	By Corollary \ref{center}, the identity element is unique and is both left and right identity.
	
	As a consequence of Corollary \ref{center}, the inverse of an element is unique and is both left and right inverse. To prove the existence, for every $[\Gamma_1]_{(g,x)}\in L_{\cX}(g,x,y)$, we pick some $[\Gamma_2]_{(g^{-1},y)}\in L_{\cX}(g^{-1},y,x)$. Then $[\Gamma_2]_{(g^{-1},y)}\circ [\Gamma_1]_{(g,x)}\in L_{\cX}(\Id,x,x)$. By Proposition \ref{prop:lifting}, there exists $\phi\in \stab_x$ such that $[L_{\phi}\circ \Gamma_2\circ \Gamma_1]_{(\Id,x)}=[\ID_x]_{(\Id,x)}$, thus $[L_{\phi}\circ \Gamma_2]_{(g^{-1}, y)}$ is an inverse to $[\Gamma_1]_{(g,x)}$. This proves Property \eqref{pp4}.
	
	For Property \eqref{pp5}, pick connected regular local uniformizers $\cU,\cO$ around $x$ and $y$ and a connected open set $V$ of $g$ such that for every $\alpha \in L(x,y,g)$ there is a representative $\Gamma_\alpha$ defined on $V \times \cU $ with image contained in $\cO$. For $x'\in \cU$, $y'\in \cO$ and $g'\in g$ such that $\rho(g',|x'|)=|y'|$, we have $|\Gamma_\alpha(g',x')|=|y'|$ for any $\alpha \in L_{\cX}(g,x,y)$. By Theorem \ref{thm:natural} there exists a $\phi\in \stab_{y}$, such that $L_{\phi}\circ \Gamma_\alpha (g',x')=y'$. By the Proposition \ref{prop:lifting}, $[L_{\phi}\circ \Gamma_\alpha]_{(g,x)}=[\Gamma_\beta]_{(g,x)}$ for some $\beta \in L_{\cX}(g,x,y)$. Since $V,\cU$ are connected, by Corollary \ref{coro:unique} we have $\Gamma_\beta = L_\phi\circ \Gamma_\alpha$ on $V\times \cU$. Therefore $[\Gamma_\beta]_{(g',x')}$ defines an element in $L_{\cX}(g',x',y')$. Then by Proposition \ref{prop:lifting}, every element in $L_{\cX}(g',x',y')$ is $[\Gamma \circ L_{\phi'}]_{(g',x')}$ precomposing with $L_{\phi'}$ for $\phi'\in\stab_{x'}$. By Corollary \ref{coro:natural}, $L_{\phi'}$ is the restriction of $L_{\delta}$ for some $\delta\in \stab_x$. Therefore any element in $L_{\cX}(g',x',y')$ comes from a restriction of $\Gamma_\beta \circ L_\delta$ for some $\delta \in \stab_x$. 
	Again by Corollary \ref{coro:unique}, $\Gamma_\beta \circ L_\delta = \Gamma_\gamma$ on $V\times \cU$ for some $\gamma \in L_{\cX}(g,x,y)$. We claim that the regular property of the polyfold implies that this restriction is injective. If $[\Gamma_\alpha]_{(g',x')} =[\Gamma_\beta]_{(g',x')}$ for $\alpha,\beta \in L_{\cX}(g,x,y)$, then 
	 by Proposition \ref{prop:lifting} and Corollary \ref{coro:unique} there exists $\psi \in \stab_x$ such that $[\Gamma_\alpha]_{(g',x')}=[\Gamma_\beta\circ L_\psi]_{(g',x')} = [\Gamma_\beta]_{(g',x')}$. Therefore $[L_\psi]_{x'} = [\Id]_{x'}$. Then by Definition \ref{def:regular}, $L_\psi = \Id$ on the local uniformizer $\cU$, which means $[\Gamma_\alpha]_{(g,x)} = [\Gamma_\beta]_{(g,x)}$. 
	\end{proof}

\begin{proposition}[Proposition \ref{prop:regular}]
	Let $(\cX,\bX)$ be an effective groupoid. Assume for every $x\in \cX$, there exists a local uniformizer $\cU$ around $x$, such that for any connected uniformizer $\cV \subset \cU$ around $x$, we have $\cV\backslash \cup_{\phi\ne \Id \in \stab_x} \Fix(\phi)$ is connected, where $\Fix(\phi)$ is the fixed set of $L_{\phi}$. Then $(\cX,\bX)$ is regular.
\end{proposition}
\begin{proof}
	Since the ep-groupoid is effective, the first property of regularity is satisfied for any local uniformizer by Theorem \ref{thm:natural}. To proved the second property, we show that $\cU$ is a regular local uniformizer. Consider any connected uniformizer $\cV \subset \cU$ and map $\Phi: \cV\to\stab_x$ such that $L_\Phi$ is sc-smooth, for $\phi\in \stab_x$ we define 
	$$S_\phi:=\{z|z\in \cV \backslash \cup_{\phi\ne \Id} \Fix(\phi), \Phi(z)=\phi \}.$$
	We claim $S_{\phi}$ is open, assume otherwise, there is point $p\in S_\phi$, and $p_i\in S_{\psi}$ converging to $p$ for $\psi\ne \phi$. Then
	$$\lim_{i\to\infty} L_{\Phi(p_i)}(p_i)=\lim_{i\to\infty}L_{\psi}(p_i)=L_{\psi}(p)\ne L_{\phi}(p)=L_{\Phi(p)}(p)$$ 
	contradicts the continuity of $L_{\Phi(z)}(z)$. Since $\cV\backslash \cup_{\phi\ne \Id} \Fix(\phi)$ is connected, hence there is a $\phi\in \stab_x$, such that  $\cV\backslash \cup_{\phi\ne \Id} \Fix(\phi) = S_\phi$. That is   
	$$L_{\Phi(z)}(z)=L_{\phi}(z) \qquad \forall z\in \cV \backslash \cup_{\phi\ne \Id} \Fix(\phi).$$ 
	By the first property of regularity, $\cV\backslash \cup_{\phi\ne \Id} \Fix(\phi)$ is dense $\cV$, thus
		$$L_{\Phi(z)}(z)=L_{\phi}(z) \qquad \forall z\in \cV.$$ 
\end{proof}	

We mention here some other consequences of the regular property of polyfolds.

\begin{proposition}\label{allnatural}
	If $(\cX,\bX)$ is an effective and regular ep-groupoid and we have two equivalences $F,G: (\cW,\bW) \to (\cX,\bX)$, such that $|F|=|G|$. Then $F$ and $G$ are naturally equivalent.
\end{proposition}	

\begin{proof}
	By Lemma \ref{lemma:localnatural}, for any $w\in \cW$, there exists a morphism $\phi_w \in \mor(F(w),G(w))$, such that $L_{\phi_w}= G\circ F^{-1}$ near $F(w)$. Since $\cX$ is reduced, $\phi_w$ is unique. Then the assignment $w\to \phi_w$ is a $\sc^\infty$ natural equivalence between $F$ and $G$, the sc-smoothness follows from the \'etale property.
\end{proof}	

The following theorem was proven in \cite{hofer2017polyfold}, if we assume the ep-groupoids are regular.
\begin{theorem}[{\cite[Theorem 10.7]{hofer2017polyfold}}]\label{thm:unique}
	Let $(\cX,\bX)$ and $(\cY,\bY)$ be two regular ep-groupoids, if $\mathfrak{f}$ and $\mathfrak{g}$ are two generalized isomorphism from $(\cX,\bX)$ to $(\cY,\bY)$ such that $|\mathfrak{f}| = |\mathfrak{g}|$, then $\mathfrak{f} =\mathfrak{g}$.
\end{theorem}

\section{Stabilization on the Fixed Locus}\label{C}
In this section, all sc-Banach spaces are sc-Hilbert spaces. Therefore we have sc-smooth bump functions and partition of unity by \cite[Theorem 7.4]{hofer2017polyfold}. Let $V^\lambda$ be a fixed irreducible representation of $G$.  An M-polyfold subbundle of $\cE$ is an M-polyfold bundle $\cF$ together with a sc-smooth bundle inclusion $\cF \hookrightarrow \cE$. Before proving Proposition \ref{prop:stab}, we first set up some technical propositions that will be used in the proof.

\begin{proposition}\label{prop:complement}
	Let $p:\cE\to \cX$ be a $\lambda$-M-polyfold bundle and $F\subset \cE$ is a finite-dimensional $G$-invariant subbundle with constant rank. Then there is a complement $\lambda$-M-polyfold subbundle $F^\perp\subset \cE$, such that $F^\perp$ is $G$-invariant and $F^\perp\oplus F=\cE$.
\end{proposition}

\begin{proof}
	Pick an open cover $\{\cU_\alpha\}_{\alpha \in A}$ of $\cX$, such that over each $U_\alpha$, $\cE$ has a trivialization modeled on the image of a bundle retraction $R_\alpha: U_\alpha\times \E_\alpha\to U_\alpha\times \E_\alpha$ and $\Phi_\alpha: p^{-1}(\cU_\alpha)\to \Ima R_\alpha$ is the local trivialization, for some sc-Hilbert space $\E_\alpha$. Using a partition of unity $\{f_\alpha\}$ subordinated to $\{\cU_\alpha\}$, then there is a $\sc^\infty$ map $g(\cdot,\cdot):\cE\oplus \cE\to \R$ by $$(x,e_1,e_2)\to \sum f_\alpha(x)\langle \pi_\E\circ \Phi_\alpha (x,e_1), \pi_\E\circ \Phi_\alpha(x,e_2)\rangle_{\E_0}.$$ 
	Then $g$ defines a metric on $\cE$. By averaging over $G$ using Corollary \ref{lemma:avesmooth}, we can assume $g$ is $G$-invariant on $\cE$. Using the metric $g$, we can define a map $\pi_F:\cE \to \cE$ to be the orthogonal projection from $\cE$ to $F$. Since $g$ and $F$ are both $G$-invariant, $\pi_F$ is a $G$-equivariant map.  We claim that $\pi_F$ is sc-smooth.  To see that,  let $x\in \cX$, we can trivialize $F$ locally to find $n$ sc-smooth sections $s_1,\ldots,s_n$ of $F$ near $x$, such that $s_1,\ldots,s_n$ are point-wise linearly independent and span $F$. Since $F\subset \cE$ is a subbundle, in particular, the inclusion $F\subset \cE$ is sc-smooth, $s_1,\ldots,s_n$ can be treated as $sc^\infty$ sections of $\cE$. Using the Gram-Schmidt process, we can assume $s_1,\ldots,s_n$ form an orthonormal basis. Therefore the projection to $F$ near $x$ can be written as 
	$$\pi_F(y,e) = \sum_{i=1}^ng(y,e,s_i(y)) s_i(y). $$
	Therefore $\pi_F$ is sc-smooth. Since $\pi_F$ is a $G$-equivariant projection, $\Id - \pi_F$ is a  $G$-equivariant projection from $\cE$ to $\ker \pi_F$. Composing the bundle chart maps of $\mathcal{E}$ with $\Id-\pi_F$, we can get bundle charts for $\ker \pi_F$. Since $\pi_F$ is $G$-equivariant,  $F^\perp:=\ker \pi_F$ is a $G$-invariant M-polyfold subbundle of $\cE$ and is a complement of $F$. The $G$-invariant subbundle $F^\perp$ of the $\lambda$-M-polyfold bundle $\cE$ is automatically a $\lambda$-M-polyfold bundle
\end{proof}

The following two propositions discuss extensions of sections in M-polyfold bundle. 
\begin{proposition}\label{prop:ext1}
	Given a triangulation on a compact $n$-dimensional manifold $M$ and $W^\lambda$ is a $\lambda$-M-polyfold bundle over $M$. For $(x,w) \in W^\lambda_x$, we use $I_{(x,w)} \subset W^\lambda_x$ to denote the $G$-invariant subspace spanned by $(x,w) \in W^\lambda_x$. Assume over an open neighborhood $U$ of a $k$-simplex $\triangle^k\subset M$, the bundle is modeled on a bundle retraction $R: U\times \F\to U\times\F$ and there are elements $\rho_1,\ldots, \rho_{n+1} \in \F_\infty$, such that $I_{R(u,\rho_1)}, \ldots, I_{R(u,\rho_{n+1})}$ do not have nontrivial intersections with each other for all $u\in U$. If there exists a nowhere vanishing $\sc^\infty$ section $s:V\to W^\lambda$ over an open neighborhood $V$ of $\partial \triangle^k$, then there is a nowhere vanishing section $\tilde{s}$ defined on a neighborhood $O$ of $\triangle^k$, such that $s=\tilde{s}$ on a neighborhood of $\partial \triangle^n$.
\end{proposition}
\begin{proof}
	We will prove the claim for $k=n$ and the remaining cases are similar.
	Assume $\dim V^\lambda = m$, then there exist $g_1,\ldots, g_m \in G$, such that for any $v\ne 0 \in V^\lambda$, $\{g_1 v,\ldots, g_mv\}$ form a basis of $V^\lambda$. Since $I_{R(u,\rho_1)}, \ldots, I_{R(u,\rho_{n+1})}$ do not have nontrivial intersections with each other for all $u\in U$, we have $g_1R(u,\rho_1), \ldots, g_mR(u,\rho_1), \allowbreak g_1R(u,\rho_2)\ldots, g_m R(u,\rho_{n+1})$ are linearly independent for all $u \in U$, i.e. they span a $G$-invariant trivial subbundle $F\subset W^\lambda$ of constant rank. By Proposition \ref{prop:complement}, we have a decomposition over $U$:
	$$W^\lambda=F\oplus F^\perp.$$
	Let $\pi_F$ denote the projection to $F$. Let $O$ denote the open neighborhood $\triangle^n\cup V$ of $\triangle^n$ and $\triangle_2^n \subset \triangle_1^n \subset \triangle^n$ be a two smaller $n$-simplexes with $\triangle^n\backslash V \subset \triangle^n_2$. Using a smooth bump function, we have a section $s_1$ on $O$ such that $s_1=s \ne 0$ on $O\backslash \triangle_2^n$. We can find a smooth $C^0$ small perturbation $p$ on $O$ in the finite dimensional bundle $F$, such that
	\begin{enumerate}
		\item $p = 0$  on $O\backslash \triangle^n_1$;
		\item $s_1+p \ne 0$ on  on $\triangle^n_1\backslash \triangle^n_2$;
		\item $\pi_F\circ s_1+p$ is a transverse section of $F$ near the $\partial \triangle^n_2$.
	\end{enumerate}
	By dimension reason, transversality here implies $\pi_F\circ s_1+p\ne 0$ near $\partial \triangle^n_2$.  Since $\dim F > n$,  $\pi_F\circ s_1+p$ can be extended smoothly to a section of $F$ from $O \backslash \triangle^n_2$ to $O$ and is nowhere zero on $\triangle_2^n$.  Denote the extension  by $s_F$. That is we have $s_F = \pi_F\circ s_1+p$ on $O\backslash \triangle_2^n$ and $s_F$ is nowhere zero on $\triangle_2^n$. Then
	$$\tilde{s}:= s_1 - \pi_F\circ s_1+ s_F$$
	satisfies the conditions. On $O\backslash \triangle_1^n$, $\tilde{s} = s_1=s \ne 0$; On $\triangle_1^n \backslash \triangle_2^n$, $\tilde{s} = s_1+p \ne 0$; On $\triangle_2^n$, $\pi_F\circ \tilde{s} = s_F \ne 0$. Moreover $\tilde{s} = s$ on $V':=O\backslash \triangle_1^n\subset V\cap O$. 
\end{proof}

\begin{proposition}\label{prop:ext2}
	Let $M$ be a compact $n$-dimensional manifold and $W^\lambda \to M$ be an infinite-dimensional $\lambda$-M-polyfold bundle.  Then for any triangulation on $M$, there exists a refined triangulation, such that if we have a trivial  $G$-invariant subbundle $F\subset W_\lambda$ over a neighborhood of an $n$-simplex $\triangle$ , then we can extend $F$ from $\triangle$ to $M$ as a trivial $G$-invariant subbundle of the same rank\footnote{But it may change the bundle outside the simplex.}.
\end{proposition}
\begin{proof}[Proof of Proposition \ref{prop:ext2}]
	Given a triangulation, we can always find a refinement of triangulation such that the conditions in Proposition \ref{prop:ext1} hold because $W^\lambda$ is infinite dimensional. Since $F$ is a trivial $\lambda$-bundle, we can find sections $s_1,\ldots,s_m$ of $F$, such that $F=\oplus_{i=1}^m I_{s_i}$, where $I_{s_i}$ the trivial $G$-bundle generated by $s_1$.  We extend $s_1$ first, pick all the $0$-simplexes in $M$ not contained in $\triangle^n$, we can assign nonzero vectors to them, then by Proposition \ref{prop:ext1}, we can extend $s_1$ to neighborhoods of all the $1$-simplexes not contained in $\triangle_n$. Therefore after applying Proposition \ref{prop:ext1} enough times, $s_1$ can be extended to a nonzero section in $W^\lambda$. By Proposition \ref{prop:complement}, there is a complement to $I_{s_1}$, denoted by $I_{s_1}^\perp$. Then the projection of $s_2,\ldots, s_m$ to $I_{s_1}^\perp$ are linearly independent over $\triangle^n$. Then we can extend the projection of $s_2$ in $I_{s_1}^\perp$ by the same argument. Repeat it $m$ times, we find the extension $F$. 
\end{proof}

Before proving Proposition \ref{prop:Astab}, we first state \cite[Corollary 3.3]{counter}, which will be useful for our proof.

\begin{proposition}\label{prop:coro}
	Let $s:O\to K$ be a sc-Fredholm section for a strong bundle retract $K$ and $x\in O_\infty$. Let $F\subset K$ be a trivial finite dimensional bundle defined over a neighborhood of $x$, which covers $\coker \rD s_x$.\footnote{$\rD s$ makes sense here, because $K\to O$ has a natural splitting.} Then $F$ covers $\coker \rD s$ on a neighborhood of $x$ in $O^2$.  
\end{proposition}

\begin{proposition}[Proposition \ref{prop:stab}]\label{prop:Astab}
	Under the tubular neighborhood assumption (Definition \ref{def:tube}), assume $(s^G)^{-1}(0)$ is a compact manifold. Then there is a $G$-invariant finite-dimensional trivial subbundle $\overline{W}^\lambda\subset W^\lambda_\infty$ over $(s^G)^{-1}(0)$ with constant rank, such that $\overline{W}^\lambda$ covers $\coker \rD^\lambda s$ over $(s^G)^{-1}(0)$.  
\end{proposition}
\begin{proof}[Proof of Proposition \ref{prop:Astab}]
	By Proposition \ref{prop:coro}, for every point $x$ in $M=(s^G)^{-1}(0)$, there exists a neighborhood $U\subset M$\footnote{Since $M$ is compact and Hausdorff for any $Z^G_i$ topology, the $Z^G_i$ topologies on $M$ are equivalent to each other for different $i$.} of $x$ and sections $s_1,\ldots, s_k$ of $W^\lambda$ over $U$, such that $I_{s_1}, \ldots, I_{s_k}$ have trivial intersection with each other and covers the cokernel of $\rD^\lambda s$ over $U$. We can find a finite cover of $M$, such that each open set in the cover has this property. Then we can find a triangulation such that every simplex is contained in some open set in the cover and the triangulation satisfies Proposition \ref{prop:ext2}. We order the top-dimensional simplex by $\triangle_1,\ldots, \triangle_l$.  $N$ denotes of the sum of the numbers of the vector fields over every open set in that open cover or simply take $N\gg0$. Therefore for each top simplex $\triangle_i$, we have a $G$-invariant trivial bundle $F_i$ covers the cokernel of $\rD^\lambda s$ on $\triangle_i$. 
	
	By Proposition \ref{prop:ext2}, we can find trivial subbundle $\widetilde{F_1}$ extending $F_1$ hence covers $\coker \rD^\lambda s$ over $\triangle_1$. Similarly, we could have a extension $\widetilde{F_2}$ of $F_2$, which covers $\coker \rD^\lambda s$ on $\triangle_2$. However, if $\widetilde{F_1}\cap \widetilde{F_2}\ne 0$, the rank of $\widetilde{F_1}+\widetilde{F_2}$ may be varying. To avoid this, we need to modify $F_2$ slightly and then extend it in the infinite dimensional complement $\widetilde{F_1}^\perp$, which is guaranteed by Proposition \ref{prop:complement}. By Proposition \ref{prop:ext2} and infinite dimensionality of $\widetilde{F_1}^\perp$, we can find a trivial bundle $V^\lambda \otimes \R^{2N} \subset \widetilde{F_1}^\perp$ over $M$. Let $\pi_{\R^{2N}}$ denote the projections from $W^\lambda$ to $ V^\lambda \otimes \R^{2N}$ by compositing the projections $W^\lambda\to \widetilde{F_1}^\perp$ and $\widetilde{F_1}^\perp \to V^\lambda \otimes \R^{2N}$. Then for any small $\epsilon > 0$, there exists $G$-equivariant bundle map $\tau:F_2|_{\triangle_2}\to V^\lambda \otimes \R^{2N}|_{\triangle_2}$, such that $|\tau|<\epsilon$ and $\pi_{\R^{2N}}|_{F_2}+\tau$ is injective. For $\epsilon$ small enough, $\Ima(\Id+\tau)|_{F_2}$ still covers $\coker\rD^\lambda s$ over $\triangle_2$ and is $G$-invariant.  Then $F'_2:=\pi_{\widetilde{F_1}^\perp}\Ima((\Id+\tau)|_{F_2})$ is still a trivial $G$-invariant subbundle over $\triangle_2$ of the same rank as $F_2$, $F'_2\cap \widetilde{F_1} = 0$ and $F'_2+\widetilde{F_1}$ covers $\coker \rD^\lambda s$ on $\triangle_2$. We can extend $F'_2$ to $\widetilde{F_2}$ over $M$ in $\widetilde{F_1}^\perp$. Then $\widetilde{F_1}+\widetilde{F_2}$ is a $G$-invariant subbundle and covers $\rD^\lambda s$ on $\triangle_1\cup \triangle_2$. Applying the same argument to $\triangle_3,\ldots, \triangle_l$, we can find the required $\overline{W}^\lambda$.  
\end{proof}

\begin{proposition}[Proposition \ref{prop:tang}]
Under the tubular neighborhood assumption (Definition \ref{def:tube}), we have the following.
\begin{enumerate}
	\item\label{tang:1} $\ind \rD^\lambda s$ is locally constant on $Z^G_\infty$ using $Z^G_i$ topology for any $0\le i \le \infty$.
	\item\label{tang:2} For every $k \in \N$, $\{x\in Z^G_\infty|\dim \coker \rD^\lambda s_x \le k\}$ is an open subset of $Z^G_\infty$ with $Z^G_i$ topology for any $2\le i \le \infty$. 
	\item\label{tang:3} Suppose that $s^G:Z^G\to W^G$ is transverse to $0$. Let $\overline{W}^\lambda$ be a trivial bundle asserted in Proposition \ref{prop:stab}, then $\overline{N}^\lambda : = (\rD^\lambda s)^{-1}(\overline{W}^\lambda)$ is smooth subbundle in $N^\lambda_\infty$ over $(s^G)^{-1}(0)$. Moreover, $\rank \overline{N}^\lambda_x-\rank \overline{W}^\lambda_x=\ind \rD^\lambda s_x$
\end{enumerate}
\end{proposition}
\begin{proof}
	 \eqref{tang:2} follows from Proposition \ref{prop:coro}. For \eqref{tang:3}, first note that by \eqref{tang:2} and compactness of $(s^G)^{-1}(0)$, there are finitely many $\lambda \in \Lambda$ such that $\rD^\lambda s$ is not surjective on $(s^G)^{-1}(0)$. Therefore we can construct a finite rank bundle $\oplus_{\lambda \in \Lambda} \overline{W}^\lambda$  covers the cokernel of $\rD s$ on $(s^{G})^{-1}(0)$. Since $\overline{W}^\lambda$ is trivial and $(s^G)^{-1}(0)$ is a sub-M-polyfold of $Z^G$, by Proposition \ref{prop:pull} generators of $\overline{W}^\lambda$ can be existed to a neighborhood $U\subset Z^G$ of 	$(s^G)^{-1}(0)$. Hence we have an extended trivial $G$-subbundle $\pi: \overline{W}^\lambda \to U$ of $W^\lambda|_U$. Let $\iota$ be the inclusion $\oplus_{\lambda\in \Lambda}\overline{W}^\lambda \to W|_{U}$, then
	 $$\pi^*\iota + s: \pi^*(\oplus_{\lambda \in \Lambda} \overline{W}^\lambda)\to \pi^*(W|_{U})$$
	 defines an equivariant sc-Fredholm section. Therefore $(\pi^*\iota+s)^{-1}(0)$ is a manifold containing $(s^{G})^{-1}(0)$ as a submanifold. Then $\overline{N}^\lambda$ is identified with the $\lambda$-direction of the normal bundle of $(s^G)^{-1}(0)$ in $(\pi^*\iota+s)^{-1}(0)$. The assertion on index is a fact from linear Fredholm theory.
	 
	 For \eqref{tang:1}, let $x\in Z^G_\infty$ and assume there is a neighborhood $V$ of $x$ modeled on a sc-retract $(O,\R^m_+\times \E)$ such that $O=r(U)$ for a connected neighborhood $U\subset \R^m_+\times \E$ of $(0,0)$. Then it suffices to prove $\ind \rD^\lambda s$ is constant on $V$. Let $y \in V_\infty$, then there exists a set of $\sc^+$-perturbations $\gamma_0,\gamma_1,\ldots, \gamma_n$ of $[0,1]\times W^G|_V\to [0,1]\times V$, such that the zero set of $s +\gamma_0+\sum_{i=1}^n r_i\gamma_i: \R^n\times [0,1]\times V\to \R^n\times [0,1] \times W^G$ contain a path from $(0, 0, x)\in \R^n\times [0,1]\times V$ to $(0,1,y)\in \R^n\times[0,1]\times V$ and $s + \gamma_0+ \sum_{i=1}^n r_i\gamma_i$ is transverse along the path. The existence of such perturbations follows from the proof of \cite[Proposition 6.1]{hofer2017polyfold}. Then the argument for \eqref{tang:2} can be applied to show that $\ind \rD^\lambda s_x = \ind \rD^\lambda s_y$.
\end{proof}

\section{Averaging in Polyfolds}
In this section, we prove that averaging argument works for polyfolds. The following Lemma follows from a classical argument e.g. \cite[Theorem 9.42]{rudin1976principles}, where the integration in Banach space is defined using Riemann sum, e.g. see \cite{gordon1991riemann}. 
\begin{lemma}\label{lemma:int}
		Let $M$ be a compact manifold with a volume form. Let $U\subset \R_+^m\times E$ and $V \subset \R_+^n\times F$ be two open subsets for Banach spaces $E,F$.  Assume we have a  map $f:M\times U \to V$, then
		\begin{enumerate}
			\item\label{int:1} if $f$ is continuous, then 
				$$\overline{f}: U\to \R_+^n\times F,\quad x \mapsto \int_M f(p,x) \rd p $$
			is continuous;
			\item\label{int:2} if $f$ is $C^1$, then $\overline{f}$ is $C^1$ and the differential is 
			$$\rD \overline{f}_x: \R^m\times E \to \R^n\times F, \quad v \mapsto   \int_M \rD f_{(p,x)} v \rd p. $$
		\end{enumerate}
\end{lemma}

\begin{proposition}\label{prop:bound}
	Let $\mathscr{L}(E,F)$ be the space of bounded linear maps between Banach spaces $E,F$.  Given a compact manifold $M$ and a map $D:M \to \mathscr{L}(E,F)$ such that $T: M\times E \to F, (p,e) \mapsto (p, D(p)e)$ is continuous.  Then $\sup_{p\in M} ||D(p)|| < \infty$. 
\end{proposition}
\begin{proof}
	For $p \in M$, since $T$ is continuous, for every $\epsilon > 0$, there exits a $\delta > 0$ and a neighborhood $U\subset M$ of $p$ such that $T|_{U\times B_{\delta}(0)} \subset B_{\epsilon}(0)$.  Hence for every  $q \in U$, $||D(q)|| \le \frac{\epsilon}{\delta}$. Since $M$ is compact, we have $\sup_{p\in M} ||D(p)|| < \infty$.
\end{proof}

\begin{lemma}\label{lemma:ave}
	Let $M$ be a compact manifold with a volume form. Let $U\subset \R_+^m\times \E$ and $V \subset \R_+^n\times \F$ be two open subsets. Assume we have a $\sc^1$ map $f:M\times U \to V$, then
	$$\overline{f}: U\to \R_+^n\times \F,\quad x \mapsto \int_M f(p,x)\rd p$$
	is a $\sc^1$ map.
\end{lemma}
\begin{proof}
	By \eqref{int:1} of Lemma \ref{lemma:int}, $\overline{f}$ is a $\sc^0$ map. By \cite[Proposition 1.5]{hofer2017polyfold},  for every $k\ge 1$ the induced map $f: M\times U_k \to \R_+^n\times \F_{k-1}$ is $C^1$ and the partial tangent map $Tf: M\times T U_{k-1}\to TV_{k-1},  (p,(x,u)) \mapsto (f(p,x), \rD f_{(p,x)}(0,u))$ is continuous. Then by Lemma \ref{lemma:int}, $\overline{f}: U_k \to \R^n\times \F_{k-1}$ is $C^1$ and $\overline{Tf} = T\overline{f}: TU \to T(\R_+^n\times \F)$ is $\sc^0$.  By \cite[Proposition 1.5]{hofer2017polyfold}, to show $\overline{f}$ is $\sc^1$ it suffices to show that  for $x \in U_k$, the linear map
	$$v\in \R^m\times \E_{k-1} \mapsto \int_M \rD f_{(p,x)}v \rd p$$
	defines an element in $\mathscr{L}(\R^m\times \E_{k-1}, \R^n\times \F_{k-1})$. This is because $||Df_{(p,x)}||$ is uniformly bounded for $p\in M$ and fixed $x$ by Proposition \ref{prop:bound}.
\end{proof}

By induction, we have the following corollary of Lemma \ref{lemma:ave}.
\begin{corollary}\label{lemma:avesmooth}
		Let $M$ be a compact manifold with a volume form. Let $U\subset \R_+^m\times \E$ and $V \subset \R_+^n\times \F$ be two open subsets. Assume we have a $\sc^k$ map $f:M\times U \to V$ for $0 \le k \le \infty$, then
	$$\overline{f}: U\to \R_+^n\times \F,\quad x \mapsto \int_M f(p,x)\rd p$$
	is a $\sc^k$ map.
\end{corollary}

The following two corollaries are direct consequences of Corollary \ref{lemma:avesmooth}. 
\begin{corollary}\label{coro:equi}
	Let $G$ acts on a polyfold $Z$ by $(\rho,\mathfrak{P})$. Assume we have two $G$-invariant neighborhoods $V\subset U$ and a sc-smooth function $f:Z\to [0,1]$ such that $f|_V  = 1$ and $\supp f \subset U$. Let $\mu$ be a Haar measure on $G$ with $\mu(G) = 1$. Then $\overline{f}:= \int_G \mathfrak{P}(g,\cdot)^*f \rd \mu $ is a sc-smooth function on $Z$ such that $\overline{f}|_V = 1$ and $\supp \overline{f} \subset U$. 
\end{corollary}

The following corollary is used in the proof of Theorem \ref{thm:equitran}.
\begin{corollary}\label{coro:ave}
	Let $\cN \to \cX$ and $\cE \to \cX$ be two $G$-M-polyfold bundles. Let $\mu$ be a Haar measure on $G$ with $\mu(G) = 1$. Assume we have a sc-smooth bundle map $f:\cN \to \cE$ then 
	$$\overline{f}: \cN \to \cE, \quad v \mapsto \int_G g\cdot f(g^{-1} v) \rd g$$
	is a $G$-equivariant sc-smooth bundle map. 
\end{corollary}

\bibliographystyle{plain} 
\bibliography{ref}
\end{document}